\documentclass[11pt,a4paper,leqno]{article}
\usepackage{a4wide}
\usepackage{latexsym}
\usepackage{amsmath}
\usepackage{amssymb}
\usepackage{stmaryrd}
\usepackage{graphicx}

\newtheorem{defin}{Definition}
\newtheorem{lemma}{Lemma} 
\newtheorem{prop}{Proposition}
\newtheorem{theo}{Theorem}
\newtheorem{corol}{Corollary}
\newenvironment{proof}{\medskip\par\noindent{\bf Proof}}{\hfill $\Box$
\medskip\par}
\begin{document}
\title{On parametric Borel summability for linear singularly perturbed 
Cauchy problems with linear fractional transforms}
\author{Alberto Lastra, St\'ephane Malek}
\date{February, 12 2018}
\maketitle
\thispagestyle{empty}
{ \small \begin{center}
{\bf Abstract}
\end{center}
We consider a family of linear singularly perturbed Cauchy problems which combines partial differential operators and linear
fractional transforms. This work is the sequel of a study initiated in \cite{ma}. We construct a collection of
holomorphic solutions on a full covering by sectors of a neighborhood of the origin in $\mathbb{C}$ with respect to
the perturbation parameter $\epsilon$. This set is built up through classical and special Laplace transforms along
piecewise linear paths of functions which possess exponential or super exponential growth/decay on horizontal strips.
A fine structure which entails two levels of Gevrey asymptotics of order 1 and so-called order $1^{+}$ is witnessed. Furthermore,
unicity properties regarding the $1^{+}$ asymptotic layer are observed and follow from results on summability w.r.t a particular
strongly regular sequence recently obtained in \cite{lamasa}.

} \bigskip \bigskip

\section{Introduction}

In this paper, we aim attention at a family of linear singularly perturbed equations that involve linear fractional
transforms and partial derivatives of the form
\begin{equation}
\mathcal{P}(t,z,\epsilon,\{m_{k,t,\epsilon}\}_{k \in I},\partial_{t},\partial_{z})y(t,z,\epsilon)=0  \label{SPEm_intro}
\end{equation}
where $\mathcal{P}(t,z,\epsilon,\{U_{k}\}_{k \in I},V_{1},V_{2})$ is a polynomial in $V_{1},V_{2}$, linear in
$U_{k}$, with holomorphic coefficients relying on $t,z,\epsilon$ in the vicinity of the origin in
$\mathbb{C}^{2}$, where $m_{k,t,\epsilon}$ stands for the Moebius operator acting on the time variable
$m_{k,t,\epsilon}y(t,z,\epsilon) = y(\frac{t}{1 + k \epsilon t},z,\epsilon)$ for $k$ belonging to some finite
subset $I$ of $\mathbb{N}$.

More precisely, we assume that the operator $\mathcal{P}$ can be factorized in the following manner
$\mathcal{P} = \mathcal{P}_{1}\mathcal{P}_{2}$ where $\mathcal{P}_{1}$ and $\mathcal{P}_{2}$ are linear operators with the
specific shapes
\begin{multline*}
\mathcal{P}_{1}(t,z,\epsilon,\{ m_{k,t,\epsilon} \}_{k \in I},\partial_{t},\partial_{z}) = 
P(\epsilon t^{2}\partial_{t})\partial_{z}^{S}
- \sum_{\underline{k} = (k_{0},k_{1},k_{2}) \in \mathcal{A}}
c_{\underline{k}}(z,\epsilon) m_{k_{2},t,\epsilon} (t^{2}\partial_{t})^{k_0} \partial_{z}^{k_1},\\
\mathcal{P}_{2}(t,z,\epsilon,\partial_{t},\partial_{z}) =
P_{\mathcal{B}}(\epsilon t^{2}\partial_{t}) \partial_{z}^{S_{\mathcal{B}}} -
\sum_{\underline{l}=(l_{0},l_{1},l_{2}) \in \mathcal{B}} d_{\underline{l}}(z,\epsilon)
t^{l_0}\partial_{t}^{l_1}\partial_{z}^{l_{2}}.
\end{multline*}
Here, $\mathcal{A}$ and $\mathcal{B}$ are finite subsets of $\mathbb{N}^{3}$ and $S,S_{\mathcal{B}} \geq 1$ are integers
that are submitted to the constraints (\ref{cond_SPCP_first}) and (\ref{SB_underline_l_constraints})
with (\ref{d_l01_defin}). Moreover, $P(X)$ and $P_{\mathcal{B}}(X)$ represent polynomials that are not identically
vanishing with complex coefficients and suffer the property that their roots belong to the open right plane
$\mathbb{C}_{+} = \{ z \in \mathbb{C} / \mathrm{Re}(z) > 0 \}$ and avoid a finite set of suitable unbounded sectors
$S_{d_p} \subset \mathbb{C}_{+}$, $0 \leq p \leq \iota-1$ centered at 0 with bisecting directions $d_{p} \in \mathbb{R}$.
The coefficients $c_{\underline{k}}(z,\epsilon)$ and $d_{\underline{l}}(z,\epsilon)$ for
$\underline{k} \in \mathcal{A}$, $\underline{l} \in \mathcal{B}$ define holomorphic functions on some polydisc centered
at the origin in $\mathbb{C}^{2}$. We consider the equation (\ref{SPEm_intro}) together with a set of initial
Cauchy data
\begin{equation}
(\partial_{z}^{j}y)(t,0,\epsilon) = \begin{cases}
                                    \psi_{j,k}(t,\epsilon) & \text{if } k \in \llbracket -n,n \rrbracket\\
                                    \psi_{j,d_{p}}(t,\epsilon) & \text{if } 0 \leq p \leq \iota-1 
                                    \end{cases} \label{SPEm_ic_1}
\end{equation}
for $0 \leq j \leq S_{\mathcal{B}}-1$ and
\begin{equation}
(\partial_{z}^{h}\mathcal{P}_{2}(t,z,\epsilon,\partial_{t},\partial_{z})y)(t,0,\epsilon) =
                                    \begin{cases}
                                    \varphi_{h,k}(t,\epsilon) & \text{if } k \in \llbracket -n,n \rrbracket\\
                                    \varphi_{h,d_{p}}(t,\epsilon) & \text{if } 0 \leq p \leq \iota-1 \label{SPEm_ic_2} 
                                    \end{cases}
\end{equation}
for $0 \leq h \leq S-1$ and some integer $n \geq 1$. We write $\llbracket -n,n \rrbracket$ for the set of integer numbers $m$ such that $-n\le m\le n$.  For $0 \leq j \leq S_{\mathcal{B}}-1$, $0 \leq h \leq S-1$,
the functions $\psi_{j,k}(t,\epsilon)$ and $\varphi_{h,k}(t,\epsilon)$
(resp. $\psi_{j,d_{p}}(t,\epsilon)$ and $\varphi_{h,d_{p}}(t,\epsilon)$) are holomorphic
on products $\mathcal{T} \times \mathcal{E}_{HJ_n}^{k}$ for $k \in \llbracket -n,n \rrbracket$
(resp. on $\mathcal{T} \times \mathcal{E}_{S_{d_p}}$ for $0 \leq p \leq \iota-1$), where $\mathcal{T}$ is a fixed open
bounded sector centered at 0 with bisecting direction $d=0$ and
$\underline{\mathcal{E}} = \{ \mathcal{E}_{HJ_n}^{k} \}_{k \in \llbracket -n,n \rrbracket}
\cup \{ \mathcal{E}_{S_{d_p}} \}_{0 \leq p \leq \iota-1}$ represents a collection of open bounded sectors centered
at 0 whose union form a covering of $\mathcal{U} \setminus \{ 0 \}$, where $\mathcal{U}$ stands for some
neighborhood of 0 in $\mathbb{C}$ (the complete list of constraints attached to $\underline{\mathcal{E}}$ is provided
at the beginning of Subsection 3.3).

This work is a continuation of a study harvested in the paper \cite{ma} dealing with small step size difference-differential
Cauchy problems of the form
\begin{equation}
\epsilon \partial_{s} \partial_{z}^{S}X_{i}(s,z,\epsilon) =
\mathcal{Q}(s,z,\epsilon,\{ T_{k,\epsilon} \}_{k \in J}, \partial_{s},\partial_{z})X_{i}(s,z,\epsilon)
+ P(z,\epsilon,X_{i}(s,z,\epsilon)) \label{SPE_shift_intro}
\end{equation}
for given initial Cauchy conditions $(\partial_{z}^{j}X_{i})(s,0,\epsilon) = x_{j,i}(s,\epsilon)$, for
$0 \leq i \leq \nu - 1$, $0 \leq j \leq S-1$, where $\nu,S \geq 2$ are integers, $\mathcal{Q}$ is some differential
operator which is polynomial in time $s$, holomorphic near the origin in $z,\epsilon$, that includes shift operators
acting on time, $T_{k,\epsilon}X_{i}(s,z,\epsilon) = X_{i}(s + k \epsilon,z,\epsilon)$ for $k \in J$ that represents
a finite subset of $\mathbb{N}$ and $P$ is some polynomial. Indeed, by performing the change of variable $t=1/s$,
the equation (\ref{SPEm_intro}) maps into a singularly perturbed linear PDE combined with small shifts $T_{k,\epsilon}$,
$k \in I$. The initial data $x_{j,i}(s,\epsilon)$ were supposed to define holomorphic functions on products
$(\mathcal{S} \cap \{ |s| > h \} ) \times \mathcal{E}_{i} \subset \mathbb{C}^{2}$ for some $h > 0$ large enough,
where $\mathcal{S}$ is a fixed open unbounded sector centered at 0 and
$\overline{\mathcal{E}} = \{ \mathcal{E}_{i} \}_{0 \leq i \leq \nu-1}$ forms a set of sectors which covers the vicinity
of the origin. Under appropriate restrictions regarding the shape of (\ref{SPE_shift_intro}) and the inputs
$x_{j,i}(s,\epsilon)$, we have built up bounded actual holomorphic solutions written as Laplace transforms
$$ X_{i}(s,z,\epsilon) = \int_{L_{e_i}} V_{i}(\tau,z,\epsilon) \exp( - \frac{s \tau}{\epsilon} ) d\tau $$
along halflines $L_{e_i} = \mathbb{R}_{+} e^{\sqrt{-1}e_{i}}$ contained in $\mathbb{C}_{+} \cup \{ 0 \}$ and, following
an approach by G. Immink (see \cite{im1}), written as truncated Laplace transforms
$$ X_{i}(s,z,\epsilon) = \int_{0}^{\Gamma_{i}\log(\Omega_{i}s/\epsilon)} V_{i}(\tau,z,\epsilon)
\exp( - \frac{s \tau}{\epsilon} ) d\tau $$
provided that $\Gamma_{i} \in \mathbb{C}_{-} = \{ z \in \mathbb{C}/ \mathrm{Re}(z) < 0 \}$, for well chosen
$\Omega_{i} \in \mathbb{C}^{\ast}$. In general, these truncated Laplace transforms do not fulfill the equation
(\ref{SPE_shift_intro}) but they are constructed in a way that all differences $X_{i+1} - X_{i}$ define flat
functions w.r.t $s$ on the intersections $\mathcal{E}_{i+1} \cap \mathcal{E}_{i}$. We have shown the existence of a
formal power series $\hat{X}(s,z,\epsilon) = \sum_{l \geq 0} h_{l}(s,z) \epsilon^{l}$ with coefficients
$h_{l}$ determining bounded holomorphic functions on $(\mathcal{S} \cap \{ |s| > h \}) \times D(0,\delta)$ for some
$\delta>0$, which solves (\ref{SPE_shift_intro}) and represents the $1-$Gevrey asymptotic expansion of each
$X_{i}$ w.r.t $\epsilon$ on $\mathcal{E}_{i}$, $0 \leq i \leq \nu-1$ (see Definition 7). Besides a precised
hierarchy that involves actually two levels of asymptotics has been uncovered. Namely, each function $X_{i}$
can be split into a sum of a convergent series, a piece $X_{i}^{1}$ which possesses an asymptotic expansion
of Gevrey order 1 w.r.t $\epsilon$ and a part $X_{i}^{2}$ whose asymptotic expansion is of Gevrey order $1^{+}$ as
$\epsilon$ tends to 0 on $\mathcal{E}_{i}$ (see Definition 8). However two major drawbacks of this result may be
pointed out. Namely, some part of the family $\{ X_{i} \}_{0 \leq i \leq \nu-1}$ do not define solutions of
(\ref{SPE_shift_intro}) and no unicity information were obtained concerning the $1^{+}-$Gevrey asymptotic expansion
(related to so-called $1^{+}-$summability as defined in \cite{im1}, \cite{im2}, \cite{im3}).

In this work, our objective is similar to the former one in \cite{ma}. Namely, we plan to construct actual holomorphic
solutions $y_{k}(t,z,\epsilon)$, $k \in \llbracket -n,n \rrbracket$ (resp. $y_{d_p}(t,z,\epsilon)$,
$0 \leq p \leq \iota-1$ ) to the problem (\ref{SPEm_intro}), (\ref{SPEm_ic_1}), (\ref{SPEm_ic_2}) on domains
$\mathcal{T} \times D(0,\delta) \times \mathcal{E}_{HJ_n}^{k}$
(resp. $\mathcal{T} \times D(0,\delta) \times \mathcal{E}_{S_{d_p}}$) for some small radius $\delta>0$ and to analyze
the nature of their asymptotic expansions as $\epsilon$ approaches 0. The main novelty is that we can now build solutions
to (\ref{SPEm_intro}), (\ref{SPEm_ic_1}), (\ref{SPEm_ic_2}) on a full covering
$\underline{\mathcal{E}}$ of a neighborhood of 0 w.r.t $\epsilon$. Besides, a structure with two levels of
Gevrey $1$ and $1^{+}$ asymptotics is also observed and unicity information leading to $1^{+}-$summability
is achieved according to a refined version of the Ramis-Sibuya Theorem obtained in \cite{ma} and to the recent progress
on so-called summability for a strongly regular sequence obtained by the authors and J. Sanz in \cite{lamasa} and \cite{sa}.

The manufacturing of the solutions $y_{k}$ and $y_{d_p}$ is divided in two main parts and can be outlined as follows.

We first set the problem
\begin{equation}
\mathcal{P}_{1}(t,z,\epsilon,\{ m_{k,t,\epsilon} \}_{k \in I},\partial_{t},\partial_{z})u(t,z,\epsilon) = 0
\label{SPEm_1_intro}
\end{equation}
for the given Cauchy inputs
\begin{equation}
(\partial_{z}^{h}u)(t,0,\epsilon) =
                                    \begin{cases}
                                    \varphi_{h,k}(t,\epsilon) & \text{if } k \in \llbracket -n,n \rrbracket\\
                                    \varphi_{h,d_{p}}(t,\epsilon) & \text{if } 0 \leq p \leq \iota-1 \label{SPEm_1_ic} 
                                    \end{cases}
\end{equation}
for $0 \leq h \leq S-1$. Under the restriction (\ref{cond_SPCP_first}) and suitable control on the initial data
(displayed through (\ref{xi_larger_sigma}), (\ref{norm_w_initial_small}) and (\ref{norm_Sdp_initial_wj})), one can
build a first collection of actual solutions to (\ref{SPEm_1_intro}), (\ref{SPEm_1_ic}) as special Laplace transforms
$$ u_{k}(t,z,\epsilon) = \int_{P_k} w_{HJ_n}(u,z,\epsilon) \exp( - \frac{u}{\epsilon t} ) \frac{du}{u} $$
which are bounded holomorphic on $\mathcal{T} \times D(0,\delta) \times \mathcal{E}_{HJ_n}^{k}$, where
$w_{HJ_n}$ defines a holomorphic function on a domain
$HJ_{n} \times D(0,\delta) \times D(0,\epsilon_{0}) \setminus \{ 0 \}$ for some radii $\delta,\epsilon_{0}>0$ and 
$HJ_{n}$ represents the union of two sets of consecutively overlapping horizontal strips
$$ H_{k} = \{ z \in \mathbb{C} / a_{k} \leq \mathrm{Im}(z) \leq b_{k}, \ \ \mathrm{Re}(z) \leq 0 \} \ \ , \ \ 
J_{k} = \{ z \in \mathbb{C} / c_{k} \leq \mathrm{Im}(z) \leq d_{k}, \ \ \mathrm{Re}(z) \leq 0 \}$$
as described at the beginning of Subsection 3.1 and $P_{k}$ is the union of a segment joining 0 and some
well chosen point $A_{k} \in H_{k}$ and the horizontal halfline $\{ A_{k} - s/ s \geq 0 \}$, for
$k \in \llbracket -n,n \rrbracket$. Moreover, $w_{HJ_n}(\tau,z,\epsilon)$ has (at most) super exponential decay w.r.t $\tau$
on $H_{k}$ (see (\ref{bds_WHJn_Hk})) and (at most) super exponential growth w.r.t $\tau$ along
$J_{k}$ (see (\ref{bds_WHJn_Jk})), uniformly in $z \in D(0,\delta)$, provided that
$\epsilon \in D(0,\epsilon_{0}) \setminus \{ 0 \}$ (Theorem 1).

The idea of considering function spaces sharing both super exponential growth and decay on strips and Laplace
transforms along piecewise linear paths departs from the next example worked out by B. Braaksma, B. Faber and
G. Immink in \cite{brfaim} (see also \cite{fab}),
\begin{equation}
h(s+1) - as^{-1}h(s) = s^{-1} \label{DE_BFI}
\end{equation}
for a real number $a>0$, for which solutions are given as special Laplace transforms
$$ h_{n}(s) = \int_{C_{n}} e^{-s \tau} e^{\tau - a} e^{ae^{\tau}} d\tau $$
for each $n \in \mathbb{Z}$, where $C_{n}$ is a path connecting 0 and $+\infty + i\theta$ for some
$\theta \in ( \frac{\pi}{2} + 2n\pi, \frac{3\pi}{2} + 2n\pi)$ built up with the help of a segment and a horizontal
halfline as above for the path $P_{k}$. The function $\tau \mapsto e^{\tau-a} e^{ae^{\tau}}$ has super exponential decay
(resp. growth) on a set of strips $-H_{k}$ (resp. $-J_{k}$) as explained in the example after Definition 3.
Furthermore, the functions $h_{n}(s)$ possess an asymptotic expansion of Gevrey order 1,
$\hat{h}(s) = \sum_{l \geq 1} h_{l} s^{-l}$ that formally solves (\ref{DE_BFI}), as $s \rightarrow \infty$ on
$\mathbb{C}_{+}$.

On the other hand, a second set of solutions to (\ref{SPEm_1_intro}), (\ref{SPEm_1_ic}) can be found as usual
Laplace transforms
$$ u_{d_p}(t,z,\epsilon) = \int_{L_{\gamma_{d_p}}} w_{d_p}(u,z,\epsilon) \exp( -\frac{u}{\epsilon t} ) \frac{du}{u} $$
along halflines $L_{\gamma_{d_p}}=\mathbb{R}_{+}e^{\sqrt{-1}\gamma_{d_p}} \subset S_{d_p} \cup \{ 0 \}$, that define
bounded holomorphic functions on $\mathcal{T} \times D(0,\delta) \times \mathcal{E}_{S_{d_p}}$, where
$w_{d_p}(\tau,z,\epsilon)$ represents a holomorphic function on $(S_{d_p} \cup D(0,r)) \times D(0,\delta) \times
D(0,\epsilon_{0}) \setminus \{ 0 \}$ with (at most) exponential growth w.r.t $\tau$ on $S_{d_p}$, uniformly
in $z \in D(0,\delta)$, whenever $\epsilon \in D(0,\epsilon_{0}) \setminus \{ 0 \}$, $0 \leq p \leq \iota-1$ (Theorem 1).

In a second stage, we focus on both problems
\begin{equation}
\mathcal{P}_{2}(t,z,\epsilon,\partial_{t},\partial_{z})y(t,z,\epsilon) = u_{k}(t,z,\epsilon) \label{SPE_2_1}
\end{equation}
with Cauchy data
\begin{equation}
(\partial_{z}^{j}y)(t,0,\epsilon) = \psi_{j,k}(t,\epsilon) \label{SPE_2_1_i_c}
\end{equation}
for $0 \leq j \leq S_{\mathcal{B}}-1$, $k \in \llbracket -n,n \rrbracket$ and
\begin{equation}
\mathcal{P}_{2}(t,z,\epsilon,\partial_{t},\partial_{z})y(t,z,\epsilon) = u_{d_p}(t,z,\epsilon) \label{SPE_2_2}
\end{equation}
under the conditions
\begin{equation}
(\partial_{z}^{j}y)(t,0,\epsilon) = \psi_{j,d_{p}}(t,\epsilon) \label{SPE_2_2_i_c}
\end{equation}
for $0 \leq j \leq S_{\mathcal{B}}-1$, $0 \leq p \leq \iota-1$. We first observe that the coupling of the problems
(\ref{SPEm_1_intro}), (\ref{SPEm_1_ic}) together with (\ref{SPE_2_1}), (\ref{SPE_2_1_i_c}) and
(\ref{SPE_2_2}), (\ref{SPE_2_2_i_c}) is equivalent to our initial question of searching for solutions to
(\ref{SPEm_intro}) under the requirements (\ref{SPEm_ic_1}), (\ref{SPEm_ic_2}).

The approach which consists to consider equations presented in factorized form follows from a series of works by the same authors
\cite{lama1}, \cite{lama2}, \cite{lama3}. In our situation, the operator $\mathcal{P}_{1}$ cannot contain arbitrary
polynomials in $t$ neither general derivatives $\partial_{t}^{l_1}$, $l_{1} \geq 1$, since $w_{HJ_n}(\tau,z,\epsilon)$
would solve some equation of the form (\ref{1_aux_CP}) with exponential coefficients which would also contain convolution operators like those appearing in equation (\ref{ACP_forcterm_w}). But the spaces of functions with super exponential decay
are not stable under the action of these integral transforms. Those specific Banach spaces are however crucial to get
bounded (or at least with exponential growth) solutions $w_{HJ_n}(\tau,z,\epsilon)$ to (\ref{1_aux_CP})
leading to the existence of the special Laplace transforms $u_{k}(t,z,\epsilon)$ along the paths $P_{k}$. In order
to deal with more general sets of equations, we compose $\mathcal{P}_{1}$ with suitable differential operators
$\mathcal{P}_{2}$ which do not enmesh Moebius transforms. In this work, we have decided to focus only on linear problems.
We postpone the study of nonlinear equations for future investigation.

Taking for granted that the constraints (\ref{SB_underline_l_constraints}) and (\ref{d_l01_defin}) are observed, under
adequate handling on the Cauchy inputs (\ref{SPE_2_1_i_c}), (\ref{SPE_2_2_i_c}) (detailed in
(\ref{normRHJ_vj_Ivj}), (\ref{normSdp_vj_Ivj})), one can exhibit a foremost set of actual solutions to
(\ref{SPE_2_1}), (\ref{SPE_2_1_i_c}) as special Laplace transforms
$$ y_{k}(t,z,\epsilon) = \int_{P_{k}} v_{HJ_n}(u,z,\epsilon) \exp( -\frac{u}{\epsilon t} ) \frac{du}{u} $$
that define bounded holomorphic functions on $\mathcal{T}\times D(0,\delta) \times \mathcal{E}_{HJ_n}^{k}$
where $v_{HJ_n}(\tau,z,\epsilon)$ represents a holomorphic function on $HJ_{n} \times D(0,\delta) \times
D(0,\epsilon_{0}) \setminus \{ 0 \}$ with (at most) exponential growth w.r.t $\tau$ along $H_{k}$
(see (\ref{bds_vHJn_Hk})) and withstanding (at most) super exponential growth w.r.t $\tau$ within $J_{k}$
(see (\ref{bds_vHJn_Jk})), uniformly in $z \in D(0,\delta)$ when $\epsilon \in D(0,\epsilon_{0}) \setminus \{ 0 \}$,
$k \in \llbracket -n,n \rrbracket$ (Theorem 2).

Furthermore, a second group of solutions to (\ref{SPE_2_2}), (\ref{SPE_2_2_i_c}) is achieved through usual Laplace
transforms
$$ y_{d_p}(t,z,\epsilon) = \int_{L_{\gamma_{d_p}}} v_{d_p}(u,z,\epsilon) \exp( -\frac{u}{\epsilon t} ) \frac{du}{u} $$
defining holomorphic bounded functions on $\mathcal{T} \times D(0,\delta) \times \mathcal{E}_{S_{d_p}}$, where
$v_{d_p}(\tau,z,\epsilon)$ stands for a holomorphic function on $(S_{d_p} \cup D(0,r)) \times D(0,\delta)
\times D(0,\epsilon_{0}) \setminus \{ 0 \}$ with (at most) exponential growth w.r.t $\tau$ on
$S_{d_p}$, uniformly in $z \in D(0,\delta)$, for all $\epsilon \in D(0,\epsilon_{0}) \setminus \{ 0 \}$,
$0 \leq p \leq \iota-1$ (Theorem 2).

As a result, the merged family $\{ y_{k} \}_{k \in \llbracket -n,n \rrbracket}$ and
$\{ y_{d_p} \}_{0 \leq p \leq \iota - 1}$ defines a set of solutions on a full covering
$\underline{\mathcal{E}}$ of some neighborhood of 0 w.r.t $\epsilon$. It remains to describe the structure of their
asymptotic expansions as $\epsilon$ tend to 0. As in our previous work, we see that a double layer of Gevrey
asymptotics arise. Namely, each function $y_{k}(t,z,\epsilon)$, $k \in \llbracket -n,n \rrbracket$
(resp. $y_{d_p}(t,z,\epsilon)$, $0 \leq p \leq \iota-1$) can be decomposed as a sum of a convergent power series
in $\epsilon$, a piece $y_{k}^{1}(t,z,\epsilon)$ (resp. $y_{d_p}^{1}(t,z,\epsilon))$
that possesses an asymptotic expansion $\hat{y}^{1}(t,z,\epsilon) =
\sum_{l \geq 0} y_{l}^{1}(t,z) \epsilon^{l}$ of Gevrey order 1 w.r.t $\epsilon$ on $\mathcal{E}_{HJ_n}^{k}$
(resp. on $\mathcal{E}_{S_{d_p}}$) and a last tail $y_{k}^{2}(t,z,\epsilon)$ (resp. $y_{d_p}^{2}(t,z,\epsilon)$)
whose asymptotic expansion
$\hat{y}^{2}(t,z,\epsilon) = \sum_{l \geq 0} y_{l}^{2}(t,z) \epsilon^{l}$ is of Gevrey order $1^{+}$ as
$\epsilon$ becomes close to 0 on $\mathcal{E}_{HJ_n}^{k}$ (resp. on $\mathcal{E}_{S_{d_p}}$).
Furthermore, the functions $y_{\pm n}^{2}(t,z,\epsilon)$ and $y_{d_p}^{2}(t,z,\epsilon)$ are the restrictions of a
common holomorphic function $y^{2}(t,z,\epsilon)$ on $\mathcal{T} \times D(0,\delta)
\times ( \mathcal{E}_{HJ_n}^{-n} \cup \mathcal{E}_{HJ_n}^{n} \cup_{p=0}^{\iota-1} \mathcal{E}_{S_{d_p}} )$
which is the unique asymptotic expansion of $\hat{y}^{2}(t,z,\epsilon)$ of order $1^{+}$ called
$1^{+}-$sum in this work that can be reconstructed through an analog of a Borel/Laplace transform in the
framework of $\mathbb{M}-$summability for the strongly regular sequence
$\mathbb{M} = (M_{n})_{n \geq 0}$ with $M_{n} = (n/\mathrm{Log}(n+2))^{n}$ (Definition 8). On the other hand, the functions $y_{d_p}^{1}(t,z,\epsilon)$ represent $1-$sums of $\hat{y}^{1}$ w.r.t $\epsilon$ on $\mathcal{E}_{S_{d_p}}$ whenvener its aperture is strictly larger than $\pi$ in the classical sense as defined in reference books such as \cite{ba1}, \cite{ba2} or \cite{co}
(Theorem 3). These informations regarding Gevrey asymptotics complemented by unicity features is achieved through
a refinement of a version of the Ramis-Sibuya theorem obtained in \cite{ma} (Proposition 23) and the flatness properties
(\ref{log_flat_difference_yk_plus_1_minus_yk_HJn}), (\ref{exp_flat_difference_yk_plus_1_minus_yk_Sdp}),
(\ref{difference_y_HJn_Sd0}) and (\ref{difference_y_HJn_Sdiota}) for the differences of neighboring functions among the
two families $\{ y_{k} \}_{k \in \llbracket -n,n \rrbracket}$ and $\{ y_{d_p} \}_{0 \leq p \leq \iota-1}$.\bigskip

\noindent The paper is organized as follows.\medskip

In Section 2, we consider a first ancillary Cauchy problem with exponentially growing coefficients. We construct holomorphic
solutions belonging to the Banach space of functions with super exponential growth (resp. decay) on horizontal strips and
exponential growth on unbounded sectors. These Banach spaces and their properties under the action of linear continuous maps are
described in Subsections 2.1 and 2.2.

In Section 3, we provide solutions to the problem (\ref{SPEm_1_intro}), (\ref{SPEm_1_ic}) with the help of the problem solved in
Section 2. Namely, in Section 3.1, we construct the solutions $u_{k}(t,z,\epsilon)$ as special Laplace transforms, along piecewise
linear paths, on the sectors $\mathcal{E}_{HJ_n}^{k}$ w.r.t $\epsilon$, $k \in \llbracket -n,n \rrbracket$. In Section 3.2,
we build up the solutions $u_{d_p}(t,z,\epsilon)$ as usual Laplace transforms along halflines provided that $\epsilon$ belongs
to the sectors $\mathcal{E}_{S_{d_p}}$, $0 \leq p \leq \iota-1$. In Section 3.3, we combine both families
$\{ u_{k} \}_{k \in \llbracket -n,n \rrbracket}$ and $\{ u_{d_p} \}_{0 \leq p \leq \iota-1}$ in order to get a set of
solutions on a full covering $\underline{\mathcal{E}}$ of the origin in $\mathbb{C}^{\ast}$ and we provide bounds for the
differences of consecutive solutions (Theorem 1).

In Section 4, we concentrate on a second auxiliary convolution Cauchy problem with polynomial coefficients and forcing term that solves
the problem stated in Section 2. We establish the existence of holomorphic solutions which are part of the Banach spaces of
functions with super exponential (resp. exponential) growth on $L-$shaped domains and exponential growth on unbounded sectors.
A description of these Banach spaces and the action of integral operators on them are provided in Subsections 4.1, 4.2 and 4.3.

In Section 5, we present solutions for the problems (\ref{SPE_2_1}), (\ref{SPE_2_1_i_c}) and (\ref{SPE_2_2}), (\ref{SPE_2_2_i_c}) 
displayed as special and usual Laplace transforms forming a collection of functions on a full covering
$\underline{\mathcal{E}}$ of the origin in $\mathbb{C}^{\ast}$ (Theorem 2).

In Section 6, the structure of the asymptotic expansions of the solutions $u_{k}$, $y_{k}$ and
$u_{d_p}$, $y_{d_p}$ w.r.t $\epsilon$ (stated in Theorem 3) is described with the help of a version of the Ramis-Sibuya Theorem which entails
two Gevrey levels 1 and $1^{+}$ disclosed in Subsection 6.1. 

\section{A first auxiliary Cauchy problem with exponential coefficients}
\subsection{Banach spaces of holomorphic functions with super-exponential decay on horizontal strips}

Let $\bar{D}(0,r)$ be the closed disc centered at 0 and with radius $r>0$ and let
$\dot{D}(0,\epsilon_{0}) = D(0,\epsilon_{0}) \setminus \{ 0 \}$ be the punctured disc centered at 0 with radius
$\epsilon_{0}>0$ in $\mathbb{C}$. We consider a
closed horizontal strip $H$ described as
\begin{equation}
H = \{ z \in \mathbb{C} / a \leq \mathrm{Im}(z) \leq b, \ \ \mathrm{Re}(z) \leq 0 \} \label{defin_strip_H}
\end{equation}
for some real numbers $a<b$. For any open set $\mathcal{D} \subset \mathbb{C}$, we denote
$\mathcal{O}(\mathcal{D})$ the vector space of holomorphic functions on $\mathcal{D}$. Let
$b>1$ be a real number, we define $\zeta(b) = \sum_{n=0}^{+\infty} 1/(n+1)^{b}$. Let $M$ be a
positive real number such that $M > \zeta(b)$. We introduce the sequences $r_{b}(\beta) = 
\sum_{n = 0}^{\beta} \frac{1}{(n+1)^{b}}$ and $s_{b}(\beta) = M - r_{b}(\beta)$
for all $\beta \geq 0$.

\begin{defin} Let $\underline{\sigma} = (\sigma_{1},\sigma_{2},\sigma_{3})$ where
$\sigma_{1},\sigma_{2},\sigma_{3}>0$ be positive real numbers and $\beta \geq 0$ an
integer. Let $\epsilon \in \dot{D}(0,\epsilon_{0})$. We denote
$SED_{(\beta,\underline{\sigma},H,\epsilon)}$ the vector space of holomorphic functions $v(\tau)$ on
$\mathring{H}$ (which stands for the interior of $H$) and continuous on $H$
such that
$$ ||v(\tau)||_{(\beta,\underline{\sigma},H,\epsilon)} =  \sup_{\tau \in H} \frac{|v(\tau)|}{|\tau|}
\exp \left(-\frac{\sigma_{1}}{|\epsilon|} r_{b}(\beta)|\tau| + \sigma_{2} s_{b}(\beta) \exp( \sigma_{3}|\tau| ) \right) $$
is finite. Let $\delta>0$ be a real number. We define
$SED_{(\underline{\sigma},H,\epsilon,\delta)}$
to be the vector space of all formal series $v(\tau,z) = \sum_{\beta \geq 0} v_{\beta}(\tau) z^{\beta}/\beta!$
with coefficients $v_{\beta}(\tau) \in SED_{(\beta,\underline{\sigma},H,\epsilon)}$, for
$\beta \geq 0$ and such that
$$ ||v(\tau,z)||_{(\underline{\sigma},H,\epsilon,\delta)} = \sum_{\beta \geq 0}
||v_{\beta}(\tau)||_{(\beta,\underline{\sigma},H,\epsilon)} \frac{\delta^{\beta}}{\beta !}$$
is finite. One can ascertain that
$SED_{(\underline{\sigma},H,\epsilon,\delta)}$ equipped with the norm
$||.||_{(\underline{\sigma},H,\epsilon,\delta)}$ turns out to be a Banach space.
\end{defin}
In the next proposition, we show that the formal series belonging to the latter Banach spaces define
actual holomorphic functions that are convergent on a disc w.r.t $z$ and with super exponential decay on
the strip $H$ w.r.t $\tau$.

\begin{prop} Let $v(\tau,z) \in SED_{(\underline{\sigma},H,\epsilon,\delta)}$. Let
$0 < \delta_{1} < 1$. Then, there
exists a constant $C_{0}>0$ (depending on $||v||_{(\underline{\sigma},H,\epsilon,\delta)}$ and
$\delta_{1}$) such that
\begin{equation}
|v(\tau,z)| \leq C_{0} |\tau| \exp \left(\frac{\sigma_{1}}{|\epsilon|} \zeta(b) |\tau|
- \sigma_{2}(M-\zeta(b)) \exp( \sigma_{3} |\tau| ) \right) \label{v_up_bds}
\end{equation}
for all $\tau \in H$, all $z \in \mathbb{C}$ with $\frac{|z|}{\delta} < \delta_{1}$.
\end{prop}
\begin{proof}
Let $v(\tau,z) = \sum_{\beta \geq 0} v_{\beta}(\tau) z^{\beta}/\beta! \in
SED_{(\underline{\sigma},H,\epsilon,\delta)}$.
By construction, there exists a constant $c_{0}>0$ (depending on
$||v||_{(\underline{\sigma},H,\epsilon,\delta)}$) with
\begin{equation}
|v_{\beta}(\tau)| \leq c_{0} |\tau| \exp( \frac{\sigma_{1}}{|\epsilon|}r_{b}(\beta)|\tau| -
\sigma_{2} s_{b}(\beta) \exp( \sigma_{3} |\tau| ) ) \beta!
(\frac{1}{\delta})^{\beta}
\end{equation}
for all $\beta \geq 0$, all $\tau \in H$. Take $0 < \delta_{1} < 1$. Departing from the
definition of $\zeta(b)$, we deduce that
\begin{multline}
|v(\tau,z)| \leq c_{0} |\tau| \sum_{\beta \geq 0}
\exp(\frac{\sigma_{1}}{|\epsilon|}r_{b}(\beta)|\tau| - \sigma_{2} s_{b}(\beta) \exp( \sigma_{3} |\tau|) )
(\delta_{1})^{\beta}\\
\leq c_{0}|\tau| \exp( \frac{\sigma_{1}}{|\epsilon|}\zeta(b)|\tau| - \sigma_{2}(M - \zeta(b))
\exp( \sigma_{3} |\tau|) ) \frac{1}{1 - \delta_{1}}
\label{v_up_bds_in_proof}
\end{multline}
for all $z \in \mathbb{C}$ such that $\frac{|z|}{\delta} < \delta_{1} < 1$, all $\tau \in H$. Therefore
(\ref{v_up_bds}) is a consequence of (\ref{v_up_bds_in_proof}).
\end{proof}

In the next three propositions, we study the action of linear operators constructed as multiplication
by exponential and polynomial functions and by bounded holomorphic functions on the Banach spaces introduced above.

\begin{prop} Let $k_{0},k_{2} \geq 0$ and $k_{1} \geq 1$ be integers. Assume that the next condition
\begin{equation}
k_{1} \geq bk_{0} + \frac{bk_{2}}{\sigma_{3}} \label{multipl_operators_continuity_cond_SED} 
\end{equation}
holds. Then, for all $\epsilon \in \dot{D}(0,\epsilon_{0})$, the operator
$v(\tau,z) \mapsto \tau^{k_0}\exp(-k_{2} \tau) \partial_{z}^{-k_1}v(\tau,z)$ is a bounded linear
operator from
$(SED_{(\underline{\sigma},H,\epsilon,\delta)}, ||.||_{(\underline{\sigma},H,\epsilon,\delta)}$ into itself.
Moreover, there exists a constant $C_{1}>0$ (depending on $k_{0},k_{1},k_{2},\underline{\sigma},b$),
independent of $\epsilon$, such that
\begin{equation}
|| \tau^{k_0} \exp(-k_{2} \tau) \partial_{z}^{-k_{1}}v(\tau,z) ||_{(\underline{\sigma},H,\epsilon,\delta)}
\leq C_{1}|\epsilon|^{k_0} \delta^{k_1} ||v(\tau,z)||_{(\underline{\sigma},H,\epsilon,\delta)}
\label{multipl_operators_exp_continuity_SED}
\end{equation}
for all $v \in SED_{(\underline{\sigma},H,\epsilon,\delta)}$, all $\epsilon \in \dot{D}(0,\epsilon_{0})$.
\end{prop}
\begin{proof} Let $v(\tau,z) = \sum_{\beta \geq 0} v_{\beta}(\tau) z^{\beta}/\beta!$ belonging to
$SED_{(\underline{\sigma},H,\epsilon,\delta)}$. By definition,
\begin{equation}
|| \tau^{k_0} \exp(-k_{2} \tau) \partial_{z}^{-k_{1}}v(\tau,z) ||_{(\underline{\sigma},H,\epsilon,\delta)}
= \sum_{\beta \geq k_{1}}
|| \tau^{k_0} \exp(-k_{2} \tau) v_{\beta - k_{1}}(\tau) ||_{(\beta,\underline{\sigma},H,\epsilon)}
\frac{\delta^{\beta}}{\beta!}. \label{norm_exp_int_zk1_v}
\end{equation}
\begin{lemma} There exists a constant $C_{1.1}>0$ (depending on
$k_{0},k_{1},k_{2},\underline{\sigma},b$) such that
\begin{equation}
|| \tau^{k_0}\exp(-k_{2} \tau)v_{\beta - k_{1}}(\tau) ||_{(\beta,\underline{\sigma},H,\epsilon)}
\leq C_{1.1}|\epsilon|^{k_0}(\beta + 1)^{bk_{0} + \frac{k_{2}b}{\sigma_{3}}}
||v_{\beta - k_{1}}(\tau)||_{(\beta - k_{1},\underline{\sigma},H,\epsilon)}
\label{bds_norm_exp_v_beta_k1}
\end{equation}
for all $\beta \geq k_{1}$.
\end{lemma}
\begin{proof} First, we perform the next factorization
\begin{multline}
|\tau^{k_0}\exp(-k_{2}\tau) v_{\beta - k_{1}}(\tau)| \frac{1}{|\tau|}
\exp \left(-\frac{\sigma_{1}}{|\epsilon|} r_{b}(\beta)|\tau| + \sigma_{2} s_{b}(\beta)
\exp( \sigma_{3}|\tau| ) \right)\\
= \frac{|v_{\beta - k_{1}}(\tau)|}{|\tau|}
\exp \left(-\frac{\sigma_{1}}{|\epsilon|} r_{b}(\beta-k_{1})|\tau| + \sigma_{2} s_{b}(\beta-k_{1})
\exp( \sigma_{3}|\tau| ) \right)\\
\times \left( |\tau^{k_0}\exp( -k_{2} \tau)|
\exp( - \frac{\sigma_{1}}{|\epsilon|}( r_{b}(\beta) - r_{b}(\beta - k_{1}) )|\tau| )
\exp( \sigma_{2}( s_{b}(\beta) - s_{b}(\beta - k_{1}) ) \exp(\sigma_{3}|\tau|) ) \right)
\label{factor_v_beta_k1}
\end{multline}
On the other hand, by construction, we observe that
\begin{equation}
r_{b}(\beta) - r_{b}(\beta - k_{1}) \geq \frac{k_{1}}{(\beta + 1)^{b}} \ \ , \ \
s_{b}(\beta) - s_{b}(\beta - k_{1}) \leq -\frac{k_{1}}{(\beta + 1)^{b}} \label{difference_s_b_r_b}
\end{equation}
for all $\beta \geq k_{1}$. According to (\ref{factor_v_beta_k1}) and
(\ref{difference_s_b_r_b}), we deduce that
\begin{equation}
|| \tau^{k_0}\exp(-k_{2} \tau)v_{\beta - k_{1}}(\tau) ||_{(\beta,\underline{\sigma},H,\epsilon)}
\leq A(\beta) ||v_{\beta - k_{1}}(\tau)||_{(\beta - k_{1},\underline{\sigma},H,\epsilon)}
\label{norm_exp_v_beta_k1<norm_v_beta_k1}
\end{equation}
where
\begin{multline*}
A(\beta) = \sup_{ \tau \in H} |\tau|^{k_0} \exp(k_{2}|\tau|)
\exp( -\frac{\sigma_{1}}{|\epsilon|} \frac{k_1}{(\beta + 1)^{b}} |\tau| )\\
\times
\exp( -\sigma_{2} \frac{k_{1}}{(\beta + 1)^{b}} \exp( \sigma_{3} |\tau| ) ) \leq A_{1}(\beta)A_{2}(\beta)
\end{multline*}
with
$$ A_{1}(\beta) = \sup_{x \geq 0} x^{k_0}
\exp( -\frac{\sigma_{1}}{|\epsilon|} \frac{k_1}{(\beta + 1)^{b}} x )
$$
and 
$$ A_{2}(\beta) = \sup_{x \geq 0} \exp(k_{2}x)
\exp( -\sigma_{2} \frac{k_{1}}{(\beta + 1)^{b}} \exp( \sigma_{3} x ) )
$$
for all $\beta \geq k_{1}$. In the next step, we provide estimates for $A_{1}(\beta)$. Namely, from the
classical bounds for exponential functions
\begin{equation}
\sup_{x \geq 0} x^{m_1} \exp( -m_{2} x) \leq (\frac{m_1}{m_2})^{m_1} \exp( -m_{1} ) \label{xm1_expm2x<}
\end{equation}
for any integers $m_{1} \geq 0$ and $m_{2} > 0$, we get that
\begin{multline}
A_{1}(\beta) = |\epsilon|^{k_0} \sup_{x \geq 0} (\frac{x}{|\epsilon|})^{k_0}
\exp( - \frac{\sigma_{1} k_{1}}{(\beta + 1)^{b}} \frac{x}{|\epsilon|} )\\
\leq |\epsilon|^{k_0} \sup_{X \geq 0} X^{k_0} \exp( - \frac{\sigma_{1}k_{1}}{(\beta + 1)^{b}} X ) =
|\epsilon|^{k_0} (\frac{k_0}{\sigma_{1}k_{1}})^{k_0} \exp( -k_{0} ) (\beta + 1)^{bk_{0}} \label{A_1_bounds}
\end{multline}
for all $\beta \geq k_{1}$. In the last part, we focus on the sequence $A_{2}(\beta)$. First of all,
if $k_{2} = 0$, we observe that $A_{2}(\beta) \leq 1$ for all $\beta \geq k_{1}$.
Now, we assume that $k_{2} \geq 1$. Again, we need the
help of classical bounds for exponential functions
$$ \sup_{x \geq 0} cx - a\exp(bx) \leq \frac{c}{b}( \log(\frac{c}{ab}) - 1) $$
for all positive integers $a,b,c>0$ provided that $c > ab$. We deduce that
$$
A_{2}(\beta) \leq \exp( \frac{k_2}{\sigma_{3}}(
\log( \frac{k_{2}(\beta + 1)^{b}}{\sigma_{3} \sigma_{2} k_{1}} ) - 1) = 
\exp( -\frac{k_2}{\sigma_{3}} + \frac{k_2}{\sigma_3}\log( \frac{k_2}{\sigma_{3}\sigma_{2}k_{1}} ) )
(\beta + 1)^{\frac{k_{2}b}{\sigma_{3}}}
$$
whenever $\beta \geq k_{1}$ and $(\beta + 1)^{b} > \sigma_{2}\sigma_{3}k_{1}/k_{2}$. Besides, we also get a
constant $C_{1.0}>0$ (depending on $k_{2},\sigma_{2},k_{1},b,\sigma_{3}$) such that
$$
A_{2}(\beta) \leq C_{1.0}(\beta + 1)^{\frac{k_{2}b}{\sigma_3}} 
$$
for all $\beta \geq k_{1}$ with $(\beta + 1)^{b} \leq \sigma_{2}\sigma_{3}k_{1}/k_{2}$. In summary, we get
a constant $\tilde{C}_{1.0}>0$ (depending on $k_{2},\sigma_{2},k_{1},b,\sigma_{3}$) with
\begin{equation}
A_{2}(\beta) \leq \tilde{C}_{1.0}(\beta + 1)^{\frac{k_{2}b}{\sigma_3}} \label{A_2_bounds}
\end{equation}
for all $\beta \geq k_{1}$. Finally, gathering (\ref{norm_exp_v_beta_k1<norm_v_beta_k1}),
(\ref{A_1_bounds}) and (\ref{A_2_bounds}) yields (\ref{bds_norm_exp_v_beta_k1}).
\end{proof}
Bearing in mind the definition of the norm (\ref{norm_exp_int_zk1_v}) and the upper bounds
(\ref{bds_norm_exp_v_beta_k1}), we deduce that
\begin{multline}
|| \tau^{k_0} \exp(-k_{2} \tau) \partial_{z}^{-k_{1}}v(\tau,z) ||_{(\underline{\sigma},H,\epsilon,\delta)}
\leq \sum_{\beta \geq k_{1}} C_{1.1}|\epsilon|^{k_0}(1 + \beta)^{bk_{0} + \frac{bk_{2}}{\sigma_{3}}}\\
\times
\frac{(\beta - k_{1})!}{\beta !} ||v_{\beta - k_{1}}(\tau)||_{(\beta - k_{1},\underline{\sigma},H,\epsilon)}
\delta^{k_1} \frac{\delta^{\beta - k_{1}}}{(\beta - k_{1})!}. \label{bds_norm_exp_intz_v}
\end{multline}
In accordance with the assumption (\ref{multipl_operators_continuity_cond_SED}), we get a constant $C_{1.2}>0$
(depending on $k_{0},k_{1},k_{2},b,\sigma_{3}$) such that
\begin{equation}
(1 + \beta)^{bk_{0} + \frac{bk_{2}}{\sigma_{3}}} \frac{(\beta - k_{1})!}{\beta!} \leq C_{1.2}
\label{power_beta_factorial_C12}
\end{equation}
for all $\beta \geq k_{1}$. Lastly, clustering (\ref{bds_norm_exp_intz_v}) and
(\ref{power_beta_factorial_C12}) furnishes (\ref{multipl_operators_exp_continuity_SED}).
\end{proof}

\begin{prop} Let $k_{0},k_{2} \geq 0$ be integers. Let $\underline{\sigma} =
(\sigma_{1},\sigma_{2},\sigma_{3})$ and $\underline{\sigma}' = (\sigma_{1}',\sigma_{2}',\sigma_{3}')$ with
$\sigma_{j} > 0$ and $\sigma_{j}'>0$ for $j=1,2,3$, such that
\begin{equation}
\sigma_{1} > \sigma_{1}' \ \ , \ \ \sigma_{2} < \sigma_{2}' \ \ , \ \ \sigma_{3} = \sigma_{3}'. \label{cond_sigma_sigma'}
\end{equation}
Then, for all $\epsilon \in \dot{D}(0,\epsilon_{0})$, the operator $v(\tau,z)
\mapsto \tau^{k_0}\exp(-k_{2}\tau)v(\tau,z)$ is a bounded linear map from
$(SED_{(\underline{\sigma}',H,\epsilon,\delta)},||.||_{(\underline{\sigma}',H,\epsilon,\delta)})$ into
$(SED_{(\underline{\sigma},H,\epsilon,\delta)},||.||_{(\underline{\sigma},H,\epsilon,\delta)})$. Moreover,
there exists a constant $\check{C}_{1}>0$ (depending on
$k_{0},k_{2},\underline{\sigma},\underline{\sigma}',M,b$) such that
\begin{equation}
|| \tau^{k_0}\exp(-k_{2}\tau)v(\tau,z)||_{(\underline{\sigma},H,\epsilon,\delta)} \leq
\check{C}_{1}|\epsilon|^{k_0}||v(\tau,z)||_{(\underline{\sigma}',H,\epsilon,\delta)}
\end{equation}
for all $v \in SED_{(\underline{\sigma}',H,\epsilon,\delta)}$.
\end{prop}
\begin{proof} Take $v(\tau,z) = \sum_{\beta \geq 0} v_{\beta}(\tau) \frac{z^{\beta}}{\beta!}$ within
$SED_{(\underline{\sigma}',H,\epsilon,\delta)}$. According to Definition 1, we see that
$$ ||\tau^{k_0}\exp(-k_{2}\tau)v(\tau,z)||_{(\underline{\sigma},H,\epsilon,\delta)}
= \sum_{\beta \geq 0} ||\tau^{k_0}\exp(-k_{2}\tau)v_{\beta}(\tau)||_{(\beta,\underline{\sigma},H,\epsilon)}
\frac{\delta^{\beta}}{\beta !}
$$
\begin{lemma}
There exists a constant $\check{C}_{1}>0$ (depending on $k_{0},k_{2},\underline{\sigma},\underline{\sigma}',M,b$) such that
$$ ||\tau^{k_0}\exp(-k_{2}\tau)v_{\beta}(\tau)||_{(\beta,\underline{\sigma},H,\epsilon)} \leq
\check{C}_{1}|\epsilon|^{k_0}||v_{\beta}(\tau)||_{(\beta,\underline{\sigma}',H,\epsilon)}$$
\end{lemma}
\begin{proof} We operate the next factorization
\begin{multline*}
|\tau^{k_0}\exp(-k_{2}\tau)v_{\beta}(\tau)| \frac{1}{|\tau|}\exp \left(-\frac{\sigma_{1}}{|\epsilon|}r_{b}(\beta)|\tau|
+ \sigma_{2}s_{b}(\beta)\exp(\sigma_{3}|\tau|) \right)\\
= |v_{\beta}(\tau)|\frac{1}{|\tau|} \exp \left(-\frac{\sigma_{1}'}{|\epsilon|}r_{b}(\beta)|\tau|
+ \sigma_{2}'s_{b}(\beta)\exp(\sigma_{3}'|\tau|) \right)\\
\times |\tau^{k_0}\exp(-k_{2}\tau)| \exp( - \frac{\sigma_{1} - \sigma_{1}'}{|\epsilon|}r_{b}(\beta)|\tau| )
\exp \left( (\sigma_{2} - \sigma_{2}')s_{b}(\beta) \exp( \sigma_{3}|\tau| ) \right).
\end{multline*}
We deduce that
$$ ||\tau^{k_0}\exp(-k_{2}\tau)v_{\beta}(\tau)||_{(\beta,\underline{\sigma},H,\epsilon)} \leq
\check{A}(\beta)||v_{\beta}(\tau)||_{(\beta,\underline{\sigma}',H,\epsilon)}$$
where
\begin{multline*}
\check{A}(\beta) = \sup_{\tau \in H} |\tau^{k_0}\exp(-k_{2}\tau)|
\exp( - \frac{\sigma_{1} - \sigma_{1}'}{|\epsilon|}r_{b}(\beta)|\tau| )
\exp \left( (\sigma_{2} - \sigma_{2}')s_{b}(\beta) \exp( \sigma_{3}|\tau| ) \right) \\
\leq \check{A}_{1}(\beta) \check{A}_{2}(\beta)
\end{multline*}
with
$$ \check{A}_{1}(\beta) = \sup_{x \geq 0}
x^{k_0}\exp( - \frac{\sigma_{1} - \sigma_{1}'}{|\epsilon|}r_{b}(\beta)x ) \ \ , \ \
\check{A}_{2}(\beta) = \sup_{x \geq 0} \exp(k_{2}x)
\exp \left( (\sigma_{2} - \sigma_{2}')s_{b}(\beta) \exp( \sigma_{3}x ) \right).$$
Since $r_{b}(\beta) \geq 1$ for all $\beta \geq 0$, we deduce from (\ref{xm1_expm2x<}) that
\begin{equation}
\check{A}_{1}(\beta) \leq |\epsilon|^{k_0} \sup_{x \geq 0} (\frac{x}{|\epsilon|})^{k_0}
\exp( - (\sigma_{1} - \sigma_{1}') \frac{x}{|\epsilon|} )
\leq |\epsilon|^{k_0}( \frac{k_{0} e^{-1}}{\sigma_{1} - \sigma_{1}'})^{k_0}. \label{checkA1_bds}
\end{equation}
In order to handle the sequence $\check{A}_{2}(\beta)$, we observe that
$s_{b}(\beta) \geq M - \zeta(b) > 0$, for all $\beta \geq 0$. Therefore, we see that
$$ \check{A}_{2}(\beta) \leq \sup_{x \geq 0} \exp \left( k_{2}x + (\sigma_{2} - \sigma_{2}')(M - \zeta(b))
\exp(\sigma_{3}x) \right) $$
which is a finite upper bound for all $\beta \geq 0$.
\end{proof}
As a consequence, Proposition 3 follows directly from Lemma 2.
\end{proof}

\begin{prop} Let $c(\tau,z,\epsilon)$ be a holomorphic function on $\mathring{H} \times D(0,\rho) \times 
D(0,\epsilon_{0})$,
continuous on $H \times D(0,\rho) \times D(0,\epsilon_{0})$, for some $\rho>0$, bounded by a constant $M_{c}>0$ on
$H \times D(0,\rho) \times D(0,\epsilon_{0})$. Let $0 < \delta < \rho$. Then, the linear map
$v(\tau,z) \mapsto c(\tau,z,\epsilon)v(\tau,z)$ is bounded from
$(SED_{(\underline{\sigma},H,\epsilon,\delta)},||.||_{(\underline{\sigma},H,\epsilon,\delta)})$ into itself, for
all $\epsilon \in \dot{D}(0,\epsilon_{0})$. Furthermore, one can choose a constant $\breve{C}_{1}>0$ (depending on
$M_{c},\delta,\rho$) independent of $\epsilon$ such that
\begin{equation}
||c(\tau,z,\epsilon)v(\tau,z)||_{(\underline{\sigma},H,\epsilon,\delta)} \leq \breve{C}_{1}
||v(\tau,z)||_{(\underline{\sigma},H,\epsilon,\delta)}
\end{equation}
for all $v \in SED_{(\underline{\sigma},H,\epsilon,\delta)}$.
\end{prop}
\begin{proof} We expand $c(\tau,z,\epsilon) = \sum_{\beta \geq 0} c_{\beta}(\tau,\epsilon) z^{\beta}/\beta!$ as a
convergent Taylor series w.r.t $z$ on $D(0,\rho)$ and we set $M_{c}>0$ with
$$ \sup_{\tau \in H, z \in \bar{D}(0,\rho),\epsilon \in \mathcal{E}} |c(\tau,z,\epsilon)| \leq M_{c}. $$
Let $v(\tau,z) = \sum_{\beta \geq 0} v_{\beta}(\tau) z^{\beta}/\beta!$ belonging to
$SED_{(\underline{\sigma},H,\epsilon,\delta)}$. According to Definition 1, we get that
\begin{equation}
||c(\tau,z,\epsilon)v(\tau,z)||_{(\underline{\sigma},H,\epsilon,\delta)} \leq 
\sum_{\beta \geq 0} \left( \sum_{\beta_{1} + \beta_{2} = \beta}
||c_{\beta_{1}}(\tau,\epsilon)v_{\beta_{2}}(\tau)||_{(\beta,\underline{\sigma},H,\epsilon)}
\frac{\beta!}{\beta_{1}!\beta_{2}!} \right) \frac{\delta^{\beta}}{\beta!}.
\label{maj_norm_c_bded_v}
\end{equation}
Besides, the Cauchy formula implies the next estimates
$$ \sup_{\tau \in H,\epsilon \in \mathcal{E}} |c_{\beta}(\tau,\epsilon)| \leq M_{c} (\frac{1}{\delta'})^{\beta} \beta!
$$
for any $\delta < \delta' < \rho$, for all $\beta \geq 0$. By construction of the norm, since
$r_{b}(\beta) \geq r_{b}(\beta_{2})$ and $s_{b}(\beta) \leq s_{b}(\beta_{2})$ whenever $\beta_{2} \leq \beta$, we
deduce that
\begin{equation}
||c_{\beta_{1}}(\tau,\epsilon)v_{\beta_{2}}(\tau)||_{(\beta,\underline{\sigma},H,\epsilon)}
\leq M_{c} \beta_{1}!(\frac{1}{\delta'})^{\beta_{1}}||v_{\beta_{2}}(\tau)||_{(\beta,\underline{\sigma},H,\epsilon)}
\leq M_{c} \beta_{1}!(\frac{1}{\delta'})^{\beta_{1}}||v_{\beta_{2}}(\tau)||_{(\beta_{2},\underline{\sigma},H,\epsilon)}
\label{norm_beta_c_bded_v}
\end{equation}
for all $\beta_{1},\beta_{2} \geq 0$ with $\beta_{1} + \beta_{2}=\beta$. Gathering (\ref{maj_norm_c_bded_v})
and (\ref{norm_beta_c_bded_v}) yields the desired bounds
$$ ||c(\tau,z,\epsilon)v(\tau,z)||_{(\underline{\sigma},H,\epsilon,\delta)} \leq M_{c}
(\sum_{\beta \geq 0} (\frac{\delta}{\delta'})^{\beta} )||v(\tau,z)||_{(\underline{\sigma},H,\epsilon,\delta)}. $$
\end{proof}

\subsection{Banach spaces of holomorphic functions with super exponential growth on horizontal strips and exponential
growth on sectors}

We keep the notations of the previous subsection 2.1. We consider a closed horizontal strip
\begin{equation}
J = \{ z \in \mathbb{C} / c \leq \mathrm{Im}(z) \leq d, \ \ \mathrm{Re}(z) \leq 0 \} \label{defin_strip_J}
\end{equation}
for some real numbers $c<d$. We denote $S_{d}$ an unbounded open sector with bisecting direction $d \in \mathbb{R}$
centered at 0 such that $S_{d} \subset \mathbb{C}_{+} = \{ z \in \mathbb{C} / \mathrm{Re}(z) > 0 \}$.

\begin{defin}
Let $\underline{\varsigma} = (\sigma_{1},\varsigma_{2},\varsigma_{3})$ where
$\sigma_{1},\varsigma_{2},\varsigma_{3}>0$ be positive real numbers and $\beta \geq 0$ be an
integer. Take $\epsilon \in \dot{D}(0,\epsilon_{0})$. We designate
$SEG_{(\beta,\underline{\varsigma},J,\epsilon)}$ as the vector space of holomorphic functions $v(\tau)$ on
$\mathring{J}$ and continuous on $J$
such that
$$ ||v(\tau)||_{(\beta,\underline{\varsigma},J,\epsilon)} =  \sup_{\tau \in J} \frac{|v(\tau)|}{|\tau|}
\exp \left(-\frac{\sigma_{1}}{|\epsilon|} r_{b}(\beta)|\tau| - \varsigma_{2} r_{b}(\beta)
\exp( \varsigma_{3}|\tau| ) \right) $$
is finite. Similarly, we denote $EG_{(\beta,\sigma_{1},S_{d} \cup D(0,r),\epsilon)}$ the vector space of
holomorphic functions $v(\tau)$ on
$S_{d} \cup D(0,r)$ and continuous on $\bar{S_d} \cup \bar{D}(0,r)$
such that
$$ ||v(\tau)||_{(\beta,\sigma_{1},S_{d} \cup D(0,r),\epsilon)} =  \sup_{\tau \in \bar{S_d} \cup \bar{D}(0,r)}
\frac{|v(\tau)|}{|\tau|}
\exp ( -\frac{\sigma_{1}}{|\epsilon|} r_{b}(\beta)|\tau| ) $$
is finite. Let us choose $\delta>0$ a real number. We define
$SEG_{(\underline{\varsigma},J,\epsilon,\delta})$
to be the vector space of all formal series $v(\tau,z) = \sum_{\beta \geq 0} v_{\beta}(\tau) z^{\beta}/\beta!$
with coefficients $v_{\beta}(\tau) \in SEG_{(\beta,\underline{\varsigma},J,\epsilon)}$, for
$\beta \geq 0$ and such that
$$ ||v(\tau,z)||_{(\underline{\varsigma},J,\epsilon,\delta)} = \sum_{\beta \geq 0}
||v_{\beta}(\tau)||_{(\beta,\underline{\varsigma},J,\epsilon)} \frac{\delta^{\beta}}{\beta !}$$
is finite. Likewise, we set $EG_{(\sigma_{1},S_{d} \cup D(0,r),\epsilon,\delta})$
as the vector space of all formal series $v(\tau,z) = \sum_{\beta \geq 0} v_{\beta}(\tau) z^{\beta}/\beta!$
with coefficients $v_{\beta}(\tau) \in EG_{(\beta,\sigma_{1},S_{d} \cup D(0,r),\epsilon)}$, for
$\beta \geq 0$ with
$$ ||v(\tau,z)||_{(\sigma_{1},S_{d} \cup D(0,r),\epsilon,\delta)} = \sum_{\beta \geq 0}
||v_{\beta}(\tau)||_{(\beta,\sigma_{1},S_{d} \cup D(0,r),\epsilon)} \frac{\delta^{\beta}}{\beta !}$$
being finite.
\end{defin}
\noindent {\bf Remark.} These Banach spaces are slight modifications of those introduced in the former work \cite{ma}
of the second author. The next proposition will be enounced without proof since it follows exactly the same steps as
Proposition 1 above. It states that the formal series appertaining to the latter Banach spaces turn out to be
holomorphic functions on some disc w.r.t $z$ and with super exponential growth (resp. exponential growth) w.r.t $\tau$
on the strip $J$ (resp. on the domain $S_{d} \cup D(0,r)$). 

\begin{prop}

1) Let $v(\tau,z) \in SEG_{(\underline{\varsigma},J,\epsilon,\delta)}$. Take some real number $0 < \delta_{1} < 1$.
Then, there exists a constant $C_{2}>0$ depending on
$||v||_{(\underline{\varsigma},J,\epsilon,\delta)}$ and $\delta_{1}$ such that
\begin{equation}
|v(\tau,z)| \leq C_{2}|\tau| \exp \left( \frac{\sigma_{1}}{|\epsilon|} \zeta(b) |\tau| +
\varsigma_{2} \zeta(b) \exp( \varsigma_{3} |\tau| ) \right) 
\end{equation}
for all $\tau \in J$, all $z \in \mathbb{C}$ with $\frac{|z|}{\delta} < \delta_{1}$.

\noindent 2) Let us take $v(\tau,z) \in EG_{(\sigma_{1},S_{d} \cup D(0,r),\epsilon,\delta)}$. Choose some real number
$0 < \delta_{1} < 1$.
Then, there exists a constant $C'_{2}>0$ depending on
$||v||_{(\sigma_{1},S_{d} \cup D(0,r),\epsilon,\delta)}$
and $\delta_{1}$ such that
\begin{equation}
|v(\tau,z)| \leq C'_{2}|\tau| \exp ( \frac{\sigma_{1}}{|\epsilon|} \zeta(b) |\tau| )
\end{equation}
for all $\tau \in S_{d} \cup D(0,r)$, all $z \in \mathbb{C}$ with $\frac{|z|}{\delta} < \delta_{1}$.
\end{prop}

In the next coming propositions, we study the same linear operators as defined in Propositions 2,3 and 4 but acting on
the Banach spaces described in Definition 2.

\begin{prop} Let us choose integers $k_{0},k_{2} \geq 0$ and $k_{1} \geq 1$. 

\noindent 1) We take for granted that the next
constraint 
$$ k_{1} \geq bk_{0} + \frac{bk_{2}}{\varsigma_{3}} \label{multipl_operators_continuity_cond_SEG}$$
holds. Then, for all $\epsilon \in \dot{D}(0,\epsilon_{0})$, the linear map 
$v(\tau,z) \mapsto \tau^{k_0}\exp(-k_{2} \tau) \partial_{z}^{-k_1}v(\tau,z)$ is bounded from
$(SEG_{(\underline{\varsigma},J,\epsilon,\delta)}, ||.||_{(\underline{\varsigma},J,\epsilon,\delta)}$ into itself.
Moreover, there exists a constant $C_{3}>0$ (depending on
$k_{0},k_{1},k_{2},\underline{\varsigma},b$),
independent of $\epsilon$, such that
\begin{equation}
|| \tau^{k_0} \exp(-k_{2} \tau) \partial_{z}^{-k_{1}}v(\tau,z) ||_{(\underline{\varsigma},J,\epsilon,\delta)}
\leq C_{3}|\epsilon|^{k_0} \delta^{k_1} ||v(\tau,z)||_{(\underline{\varsigma},J,\epsilon,\delta)}
\label{multipl_operators_exp_continuity_SEG}
\end{equation}
for all $v(\tau,z) \in SEG_{(\underline{\varsigma},J,\epsilon,\delta)}$, all $\epsilon \in \dot{D}(0,\epsilon_{0})$.

\noindent 2) We suppose that the next restriction 
$$
 k_{1} \geq bk_{0}
$$
holds. Then, for all $\epsilon \in \dot{D}(0,\epsilon_{0})$, the linear map 
$v(\tau,z) \mapsto \tau^{k_0}\exp(-k_{2} \tau) \partial_{z}^{-k_1}v(\tau,z)$ is bounded
from $EG_{(\sigma_{1},S_{d} \cup D(0,r),\epsilon,\delta)}$ into itself.
Moreover, there exists a constant $C'_{3}>0$ (depending on
$k_{0},k_{1},k_{2},\sigma_{1},r,b$),
independent of $\epsilon$, such that
\begin{equation}
|| \tau^{k_0} \exp(-k_{2} \tau) \partial_{z}^{-k_{1}}v(\tau,z) ||_{(\sigma_{1},S_{d} \cup D(0,r),\epsilon,\delta)}
\leq C'_{3}|\epsilon|^{k_0} \delta^{k_1} ||v(\tau,z)||_{(\sigma_{1},S_{d} \cup D(0,r),\epsilon,\delta)}
\label{multipl_operators_exp_continuity_EG}
\end{equation}
for all $v(\tau,z) \in EG_{(\sigma_{1},S_{d} \cup D(0,r),\epsilon,\delta)}$, all $\epsilon \in \mathcal{E}$.
\end{prop}  
\begin{proof} We only perform a sketch of proof since the lines of arguments are bordering the ones used in Proposition 2.
For the first point 1), we are reduced to show the next lemma
\begin{lemma} Let $v_{\beta - k_{1}}(\tau)$ in $SEG_{(\beta - k_{1},\underline{\varsigma},J,\epsilon)}$, for all
$\beta \geq k_{1}$. There exists a constant $C_{3.1}>0$ (depending on $k_{0},k_{1},k_{2},\underline{\varsigma},b$)
such that
$$
|| \tau^{k_0}\exp(-k_{2} \tau)v_{\beta - k_{1}}(\tau) ||_{(\beta,\underline{\varsigma},J,\epsilon)}
\leq C_{3.1}|\epsilon|^{k_0}(\beta + 1)^{bk_{0} + \frac{k_{2}b}{\varsigma_{3}}}
||v_{\beta - k_{1}}(\tau)||_{(\beta - k_{1},\underline{\varsigma},J,\epsilon)}
$$
for all $\beta \geq k_{1}$. 
\end{lemma}
\begin{proof} We use the factorization
\begin{multline*}
|\tau^{k_0}\exp(-k_{2}\tau) v_{\beta - k_{1}}(\tau)| \frac{1}{|\tau|}
\exp \left(-\frac{\sigma_{1}}{|\epsilon|} r_{b}(\beta)|\tau| - \varsigma_{2} r_{b}(\beta)
\exp( \varsigma_{3}|\tau| ) \right)\\
= \frac{|v_{\beta - k_{1}}(\tau)|}{|\tau|}
\exp \left(-\frac{\sigma_{1}}{|\epsilon|} r_{b}(\beta-k_{1})|\tau| - \varsigma_{2} r_{b}(\beta-k_{1})
\exp( \varsigma_{3}|\tau| ) \right)\\
\times \left( |\tau^{k_0}\exp( -k_{2} \tau)|
\exp( - \frac{\sigma_{1}}{|\epsilon|}( r_{b}(\beta) - r_{b}(\beta - k_{1}) )|\tau| )
\exp( -\varsigma_{2}( r_{b}(\beta) - r_{b}(\beta - k_{1}) ) \exp(\varsigma_{3}|\tau|) ) \right).
\end{multline*}
In accordance with (\ref{difference_s_b_r_b}), we get that
$$
|| \tau^{k_0}\exp(-k_{2} \tau)v_{\beta - k_{1}}(\tau) ||_{(\beta,\underline{\varsigma},J,\epsilon)}
\leq B(\beta) ||v_{\beta - k_{1}}(\tau)||_{(\beta - k_{1},\underline{\varsigma},J,\epsilon)}
$$
where
\begin{multline*}
B(\beta) = \sup_{ \tau \in J} |\tau|^{k_0} \exp(k_{2}|\tau|)
\exp( -\frac{\sigma_{1}}{|\epsilon|} \frac{k_1}{(\beta + 1)^{b}} |\tau| )\\
\times
\exp( -\varsigma_{2} \frac{k_{1}}{(\beta + 1)^{b}} \exp( \varsigma_{3} |\tau| ) ) \leq B_{1}(\beta)B_{2}(\beta)
\end{multline*}
with
$$ B_{1}(\beta) = \sup_{x \geq 0} x^{k_0}
\exp( -\frac{\sigma_{1}}{|\epsilon|} \frac{k_1}{(\beta + 1)^{b}} x )
$$
and 
$$ B_{2}(\beta) = \sup_{x \geq 0} \exp(k_{2}x)
\exp( -\varsigma_{2} \frac{k_{1}}{(\beta + 1)^{b}} \exp( \varsigma_{3} x ) )
$$
for all $\beta \geq k_{1}$. From the estimates (\ref{A_1_bounds}), we deduce that
$$B_{1}(\beta) \leq |\epsilon|^{k_0} (\frac{k_0}{\sigma_{1}k_{1}})^{k_0} \exp( -k_{0} ) (\beta + 1)^{bk_{0}}$$
for all $\beta \geq k_{1}$. Bearing in mind the estimates (\ref{A_2_bounds}), we get a constant
$\tilde{C}_{3.0}>0$ (depending on $k_{2},\varsigma_{2},k_{1},b,\varsigma_{3}$) with
$$ B_{2}(\beta) \leq \tilde{C}_{3.0}(\beta + 1)^{\frac{k_{2}b}{\varsigma_{3}}} $$
for all $\beta \geq k_{1}$, provided that $k_{2} \geq 1$. When $k_{2}=0$, we obviously see that
$B_{2}(\beta) \leq 1$ for all $\beta \geq k_{1}$. The lemma 3 follows.
\end{proof}

\noindent In order to explain the second point 2), we need to check the next lemma

\begin{lemma} Let $v_{\beta - k_{1}}(\tau)$ in $EG_{(\beta - k_{1},\sigma_{1},S_{d} \cup D(0,r),\epsilon)}$, for all
$\beta \geq k_{1}$. There exists a constant $C'_{3.1}>0$ (depending on $k_{0},k_{1},k_{2},\sigma_{1},r,b$)
such that
$$
|| \tau^{k_0}\exp(-k_{2} \tau)v_{\beta - k_{1}}(\tau) ||_{(\beta,\sigma_{1},S_{d} \cup D(0,r),\epsilon)}
\leq C'_{3.1}|\epsilon|^{k_0}(\beta + 1)^{bk_{0}}
||v_{\beta - k_{1}}(\tau)||_{(\beta - k_{1},\sigma_{1},S_{d} \cup D(0,r),\epsilon)}
$$
for all $\beta \geq k_{1}$. 
\end{lemma}
\begin{proof} We need the help of the factorization
\begin{multline*}
|\tau^{k_0}\exp(-k_{2}\tau) v_{\beta - k_{1}}(\tau)| \frac{1}{|\tau|}
\exp (-\frac{\sigma_{1}}{|\epsilon|} r_{b}(\beta)|\tau| )
= \frac{|v_{\beta - k_{1}}(\tau)|}{|\tau|}
\exp (-\frac{\sigma_{1}}{|\epsilon|} r_{b}(\beta-k_{1})|\tau| ) \\
\times |\tau^{k_0}\exp( -k_{2} \tau)|
\exp( - \frac{\sigma_{1}}{|\epsilon|}( r_{b}(\beta) - r_{b}(\beta - k_{1}) )|\tau| ).
\end{multline*}
Due to the fact that there exists a constant $C'_{3.2}>0$ (depending on $k_{2},r$) such that
$|\exp(-k_{2}\tau)| \leq C'_{3.2}$ for all $\tau \in S_{d} \cup D(0,r)$ and according to (\ref{difference_s_b_r_b}),
we obtain that
$$
|| \tau^{k_0}\exp(-k_{2} \tau)v_{\beta - k_{1}}(\tau) ||_{(\beta,\sigma_{1},S_{d} \cup D(0,r),\epsilon)}
\leq C(\beta) ||v_{\beta - k_{1}}(\tau)||_{(\beta - k_{1},\sigma_{1},S_{d} \cup D(0,r),\epsilon)}
$$
where
$$
C(\beta) = C'_{3.2}\sup_{ \tau \in S_{d} \cup D(0,r)} |\tau|^{k_0}
\exp( -\frac{\sigma_{1}}{|\epsilon|} \frac{k_1}{(\beta + 1)^{b}} |\tau| )\\
\leq C'_{3.2}C_{1}(\beta)
$$
with
$$ C_{1}(\beta) = \sup_{x \geq 0} x^{k_0}
\exp( -\frac{\sigma_{1}}{|\epsilon|} \frac{k_1}{(\beta + 1)^{b}} x )
$$
for all $\beta \geq k_{1}$. Again, keeping in view the estimates (\ref{A_1_bounds}), we deduce that
$$C_{1}(\beta) \leq |\epsilon|^{k_0} (\frac{k_0}{\sigma_{1}k_{1}})^{k_0} \exp( -k_{0} ) (\beta + 1)^{bk_{0}}$$
for all $\beta \geq k_{1}$. The lemma 4 follows.
\end{proof}
\end{proof}

\begin{prop} Let $k_{0},k_{2} \geq 0$ be integers.\\
1) We select $\underline{\varsigma} = (\sigma_{1},\varsigma_{2},\varsigma_{3})$ and
$\underline{\varsigma}' = (\sigma_{1}',\varsigma_{2}',\varsigma_{3}')$ with $\sigma_{1},\sigma_{1}'>0$, $\varsigma_{j},\varsigma_{j}'>0$ for
$j=2,3$ in order that
\begin{equation}
\sigma_{1} > \sigma_{1}' \ \ , \ \  \varsigma_{2} > \varsigma_{2}' \ \ , \ \ \varsigma_{3} = \varsigma_{3}'.
\end{equation}
Then, for all $\epsilon \in \dot{D}(0,\epsilon_{0})$, the map $v(\tau,z) \mapsto \tau^{k_0}\exp(-k_{2}\tau)v(\tau,z)$
is a bounded linear operator from
$(SEG_{(\underline{\varsigma}',J,\epsilon,\delta)},||.||_{(\underline{\varsigma}',J,\epsilon,\delta)})$ into
$(SEG_{(\underline{\varsigma},J,\epsilon,\delta)},||.||_{(\underline{\varsigma},J,\epsilon,\delta)})$. Furthermore,
there exists a constant $\check{C}_{3}>0$ (depending on $k_{0},k_{2},\underline{\varsigma},\underline{\varsigma}'$)
such that
\begin{equation}
|| \tau^{k_0}\exp( -k_{2} \tau)v(\tau,z) ||_{(\underline{\varsigma},J,\epsilon,\delta)}
\leq \check{C}_{3} |\epsilon|^{k_0} || v(\tau,z) ||_{(\underline{\varsigma}',J,\epsilon,\delta)}
\end{equation}
for all $v \in SEG_{(\underline{\varsigma}',J,\epsilon,\delta)}$.\\
2) Let $\sigma_{1},\sigma_{1}'>0$ such that
\begin{equation}
\sigma_{1} > \sigma_{1}'. 
\end{equation}
Then, for all $\epsilon \in \dot{D}(0,\epsilon_{0})$, the linear map $v(\tau,z) \mapsto \tau^{k_0}\exp(-k_{2}\tau)v(\tau,z)$
is bounded from the Banach space $(EG_{(\sigma_{1}',S_{d} \cup D(0,r),\epsilon,\delta)},
||.||_{(\sigma_{1}',S_{d} \cup D(0,r),\epsilon,\delta)})$ into $(EG_{(\sigma_{1},S_{d} \cup D(0,r),\epsilon,\delta)},
||.||_{(\sigma_{1},S_{d} \cup D(0,r),\epsilon,\delta)})$. Besides, there exists a constant $\check{C}_{3}'>0$
(depending on $k_{0},k_{2},r,\sigma_{1},\sigma_{1}'$) such that
\begin{equation}
|| \tau^{k_0}\exp( -k_{2} \tau)v(\tau,z) ||_{(\sigma_{1},S_{d} \cup D(0,r),\epsilon,\delta)}
\leq \check{C}_{3}' |\epsilon|^{k_0} || v(\tau,z) ||_{(\sigma_{1}',S_{d} \cup D(0,r),\epsilon,\delta)}
\end{equation}
for all $v \in EG_{(\sigma_{1}',S_{d} \cup D(0,r),\epsilon,\delta)}$.\\
\end{prop}
\begin{proof}
As in Proposition 6, we only provide an outline of the proof since it keeps very close to the one of Proposition 3.
Concerning the first item 1), we are scaled down to show the next lemma
\begin{lemma}
There exists a constant $\check{C}_{3}>0$ (depending on $k_{0},k_{2},\underline{\varsigma},\underline{\varsigma}'$)
such that
$$ ||\tau^{k_0}\exp(-k_{2}\tau)v_{\beta}(\tau)||_{(\beta,\underline{\varsigma},J,\epsilon)}
\leq \check{C}_{3}|\epsilon|^{k_0}||v_{\beta}(\tau)||_{(\beta,\underline{\varsigma}',J,\epsilon)} $$
\end{lemma}
\begin{proof}
We perform the factorization
\begin{multline*}
|\tau^{k_0}\exp(-k_{2}\tau)v_{\beta}(\tau)| \frac{1}{|\tau|}\exp \left(-\frac{\sigma_{1}}{|\epsilon|}r_{b}(\beta)|\tau|
- \varsigma_{2}r_{b}(\beta)\exp(\varsigma_{3}|\tau|) \right)\\
= |v_{\beta}(\tau)|\frac{1}{|\tau|} \exp \left(-\frac{\sigma_{1}'}{|\epsilon|}r_{b}(\beta)|\tau|
- \varsigma_{2}'r_{b}(\beta)\exp(\varsigma_{3}'|\tau|) \right)\\
\times |\tau^{k_0}\exp(-k_{2}\tau)| \exp( - \frac{\sigma_{1} - \sigma_{1}'}{|\epsilon|}r_{b}(\beta)|\tau| )
\exp \left( -(\varsigma_{2} - \varsigma_{2}')r_{b}(\beta) \exp( \varsigma_{3}|\tau| ) \right).
\end{multline*}
We get that
$$ ||\tau^{k_0}\exp(-k_{2}\tau)v_{\beta}(\tau)||_{(\beta,\underline{\varsigma},J,\epsilon)} \leq
\check{B}(\beta)||v_{\beta}(\tau)||_{(\beta,\underline{\varsigma}',J,\epsilon)}$$
where
\begin{multline*}
\check{B}(\beta) = \sup_{\tau \in J} |\tau|^{k_0}\exp(k_{2}|\tau|)
\exp( - \frac{\sigma_{1} - \sigma_{1}'}{|\epsilon|}r_{b}(\beta)|\tau| )
\exp \left( -(\varsigma_{2} - \varsigma_{2}')r_{b}(\beta) \exp( \varsigma_{3}|\tau| ) \right) \\
\leq \check{B}_{1}(\beta) \check{B}_{2}(\beta)
\end{multline*}
with
$$ \check{B}_{1}(\beta) = \sup_{x \geq 0}
x^{k_0}\exp( - \frac{\sigma_{1} - \sigma_{1}'}{|\epsilon|}r_{b}(\beta)x ) \ \ , \ \
\check{B}_{2}(\beta) = \sup_{x \geq 0} \exp(k_{2}x)
\exp \left( -(\varsigma_{2} - \varsigma_{2}')r_{b}(\beta) \exp( \varsigma_{3}x ) \right).$$
With the help of (\ref{checkA1_bds}), we check that
$$ \check{B}_{1}(\beta) 
\leq |\epsilon|^{k_0}( \frac{k_{0} e^{-1}}{\sigma_{1} - \sigma_{1}'})^{k_0}$$
and since $r_{b}(\beta) \geq 1$ for all $\beta \geq 0$, we deduce
$$ \check{B}_{2}(\beta) \leq \sup_{x \geq 0} \exp \left( k_{2}x - (\varsigma_{2} - \varsigma_{2}')
\exp(\varsigma_{3}x) \right) $$
which is a finite majorant for all $\beta \geq 0$. The lemma follows.
\end{proof}
Regarding the second item 2), it boils down to the next lemma
\begin{lemma}
There exists a constant $\check{C}_{3}'>0$ (depending on $k_{0},k_{2},r,\sigma_{1},\sigma_{1}'$)
such that
$$ ||\tau^{k_0}\exp(-k_{2}\tau)v_{\beta}(\tau)||_{(\beta,\sigma_{1},S_{d} \cup D(0,r),\epsilon)}
\leq \check{C}_{3}'|\epsilon|^{k_0}||v_{\beta}(\tau)||_{(\beta,\sigma_{1}',S_{d} \cup D(0,r),\epsilon)} $$
\end{lemma}
\begin{proof}
Again we need to factorize the next expression
\begin{multline*}
|\tau^{k_0}\exp(-k_{2}\tau)v_{\beta}(\tau)| \frac{1}{|\tau|}\exp (-\frac{\sigma_{1}}{|\epsilon|}r_{b}(\beta)|\tau| )
= |v_{\beta}(\tau)|\frac{1}{|\tau|} \exp ( -\frac{\sigma_{1}'}{|\epsilon|}r_{b}(\beta)|\tau| )\\
\times |\tau^{k_0}\exp(-k_{2}\tau)| \exp( - \frac{\sigma_{1} - \sigma_{1}'}{|\epsilon|}r_{b}(\beta)|\tau| ).
\end{multline*}
By construction, we can select a constant $\check{C}_{3.1}'>0$ (depending on $k_{2},r$) such that
$|\exp(-k_{2}\tau)| \leq \check{C}_{3.1}'$ for all $\tau \in S_{d} \cup D(0,r)$. We deduce that
\begin{equation}
||\tau^{k_0}\exp(-k_{2}\tau)v_{\beta}(\tau)||_{(\beta,\sigma_{1},S_{d} \cup D(0,r),\epsilon)}
\leq \check{C}(\beta) ||v_{\beta}(\tau)||_{(\beta,\sigma_{1}',S_{d} \cup D(0,r),\epsilon)}
\end{equation}
where
$$ \check{C}(\beta) \leq \check{C}_{3.1}' \sup_{\tau \in S_{d} \cup D(0,r)} |\tau|^{k_0}
\exp( - \frac{\sigma_{1} - \sigma_{1}'}{|\epsilon|}r_{b}(\beta)|\tau| ) \leq
\check{C}_{3.1}'\check{C}_{1}(\beta) $$
with
$$ \check{C}_{1}(\beta) = \sup_{x \geq 0}
x^{k_0}\exp( - \frac{\sigma_{1} - \sigma_{1}'}{|\epsilon|}r_{b}(\beta)x ).$$
Through (\ref{checkA1_bds}) we notice that
$$ \check{C}_{1}(\beta) 
\leq |\epsilon|^{k_0}( \frac{k_{0} e^{-1}}{\sigma_{1} - \sigma_{1}'})^{k_0}$$
for all $\beta \geq 0$. This yields the lemma.
\end{proof}
\end{proof}
The next proposition will be stated without proof since its explanation can be disclosed following exactly the same steps
and arguments as in Proposition 4.
\begin{prop}
1) Consider a holomorphic function $c(\tau,z,\epsilon)$ on $\mathring{J} \times D(0,\rho) \times D(0,\epsilon_{0})$,
continuous
on $J \times D(0,\rho) \times D(0,\epsilon_{0})$, for some $\rho>0$, bounded by a constant $M_{c}>0$ on
$J \times D(0,\rho) \times D(0,\epsilon_{0})$. We set $0 < \delta < \rho$. Then, the operator
$v(\tau,z) \mapsto c(\tau,z,\epsilon)v(\tau,z)$ is bounded from
$(SEG_{(\underline{\varsigma},J,\epsilon,\delta)},||.||_{(\underline{\varsigma},J,\epsilon,\delta)})$ into itself, for
all $\epsilon \in \dot{D}(0,\epsilon_{0})$. Besides, one can select a constant $\breve{C}_{3}>0$
(depending on $M_{c},\delta,\rho$) such that
$$ ||c(\tau,z,\epsilon)v(\tau,z)||_{(\underline{\varsigma},J,\epsilon,\delta)} \leq
\breve{C}_{3}||v(\tau,z)||_{(\underline{\varsigma},J,\epsilon,\delta)} $$
for all $v \in SEG_{(\underline{\varsigma},J,\epsilon,\delta)}$.\\
2) Let us take a function $c(\tau,z,\epsilon)$ holomorphic on $(S_{d} \cup D(0,r)) \times D(0,\rho) \times
D(0,\epsilon_{0})$,
continuous on $(\bar{S_{d}} \cup \bar{D}(0,r)) \times D(0,\rho) \times D(0,\epsilon_{0})$, for some $\rho>0$ and
bounded by a constant $M_{c}>0$ on $(\bar{S_{d}} \cup \bar{D}(0,r)) \times D(0,\rho) \times D(0,\epsilon_{0})$. Let
$0 < \delta < \rho$. Then, the linear map $v(\tau,z) \mapsto c(\tau,z,\epsilon)v(\tau,z)$ is bounded from
$(EG_{(\sigma_{1},S_{d} \cup D(0,r),\epsilon,\delta)},||.||_{(\sigma_{1},S_{d} \cup D(0,r),\epsilon,\delta)})$
into itself, for all $\epsilon \in \dot{D}(0,\epsilon_{0})$. Furthermore, one can sort a constant $\breve{C}_{3}'>0$
(depending on $M_{c},\delta,\rho$) with
$$ ||c(\tau,z,\epsilon)v(\tau,z)||_{(\sigma_{1},S_{d} \cup D(0,r),\epsilon,\delta)} \leq
\breve{C}_{3}'||v(\tau,z)||_{(\sigma_{1},S_{d} \cup D(0,r),\epsilon,\delta)} $$
for all $v \in EG_{(\sigma_{1},S_{d} \cup D(0,r),\epsilon,\delta)}$.\\
\end{prop}

\subsection{An auxiliary Cauchy problem whose coefficients suffer exponential growth on strips
and polynomial growth on unbounded sectors}

We start this subsection by introducing some notations. Let $\mathcal{A}$ be a finite subset of
$\mathbb{N}^{3}$. For all $\underline{k} = (k_{0},k_{1},k_{2}) \in \mathcal{A}$, we consider a bounded
holomorphic function $c_{\underline{k}}(z,\epsilon)$ on a polydisc $D(0,\rho) \times D(0,\epsilon_{0})$ for some
radii $\rho,\epsilon_{0}>0$. Let $S \geq 1$ be an integer and let $P(\tau)$ be a polynomial (not identically equal to 0) with complex coefficients
whose roots belong to the open right halfplane $\mathbb{C}_{+} = \{ z \in \mathbb{C} / \mathrm{Re}(z) > 0 \}$.

We consider the following equation
\begin{equation}
\partial_{z}^{S}w(\tau,z,\epsilon) = \sum_{\underline{k}=(k_{0},k_{1},k_{2}) \in \mathcal{A}}
\frac{c_{\underline{k}}(z,\epsilon)}{P(\tau)} \epsilon^{-k_0}\tau^{k_0} \exp(- k_{2}\tau) \partial_{z}^{k_1}
w(\tau,z,\epsilon) \label{1_aux_CP}
\end{equation}
Let us now enounce the principal statement of this subsection.

\begin{prop} 1) We impose the next requirements\\
a) There exist $\underline{\sigma} = (\sigma_{1},\sigma_{2},\sigma_{3})$ for $\sigma_{1},\sigma_{2},\sigma_{3}>0$
and $b>1$ being real numbers such that for all $\underline{k} = (k_{0},k_{1},k_{2}) \in \mathcal{A}$, we have
\begin{equation}
S \geq k_{1} + bk_{0} + \frac{b k_{2}}{\sigma_{3}} \ \ , \ \ S > k_{1} \label{cond_ex_uniq_sol_1_aux_CP_SED_H}
\end{equation}
b) For all $0 \leq j \leq S-1$, we consider a function $\tau \mapsto w_{j}(\tau,\epsilon)$ that belong to
the Banach space $SED_{(0,\underline{\sigma}',H,\epsilon)}$ for all $\epsilon \in \dot{D}(0,\epsilon_{0})$,
for some closed horizontal strip $H$ described in (\ref{defin_strip_H}) and for a tuple
$\underline{\sigma}'=(\sigma_{1}',\sigma_{2}',\sigma_{3}')$ with $\sigma_{1}>\sigma_{1}'>0$,
$\sigma_{2}<\sigma_{2}'$ and $\sigma_{3}=\sigma_{3}'$.

Then, there exist some constants $I,R>0$ and $0 < \delta < \rho$ (independent of $\epsilon$) such that if one assumes that
\begin{equation}
\sum_{j=0}^{S-1-h} ||w_{j+h}(\tau,\epsilon)||_{(0,\underline{\sigma}',H,\epsilon)} \frac{\delta^{j}}{j!} \leq I
\label{initial_data_1_aux_CP_small_H}
\end{equation}
for all $0 \leq h \leq S-1$, for all $\epsilon \in \dot{D}(0,\epsilon_{0})$, the equation
(\ref{1_aux_CP}) with initial data
\begin{equation}
(\partial_{z}^{j}w)(\tau,0,\epsilon) = w_{j}(\tau,\epsilon) \ \ , \ \ 0 \leq j \leq S-1,
\label{1_aux_CP_initial_data}
\end{equation}
has a unique solution $w(\tau,z,\epsilon)$ in the space $SED_{(\underline{\sigma},H,\epsilon,\delta)}$, for all
$\epsilon \in \dot{D}(0,\epsilon_{0})$ and satisfies furthermore the estimates
\begin{equation}
||w(\tau,z,\epsilon)||_{(\underline{\sigma},H,\epsilon,\delta)} \leq \delta^{S}R + I
\label{norm_w_bd_in_epsilon_H}
\end{equation}
for all $\epsilon \in \dot{D}(0,\epsilon_{0})$.\\
2) We demand the next restrictions\\
a) There exist $\underline{\varsigma}=(\sigma_{1},\varsigma_{2},\varsigma_{3})$ where
$\sigma_{1},\varsigma_{2},\varsigma_{3}>0$ and $b>1$ real numbers taken in way that all
$\underline{k}=(k_{0},k_{1},k_{2}) \in \mathcal{A}$ we have
\begin{equation}
S \geq k_{1} + bk_{0} + \frac{bk_{2}}{\varsigma_{3}} \ \ , \ \ S > k_{1}. 
\end{equation}
b) For all $0 \leq j \leq S-1$, we choose a function $\tau \mapsto w_{j}(\tau,\epsilon)$ belonging to
the Banach space $SEG_{(0,\underline{\varsigma}',J,\epsilon)}$ for all $\epsilon \in \dot{D}(0,\epsilon_{0})$,
for some closed horizontal strip $J$ displayed in (\ref{defin_strip_J}) and for a tuple
$\underline{\varsigma}'=(\sigma_{1}',\varsigma_{2}',\varsigma_{3}')$ with $\sigma_{1}>\sigma_{1}'>0$,
$\varsigma_{2}>\varsigma_{2}'>0$ and $\varsigma_{3}=\varsigma_{3}'$.

Then, there exist some constants $I,R>0$ and $0 < \delta < \rho$ (independent of $\epsilon$) such that if one
takes for granted that
\begin{equation}
\sum_{j=0}^{S-1-h} ||w_{j+h}(\tau,\epsilon)||_{(0,\underline{\varsigma}',J,\epsilon)} \frac{\delta^{j}}{j!} \leq I
\label{initial_data_1_aux_CP_small_J}
\end{equation}
for all $0 \leq h \leq S-1$, for all $\epsilon \in \dot{D}(0,\epsilon_{0})$, the equation
(\ref{1_aux_CP}) with initial data (\ref{1_aux_CP_initial_data})
has a unique solution $w(\tau,z,\epsilon)$ in the space $SEG_{(\underline{\varsigma},J,\epsilon,\delta)}$, for all
$\epsilon \in \dot{D}(0,\epsilon_{0})$ and fulfills the next constraint
\begin{equation}
||w(\tau,z,\epsilon)||_{(\underline{\varsigma},J,\epsilon,\delta)} \leq \delta^{S}R + I
\label{norm_w_bd_in_epsilon_J}
\end{equation}
for all $\epsilon \in \dot{D}(0,\epsilon_{0})$.\\
3) We ask for the next conditions.\\
a) We fix some real number $\sigma_{1}>0$ and assume the existence of $b>1$ a real number such that for all
$\underline{k}=(k_{0},k_{1},k_{2}) \in \mathcal{A}$ we have
\begin{equation}
S \geq k_{1} + bk_{0} \ \ , \ \ S > k_{1}. 
\end{equation}
b) For all $0 \leq j \leq S-1$, we select a function $\tau \mapsto w_{j}(\tau,\epsilon)$ that belong to
the Banach space $EG_{(0,\sigma_{1}',S_{d} \cup D(0,r),\epsilon)}$ for all
$\epsilon \in \dot{D}(0,\epsilon_{0})$,
for some open unbounded sector $S_{d}$ with bisecting direction $d$ with $S_{d} \subset \mathbb{C}_{+}$
and $D(0,r)$ a disc centered at 0 with radius $r$, for some $0 < \sigma_{1}' < \sigma_{1}$. The sector
$S_{d}$ and the disc $D(0,r)$ are chosen in a way that $\bar{S}_{d} \cup \bar{D}(0,r)$ does not contain any root of
the polynomial $P(\tau)$.

Then, some constants $I,R>0$ and $0 < \delta < \rho$ (independent of $\epsilon$) can be sorted if one
accepts that
\begin{equation}
\sum_{j=0}^{S-1-h} ||w_{j+h}(\tau,\epsilon)||_{(0,\sigma_{1}',S_{d} \cup D(0,r),\epsilon)} \frac{\delta^{j}}{j!} \leq I
\label{initial_data_1_aux_CP_small_S}
\end{equation}
for all $0 \leq h \leq S-1$, for all $\epsilon \in \dot{D}(0,\epsilon_{0})$, the equation
(\ref{1_aux_CP}) with initial data (\ref{1_aux_CP_initial_data})
has a unique solution $w(\tau,z,\epsilon)$ in the space $EG_{(\sigma_{1},S_{d} \cup D(0,r),\epsilon,\delta)}$, for all
$\epsilon \in \dot{D}(0,\epsilon_{0})$, with the bounds
\begin{equation}
||w(\tau,z,\epsilon)||_{(\sigma_{1},S_{d} \cup D(0,r),\epsilon,\delta)} \leq \delta^{S}R + I
\label{norm_w_bd_in_epsilon_S}
\end{equation}
for all $\epsilon \in \dot{D}(0,\epsilon_{0})$.\\
\end{prop}
\begin{proof}
Within the proof, we only plan to provide a detailed description of the point 1) since the same lines of arguments
apply for the points 2) and 3) by making use of Propositions 6,7 and 8 instead of Propositions 2,3 and 4. We consider
the function
$$ W_{S}(\tau,z,\epsilon) = \sum_{j=0}^{S-1} w_{j}(\tau,\epsilon) \frac{z^j}{j!} $$
where $w_{j}(\tau,\epsilon)$ is displayed in 1)b) above. We introduce a map $A_{\epsilon}$ defined as
\begin{multline*}
A_{\epsilon}(U(\tau,z)) :=  \sum_{\underline{k}=(k_{0},k_{1},k_{2}) \in \mathcal{A}}
\frac{c_{\underline{k}}(z,\epsilon)}{P(\tau)} \epsilon^{-k_0}\tau^{k_0} \exp(- k_{2}\tau) \partial_{z}^{k_{1}-S}
U(\tau,z)\\
+ \sum_{\underline{k}=(k_{0},k_{1},k_{2}) \in \mathcal{A}}
\frac{c_{\underline{k}}(z,\epsilon)}{P(\tau)} \epsilon^{-k_0}\tau^{k_0} \exp(- k_{2}\tau) \partial_{z}^{k_1}
W_{S}(\tau,z,\epsilon).
\end{multline*}
In the forthcoming lemma, we show that $A_{\epsilon}$ represents a Lipschitz shrinking map from and into a small
ball centered at the origin in the space $SED_{(\underline{\sigma},H,\epsilon,\delta)}$.
\begin{lemma}
Under the constraint (\ref{cond_ex_uniq_sol_1_aux_CP_SED_H}), let us consider a positive real number $I>0$ such that
$$ \sum_{j=0}^{S-1-h} ||w_{j+h}(\tau,\epsilon)||_{(0,\underline{\sigma}',H,\epsilon)} \frac{\delta^j}{j!}
\leq I
$$
for all $0 \leq h \leq S-1$, for $\epsilon \in \dot{D}(0,\epsilon_{0})$. Then, for an appropriate
choice of $I$,\\
a) There exists a constant $R>0$ (independent of $\epsilon$) such that
\begin{equation}
||A_{\epsilon}(U(\tau,z))||_{(\underline{\sigma},H,\epsilon,\delta)} \leq R \label{A_epsilon_ball_in_ball}
\end{equation}
for all $U(\tau,z) \in B(0,R)$, for all $\epsilon \in \dot{D}(0,\epsilon_{0})$, where
$B(0,R)$ is the closed ball centered at 0 with radius $R$ in $SED_{(\underline{\sigma},H,\epsilon,\delta)}$.\\
b) The next inequality
\begin{equation}
||A_{\epsilon}(U_{1}(\tau,z)) - A_{\epsilon}(U_{2}(\tau,z))||_{(\underline{\sigma},H,\epsilon,\delta)} \leq
\frac{1}{2}||U_{1}(\tau,z) - U_{2}(\tau,z)||_{(\underline{\sigma},H,\epsilon,\delta)}
\label{A_epsilon_shrink}
\end{equation}
holds for all $U_{1},U_{2} \in B(0,R)$, all $\epsilon \in \dot{D}(0,\epsilon_{0})$.
\end{lemma}
\begin{proof} Since $r_{b}(\beta) \geq r_{b}(0)$ and $s_{b}(\beta) \leq s_{b}(0)$ for all $\beta \geq 0$, we notice
that for any $0 \leq h \leq S-1$ and $0 \leq j \leq S-1-h$,
$$ ||w_{j+h}(\tau,\epsilon)||_{(j,\underline{\sigma}',H,\epsilon)} \leq
||w_{j+h}(\tau,\epsilon)||_{(0,\underline{\sigma}',H,\epsilon)} $$
holds. We deduce that $\partial_{z}^{h}W_{S}(\tau,z,\epsilon)$ belongs to
$SED_{(\underline{\sigma}',H,\epsilon,\delta)}$ and moreover that
\begin{equation}
||\partial_{z}^{h}W_{S}(\tau,z,\epsilon)||_{(\underline{\sigma}',H,\epsilon,\delta)}
\leq \sum_{j=0}^{S-1-h} ||w_{j+h}(\tau,\epsilon)||_{(0,\underline{\sigma}',H,\epsilon)} \frac{\delta^j}{j!}
\leq I, \label{norm_partial_z_WS}
\end{equation}
for all $0 \leq h \leq S-1$. We start by focusing our attention to the estimates (\ref{A_epsilon_ball_in_ball}). Let $U(\tau,z)$ belonging
to $SED_{(\underline{\sigma},H,\epsilon,\delta)}$ with $||U(\tau,z)||_{(\underline{\sigma},H,\epsilon,\delta)} \leq R$.
Assume that $0 < \delta < \rho$. We put
$$ M_{\underline{k}} = \sup_{\tau \in H,z \in D(0,\rho), \epsilon \in D(0,\epsilon_{0})}
\left| \frac{c_{\underline{k}}(z,\epsilon)}{P(\tau)} \right|
$$
for all $\underline{k} \in \mathcal{A}$. Taking for granted the assumption (\ref{cond_ex_uniq_sol_1_aux_CP_SED_H}) and
according to Propositions 2 and 4, for all $\underline{k} \in \mathcal{A}$, we get two constants $C_{1}>0$
(depending on $k_{0},k_{1},k_{2},S,\underline{\sigma},b$) and $\breve{C}_{1}>0$ (depending on $M_{\underline{k}}$,
$\delta$,$\rho$) such that
\begin{multline}
|| \frac{c_{\underline{k}}(z,\epsilon)}{P(\tau)} \epsilon^{-k_0} \tau^{k_0} \exp( -k_{2}\tau )
\partial_{z}^{k_{1}-S}U(\tau,z) ||_{(\underline{\sigma},H,\epsilon,\delta)} \\
\leq
\breve{C}_{1}C_{1} \delta^{S-k_{1}}
|| U(\tau,z) ||_{(\underline{\sigma},H,\epsilon,\delta)} = \breve{C}_{1}C_{1} \delta^{S-k_{1}}R
\label{A_epsilon_ball_in_ball_lin_part}
\end{multline}
On the other hand, in agreement with Propositions 3 and 4 and with the help of
(\ref{norm_partial_z_WS}), we obtain two constants $\check{C}_{1}>0$ (depending on
$k_{0},k_{2},\underline{\sigma},\underline{\sigma}',M,b$) and $\breve{C}_{1}>0$ (depending on
$M_{\underline{k}},\delta,\rho$) with
\begin{multline}
|| \frac{c_{\underline{k}}(z,\epsilon)}{P(\tau)} \epsilon^{-k_0} \tau^{k_0} \exp( -k_{2}\tau )
\partial_{z}^{k_1}W_{S}(\tau,z,\epsilon) ||_{(\underline{\sigma},H,\epsilon,\delta)} \\
\leq
\breve{C}_{1}\check{C}_{1}
|| \partial_{z}^{k_1}W_{S}(\tau,z,\epsilon) ||_{(\underline{\sigma}',H,\epsilon,\delta)} \leq \breve{C}_{1}\check{C}_{1}I
\label{A_epsilon_ball_in_ball_nonhomog_part} 
\end{multline}
Now, we choose $\delta,R,I>0$ in such a way that
\begin{equation}
\sum_{\underline{k} \in \mathcal{A}} (\breve{C}_{1}C_{1} \delta^{S-k_{1}}R + \breve{C}_{1}\check{C}_{1}I) \leq R
\label{cond_A_epsilon_ball_in_ball}
\end{equation}
holds. Assembling (\ref{A_epsilon_ball_in_ball_lin_part}) and (\ref{A_epsilon_ball_in_ball_nonhomog_part}) under 
(\ref{cond_A_epsilon_ball_in_ball}) allows (\ref{A_epsilon_ball_in_ball}) to hold.

In a second part, we turn to the estimates (\ref{A_epsilon_shrink}). Let $R>0$ with $U_{1},U_{2}$ belonging
to $SED_{(\underline{\sigma},H,\epsilon,\delta)}$ inside the ball $B(0,R)$. By means of
(\ref{A_epsilon_ball_in_ball_lin_part}), we see that
\begin{multline}
|| \frac{c_{\underline{k}}(z,\epsilon)}{P(\tau)} \epsilon^{-k_0} \tau^{k_0} \exp( -k_{2}\tau )
\partial_{z}^{k_{1}-S}(U_{1}(\tau,z) - U_{2}(\tau,z)) ||_{(\underline{\sigma},H,\epsilon,\delta)} \\
\leq
\breve{C}_{1}C_{1} \delta^{S-k_{1}}
||U_{1}(\tau,z) - U_{2}(\tau,z)||_{(\underline{\sigma},H,\epsilon,\delta)}
\label{A_epsilon_ball_in_ball_shrink}
\end{multline}
where $C_{1},\breve{C}_{1}>0$ are given above. We select $\delta>0$ small enough in order that
\begin{equation}
\sum_{\underline{k} \in \mathcal{A}} \breve{C}_{1}C_{1}\delta^{S-k_{1}} \leq 1/2. \label{cond_A_epsilon_shrink}
\end{equation}
Therefore, (\ref{A_epsilon_ball_in_ball_shrink}) under (\ref{cond_A_epsilon_shrink}) supports that
(\ref{A_epsilon_shrink}) holds.

At last, we sort $\delta,R,I$ in a way that both (\ref{cond_A_epsilon_ball_in_ball}) and (\ref{cond_A_epsilon_shrink})
hold at the same time. Lemma 7 follows.
\end{proof}
Let the constraint (\ref{cond_ex_uniq_sol_1_aux_CP_SED_H}) be fulfilled. We choose the constants $I,R,\delta$ as in Lemma 7.
We select the initial data $w_{j}(\tau,\epsilon)$, $0 \leq j \leq S-1$ and a tuple $\underline{\sigma}'$ in a way that
the restriction (\ref{initial_data_1_aux_CP_small_H}) holds. Owing to Lemma 7 and to the classical contractive mapping
theorem on complete metric spaces, we deduce that the map $A_{\epsilon}$ has a unique fixed point called
$U(\tau,z,\epsilon)$ (depending analytically on $\epsilon \in \dot{D}(0,\epsilon_{0})$) in the closed
ball $B(0,R) \subset SED_{(\underline{\sigma},H,\epsilon,\delta)}$, for all $\epsilon \in
\dot{D}(0,\epsilon_{0})$. This means that
$A_{\epsilon}(U(\tau,z,\epsilon)) = U(\tau,z,\epsilon)$ with
$||U(\tau,z,\epsilon)||_{(\underline{\sigma},H,\epsilon,\delta)} \leq R$. As a result, we get that the next expression
$$ w(\tau,z,\epsilon) = \partial_{z}^{-S}U(\tau,z,\epsilon) + W_{S}(\tau,z,\epsilon) $$
solves the equation (\ref{1_aux_CP}) with initial data (\ref{1_aux_CP_initial_data}). It remains to show that
$w(\tau,z,\epsilon)$ belongs to $SED_{(\underline{\sigma},H,\epsilon,\delta)}$ and to check the bounds
(\ref{norm_w_bd_in_epsilon_H}). By application of Proposition 2 for $k_{0}=k_{2}=0$ and $k_{1}=S$ we check that
\begin{equation}
||\partial_{z}^{-S}U(\tau,z,\epsilon)||_{(\underline{\sigma},H,\epsilon,\delta)} \leq
\delta^{S}||U(\tau,z,\epsilon)||_{(\underline{\sigma},H,\epsilon,\delta)}
\label{norm_partial_z_minus_S_U}
\end{equation}
Gathering (\ref{norm_partial_z_WS}) and (\ref{norm_partial_z_minus_S_U}) yields the fact that
$w(\tau,z,\epsilon)$ belongs to $SED_{(\underline{\sigma},H,\epsilon,\delta)}$ through the bounds
(\ref{norm_w_bd_in_epsilon_H}).
\end{proof}

\section{Sectorial analytic solutions in a complex parameter of a singular perturbed Cauchy problem involving fractional linear transforms}

Let $\mathcal{A}$ be a finite subset of
$\mathbb{N}^{3}$. For all $\underline{k} = (k_{0},k_{1},k_{2}) \in \mathcal{A}$, we denote
$c_{\underline{k}}(z,\epsilon)$ a bounded holomorphic function on a polydisc $D(0,\rho) \times D(0,\epsilon_{0})$ for
given radii $\rho,\epsilon_{0}>0$. Let $S \geq 1$ be an integer and let $P(\tau)$ be
a polynomial (not identically equal to 0) with complex coefficients selected in a way that its
roots belong to the open right halfplane $\mathbb{C}_{+} = \{ z \in \mathbb{C} / \mathrm{Re}(z) > 0 \}$. We focus on the
following singularly perturbed Cauchy problem that incorporates fractional linear transforms
\begin{equation}
P(\epsilon t^{2}\partial_{t})\partial_{z}^{S}u(t,z,\epsilon)
= \sum_{\underline{k} = (k_{0},k_{1},k_{2}) \in \mathcal{A}}
c_{\underline{k}}(z,\epsilon) \left((t^{2}\partial_{t})^{k_0}\partial_{z}^{k_1}u \right)( \frac{t}{1 + k_{2}\epsilon t},
z,\epsilon) \label{SPCP_first}
\end{equation}
for given initial data
\begin{equation}
(\partial_{z}^{j}u)(t,0,\epsilon) = \varphi_{j}(t,\epsilon) \ \ , \ \ 0 \leq j \leq S-1. \label{SPCP_first_i_d} 
\end{equation}
We put the next assumption on the set $\mathcal{A}$. There exist two real numbers $\xi>0$ and $b>1$ such that for
all $\underline{k}=(k_{0},k_{1},k_{2}) \in \mathcal{A}$,
\begin{equation}
S \geq k_{1} + bk_{0} + \frac{bk_{2}}{\xi} \ \ , \ \ S > k_{1}. \label{cond_SPCP_first} 
\end{equation}

\subsection{Construction of holomorphic solutions on a prescribed sector w.r.t $\epsilon$ using Banach spaces of
functions with super exponential growth and decay on strips}

Let $n \geq 1$ be an integer. We denote $\llbracket -n,n \rrbracket$ the set of integers
$\{ j \in \mathbb{N}, -n \leq j \leq n \}$. We consider two sets of closed horizontal strips
$\{ H_{k} \}_{k \in \llbracket -n,n
\rrbracket }$ and $\{ J_{k} \}_{k \in \llbracket -n,n \rrbracket }$ fulfilling the next conditions. If one displays
the strips $H_{k}$ and $J_{k}$ as follows,
$$ H_{k} = \{ z \in \mathbb{C} / a_{k} \leq \mathrm{Im}(z) \leq b_{k}, \ \ \mathrm{Re}(z) \leq 0 \} \ \ , \ \ 
J_{k} = \{ z \in \mathbb{C} / c_{k} \leq \mathrm{Im}(z) \leq d_{k}, \ \ \mathrm{Re}(z) \leq 0 \}$$
then, the real numbers $a_{k},b_{k},c_{k},d_{k}$ are asked to fulfill the next constraints.\\
1) The origin 0 belongs to $(c_{0},d_{0})$.\\
2) We have $c_{k} < a_{k} < d_{k}$ and $c_{k+1} < b_{k} < d_{k+1}$ for $-n \leq k \leq n-1$ together with
$c_{n} < a_{n} < d_{n}$ and $b_{n} > d_{n}$. In other words the strips
$J_{-n},H_{-n},J_{-n+1},\ldots,J_{n-1},H_{n-1},J_{n},H_{n}$ are consecutively overlapping.\\
3) We have $a_{k+1}>b_{k}$ and $c_{k+1}>d_{k}$ for $-n \leq k \leq n-1$. Namely, the strips $H_{k}$ (resp. $J_{k}$)
are disjoints for $k \in \llbracket -n,n \rrbracket$.

We denote $HJ_{n} = \{ z \in \mathbb{C} / c_{-n} \leq \mathrm{Im}(z) \leq b_{n}, \mathrm{Re}(z) \leq 0 \}$. We notice
that $HJ_{n}$ can be written as the union $\cup_{k \in \llbracket -n,n \rrbracket} H_{k} \cup J_{k}$.

\begin{figure}
	\centering
		\includegraphics[width=0.4\textwidth]{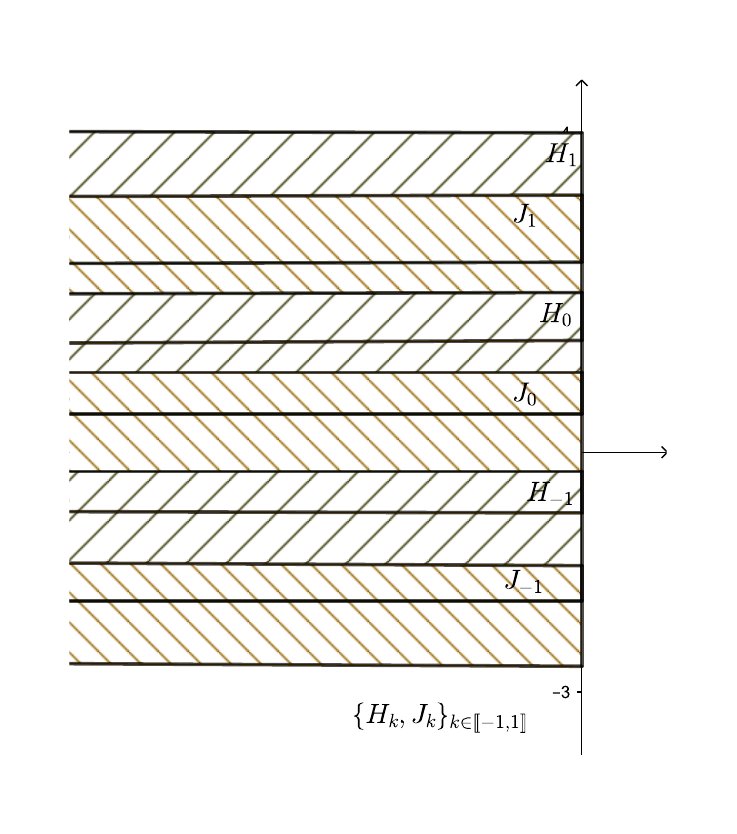}
		\caption{Example of configuration for the sets $H_{k}$ and $J_k$}
		\label{fig1}
\end{figure}

An example of configuration is shown in Figure~\ref{fig1}.

\begin{defin}
Let $n \geq 1$ be an integer. Let $w(\tau,\epsilon)$ be a holomorphic function on
$\mathring{HJ}_{n} \times \dot{D}(0,\epsilon_{0})$ (where $\mathring{HJ}_{n}$ denotes the interior
of $HJ_{n}$), continuous on $HJ_{n} \times \dot{D}(0,\epsilon_{0})$. Assume that for all
$\epsilon \in \dot{D}(0,\epsilon_{0})$, for all $k \in \llbracket -n,n \rrbracket$, the function
$\tau \mapsto w(\tau,\epsilon)$ belongs to the Banach spaces $SED_{(0,\underline{\sigma}',H_{k},\epsilon)}$
and $SEG_{(0,\underline{\varsigma}',J_{k},\epsilon)}$ with
$\underline{\sigma}' = (\sigma_{1}',\sigma_{2}',\sigma_{3}')$ and
$\underline{\varsigma}' = (\sigma_{1}',\varsigma_{2}',\varsigma_{3}')$ for some $\sigma_{1}'>0$ and
$\sigma_{j}',\varsigma_{j}'>0$ for $j=2,3$. Moreover, there exists a constant $I_{w}>0$ independent of
$\epsilon$, such that
\begin{equation}
||w(\tau,\epsilon)||_{(0,\underline{\sigma}',H_{k},\epsilon)} \leq I_{w} \ \ , \ \
||w(\tau,\epsilon)||_{(0,\underline{\varsigma}',J_{k},\epsilon)} \leq I_{w}, \label{bounds_w_initial}
\end{equation}
for all $k \in \llbracket -n,n \rrbracket$ and all $\epsilon \in \dot{D}(0,\epsilon_{0})$.

\noindent Let $\mathcal{E}_{HJ_{n}}$ be an open sector centered at 0 inside the disc $D(0,\epsilon_{0})$ with aperture
strictly less than $\pi$ and $\mathcal{T}$ be a bounded open sector centered at 0 with bisecting
direction $d=0$ chosen in a way that
\begin{equation}
\pi - \mathrm{arg}(t) - \mathrm{arg}(\epsilon) \in (-\frac{\pi}{2} + \delta_{HJ_{n}},\frac{\pi}{2} - \delta_{HJ_{n}})
\label{rel_t_epsilon_mathcal_E_T} 
\end{equation}
for some small $\delta_{HJ_{n}}>0$, for all $\epsilon \in \mathcal{E}_{HJ_{n}}$ and $t \in \mathcal{T}$.

\noindent We say that the set $(w(\tau,\epsilon),\mathcal{E}_{HJ_{n}},\mathcal{T})$ is
$(\underline{\sigma}',\underline{\varsigma}')-$admissible.
\end{defin}

\noindent {\bf Example:} Let $w(\tau,\epsilon) = \tau \exp(a\exp(-\tau))$ for some real number $a>0$. One can notice that 
$$ |w(\tau,\epsilon)| \leq |\tau| \exp \left( a \cos(\mathrm{Im}(\tau)) \exp(-\mathrm{Re}(\tau)) \right) $$
for all $\tau \in \mathbb{C}$, all $\epsilon \in \mathbb{C}$. For all $k \in \mathbb{Z}$, let $H_{k}$ be the closed
strip defined as
$$ H_{k} = \{ z \in \mathbb{C} / \ \  \frac{\pi}{2} + \eta + 2k\pi \leq \mathrm{Im}(z) \leq
\frac{3\pi}{2} - \eta + 2k \pi, \ \ \mathrm{Re}(z) \leq 0 \} $$
for some real number $\eta>0$ and let $J_{k}$ be the closed strip described as
$$ J_{k} = \{ z \in \mathbb{C} / \ \ \frac{3\pi}{2} - \eta - \eta_{1} + 2(k-1)\pi \leq \mathrm{Im}(z)
\leq \frac{\pi}{2} + \eta + \eta_{1} + 2k\pi, \ \ \mathrm{Re}(z) \leq 0 \} $$
for some $\eta_{1}>0$. Provided that $\eta$ and $\eta_{1}$ are small enough, we can check that all the constraints
1) to 3) listed above are fulfilled for any fixed $n \geq 1$, for $k \in \llbracket -n,n \rrbracket$.

By construction, we get a constant $\Delta_{\eta}>0$ (depending on $\eta$) with
$\cos(\mathrm{Im}(\tau)) \leq -\Delta_{\eta}$ provided that $\tau \in H_{k}$, for all $k \in \mathbb{Z}$.
Let $m>0$ be a fixed real number. We first show that there exists $K_{m,k}>0$ (depending on $m$ and $k$)
such that
$$ -\mathrm{Re}(\tau) \geq K_{m,k}|\tau| $$
for all $\mathrm{Re}(\tau) \leq -m$ provided that $\tau \in H_{k}$. Indeed, if one puts
$$ y_{k} = \max \{ |y| / y \in [ \frac{\pi}{2} + \eta + 2k\pi, \frac{3\pi}{2} - \eta + 2k\pi ] \} $$
then the next inequality holds
$$ \frac{-\mathrm{Re}(\tau)}{|\tau|} \geq \min_{x \geq m}
\frac{x}{(x^{2} + y_{k}^{2})^{1/2}} = K_{m,k} > 0 $$
for all $\tau \in \mathbb{C}$ such that $\mathrm{Re}(\tau) \leq -m$ and $\tau \in H_{k}$. Now, we set
$K_{m;n} = \min_{k \in \llbracket -n,n \rrbracket} K_{m,k}$. As a result, we deduce the
existence of a constant $\Omega_{m,k}>0$ (depending on $m$,$k$ and $a$) such that
$$ |w(\tau,\epsilon)| \leq \Omega_{m,k}|\tau|\exp( -a \Delta_{\eta} \exp( K_{m;n}|\tau| ) ) $$
for all $\tau \in H_{k}$.

On the
other hand, we only have the upper bound $\cos(\mathrm{Im}(\tau)) \leq 1$ when $\tau \in J_{k}$, for all
$k \in \mathbb{Z}$. Since $-\mathrm{Re}(\tau) \leq |\tau|$, for all $\tau \in \mathbb{C}$, we deduce that
$$ |w(\tau,\epsilon)| \leq |\tau| \exp( a \exp(|\tau|))$$
whenever $\tau$ belongs to $J_{k}$, for all $\epsilon \in \mathbb{C}$. As a result, the function
$w(\tau,\epsilon)$ fulfills all the requirements asked in
Definition 3 for
$$ \underline{\sigma}'=(\sigma_{1}',a \Delta_{\eta}/(M-1),K_{m;n}) \ \ , \ \ 
\underline{\varsigma}'=(\sigma_{1}',a,1)$$
for any given $\sigma_{1}'>0$.

Let $n \geq 1$ be an integer and let us take some integer $k \in \llbracket -n,n \rrbracket$.
For each $0 \leq j \leq S-1$ and each integer $k \in \llbracket -n,n \rrbracket$, let
$\{ w_{j}(\tau,\epsilon), \mathcal{E}_{HJ_{n}}^{k},\mathcal{T} \}$ be a
$(\underline{\sigma}',\underline{\varsigma}')-$admissible set. As initial data (\ref{SPCP_first_i_d}), we set
\begin{equation}
\varphi_{j,\mathcal{E}_{HJ_{n}}^{k}}(t,\epsilon) = \int_{P_k} w_{j}(u,\epsilon)
\exp( -\frac{u}{\epsilon t} ) \frac{du}{u} \label{SPCP_first_i_d_k}
\end{equation}
where the integration path $P_{k}$ is built as the union of two paths $P_{k,1}$ and $P_{k,2}$ described as follows.
$P_{k,1}$ is a segment joining the origin 0 and a prescribed point $A_{k} \in H_{k}$ and
$P_{k,2}$ is the horizontal line $\{ A_{k} - s / s \geq 0 \}$. According to (\ref{rel_t_epsilon_mathcal_E_T}),
we choose the point $A_{k}$ with
$|\mathrm{Re}(A_k)|$ suitably large in a way that
\begin{equation}
\mathrm{arg}(A_{k}) - \mathrm{arg}(\epsilon) - \mathrm{arg}(t) \in ( -\frac{\pi}{2} + \eta_{k},
\frac{\pi}{2} - \eta_{k} ) \label{choice_a_k}
\end{equation}
for some $\eta_{k}>0$ close to 0, provided that $\epsilon$ belongs to the sector
$\mathcal{E}_{HJ_{n}}^{k}$.
\begin{lemma} The function $\varphi_{j,\mathcal{E}_{HJ_{n}}^{k}}(t,\epsilon)$ defines a bounded holomorphic function
on $(\mathcal{T} \cap D(0,r_{\mathcal{T}})) \times \mathcal{E}_{HJ_{n}}^{k}$
for some well selected radius $r_{\mathcal{T}}>0$.
\end{lemma}
\begin{proof} We set
$$ \varphi_{j,\mathcal{E}_{HJ_{n}}^{k}}^{1}(t,\epsilon) = \int_{P_{k,1}} w_{j}(u,\epsilon)
\exp( -\frac{u}{\epsilon t} ) \frac{du}{u}
$$
Since the path $P_{k,1}$ crosses the domains $H_{q},J_{q}$ for some $q \in \llbracket -n,n \rrbracket$, due to
(\ref{bounds_w_initial}), we have the coarse
upper bounds
$$ |w_{j}(\tau,\epsilon)| \leq I_{w_{j}}|\tau| \exp \left( \frac{\sigma_{1}'}{|\epsilon|}|\tau| +
\varsigma_{2}'\exp( \varsigma_{3}' |\tau| ) \right)
$$
for all $\tau \in P_{k,1}$. We deduce the next estimates
\begin{multline*}
|\int_{P_{k,1}} w_{j}(u,\epsilon)
\exp( -\frac{u}{\epsilon t} ) \frac{du}{u}| \leq 
\int_{0}^{|A_{k}|} I_{w_j} \rho \exp \left( \frac{\sigma_{1}'}{|\epsilon|}\rho +
\varsigma_{2}'\exp( \varsigma_{3}' \rho ) \right)\\
\times
\exp( -\frac{\rho}{|\epsilon t|} \cos( \mathrm{arg}(A_{k}) - \mathrm{arg}(\epsilon t) ) ) \frac{d\rho}{\rho}.
\end{multline*}
From the choice of $A_{k}$ fulfilling (\ref{choice_a_k}), we can find some
real number $\delta_{1}>0$ with
$\cos( \mathrm{arg}(A_{k}) - \mathrm{arg}(\epsilon t) ) \geq \delta_{1}$
for all $\epsilon \in \mathcal{E}_{HJ_n}^{k}$. We choose $\delta_{2}>0$ and take $t \in \mathcal{T}$
with $|t| \leq \delta_{1}/(\delta_{2} + \sigma_{1}')$. Then, we get
$$
|\varphi_{j,\mathcal{E}_{HJ_{n}}^{k}}^{1}(t,\epsilon)| \leq
I_{w_j} \int_{0}^{|A_{k}|} \exp( \varsigma_{2}'\exp( \varsigma_{3}' \rho ) ) \exp( -\frac{\rho}{|\epsilon|}\delta_{2} )
d \rho
$$
which implies that $\varphi_{j,\mathcal{E}_{HJ_{n}}^{k}}^{1}(t,\epsilon)$ is bounded holomorphic on
$(\mathcal{T} \cap D(0,\frac{\delta_{1}}{\delta_{2} + \sigma_{1}'})) \times \mathcal{E}_{HJ_{n}}^{k}$.

In a second part, we put
$$ \varphi_{j,\mathcal{E}_{HJ_{n}}^{k}}^{2}(t,\epsilon) = \int_{P_{k,2}} w_{j}(u,\epsilon)
\exp( -\frac{u}{\epsilon t} ) \frac{du}{u}
$$
Since the path $P_{k,2}$ is enclosed in the strip $H_{k}$, using the hypothesis (\ref{bounds_w_initial}), we check the
next estimates
\begin{multline}
|\int_{P_{k,2}} w_{j}(u,\epsilon)
\exp( -\frac{u}{\epsilon t} ) \frac{du}{u}| \\
\leq 
\int_{0}^{+\infty} I_{w_j}|A_{k}-s|
\exp \left( \frac{\sigma_{1}'}{|\epsilon|}|A_{k}-s| - \sigma_{2}'(M-1)\exp(\sigma_{3}'|A_{k}-s|) \right)\\
\times \exp( -\frac{|A_{k}-s|}{|\epsilon t|} \cos( \mathrm{arg}(A_{k}-s) - \mathrm{arg}(\epsilon) - \mathrm{arg}(t) ) )
\frac{ds}{|A_{k}-s|} \label{int_Pk2_w_j}
\end{multline}
From the choice of $A_{k}$ fulfilling (\ref{choice_a_k}), we observe that
\begin{equation}
\mathrm{arg}(A_{k}-s) - \mathrm{arg}(\epsilon) - \mathrm{arg}(t) \in (-\frac{\pi}{2} + \eta_{k},
\frac{\pi}{2} - \eta_{k} ) \label{argument_ak_minus_s}
\end{equation}
for all $s \geq 0$, provided that $\epsilon \in \mathcal{E}_{HJ_{n}}^{k}$. Consequently, we can select some
$\delta_{1}>0$ with $\cos( \mathrm{arg}(A_{k}-s) - \mathrm{arg}(\epsilon) - \mathrm{arg}(t) ) > \delta_{1}$.  We sort
$\delta_{2}>0$ and take $t \in \mathcal{T}$ with $|t| \leq \delta_{1}/(\delta_{2} + \sigma_{1}')$. On the other hand, we may sort
a constant $K_{A_{k}}>0$ (depending on $A_{k}$) for which
$$ |A_{k} - s| \geq K_{A_k}(|A_{k}| + s) $$
whenever $s \geq 0$. Subsequently, we get
\begin{multline*}
|\varphi_{j,\mathcal{E}_{HJ_{n}}^{k}}^{2}(t,\epsilon)| \leq I_{w_j}
\int_{0}^{+\infty} \exp \left( -\sigma_{2}'(M-1) \exp( \sigma_{3}'|A_{k}-s| ) \right)
\exp(-\frac{|A_{k}-s|}{|\epsilon|}\delta_{2}) ds \\
\leq I_{w_j} \int_{0}^{+\infty} \exp( -\frac{K_{A_k}\delta_{2}}{|\epsilon|} (|A_{k}| + s) ) ds =
\frac{I_{w_j}}{K_{A_k}\delta_{2}} |\epsilon| \exp( -\frac{K_{A_k}\delta_{2}}{|\epsilon|} |A_{k}| ).
\end{multline*}
As a consequence, $\varphi_{j,\mathcal{E}_{HJ_{n}}^{k}}^{2}(t,\epsilon)$ represents a bounded holomorphic function
on $(\mathcal{T} \cap D(0, \delta_{1}/(\delta_{2} + \sigma_{1}'))) \times \mathcal{E}_{HJ_{n}}^{k}$. Lemma 8 follows.
\end{proof}

\begin{prop} We make the assumption that the real number $\xi$ introduced in (\ref{cond_SPCP_first}) conforms the next
inequality
\begin{equation}
\xi \leq \min(\sigma_{3}',\varsigma_{3}'). \label{xi_larger_sigma}
\end{equation}
1) There exist some constants $I,\delta>0$ (independent of $\epsilon$) selected in a way that if one assumes
that
\begin{equation}
\sum_{j=0}^{S-1-h} ||w_{j+h}(\tau,\epsilon)||_{(0,\underline{\sigma}',H_{k},\epsilon)}
\frac{\delta^j}{j!} \leq I \ \ , \ \ \sum_{j=0}^{S-1-h}
||w_{j+h}(\tau,\epsilon)||_{(0,\underline{\varsigma}',J_{k},\epsilon)}
\frac{\delta^j}{j!} \leq I \label{norm_w_initial_small}
\end{equation}
for all $0 \leq h \leq S-1$, all $\epsilon \in \dot{D}(0,\epsilon_{0})$, all
$k \in \llbracket -n,n \rrbracket$, then the Cauchy problem (\ref{SPCP_first}), (\ref{SPCP_first_i_d})
with initial data given by (\ref{SPCP_first_i_d_k}) has a solution $u_{\mathcal{E}_{HJ_{n}}^{k}}(t,z,\epsilon)$ which turns
out to be bounded and holomorphic on a domain
$(\mathcal{T} \cap D(0,r_{\mathcal{T}})) \times D(0,\delta\delta_{1}) \times \mathcal{E}_{HJ_{n}}^{k}$ for some fixed
radius $r_{\mathcal{T}}>0$ and $0 < \delta_{1} < 1$.

Furthermore,
$u_{\mathcal{E}_{HJ_{n}}^{k}}$ can be written as a special Laplace transform
\begin{equation}
u_{\mathcal{E}_{HJ_{n}}^{k}}(t,z,\epsilon) = \int_{P_{k}} w_{HJ_{n}}(u,z,\epsilon)
\exp( -\frac{u}{\epsilon t} ) \frac{du}{u} \label{u_E_HJn_k_Laplace}
\end{equation}
where $w_{HJ_n}(\tau,z,\epsilon)$ defines a holomorphic function on
$\mathring{HJ}_{n} \times D(0,\delta \delta_{1}) \times \dot{D}(0,\epsilon_{0})$, continuous
on $HJ_{n} \times D(0,\delta \delta_{1}) \times \dot{D}(0,\epsilon_{0})$ that fulfills the next
constraints. For any choice of two tuples $\underline{\sigma} = (\sigma_{1},\sigma_{2},\sigma_{3})$ and
$\underline{\varsigma} = (\sigma_{1},\varsigma_{2},\varsigma_{3})$ with
\begin{equation}
\sigma_{1} > \sigma_{1}', 0 < \sigma_{2} < \sigma_{2}', \sigma_{3}=\sigma_{3}',
\varsigma_{2}> \varsigma_{2}',\varsigma_{3}=\varsigma_{3}' \label{relations_sigma_sigma_prim}
\end{equation}
there exist a constant $C_{H_k}>0$ and
$C_{J_k}>0$ (independent of $\epsilon$) with
\begin{equation}
|w_{HJ_n}(\tau,z,\epsilon)| \leq C_{H_k}|\tau|
\exp \left( \frac{\sigma_{1}}{|\epsilon|} \zeta(b) |\tau| - \sigma_{2}(M-\zeta(b))\exp( \sigma_{3}|\tau| ) \right)
\label{bds_WHJn_Hk}
\end{equation}
for all $\tau \in H_{k}$, all $z \in D(0,\delta \delta_{1})$ and
\begin{equation}
|w_{HJ_n}(\tau,z,\epsilon)| \leq C_{J_k}|\tau|
\exp \left( \frac{\sigma_{1}}{|\epsilon|} \zeta(b) |\tau| + \varsigma_{2} \zeta(b) \exp( \varsigma_{3}|\tau| ) \right)
\label{bds_WHJn_Jk}
\end{equation}
for all $\tau \in J_{k}$, all $z \in D(0,\delta \delta_{1})$, provided that $\epsilon \in
\dot{D}(0,\epsilon_{0})$, for each $k \in \llbracket -n,n \rrbracket$.\\
2) Let $k \in \llbracket -n,n \rrbracket$ with $k \neq n$. Then, keeping $\epsilon_{0}$ and
$r_{\mathcal{T}}$ small enough, there exist constants $M_{k,1},M_{k,2}>0$ and $M_{k,3}>1$,
independent of $\epsilon$, such that
\begin{equation}
| u_{\mathcal{E}_{HJ_{n}}^{k+1}}(t,z,\epsilon) - u_{\mathcal{E}_{HJ_{n}}^{k}}(t,z,\epsilon) |
\leq M_{k,1} \exp( -\frac{M_{k,2}}{|\epsilon|} \mathrm{Log} \frac{M_{k,3}}{|\epsilon|} ) \label{log_flat_difference_uk_plus_1_minus_uk_HJn}
\end{equation}
for all $t \in \mathcal{T} \cap D(0,r_{\mathcal{T}})$, all $\epsilon \in \mathcal{E}_{HJ_{n}}^{k} \cap
\mathcal{E}_{HJ_{n}}^{k+1} \neq \emptyset$ and
all $z \in D(0,\delta \delta_{1})$.
\end{prop}
\begin{proof}
We consider the equation (\ref{1_aux_CP}) for the given initial data
\begin{equation}
(\partial_{z}^{j}w)(\tau,0,\epsilon) = w_{j}(\tau,\epsilon) \ \ , \ \ 0 \leq j \leq S-1 \label{1_aux_CP_i_d_admissible}
\end{equation}
where $w_{j}(\tau,\epsilon)$ are given above in order to construct the functions
$\varphi_{j,\mathcal{E}_{HJ_n}^{k}}(t,\epsilon)$ in (\ref{SPCP_first_i_d_k}).

In a first step, we check that the problem (\ref{1_aux_CP}), (\ref{1_aux_CP_i_d_admissible}) possesses a unique formal
solution
\begin{equation}
w_{HJ_n}(\tau,z,\epsilon) = \sum_{\beta \geq 0} w_{\beta}(\tau,\epsilon) \frac{z^{\beta}}{\beta !} \label{formal_wHJn} 
\end{equation}
where $w_{\beta}(\tau,\epsilon)$ are holomorphic on $\mathring{HJ}_{n} \times \dot{D}(0,\epsilon_{0})$,
continuous on $HJ_{n} \times \dot{D}(0,\epsilon_{0})$. Namely, if one expands
$c_{\underline{k}}(z,\epsilon) = \sum_{\beta \geq 0} c_{\underline{k},\beta}(\epsilon) z^{\beta}/\beta!$
as Taylor series at $z=0$, the formal series (\ref{formal_wHJn}) is solution of
(\ref{1_aux_CP}), (\ref{1_aux_CP_i_d_admissible}) if and only if the next recursion holds
\begin{equation}
w_{\beta + S}(\tau,\epsilon) = \sum_{\underline{k} = (k_{0},k_{1},k_{2}) \in \mathcal{A}}
\frac{\epsilon^{-k_{0}}\tau^{k_0}}{P(\tau)} \exp( -k_{2} \tau)
\left( \sum_{\beta_{1} + \beta_{2} = \beta} \frac{c_{\underline{k},\beta_{1}}(\epsilon)}{\beta_{1}!}
\frac{w_{\beta_{2}+k_{1}}(\tau,\epsilon)}{\beta_{2}!} \beta! \right) \label{recursion_w_beta}
\end{equation}
for all $\beta \geq 0$. Since the initial data $w_{j}(\tau,\epsilon)$, for $0 \leq j \leq S-1$ are assumed to
define holomorphic functions on $\mathring{HJ}_{n} \times \dot{D}(0,\epsilon_{0})$, continuous on
$HJ_{n} \times \dot{D}(0,\epsilon_{0})$, the recursion (\ref{recursion_w_beta}) implies in particular
that all $w_{n}(\tau,\epsilon)$ for $n \geq S$ are well defined and represent holomorphic functions on
$\mathring{HJ}_{n} \times \dot{D}(0,\epsilon_{0})$, continuous on
$HJ_{n} \times \dot{D}(0,\epsilon_{0})$.

According to the assumption (\ref{cond_SPCP_first}) together with (\ref{xi_larger_sigma}) and the restriction on the size
of the initial data (\ref{norm_w_initial_small}), we notice that the requirements 1)a)b) and 2)a)b) in Proposition 9
are realized. We deduce that\\
1) The formal solution $w_{HJ_n}(\tau,z,\epsilon)$ belongs to the Banach spaces
$SED_{(\underline{\sigma},H_{k},\epsilon,\delta)}$, for all $\epsilon \in \dot{D}(0,\epsilon_{0})$,
all $k \in \llbracket -n,n \rrbracket$, for any tuple
$\underline{\sigma} = (\sigma_{1},\sigma_{2},\sigma_{3})$ chosen as in
(\ref{relations_sigma_sigma_prim}), with an upper bound $\tilde{C}_{H_k}>0$ (independent of $\epsilon$) such that
\begin{equation}
||w_{HJ_n}(\tau,z,\epsilon)||_{(\underline{\sigma},H_{k},\epsilon,\delta)} \leq \tilde{C}_{H_k}, \label{norm_wHJn_Hk}
\end{equation}
for all $\epsilon \in \dot{D}(0,\epsilon_{0})$.\\
2) The formal series $w_{HJ_n}(\tau,z,\epsilon)$ belongs to the Banach spaces
$SEG_{(\underline{\varsigma},J_{k},\epsilon,\delta)}$, for all $\epsilon \in \dot{D}(0,\epsilon_{0})$,
all $k \in \llbracket -n,n \rrbracket$, for any tuple
$\underline{\varsigma} = (\sigma_{1},\varsigma_{2},\varsigma_{3})$ selected as in
(\ref{relations_sigma_sigma_prim}). Besides, we can get a constant $\tilde{C}_{J_k}>0$ (independent of $\epsilon$) with
\begin{equation}
||w_{HJ_n}(\tau,z,\epsilon)||_{(\underline{\varsigma},J_{k},\epsilon,\delta)} \leq \tilde{C}_{J_k}, \label{norm_wHJn_Jk}
\end{equation}
for all $\epsilon \in \dot{D}(0,\epsilon_{0})$.

Bearing in mind (\ref{norm_wHJn_Hk}) and (\ref{norm_wHJn_Jk}), the application of Proposition 1 and Proposition 5 1)
yields in particular the fact that the formal series $w_{HJ_n}(\tau,z,\epsilon)$ actually defines a holomorphic function
on $\mathring{HJ}_{n} \times D(0,\delta \delta_{1}) \times \dot{D}(0,\epsilon_{0})$, continuous on
$HJ_{n} \times D(0,\delta \delta_{1}) \times \dot{D}(0,\epsilon_{0})$, for some $0 < \delta_{1} < 1$, that
satisfies moreover the estimates (\ref{bds_WHJn_Hk}) and (\ref{bds_WHJn_Jk}).

Following the same steps as in the proof of Lemma 8, one can show that for each $k \in \llbracket -n,n \rrbracket$, the
function $u_{\mathcal{E}_{HJ_n}^{k}}$ defined as a special Laplace transform
$$ u_{\mathcal{E}_{HJ_n}^{k}}(t,z,\epsilon) = \int_{P_k} w_{HJ_n}(u,z,\epsilon)
\exp( -\frac{u}{\epsilon t} ) \frac{du}{u} $$
represents a bounded holomorphic function on
$(\mathcal{T} \cap D(0,r_{\mathcal{T}})) \times D(0,\delta_{1}\delta) \times \mathcal{E}_{HJ_{n}}^{k}$ for some fixed
radius $r_{\mathcal{T}}>0$ and $0 < \delta_{1} < 1$. Besides, by a direct computation, we can check that
$u_{\mathcal{E}_{HJ_n}^{k}}(t,z,\epsilon)$
solves the problem (\ref{SPCP_first}), (\ref{SPCP_first_i_d}) with initial data (\ref{SPCP_first_i_d_k})
on $(\mathcal{T} \cap D(0,r_{\mathcal{T}})) \times D(0,\delta_{1}\delta) \times \mathcal{E}_{HJ_{n}}^{k}$.\medskip

In a second part of the proof, we focus our attention to the point 2). Take some $k \in \llbracket -n,n \rrbracket$ with
$k \neq n$. Let us choose two complex numbers
$$ h_{q} = -\varrho \mathrm{Log}( \frac{1}{\epsilon t} e^{i \chi_{q}}) $$
for $q=k,k+1$, where $0 < \varrho < 1$ and where $\chi_{q} \in \mathbb{R}$ are directions selected in a way that
\begin{equation}
i \varrho( \mathrm{arg}(t) + \mathrm{arg}(\epsilon) - \chi_{q} ) \in H_{q} \label{cond_chi_q}
\end{equation}
for all $\epsilon \in \mathcal{E}_{HJ_n}^{k} \cap \mathcal{E}_{HJ_n}^{k+1}$, all $t \in \mathcal{T}$.
Notice that such
directions $\chi_{q}$ always exist for some $0 < \varrho < 1$ small enough since by definition the aperture of
$\mathcal{E}_{HJ_n}^{k} \cap \mathcal{E}_{HJ_n}^{k+1}$ is strictly less than $\pi$, the aperture of
$\mathcal{T}$ is close to 0.
By construction, we get that $h_{q}$ belongs to $H_{q}$ for $q=k,k+1$ since $h_{q}$ can be expressed as
$$ h_{q} = -\varrho \mathrm{Log}|\frac{1}{\epsilon t}| + i \varrho
(\mathrm{arg}(t) + \mathrm{arg}(\epsilon) - \chi_{q} ). $$

From the fact that $u \mapsto w_{HJ_n}(u,z,\epsilon) \exp( -\frac{u}{\epsilon t} )/u$ is holomorphic on
the strip $\mathring{HJ}_{n}$, for any fixed $z \in D(0,\delta \delta_{1})$ and
$\epsilon \in \mathcal{E}_{HJ_n}^{k} \cap \mathcal{E}_{HJ_n}^{k+1}$, by means of a path deformation
argument (according to the classical Cauchy theorem, the integral of a holomorphic function along a closed path is vanishing)
we can rewrite the difference $u_{\mathcal{E}_{HJ_{n}}^{k+1}} - u_{\mathcal{E}_{HJ_{n}}^{k}}$ as a sum of three integrals
\begin{multline}
u_{\mathcal{E}_{HJ_{n}}^{k+1}}(t,z,\epsilon) - u_{\mathcal{E}_{HJ_{n}}^{k}}(t,z,\epsilon) =
- \int_{L_{h_{k},\infty}} w_{HJ_n}(u,z,\epsilon)
\exp( -\frac{u}{\epsilon t} ) \frac{du}{u} \\
+ \int_{L_{h_{k},h_{k+1}}} w_{HJ_n}(u,z,\epsilon)
\exp( -\frac{u}{\epsilon t} ) \frac{du}{u} + \int_{L_{h_{k+1},\infty}} w_{HJ_n}(u,z,\epsilon)
\exp( -\frac{u}{\epsilon t} ) \frac{du}{u} \label{splitting_uk_plus_1_minus_uk}
\end{multline}
where $L_{h_{q},\infty} = \{ h_{q} - s / s \geq 0 \}$ for $q=k,k+1$ are horizontal halflines and
$L_{h_{k},h_{k+1}} = \{ (1-s)h_{k} + sh_{k+1} / s \in [0,1] \}$ is a segment joining $h_{k}$ and $h_{k+1}$. This situation is shown in Figure~\ref{fig2}.

\begin{figure}
	\centering
		\includegraphics[width=0.4\textwidth]{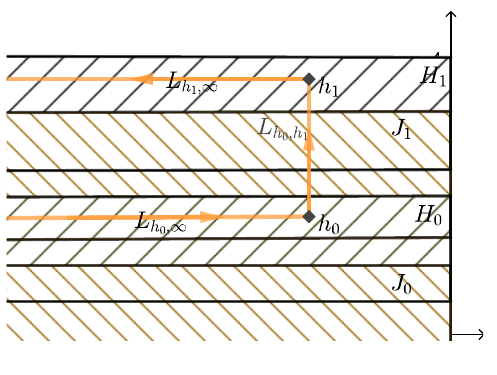}
		\caption{Integration path for the difference of solutions}
		\label{fig2}
\end{figure}

We first furnish estimates for
$$ I_{1} = \left| \int_{L_{h_{k},\infty}} w_{HJ_n}(u,z,\epsilon)
\exp( -\frac{u}{\epsilon t} ) \frac{du}{u} \right|. $$
Since the path $L_{h_{k},\infty}$ is contained inside the strip $H_{k}$, in accordance with the bounds
(\ref{bds_WHJn_Hk}), we reach the estimates
\begin{multline}
I_{1} \leq C_{H_k}\int_{0}^{+\infty} |h_{k}-s| \exp \left( \frac{\sigma_{1}}{|\epsilon|} \zeta(b) |h_{k} - s|
 - \sigma_{2}(M - \zeta(b)) \exp( \sigma_{3} |h_{k} - s| ) \right)\\
\times \exp \left( - \frac{|h_{k} - s|}{|\epsilon t|} \cos( \mathrm{arg}(h_{k} - s) - \mathrm{arg}(\epsilon) 
 - \mathrm{arg}(t) ) \right) \frac{ds}{|h_{k} - s|}
\end{multline}
Provided that $\epsilon_{0}>0$ is chosen small enough, $|\mathrm{Re}(h_{k})| = \varrho \mathrm{Log}(1/|\epsilon t|)$
becomes suitably large and implies the next range
$$ \mathrm{arg}(h_{k} - s) - \mathrm{arg}(\epsilon) - \mathrm{arg}(t) \in
(-\frac{\pi}{2} + \eta_{k}, \frac{\pi}{2} - \eta_{k}) $$
for some $\eta_{k}>0$ close to 0, according that $\epsilon$ belongs to
$\mathcal{E}_{HJ_n}^{k} \cap \mathcal{E}_{HJ_n}^{k+1}$
and $t$ is inside $\mathcal{T}$, for all $s \geq 0$. Consequently, we can select some $\delta_{1}>0$ with
\begin{equation}
 \cos( \mathrm{arg}(h_{k} - s) - \mathrm{arg}(\epsilon) - \mathrm{arg}(t) ) > \delta_{1} \label{low_bds_cos_hk_minus_s}
\end{equation}
for all $s \geq 0$, $t \in \mathcal{T}$ and
$\epsilon \in \mathcal{E}_{HJ_n}^{k} \cap \mathcal{E}_{HJ_n}^{k+1}$.
On the other hand, we can rewrite
\begin{multline*}
|h_{k} - s| = \left( ( \varrho \mathrm{Log}(\frac{1}{|\epsilon t|}) + s)^{2} + \varrho^{2}(
\mathrm{arg}(t) + \mathrm{arg}(\epsilon) - \chi_{k})^{2} \right)^{1/2}\\
= (\varrho \mathrm{Log}(\frac{1}{|\epsilon t|}) + s)
( 1 + \frac{\varrho^{2}(
\mathrm{arg}(t) + \mathrm{arg}(\epsilon) - \chi_{k})^{2}}{ ( \varrho \mathrm{Log}(\frac{1}{|\epsilon t|}) + s)^{2} })^{1/2}
\end{multline*}
provided that $|\epsilon t| < 1$ which holds if one assumes that $0 < \epsilon_{0} < 1$ and $0< r_{\mathcal{T}} < 1$.
For that reason, we get a constant $m_{k}>0$ (depending on $H_k$ and $\varrho$) such that
\begin{equation}
|h_{k} - s| \geq m_{k}( \varrho \mathrm{Log}( \frac{1}{|\epsilon t|} ) + s) \label{low_bds_hk_minus_s}
\end{equation}
for all $s \geq 0$, all $t \in \mathcal{T}$ and
$\epsilon \in \mathcal{E}_{HJ_n}^{k} \cap \mathcal{E}_{HJ_n}^{k+1}$. Now, we select $\delta_{2}>0$ and
take $t \in \mathcal{T}$ with $|t| \leq \delta_{1}/ ( \sigma_{1} \zeta(b) + \delta_{2})$. Then, gathering
(\ref{low_bds_cos_hk_minus_s}) and (\ref{low_bds_hk_minus_s}) yields
\begin{multline}
I_{1} \leq C_{H_k}\int_{0}^{+\infty} \exp \left( \frac{\sigma_{1}}{|\epsilon|} \zeta(b) |h_{k} - s| -
\frac{|h_{k} - s|}{|\epsilon t|} \delta_{1} \right) ds \leq C_{H_k}\int_{0}^{+\infty}
\exp( -\delta_{2} \frac{|h_{k} - s|}{|\epsilon|} ) ds\\
\leq C_{H_k} \exp \left(-\delta_{2} m_{k} \frac{\varrho}{|\epsilon|} \mathrm{Log}( \frac{1}{|\epsilon t|} ) \right)
\int_{0}^{+\infty} \exp( -\delta_{2} m_{k} \frac{s}{|\epsilon|}) ds\\
\leq C_{H_k} \frac{\epsilon_{0}}{\delta_{2}m_{k}}
\exp \left(-\delta_{2} m_{k} \frac{\varrho}{|\epsilon|} \mathrm{Log}( \frac{1}{|\epsilon|r_{\mathcal{T}}} ) \right) \label{I1<=}
\end{multline}
whenever $t \in \mathcal{T} \cap D(0, \delta_{1}/( \sigma_{1} \zeta(b) + \delta_{2}))$ and
$\epsilon \in \mathcal{E}_{HJ_n}^{k} \cap \mathcal{E}_{HJ_n}^{k+1}$.\medskip

Let
$$ I_{2} = \left| \int_{L_{h_{k+1},\infty}} w_{HJ_n}(u,z,\epsilon)
\exp( -\frac{u}{\epsilon t} ) \frac{du}{u} \right|. $$
In a similar manner, we can grab constants $\delta_{1},\delta_{2}>0$ and $m_{k+1}>0$ (depending on
$H_{k+1}$ and $\varrho$) with
\begin{equation}
I_{2} \leq C_{H_{k+1}} \frac{\epsilon_{0}}{\delta_{2}m_{k+1}}
\exp \left(-\delta_{2} m_{k+1} \frac{\varrho}{|\epsilon|} \mathrm{Log}( \frac{1}{|\epsilon|r_{\mathcal{T}}} ) \right) \label{I2<=}
\end{equation}
for all $t \in \mathcal{T} \cap D(0, \delta_{1}/( \sigma_{1} \zeta(b) + \delta_{2}))$ and
$\epsilon \in \mathcal{E}_{HJ_n}^{k} \cap \mathcal{E}_{HJ_n}^{k+1}$.\medskip

In a final step, we need to show estimates for
$$ I_{3} = \left| \int_{L_{h_{k},h_{k+1}}} w_{HJ_n}(u,z,\epsilon)
\exp( -\frac{u}{\epsilon t} ) \frac{du}{u} \right|. $$
We notice that the vertical segment $L_{h_{k},h_{k+1}}$ crosses the strips $H_{k},J_{k+1}$ and $H_{k+1}$ and belongs
to the union $H_{k} \cup J_{k+1} \cup H_{k+1}$. According to (\ref{bds_WHJn_Hk}) and (\ref{bds_WHJn_Jk}), we only have
the rough upper bounds
$$
|w_{HJ_{n}}(\tau,z,\epsilon)| \leq \max( C_{H_k}, C_{J_{k+1}}, C_{H_{k+1}} ) |\tau|
\exp \left( \frac{\sigma_{1}}{|\epsilon|} \zeta(b) |\tau| + \varsigma_{2} \zeta(b) \exp( \varsigma_{3}|\tau|)
\right)
$$
for all $\tau  \in H_{k} \cup J_{k+1} \cup H_{k+1}$, all $z \in D(0,\delta \delta_{1})$, all
$\epsilon \in \dot{D}(0,\epsilon_{0})$. We deduce that
\begin{multline}
I_{3} \leq \max( C_{H_k}, C_{J_{k+1}}, C_{H_{k+1}} )
\int_{0}^{1} |(1-s)h_{k} + sh_{k+1}|\\
\exp \left( \frac{\sigma_{1}}{|\epsilon|} \zeta(b) |(1-s)h_{k} + sh_{k+1}|
+ \varsigma_{2} \zeta(b) \exp( \varsigma_{3} |(1-s)h_{k} + sh_{k+1}| ) \right)\\
\times \exp \left( - \frac{ |(1-s)h_{k} + sh_{k+1}| }{ |\epsilon t| }
\cos( \mathrm{arg}( (1-s)h_{k} + sh_{k+1} ) - \mathrm{arg}(\epsilon) - \mathrm{arg}(t) ) \right)\\
\times
\frac{|h_{k+1} - h_{k}|}{|(1-s)h_{k} + sh_{k+1}|} ds \label{I3<=first}
\end{multline}
Taking for granted that $\epsilon_{0}>0$ is chosen small enough, the quantity
$|\mathrm{Re}((1-s)h_{k} + sh_{k+1})| = \varrho \mathrm{Log}(1 / |\epsilon t|)$ turns out to be large and leads to the next
variation of arguments
$$ \mathrm{arg}( (1-s)h_{k} + sh_{k+1} ) - \mathrm{arg}(\epsilon) - \mathrm{arg}(t) \in
(-\frac{\pi}{2} + \eta_{k,k+1}, \frac{\pi}{2} - \eta_{k,k+1}) $$
for some $\eta_{k,k+1}>0$ close to 0, as
$\epsilon \in \mathcal{E}_{HJ_n}^{k} \cap \mathcal{E}_{HJ_n}^{k+1}$, for $s \in [0,1]$. Therefore,
one can find $\delta_{1}>0$ with
\begin{equation}
\cos( \mathrm{arg}((1-s)h_{k} + sh_{k+1}) - \mathrm{arg}(\epsilon) - \mathrm{arg}(t) ) > \delta_{1}
\label{low_bds_cos_hk_hk_plus_1}
\end{equation}
for all $t \in \mathcal{T}$ and
$\epsilon \in \mathcal{E}_{HJ_n}^{k} \cap \mathcal{E}_{HJ_n}^{k+1}$, when $s \in [0,1]$.
Besides, we can compute the modulus
\begin{multline*}
|(1-s)h_{k} + sh_{k+1}| = \left( (\varrho \mathrm{Log}( \frac{1}{|\epsilon t|} ))^{2} +
\varrho^{2}( \mathrm{arg}(t) + \mathrm{arg}(\epsilon) - (1-s)\chi_{k} - s \chi_{k+1} )^{2} \right)^{1/2}\\
= \varrho \mathrm{Log}( \frac{1}{|\epsilon t|} )( 1 +
\frac{ ( \mathrm{arg}(t) + \mathrm{arg}(\epsilon) - (1-s)\chi_{k} - s \chi_{k+1} )^{2} }{
(\mathrm{Log}( \frac{1}{|\epsilon t|} ))^{2} } )^{1/2}
\end{multline*}
as long as $|\epsilon t|<1$, which occurs whenever $0<\epsilon_{0}<1$ and $0<r_{\mathcal{T}}<1$. Then, when
$\epsilon_{0}$ is taken small enough, we obtain two constants $m_{k,k+1}>0$ and $M_{k,k+1}>0$ with
\begin{equation}
\varrho m_{k,k+1} \mathrm{Log}( \frac{1}{|\epsilon t|} ) \leq
|(1-s)h_{k} + sh_{k+1}| \leq \varrho  M_{k,k+1} \mathrm{Log}( \frac{1}{|\epsilon t|} ) \label{bds_hk_hk_plus_1}
\end{equation}
for all $s \in [0,1]$, when $t \in \mathcal{T}$ and
$\epsilon \in \mathcal{E}_{HJ_n}^{k} \cap \mathcal{E}_{HJ_n}^{k+1}$. Moreover, we remark that
$|h_{k+1} - h_{k}| = \varrho |\chi_{k+1} - \chi_{k}|$. Bearing in mind (\ref{low_bds_cos_hk_hk_plus_1}) together with 
(\ref{bds_hk_hk_plus_1}), we deduce from (\ref{I3<=first}) that the next inequality holds
\begin{multline*}
I_{3} \leq \max( C_{H_k}, C_{J_{k+1}}, C_{H_{k+1}} ) \varrho |\chi_{k+1} - \chi_{k}|\\
\times \exp \left( \frac{\sigma_{1}}{|\epsilon|}\zeta(b) \varrho M_{k,k+1} \mathrm{Log}( \frac{1}{|\epsilon t|} )
+ \varsigma_{2} \zeta(b) \exp( \varsigma_{3} \varrho M_{k,k+1} \mathrm{Log}(\frac{1}{|\epsilon t|}) ) \right) \\
\times \exp \left( -\varrho m_{k,k+1} \frac{1}{|\epsilon t|} \mathrm{Log}( \frac{1}{|\epsilon t|}) \delta_{1} \right)
\end{multline*}
for any $t \in \mathcal{T}$ and $\epsilon \in \mathcal{E}_{HJ_n}^{k} \cap \mathcal{E}_{HJ_n}^{k+1}$.
We choose $0 < \varrho < 1$ in a way that $\varsigma_{3} \varrho M_{k,k+1} \leq 1$. Let $\psi(x) =
\varsigma_{2} \zeta(b) x^{\varsigma_{3}\varrho M_{k,k+1}} - \varrho m_{k,k+1} \delta_{1} x \mathrm{Log}(x)$. Then,
we can check that there exists $B>0$ (depending on $\zeta(b),\varrho,\varsigma_{2},\varsigma_{3},M_{k,k+1},m_{k,k+1},\delta_{1}$) such that
$$ \psi(x) \leq - \frac{ \varrho m_{k,k+1} \delta_{1} }{2} x \mathrm{Log}(x) + B $$
for all $x \geq 1$. We deduce that
\begin{multline*}
I_{3} \leq \max( C_{H_k}, C_{J_{k+1}}, C_{H_{k+1}} ) \varrho |\chi_{k+1} - \chi_{k}|\\
\times \exp \left( \frac{\sigma_{1}}{|\epsilon|}\zeta(b) \varrho M_{k,k+1} \mathrm{Log}( \frac{1}{|\epsilon t|} )
- \frac{\varrho}{2} m_{k,k+1} \delta_{1} \frac{1}{|\epsilon t|} \mathrm{Log}( \frac{1}{|\epsilon t|} ) + B \right)
\end{multline*}
whenever $t \in \mathcal{T}$ and $\epsilon \in \mathcal{E}_{HJ_n}^{k} \cap \mathcal{E}_{HJ_n}^{k+1}$. We select
$\delta_{2}>0$ and take $t \in \mathcal{T}$ with the constraint
$|t| \leq d_{k,k+1}$ where
$$ d_{k,k+1} = \frac{\varrho m_{k,k+1} \delta_{1} / 2}{ \sigma_{1} \zeta(b) \varrho M_{k,k+1} + \delta_{2} }. $$
This last choice implies in particular that
\begin{multline}
I_{3} \leq \max( C_{H_k}, C_{J_{k+1}}, C_{H_{k+1}} ) \varrho |\chi_{k+1} - \chi_{k}|
\exp \left( - \frac{\delta_{2}}{|\epsilon|} \mathrm{Log}( \frac{1}{|\epsilon t|} ) + B \right)\\
\leq \max( C_{H_k}, C_{J_{k+1}}, C_{H_{k+1}} ) \varrho |\chi_{k+1} - \chi_{k}| e^{B}
\exp \left( - \frac{\delta_{2}}{|\epsilon|} \mathrm{Log}( \frac{1}{|\epsilon|r_{\mathcal{T}}}) \right) \label{I3<=}
\end{multline}
provided that $\epsilon \in \mathcal{E}_{HJ_n}^{k} \cap \mathcal{E}_{HJ_n}^{k+1}$.

Finally, starting from the splitting (\ref{splitting_uk_plus_1_minus_uk}) and gathering the upper bounds for the three pieces of this
decomposition (\ref{I1<=}), (\ref{I2<=}) and (\ref{I3<=}), we obtain the anticipated estimates
(\ref{log_flat_difference_uk_plus_1_minus_uk_HJn}).
\end{proof}

\subsection{Construction of sectorial holomorphic solutions in the parameter $\epsilon$ with the help of Banach spaces
with exponential growth on sectors}

In the next definition, we introduce the notion of $\sigma_{1}'-$admissible set in a similar way as in Definition 3.

\begin{defin} We consider an unbounded sector $S_{d}$ with bisecting direction $d \in \mathbb{R}$ with $S_{d} \subset \mathbb{C}_{+}$ and
$D(0,r)$ a disc centered at 0 with radius $r>0$ with the property that no root of $P(\tau)$ belongs to
$\bar{S}_{d} \cup \bar{D}(0,r)$. Let $w(\tau,\epsilon)$ be a holomorphic function on
$(S_{d} \cup D(0,r)) \times \dot{D}(0,\epsilon_{0})$, continuous on $(\bar{S}_{d} \cup \bar{D}(0,r)) \times
\dot{D}(0,\epsilon_{0})$. We assume that for all $\epsilon \in \dot{D}(0,\epsilon_{0})$, the function
$\tau \mapsto w(\tau,\epsilon)$ belongs to the Banach space $EG_{(0,\sigma_{1}',S_{d} \cup D(0,r),\epsilon)}$ for given
$\sigma_{1}'>0$. Besides, the take for granted that some constant $I_{w}>0$, independent of $\epsilon$, exists with the bounds
\begin{equation}
||w(\tau,\epsilon)||_{(0,\sigma_{1}',S_{d} \cup D(0,r),\epsilon)} \leq I_{w} \label{EG_norms_w_Iw}
\end{equation}
for all $\epsilon \in \dot{D}(0,\epsilon_{0})$.

We denote $\mathcal{E}_{S_{d}}$ an open sector centered at 0 within the disc $D(0,\epsilon_{0})$, and let $\mathcal{T}$ be a bounded open sector centered at 0 with bisecting direction $d=0$ suitably chosen in a way that for all
$t \in \mathcal{T}$, all $\epsilon \in \mathcal{E}_{S_d}$, there exists a direction $\gamma_{d}$ (depending on $t$,$\epsilon$) such that
$\exp( \sqrt{-1} \gamma_{d}) \in S_{d}$
with
\begin{equation}
\gamma_{d} - \mathrm{arg}(t) - \mathrm{arg}(\epsilon) \in (-\frac{\pi}{2} + \eta, \frac{\pi}{2}- \eta) \label{relation_gamma_epsilon_t} 
\end{equation}
for some $\eta > 0$ close to 0.

The data $(w(\tau,\epsilon), \mathcal{E}_{S_d}, \mathcal{T})$ are said to be $\sigma_{1}'-$admissible.
\end{defin}

For all $0 \leq j \leq S-1$, all $0 \leq p \leq \iota - 1$ for some integer $\iota \geq 2$, we sort directions $d_{p} \in \mathbb{R}$,
unbounded sectors $S_{d_p}$ and corresponding bounded sectors $\mathcal{E}_{S_{d_p}}$, $\mathcal{T}$ such that the next given sets
$(w_{j}(\tau,\epsilon), \mathcal{E}_{S_{d_p}}, \mathcal{T})$ are $\sigma_{1}'-$admissible for some $\sigma_{1}'>0$. We assume moreover
that for each $0 \leq j \leq S-1$, $\tau \mapsto w_{j}(\tau,\epsilon)$ restricted to $S_{d_p}$ is an analytic continuation of a common
holomorphic function $\tau \mapsto w_{j}(\tau,\epsilon)$ on $D(0,r)$, for all $0 \leq p \leq \iota-1$. We adopt the convention that
$d_{p} < d_{p+1}$ and $S_{d_p} \cap S_{d_{p+1}} = \emptyset$ for all $0 \leq p \leq \iota-2$. As initial data
(\ref{SPCP_first_i_d}), we put
\begin{equation}
\varphi_{j,\mathcal{E}_{S_{d_p}}}(t,\epsilon) = \int_{L_{\gamma_{d_p}}} w_{j}(u,\epsilon) \exp( - \frac{u}{\epsilon t} ) \frac{du}{u}
\label{Laplace_varphi_j_along_halfline}
\end{equation}
where the integration path $L_{\gamma_{d_p}} = \mathbb{R}_{+}\exp(\sqrt{-1} \gamma_{d_p})$ is a halfline in direction $\gamma_{d_p}$ defined in
(\ref{relation_gamma_epsilon_t}).
\begin{lemma}
For all $0 \leq j \leq S-1$, $0 \leq p \leq \iota-1$, the Laplace integral $\varphi_{j,\mathcal{E}_{S_{d_p}}}(t,\epsilon)$
determines a bounded holomorphic function on $(\mathcal{T} \cap D(0,r_{\mathcal{T}})) \times \mathcal{E}_{S_{d_p}}$ for some
suitable radius $r_{\mathcal{T}}>0$.
\end{lemma}
\begin{proof} According to (\ref{EG_norms_w_Iw}), each function $w_{j}(\tau,\epsilon)$ satisfies the upper bounds
\begin{equation}
|w_{j}(\tau,\epsilon)| \leq I_{w_j} |\tau| \exp \left( \frac{\sigma_{1}'}{|\epsilon|} |\tau| \right) \label{bds_w_j_varsigma} 
\end{equation}
for some constant $I_{w_j}>0$, whenever $\tau \in \bar{S}_{d_p} \cup \bar{D}(0,r)$, $\epsilon \in \dot{D}(0,\epsilon_{0})$.
Besides, due to (\ref{relation_gamma_epsilon_t}), we can grasp a constant $\delta_{1}>0$ with
\begin{equation}
\cos( \gamma_{d_p} - \mathrm{arg}(t) - \mathrm{arg}(\epsilon) ) \geq \delta_{1} \label{cos >= delta1} 
\end{equation}
for any $t \in \mathcal{T}$, $\epsilon \in \mathcal{E}_{S_{d_p}}$. We choose $\delta_{2}>0$ and take $t \in \mathcal{T}$ with
$|t| \leq \frac{\delta_{1}}{\delta_{2} + \sigma_{1}'}$. Then, collecting (\ref{bds_w_j_varsigma}) and (\ref{cos >= delta1}) allows us to
write
\begin{multline}
|\varphi_{j,\mathcal{E}_{S_{d_p}}}(t,\epsilon)| \leq \int_{0}^{+\infty} I_{w_j} \rho \exp( \frac{\sigma_{1}'}{|\epsilon|} \rho)
\exp( -\frac{\rho}{|\epsilon t|} \cos( \gamma_{d_p} - \mathrm{arg}(t) - \mathrm{arg}(\epsilon) ) \frac{d \rho}{\rho}\\
\leq I_{w_j} \int_{0}^{+\infty} \exp( -\frac{\rho}{|\epsilon|} \delta_{2} ) d\rho = I_{w_j} \frac{|\epsilon|}{\delta_{2}}
\end{multline}
which implies in particular that $\varphi_{j,\mathcal{E}_{S_{d_p}}}(t,\epsilon)$ is holomorphic and bounded on
$(\mathcal{T} \cap D(0, \frac{\delta_{1}}{\delta_{2} + \sigma_{1}'})) \times \mathcal{E}_{S_{d_p}}$.
\end{proof}

In the next proposition, we construct actual holomorphic solutions of the problem (\ref{SPCP_first}), (\ref{SPCP_first_i_d})
as Laplace transforms along halflines.
\begin{prop} 1) There exist two constants $I,\delta>0$ (independent of $\epsilon$) such that if one takes for granted that
\begin{equation}
\sum_{j=0}^{S-1-h} ||w_{j+h}(\tau,\epsilon)||_{(0,\sigma_{1}',S_{d_p} \cup D(0,r),\epsilon)} \frac{\delta^j}{j!} \leq I
\label{norm_Sdp_initial_wj}
\end{equation}
for all $0 \leq h \leq S-1$, all $\epsilon \in \dot{D}(0,\epsilon_{0})$, all $0 \leq p \leq \iota-1$, then the
Cauchy problem (\ref{SPCP_first}), (\ref{SPCP_first_i_d}) for initial conditions given by
(\ref{Laplace_varphi_j_along_halfline}) possesses a solution $u_{\mathcal{E}_{S_{d_p}}}(t,z,\epsilon)$ which represents a
bounded holomorphic function on a domain $(\mathcal{T} \cap D(0,r_{\mathcal{T}})) \times D(0,\delta_{1}\delta) \times
\mathcal{E}_{S_{d_p}}$, for suitable radius $r_{\mathcal{T}}>0$ and with $0 < \delta_{1} < 1$. Additionally,
$u_{\mathcal{E}_{S_{d_p}}}$ turns out to be a Laplace transform
\begin{equation}
u_{\mathcal{E}_{S_{d_p}}}(t,z,\epsilon) = \int_{L_{\gamma_{d_p}}} w_{S_{d_p}}(u,z,\epsilon) \exp( -\frac{u}{\epsilon t} ) \frac{du}{u}
\label{u_E_Sdp_Laplace}
\end{equation}
where $w_{S_{d_p}}(u,z,\epsilon)$ stands for a holomorphic function on $(S_{d_p} \cup D(0,r)) \times D(0,\delta \delta_{1})
\times \dot{D}(0,\epsilon_{0})$, continuous on $(\bar{S}_{d_p} \cup \bar{D}(0,r)) \times D(0,\delta \delta_{1})
\times \dot{D}(0,\epsilon_{0})$ which obeys the following restriction : for any choice of $\sigma_{1} > \sigma_{1}'$, we can find
a constant $C_{S_{d_p}}>0$ (independent of $\epsilon$) with
\begin{equation}
|w_{S_{d_p}}(\tau,z,\epsilon)| \leq C_{S_{d_p}} |\tau| \exp( \frac{\sigma_{1}}{|\epsilon|} \zeta(b) |\tau| )
\label{bds_w_Sdp}
\end{equation}
for all $\tau \in S_{d_p} \cup D(0,r)$, all $z \in D(0,\delta \delta_{1})$, whenever $\epsilon \in \dot{D}(0,\epsilon_{0})$.

\noindent 2) Let $0 \leq p \leq \iota-2$. Provided that $r_{\mathcal{T}}>0$ is taken small enough, there exist two constants
$M_{p,1},M_{p,2}>0$ (independent of $\epsilon$) such that
\begin{equation}
|u_{\mathcal{E}_{S_{d_{p+1}}}}(t,z,\epsilon) - u_{\mathcal{E}_{S_{d_p}}}(t,z,\epsilon)| \leq M_{p,1}\exp( - \frac{M_{p,2}}{|\epsilon|} )
\label{difference_u_Sdp_exp_small}
\end{equation}
for all $t \in \mathcal{T} \cap D(0,r_{\mathcal{T}})$, all $\epsilon \in \mathcal{E}_{S_{d_{p+1}}} \cap \mathcal{E}_{S_{d_p}} \neq \emptyset$
and all $z \in D(0,\delta \delta_{1})$.
\end{prop}
\begin{proof}
The first step follows the one performed in Proposition 10. Namely, we can check that the problem (\ref{1_aux_CP}) with initial data
\begin{equation}
(\partial_{z}^{j}w)(\tau,0,\epsilon) = w_{j}(\tau,\epsilon) \ \ , \ \ 0 \leq j \leq S-1 \label{1_aux_CP_i_d_Sd} 
\end{equation}
given above in the $\sigma_{1}'-$admissible sets appearing in the Laplace integrals (\ref{Laplace_varphi_j_along_halfline}), owns a unique
formal solution
\begin{equation}
w_{S_{d_p}}(\tau,z,\epsilon) = \sum_{\beta \geq 0} w_{\beta}(\tau,\epsilon) \frac{z^{\beta}}{\beta !} \label{defin_w_S_dp}
\end{equation}
where $w_{\beta}(\tau,\epsilon)$ define holomorphic functions on $(S_{d} \cup D(0,r)) \times \dot{D}(0,\epsilon_{0})$,
continuous on $(\bar{S}_{d} \cup \bar{D}(0,r)) \times \dot{D}(0,\epsilon_{0})$. Namely, the formal expansion
(\ref{defin_w_S_dp}) solves (\ref{1_aux_CP}) together with (\ref{1_aux_CP_i_d_Sd}) if and only if the recursion
(\ref{recursion_w_beta}) holds. As a result, it implies that all the coefficients $w_{n}(\tau,\epsilon)$ for $n \geq S$ represent
holomorphic functions on $(S_{d_p} \cup D(0,r)) \times \dot{D}(0,\epsilon_{0})$, continuous on
$(\bar{S}_{d_p} \cup \bar{D}(0,r)) \times \dot{D}(0,\epsilon_{0})$ since this property already holds for the initial data
$w_{j}(\tau,\epsilon)$, $0 \leq j \leq S-1$, under our assumption (\ref{EG_norms_w_Iw}).

The assumption (\ref{cond_SPCP_first}) and the control on the norm range of the initial data (\ref{norm_Sdp_initial_wj}), let us
figure out that the demands 3)a)b) in Proposition 9 are scored. In particular, the formal series $w_{S_{d_p}}(\tau,z,\epsilon)$ is located
in the Banach space $EG_{(\sigma_{1},S_{d_p} \cup D(0,r),\epsilon,\delta)}$, for all
$\epsilon \in \dot{D}(0,\epsilon_{0})$, for any real number $\sigma_{1}>\sigma_{1}'$, with a constant $\tilde{C}_{S_{d_p}}>0$
(independent of $\epsilon$) for which
$$ ||w_{S_{d_p}}(\tau,z,\epsilon)||_{(\sigma_{1},S_{d_p} \cup D(0,r),\epsilon,\delta)} \leq \tilde{C}_{S_{d_p}} $$
holds for all $\epsilon \in \dot{D}(0,\epsilon_{0})$. With the help of Proposition 5 2), we notice that the formal expansion
$w_{S_{d_p}}(\tau,z,\epsilon)$ turns out to be an actual holomorphic function on $(S_{d_p} \cup D(0,r)) \times D(0,\delta \delta_{1})
\times \dot{D}(0,\epsilon_{0})$, continuous on $(\bar{S}_{d_p} \cup \bar{D}(0,r)) \times D(0,\delta \delta_{1})
\times \dot{D}(0,\epsilon_{0})$ for some $0 < \delta_{1} < 1$, that conforms to the bounds (\ref{bds_w_Sdp}).

By proceeding with the same lines of arguments as in Lemma 9, one can see that the function $u_{\mathcal{E}_{S_{d_p}}}$ defined as Laplace
transform
$$ u_{\mathcal{E}_{S_{d_p}}}(t,z,\epsilon) = \int_{L_{\gamma_{d_p}}} w_{S_{d_p}}(u,z,\epsilon) \exp( -\frac{u}{\epsilon t} ) \frac{du}{u} $$
represents a bounded holomorphic function on $(\mathcal{T} \cap D(0,r_{\mathcal{T}})) \times D(0,\delta \delta_{1}) \times
\mathcal{E}_{S_{d_p}}$, for suitably small radius $r_{\mathcal{T}}>0$ and given $0 < \delta_{1} < 1$. Furthermore, by direct inspection,
one can testify that $u_{\mathcal{E}_{S_{d_p}}}(t,z,\epsilon)$ solves the problem (\ref{SPCP_first}), (\ref{SPCP_first_i_d}) for initial
conditions (\ref{Laplace_varphi_j_along_halfline}) on $(\mathcal{T} \cap D(0,r_{\mathcal{T}})) \times D(0,\delta \delta_{1}) \times
\mathcal{E}_{S_{d_p}}$.

In the last part of the proof, we concentrate on the second point 2). Let $0 \leq p \leq \iota-2$. We depart from the observation that the
maps $u \mapsto w_{S_{d_q}}(u,z,\epsilon) \exp( -\frac{u}{\epsilon t} )/u$, for $q=p,p+1$, represent analytic continuations on the
sectors $S_{d_q}$ of a common analytic function defined on $D(0,r)$ (since
$w_{S_{d_p}}(u,z,\epsilon) = w_{S_{d_{p+1}}}(u,z,\epsilon)$ for $u \in D(0,r)$), for all fixed $z \in D(0,\delta \delta_{1})$ and
$\epsilon \in \mathcal{E}_{S_{d_p}} \cap \mathcal{E}_{S_{d_{p+1}}}$. Therefore, by carrying out a path deformation inside the domain
$S_{d_p} \cup S_{d_{p+1}} \cup D(0,r)$, we can recast the difference $u_{\mathcal{E}_{S_{d_{p+1}}}} - u_{\mathcal{E}_{S_{d_p}}}$ as a sum of
three paths integrals
\begin{multline}
u_{\mathcal{E}_{S_{d_{p+1}}}}(t,z,\epsilon) - u_{\mathcal{E}_{S_{d_p}}}(t,z,\epsilon) =\\
-\int_{L_{\gamma_{d_p},r/2}} w_{S_{d_p}}(u,z,\epsilon) \exp( -\frac{u}{\epsilon t} ) \frac{du}{u}
+ \int_{C_{\gamma_{d_p},\gamma_{d_{p+1}},r/2}} w_{S_{d_p}}(u,z,\epsilon) \exp( -\frac{u}{\epsilon t} ) \frac{du}{u} \\
+ \int_{L_{\gamma_{d_{p+1}},r/2}} w_{S_{d_{p+1}}}(u,z,\epsilon) \exp( -\frac{u}{\epsilon t} ) \frac{du}{u} \label{decomp_difference_u_Sdp}
\end{multline}
where $L_{\gamma_{d_p},r/2} = [r/2,+\infty)\exp( \sqrt{-1} \gamma_{d_q})$ are unbounded segments for $q=p,p+1$, 
$C_{\gamma_{d_p},\gamma_{d_{p+1}},r/2}$ stands for the arc of circle with radius $r/2$ joining the points
$\frac{r}{2}\exp(\sqrt{-1}\gamma_{d_p})$ and $\frac{r}{2}\exp(\sqrt{-1}\gamma_{d_{p+1}})$.

As an initial step, we provide estimates for
$$ I_{1} = \left| \int_{L_{\gamma_{d_p},r/2}} w_{S_{d_p}}(u,z,\epsilon) \exp( -\frac{u}{\epsilon t} ) \frac{du}{u} \right|. $$
Due to the bounds (\ref{bds_w_Sdp}), we check that
$$ I_{1} \leq \int_{r/2}^{+\infty} C_{S_{d_p}} \rho \exp( \frac{\sigma_{1}}{|\epsilon|} \zeta(b) \rho )
\exp( - \frac{\rho}{|\epsilon t|} \cos( \gamma_{d_p} - \mathrm{arg}(t) - \mathrm{arg}(\epsilon) ) ) \frac{d\rho}{\rho} $$
for all $t \in \mathcal{T}$, $\epsilon \in \mathcal{E}_{S_{d_p}} \cap \mathcal{E}_{S_{d_{p+1}}}$. Besides, the lower bounds
(\ref{cos >= delta1}) hold for some constant
$\delta_{1}>0$ when $t \in \mathcal{T}$ and $\epsilon \in \mathcal{E}_{S_{d_p}} \cap \mathcal{E}_{S_{d_{p+1}}}$. Hence, if we select
$\delta_{2}>0$ and choose
$t \in \mathcal{T}$ with $|t| \leq \frac{\delta_{1}}{\delta_{2} + \sigma_{1} \zeta(b)}$, we get
\begin{equation}
I_{1} \leq C_{S_{d_p}} \int_{r/2}^{+\infty} \exp( -\frac{\rho}{|\epsilon|} \delta_{2} ) d\rho =
C_{S_{d_p}} \frac{|\epsilon|}{\delta_{2}} \exp( -\frac{r \delta_{2}}{2 |\epsilon|} ) \label{I1_Sdp}
\end{equation}
for all $\epsilon \in \mathcal{E}_{S_{d_{p+1}}} \cap \mathcal{E}_{S_{d_p}}$. Now, let
$$ I_{2} = \left| \int_{L_{\gamma_{d_{p+1}},r/2}} w_{S_{d_{p+1}}}(u,z,\epsilon) \exp( -\frac{u}{\epsilon t} ) \frac{du}{u} \right|. $$
With a comparable approach, we can obtain two constants $\delta_{1},\delta_{2}>0$ with
\begin{equation}
I_{2} \leq C_{S_{d_{p+1}}} \frac{|\epsilon|}{\delta_{2}} \exp( -\frac{r \delta_{2}}{2 |\epsilon|} ) \label{I2_Sdp} 
\end{equation}
for $t \in \mathcal{T} \cap D(0,\frac{\delta_{1}}{\delta_{2} + \sigma_{1} \zeta(b)})$ and
$\epsilon \in \mathcal{E}_{S_{d_{p+1}}} \cap \mathcal{E}_{S_{d_p}}$.

In a closing step, we focus on
$$ I_{3} = \left| \int_{C_{\gamma_{d_p},\gamma_{d_{p+1}},r/2}} w_{S_{d_p}}(u,z,\epsilon) \exp( -\frac{u}{\epsilon t} ) \frac{du}{u} \right|. $$
Again, according to (\ref{bds_w_Sdp}), we guarantee that
$$ I_{3} \leq C_{S_{d_p}} \int_{\gamma_{d_p}}^{\gamma_{d_{p+1}}} \frac{r}{2} \exp( \frac{\sigma_{1}}{|\epsilon|} \zeta(b) \frac{r}{2})
\exp( -\frac{r/2}{|\epsilon t|} \cos( \theta - \mathrm{arg}(t) -\mathrm{arg}(\epsilon) ) ) d\theta. $$
By construction, we also get a constant $\delta_{1}>0$ for which
$$ \cos( \theta - \mathrm{arg}(t) - \mathrm{arg}(\epsilon) ) \geq \delta_{1} $$
when $\epsilon \in \mathcal{E}_{S_{d_{p+1}}} \cap \mathcal{E}_{S_{d_p}}$, $t \in \mathcal{T}$ and $\theta \in (\gamma_{d_p},\gamma_{d_{p+1}})$.
As a consequence, if one takes $\delta_{2}>0$ and selects $t \in \mathcal{T}$ with
$|t| \leq \frac{\delta_{1}}{\sigma_{1} \zeta(b) + \delta_{2}}$. Then,
\begin{equation}
I_{3} \leq C_{S_{d_p}} (\gamma_{d_{p+1}} - \gamma_{d_{p}}) \frac{r}{2} \exp( -\frac{r \delta_{2}}{2 |\epsilon|} ) \label{I3_Sdp}
\end{equation}
for all $\epsilon \in \mathcal{E}_{S_{d_{p+1}}} \cap \mathcal{E}_{S_{d_p}}$.

At last, departing from the decomposition (\ref{decomp_difference_u_Sdp}) and clustering the bounds (\ref{I1_Sdp}), (\ref{I2_Sdp}) and
(\ref{I3_Sdp}), we reach our expected estimates (\ref{difference_u_Sdp_exp_small}).
\end{proof}

\subsection{Construction of a finite set of holomorphic solutions when the parameter $\epsilon$ belongs to a good covering of the
origin in $\mathbb{C}^{\ast}$}

Let $n \geq 1$ and $\iota \geq 2$ be integers. We consider two collections of open bounded sectors
$\{ \mathcal{E}_{HJ_n}^{k} \}_{k \in \llbracket -n, n \rrbracket}$, $\{ \mathcal{E}_{S_{d_p}} \}_{0 \leq p \leq \iota-1}$ and
a bounded sector $\mathcal{T}$ with bisecting direction $d=0$ together with a family of functions $w_{j}(\tau,\epsilon)$,
$0 \leq j \leq S-1$ for which the data $(w_{j}(\tau,\epsilon), \mathcal{E}_{HJ_n}^{k}, \mathcal{T})$ are
$(\underline{\sigma}',\underline{\varsigma}')-$admissible in the sense of Definition 3 for some tuples
$\underline{\sigma}' = (\sigma_{1}',\sigma_{2}',\sigma_{3}')$
and $\underline{\varsigma}' = (\sigma_{1}',\varsigma_{2}',\varsigma_{3}')$ (where $\sigma_{1}'>0$, $\sigma_{j}',\varsigma_{j}'>0$ for $j=2,3$)
for $k \in \llbracket -n,n \rrbracket$
and $(w_{j}(\tau,\epsilon), \mathcal{E}_{S_{d_p}}, \mathcal{T})$ are $\sigma_{1}'-$admissible according to Definition 4 for
$0 \leq p \leq \iota-1$.\medskip

\noindent We make the next additional assumptions:\medskip

\noindent 1) For each $0 \leq j \leq S-1$, the map $\tau \mapsto w_{j}(\tau,\epsilon)$ restricted to $S_{d_p}$, for $0 \leq p \leq \iota-1$
and to $\mathring{HJ}_{n}$ is the analytic continuation of a common holomorphic function $\tau \mapsto w_{j}(\tau,\epsilon)$
on $D(0,r)$, for all $\epsilon \in \dot{D}(0,\epsilon_{0})$. Moreover, the radius $r$ is taken small enough such that
$D(0,r) \cap \{ z \in \mathbb{C} / \mathrm{Re}(z) \leq 0 \} \subset J_{0}$.\\
2) We assume that $d_{p} < d_{p+1}$ and $S_{d_p} \cap S_{d_{p+1}} = \emptyset$ for $0 \leq p \leq \iota-2$.\\
3) We take for granted that\\
3.1) $\mathcal{E}_{HJ_n}^{k} \cap \mathcal{E}_{HJ_n}^{k+1} \neq \emptyset$ for $-n \leq k \leq n-1$.\\
3.2) $\mathcal{E}_{S_{d_{p+1}}} \cap \mathcal{E}_{S_{d_p}} \neq \emptyset$ for $0 \leq p \leq \iota-2$.\\
3.3) $\mathcal{E}_{HJ_n}^{-n} \cap \mathcal{E}_{S_{d_0}} \neq \emptyset$ and
$\mathcal{E}_{HJ_n}^{n} \cap \mathcal{E}_{S_{d_{\iota-1}}} \neq \emptyset$.\\
4) We ask that
$$ ( \bigcup_{k=-n}^{n} \mathcal{E}_{HJ_n}^{k} ) \cup ( \bigcup_{p=0}^{\iota-1} \mathcal{E}_{S_{d_p}} ) = \mathcal{U}
\setminus \{ 0 \} $$
where $\mathcal{U}$ stands for some neighborhood of 0 in $\mathbb{C}$.\\
5) Among the set of sectors $\underline{\mathcal{E}} = \{ \mathcal{E}_{HJ_n}^{k} \}_{k \in \llbracket -n, n \rrbracket} \bigcup
\{ \mathcal{E}_{S_{d_p}} \}_{0 \leq p \leq \iota-1}$, every tuple of three sectors has empty intersection.\medskip

In the literature, when the requirements 3),4) and 5) hold, the set $\underline{\mathcal{E}}$ is called a good covering in
$\mathbb{C}^{\ast}$, see for instance \cite{ba1} or \cite{hssi}. An example of a good covering for $n=1$ and $\iota=2$ is displayed in Figure~\ref{fig3}

\begin{figure}
	\centering
		\includegraphics[width=0.4\textwidth]{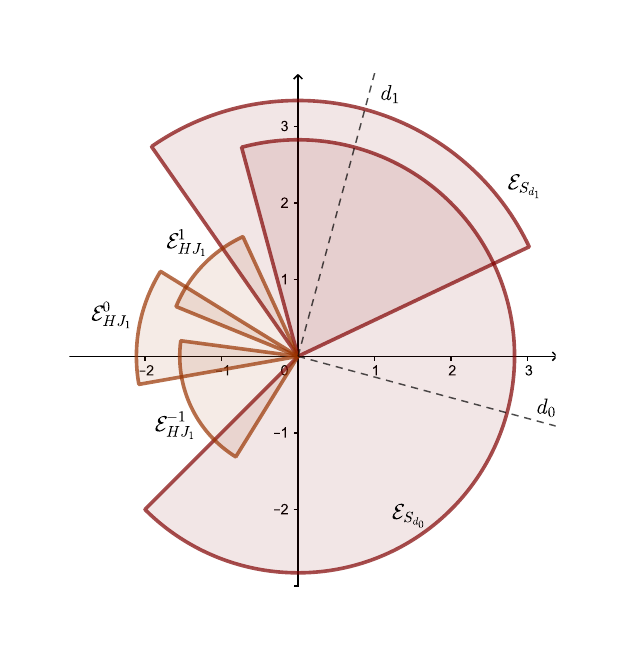}
		\caption{Example of good covering, $n=1$ and $\iota=2$}
		\label{fig3}
\end{figure}

We can state the first main result of our work.

\begin{theo}
Under the claim that the control on the initial data (\ref{norm_w_initial_small}) in Proposition 10 and
(\ref{norm_Sdp_initial_wj}) in Proposition 11 holds together with the restrictions (\ref{cond_SPCP_first}),
(\ref{xi_larger_sigma}), the next statements come forth.

1) The Cauchy problem (\ref{SPCP_first}), (\ref{SPCP_first_i_d}) with initial data given by (\ref{SPCP_first_i_d_k}) has a bounded holomorphic
solution $u_{\mathcal{E}_{HJ_n}^{k}}(t,z,\epsilon)$ on a domain
$(\mathcal{T} \cap D(0,r_{\mathcal{T}})) \times D(0,\delta \delta_{1}) \times \mathcal{E}_{HJ_n}^{k}$ for some radius $r_{\mathcal{T}}>0$
taken small enough. Furthermore, $u_{\mathcal{E}_{HJ_n}^{k}}$ can be written as a special Laplace transform (\ref{u_E_HJn_k_Laplace}) of a
function $w_{HJ_{n}}(\tau,z,\epsilon)$ fulfilling the bounds (\ref{bds_WHJn_Hk}), (\ref{bds_WHJn_Jk}). Besides, the
logarithmic tameness constraints
(\ref{log_flat_difference_uk_plus_1_minus_uk_HJn}) hold for all consecutive sectors $\mathcal{E}_{HJ_n}^{k}$,
$\mathcal{E}_{HJ_n}^{k+1}$ for $-n \leq k \leq n-1$.

2) The Cauchy problem (\ref{SPCP_first}), (\ref{SPCP_first_i_d}) for initial conditions (\ref{Laplace_varphi_j_along_halfline}) owns
a solution $u_{\mathcal{E}_{S_{d_p}}}(t,z,\epsilon)$ which is bounded and holomorphic on
$(\mathcal{T} \cap D(0,r_{\mathcal{T}})) \times D(0,\delta \delta_{1}) \times \mathcal{E}_{S_{d_p}}$ for some well chosen radius
$r_{\mathcal{T}}>0$. Moreover, $u_{\mathcal{E}_{S_{d_p}}}$ can be expressed through a Laplace transform (\ref{u_E_Sdp_Laplace})
of a function $w_{S_{d_p}}(\tau,z,\epsilon)$ that undergoes (\ref{bds_w_Sdp}). Conjointly, the flatness estimates
(\ref{difference_u_Sdp_exp_small}) occur for any neighboring sectors $\mathcal{E}_{S_{d_{p+1}}}$,
$\mathcal{E}_{S_{d_{p}}}$, $0 \leq p \leq \iota-2$.

3) Provided that $r_{\mathcal{T}}>0$ is close to 0, there exist constants $M_{n,1},M_{n,2}>0$ (independent of $\epsilon$) with
\begin{equation}
| u_{\mathcal{E}_{HJ_n}^{-n}}(t,z,\epsilon) - u_{\mathcal{E}_{S_{d_0}}}(t,z,\epsilon) | \leq M_{n,1} \exp( -\frac{M_{n,2}}{|\epsilon|} )
\label{difference_u_HJn_Sd0}
\end{equation}
for all $\epsilon \in \mathcal{E}_{HJ_n}^{-n} \cap \mathcal{E}_{S_{d_0}}$ and
\begin{equation}
| u_{\mathcal{E}_{HJ_n}^{n}}(t,z,\epsilon) - u_{\mathcal{E}_{S_{d_{\iota-1}}}}(t,z,\epsilon) | \leq M_{n,1} \exp( -\frac{M_{n,2}}{|\epsilon|} )
\label{difference_u_HJn_Sdiota}
\end{equation}
for all $\epsilon \in \mathcal{E}_{HJ_n}^{n} \cap \mathcal{E}_{S_{d_{\iota-1}}}$ whenever $t \in \mathcal{T} \cap D(0, r_{\mathcal{T}})$ and
$z \in D(0,\delta \delta_{1})$.
\end{theo}
\begin{proof}
The first two points 1) and 2) merely rephrase the statements already obtained in Propositions 10 and 11. It remains to show that the two
exponential bounds (\ref{difference_u_HJn_Sd0}) and (\ref{difference_u_HJn_Sdiota}) hold. We aim our attention only at the first estimates
(\ref{difference_u_HJn_Sd0}), the second ones (\ref{difference_u_HJn_Sdiota}) being of the same nature.

By construction, according to our additional assumption 1) described above, the functions
$\tau \mapsto w_{HJ_n}(\tau,z,\epsilon)$ on $\mathring{HJ}_{n}$ and $\tau \mapsto w_{S_{d_0}}(\tau,z,\epsilon)$
on $S_{d_0}$ are the restrictions of an holomorphic function denoted $\tau \mapsto w_{HJ_{n},S_{d_0}}(\tau,z,\epsilon)$
on $\mathring{HJ}_{n} \cup D(0,r) \cup S_{d_0}$, for all $z \in D(0,\delta \delta_{1})$, $\epsilon \in \dot{D}(0,\epsilon_{0})$.
As a consequence, we can realize a path deformation within the domain $\mathring{HJ}_{n} \cup D(0,r) \cup S_{d_0}$ and break up the difference
$u_{\mathcal{E}_{HJ_n}^{-n}} - u_{\mathcal{E}_{S_{d_0}}}$ into a sum of four path integrals
\begin{multline}
u_{\mathcal{E}_{HJ_n}^{-n}}(t,z,\epsilon) - u_{\mathcal{E}_{S_{d_0}}}(t,z,\epsilon)
= -\int_{L_{\gamma_{d_0},r/2}} w_{S_{d_0}}(u,z,\epsilon) \exp( -\frac{u}{\epsilon t} ) \frac{du}{u}\\
+ \int_{C_{\gamma_{d_0},P_{-n,1},r/2}} w_{S_{d_0}}(u,z,\epsilon) \exp( -\frac{u}{\epsilon t} ) \frac{du}{u}
+ \int_{P_{-n,1,r/2}} w_{HJ_n}(u,z,\epsilon) \exp( -\frac{u}{\epsilon t} ) \frac{du}{u} \\+
\int_{P_{-n,2}} w_{HJ_n}(u,z,\epsilon) \exp( -\frac{u}{\epsilon t} ) \frac{du}{u} \label{difference_u_HJn_Sd0_decomposition}
\end{multline}
where $L_{\gamma_{d_0},r/2} = [r/2,+\infty) \exp( \sqrt{-1} \gamma_{d_0})$ is an unbounded segment,
$C_{\gamma_{d_0},P_{-n,1},r/2}$ represents an arc of circle with radius $r/2$ joining the two points
$(r/2)\exp( \sqrt{-1}\gamma_{d_0})$ and \\
$(r/2)\exp( \sqrt{-1} \mathrm{arg}(A_{-n}))$, $P_{-n,1,r/2}$ stands for the segment
linking $(r/2)\exp( \sqrt{-1} \mathrm{arg}(A_{-n}))$ and $A_{-n}$ and finally as introduced earlier $P_{-n,2}$ denotes the horizontal
line $\{ A_{-n} - s / s \geq 0 \}$. An illustrative example is shown in Figure~\ref{fig4}.

\begin{figure}
	\centering
		\includegraphics[width=0.5\textwidth]{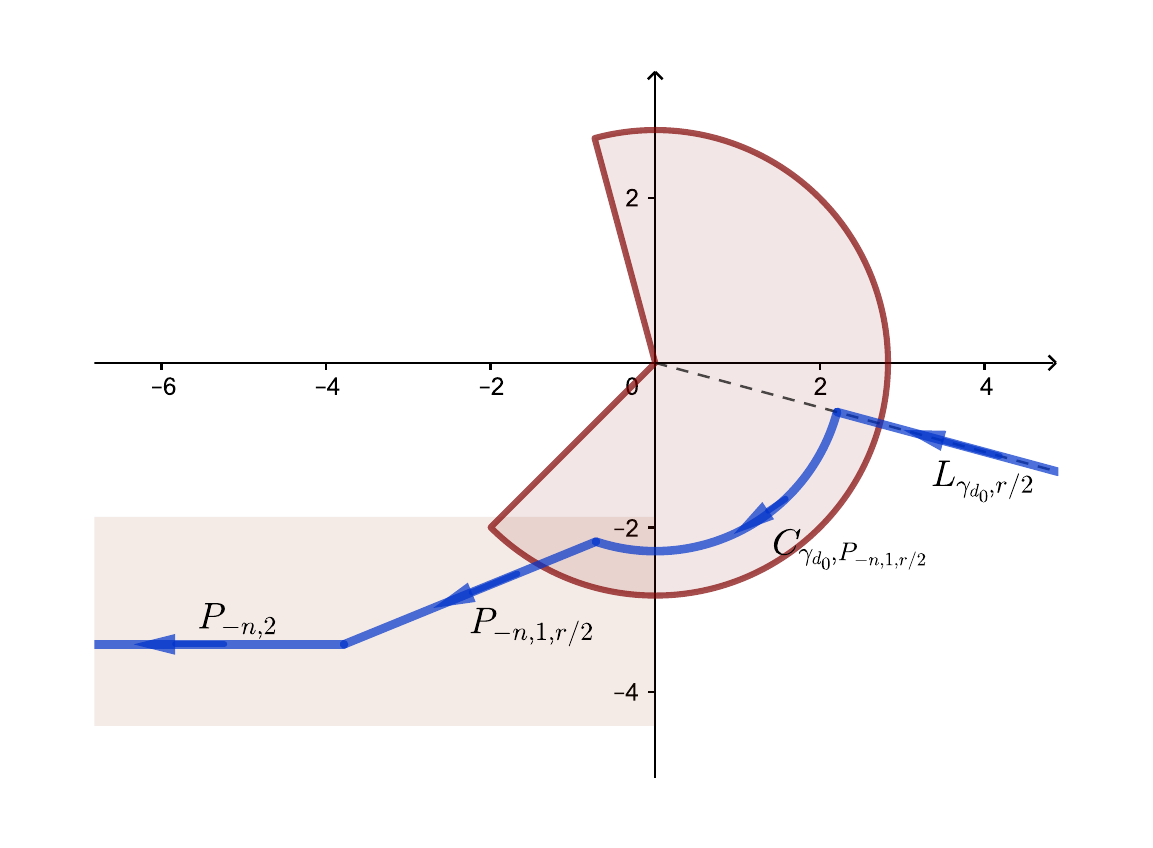}
		\caption{Deformation of the integration path}
		\label{fig4}
\end{figure}

Let
$$ J_{1} = \left| \int_{L_{\gamma_{d_0},r/2}} w_{S_{d_0}}(u,z,\epsilon) \exp( -\frac{u}{\epsilon t} ) \frac{du}{u} \right|.$$
In accordance with the bounds (\ref{I1_Sdp}), we can select $\delta_{2}>0$ and find $\delta_{1}>0$ with a constant
$C_{S_{d_0}}>0$ (independent of $\epsilon$) for which
\begin{equation}
J_{1} \leq C_{S_{d_0}} \frac{|\epsilon|}{\delta_{2}} \exp( -\frac{r \delta_{2}}{2 |\epsilon|} ) \label{J1_bds}
\end{equation}
holds whenever $t \in \mathcal{T} \cap D(0, \frac{\delta_{1}}{\delta_{2} + \sigma_{1} \zeta(b)})$ and
$\epsilon \in \mathcal{E}_{HJ_n}^{-n} \cap \mathcal{E}_{S_{d_0}}$.

Now, consider
$$ J_{2} = \left| \int_{C_{\gamma_{d_0},P_{-n,1},r/2}} w_{S_{d_0}}(u,z,\epsilon) \exp( -\frac{u}{\epsilon t} ) \frac{du}{u} \right|. $$
The function $w_{S_{d_0}}(\tau,z,\epsilon)$ suffers both the bounds (\ref{bds_w_Sdp}) since
$C_{\gamma_{d_0},P_{-n,1},r/2} \subset D(0,r)$ and also (\ref{bds_WHJn_Jk}) when
$\tau \in C_{\gamma_{d_0},P_{-n,1},r/2} \cap J_{0}$. We deduce a constant $C_{J_{0},S_{d_0}}>0$ (independent of $\epsilon$) such that
$$ |w_{S_{d_0}}(\tau,z,\epsilon)| \leq C_{J_{0},S_{d_0}} |\tau| \exp( \frac{\sigma_{1}}{|\epsilon|} \zeta(b) |\tau| ) $$
for all $\tau \in C_{\gamma_{d_0},P_{-n,1},r/2}$, $z \in D(0,\delta \delta_{1})$ and $\epsilon \in
\dot{D}(0,\epsilon_{0})$. Hence,
$$ J_{2} \leq C_{J_{0},S_{d_0}} \left| \int_{\mathrm{arg}(A_{-n})}^{\gamma_{d_0}} \frac{r}{2}
\exp( \frac{\sigma_{1}}{|\epsilon|} \zeta(b) \frac{r}{2} )
\exp( -\frac{r/2}{|\epsilon t|} \cos( \theta - \mathrm{arg}(t) - \mathrm{arg}(\epsilon) ) ) d\theta \right|.
$$
The sectors $\mathcal{E}_{HJ_n}^{-n}$ and $\mathcal{E}_{S_{d_0}}$ are suitably chosen in a way that
$\cos( \theta - \mathrm{arg}(t) - \mathrm{arg}(\epsilon) ) \geq \delta_{1}$ for some constant $\delta_{1}>0$, when
$\epsilon \in \mathcal{E}_{HJ_n}^{-n} \cap \mathcal{E}_{S_{d_0}}$, for $t \in \mathcal{T}$ and
$\theta \in (\mathrm{arg}(A_{-n}),\gamma_{d_0})$. As an issue,
\begin{equation}
J_{2} \leq C_{J_{0},S_{d_0}} |\gamma_{d_0} - \mathrm{arg}(A_{-n})| \frac{r}{2} \exp(- \frac{r \delta_{2}}{2 |\epsilon|} ) \label{J2_bds}
\end{equation}
when $\epsilon \in \mathcal{E}_{HJ_n}^{-n} \cap \mathcal{E}_{S_{d_0}}$, $t \in \mathcal{T} \cap
D(0, \frac{\delta_1}{\sigma_{1} \zeta(b) + \delta_{2}})$, for some fixed $\delta_{2}>0$.

We put
$$ J_{3} = \left| \int_{P_{-n,1,r/2}} w_{HJ_n}(u,z,\epsilon) \exp( -\frac{u}{\epsilon t} ) \frac{du}{u} \right|.$$
Owing to the fact that the path $P_{-n,1,r/2}$ lies across the domains $H_{q},J_{q}$ for $-n \leq q \leq 0$, the bounds
(\ref{bds_WHJn_Hk}) and (\ref{bds_WHJn_Jk}) entail that
$$ |w_{HJ_n}(\tau,z,\epsilon) \leq \max_{q \in \llbracket -n,0 \rrbracket}(C_{H_{q}},C_{J_{q}})
|\tau| \exp \left( \frac{\sigma_{1}}{|\epsilon|}\zeta(b) |\tau| + \varsigma_{2} \zeta(b) \exp( \varsigma_{3}|\tau| ) \right)
$$
for $\tau \in P_{-n,1,r/2}$, all $z \in D(0,\delta \delta_{1})$, all $\epsilon \in \dot{D}(0,\epsilon_{0})$. Therefore,
\begin{multline*}
J_{3} \leq \int_{r/2}^{|A_{-n}|} \max_{q \in \llbracket -n,0 \rrbracket}(C_{H_{q}},C_{J_{q}})
\rho \exp \left( \frac{\sigma_{1}}{|\epsilon|} \zeta(b) \rho + \varsigma_{2} \zeta(b) \exp( \varsigma_{3} \rho ) \right)\\
\times
\exp( -\frac{\rho}{|\epsilon t|} \cos( \mathrm{arg}(A_{-n}) - \mathrm{arg}(\epsilon t) ) ) \frac{d\rho}{\rho}.
\end{multline*}
Besides, according to (\ref{choice_a_k}), there exists some $\delta_{1}>0$ with
$\cos( \mathrm{arg}(A_{-n}) - \mathrm{arg}(\epsilon t) ) \geq \delta_{1}$ for
$\epsilon \in \mathcal{E}_{HJ_n}^{-n} \cap \mathcal{E}_{S_{d_0}}$. Let $\delta_{2}>0$ and take $t \in \mathcal{T}$ with
$|t| \leq \frac{\delta_{1}}{\delta_{2} + \sigma_{1} \zeta(b) }$. We obtain
\begin{multline}
J_{3} \leq \max_{q \in \llbracket -n,0 \rrbracket}(C_{H_{q}},C_{J_{q}})
\int_{r/2}^{|A_{-n}|} \exp( \varsigma_{2} \zeta(b) \exp( \varsigma_{3} \rho ) ) \exp( -\frac{\rho}{|\epsilon|} \delta_{2} ) d\rho\\
\leq \max_{q \in \llbracket -n,0 \rrbracket}(C_{H_{q}},C_{J_{q}})
\exp( \varsigma_{2} \zeta(b) \exp( \varsigma_{3}|A_{-n}| ))
\frac{|\epsilon|}{\delta_{2}} \exp( -\frac{r}{2 |\epsilon|} \delta_{2} ) \label{J3_bds}
\end{multline}
provided that $\epsilon \in \mathcal{E}_{HJ_n}^{-n} \cap \mathcal{E}_{S_{d_0}}$.

Ultimately, let
$$ J_{4} = \left| \int_{P_{-n,2}} w_{HJ_n}(u,z,\epsilon) \exp( -\frac{u}{\epsilon t} ) \frac{du}{u} \right|. $$
For the reason that the path $P_{-n,2}$ belongs to the strip $H_{-n}$, we can use the estimates
(\ref{bds_WHJn_Hk}) in order to get
\begin{multline*}
J_{4} \leq \int_{0}^{+\infty} C_{H_{-n}} |A_{-n} - s|
\exp \left( \frac{\sigma_{1}}{|\epsilon|} \zeta(b) |A_{-n}-s| - \sigma_{2}(M - \zeta(b))\exp( \sigma_{3}|A_{-n} - s| ) \right)\\
\times \exp \left( - \frac{|A_{-n} - s|}{|\epsilon t|} \cos( \mathrm{arg}(A_{-n} - s) - \mathrm{arg}(\epsilon) - \mathrm{arg}(t) ) \right)
\frac{ds}{|A_{-n} - s|}.
\end{multline*}
From the controlled variation of arguments (\ref{argument_ak_minus_s}), we can pick up some constant $\delta_{1}>0$ for which
$$ \cos( \mathrm{arg}(A_{-n} - s) - \mathrm{arg}(\epsilon) - \mathrm{arg}(t) ) > \delta_{1} $$
for $\epsilon \in \mathcal{E}_{HJ_n}^{-n} \cap \mathcal{E}_{S_{d_0}}$ and $t \in \mathcal{T}$. We take $\delta_{2}>0$ and restrict $t$ inside
$\mathcal{T}$ in a way that $|t| \leq \frac{\delta_{1}}{\delta_{2} + \sigma_{1}\zeta(b)}$. Besides, we can find a constant
$K_{A_{-n}}>0$ (depending on $A_{-n}$) such that
$$ |A_{-n} - s| \geq K_{A_{-n}} (|A_{-n}| + s) $$
for all $s \geq 0$. Henceforth, we obtain
\begin{multline}
J_{4} \leq C_{H_{-n}} \int_{0}^{+\infty} \exp \left( - \sigma_{2}(M - \zeta(b)) \exp( \sigma_{3}|A_{-n} - s| ) \right)
\exp( - \frac{|A_{-n} -s|}{|\epsilon|} \delta_{2} ) ds\\
\leq C_{H_{-n}} \int_{0}^{+\infty} \exp \left( -\frac{K_{A_{-n}} \delta_{2}}{|\epsilon|} (|A_{-n}| + s) \right) ds
= \frac{C_{H_{-n}}|\epsilon|}{K_{A_{-n}} \delta_{2}} \exp \left( -\frac{K_{A_{-n}} \delta_{2} |A_{-n}|}{|\epsilon|} \right)
\label{J4_bds}
\end{multline}
for all $\epsilon \in \mathcal{E}_{HJ_n}^{-n} \cap \mathcal{E}_{S_{d_0}}$.

In conclusion, bearing in mind the splitting (\ref{difference_u_HJn_Sd0_decomposition}) and collecting the upper bounds
(\ref{J1_bds}), (\ref{J2_bds}), (\ref{J3_bds}) and (\ref{J4_bds}) yields the forseen estimates (\ref{difference_u_HJn_Sd0}).
\end{proof}

\section{A second auxiliary convolution Cauchy problem}

\subsection{Banach spaces of holomorphic functions with exponential growth on $L-$shaped domains}

We keep the same notations as in Section 3.1. We consider a closed horizontal strip $H$ as defined in (\ref{defin_strip_H}) with $a \neq 0$
which belongs to the set of strips $\{ H_{k} \}_{k \in \llbracket -n,n \rrbracket}$ described at the beginning of the subsection 3.1
and we single out a closed
rectangle $R_{a,b,\upsilon}$ defined as follows:\\
If $a>0$, then
\begin{equation}
R_{a,b,\upsilon} = \{ z \in \mathbb{C} / \upsilon \leq \mathrm{Re}(z) \leq 0, 0 \leq \mathrm{Im}(z) \leq b \} \label{R_ab_u_a_positive}
\end{equation} 
and if $a < 0$
\begin{equation}
R_{a,b,\upsilon} = \{ z \in \mathbb{C} / \upsilon \leq \mathrm{Re}(z) \leq 0, a \leq \mathrm{Im}(z) \leq 0 \} \label{R_ab_u_a_negative}
\end{equation}
for some negative real number $\upsilon < 0$. We denote $RH_{a,b,\upsilon}$ the $L-$shaped domain $H \cup R_{a,b,\upsilon}$. See Figure~\ref{fig51}.

\begin{figure}
	\centering
		\includegraphics[width=0.4\textwidth]{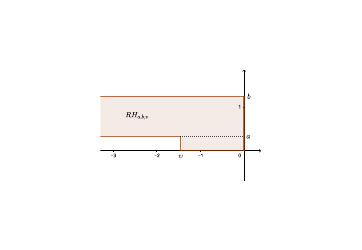}
		\includegraphics[width=0.4\textwidth]{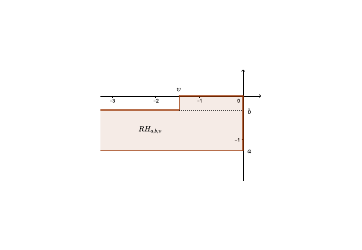}
		\caption{Examples of sets $RH_{a,b,\upsilon}=H\cup R_{a,b,\upsilon}$}
		\label{fig51}
\end{figure}

\begin{defin}
Let $\sigma_{1}>0$ be a positive real number and $\beta \geq 0$ be an integer. Let $\epsilon \in \dot{D}(0,\epsilon_{0})$. We set
$EG_{(\beta,\sigma_{1},RH_{a,b,\upsilon},\epsilon)}$ as the vector space of holomorphic functions $v(\tau)$ on the interior domain
$\mathring{RH}_{a,b,\upsilon}$, continuous on $RH_{a,b,\upsilon}$ such that the norm
$$ ||v(\tau)||_{(\beta,\sigma_{1},RH_{a,b,\upsilon},\epsilon)} = \sup_{\tau \in RH_{a,b,\upsilon}}
\frac{|v(\tau)|}{|\tau|} \exp \left( -\frac{\sigma_{1}}{|\epsilon|}r_{b}(\beta)|\tau| \right) $$
is finite. Let us take some positive real number $\delta>0$. We define
$EG_{(\sigma_{1},RH_{a,b,\upsilon},\epsilon,\delta)}$ as the vector space of all formal series
$v(\tau,z) = \sum_{\beta \geq 0} v_{\beta}(\tau) z^{\beta}/\beta!$ with coefficients $v_{\beta}(\tau)$ inside
$EG_{(\beta,\sigma_{1},RH_{a,b,\upsilon},\epsilon)}$ for all $\beta \geq 0$ and for which the norm
$$ ||v(\tau,z)||_{(\sigma_{1},RH_{a,b,\upsilon},\epsilon,\delta)} = \sum_{\beta \geq 0}
||v_{\beta}(\tau)||_{(\beta,\sigma_{1},RH_{a,b,\upsilon},\epsilon)} \frac{\delta^{\beta}}{\beta !}$$
is finite. It turns out that $EG_{(\sigma_{1},RH_{a,b,\upsilon},\epsilon,\delta)}$ endowed with the latter norm defines a Banach space.
\end{defin}
In the next proposition, we testify that the formal series belonging to the Banach space discussed above represent holomorphic functions
that are convergent in the vicinity of 0 w.r.t $z$ and with exponential growth on $RH_{a,b,\upsilon}$ regarding $\tau$. Its proof follows
the one of Proposition 1 in a straightforward manner.
\begin{prop} Let $v(\tau,z)$ chosen in $EG_{(\sigma_{1},RH_{a,b,\upsilon},\epsilon,\delta)}$. Take some $0 < \delta_{1} < 1$. Then, one can
get a constant $C_{4}>0$ (depending on $||v||_{(\sigma_{1},RH_{a,b,\upsilon},\epsilon,\delta)}$ and $\delta_{1}$) such that
\begin{equation}
|v(\tau,z)| \leq C_{4}|\tau| \exp \left( \frac{\sigma_{1}}{|\epsilon|} \zeta(b) |\tau| \right) 
\end{equation}
for all $\tau \in RH_{a,b,\upsilon}$, all $z \in D(0,\delta_{1}\delta)$.
\end{prop}
In the sequel, through the proposal of the next three propositions, we investigate the action of linear maps built as convolution products and
multiplication by bounded holomorphic functions on the Banach spaces defined above.

For all $\tau \in RH_{a,b,\upsilon}$, we denote
$L_{0,\tau}$ the path formed by the union of the segments $[0,c_{RH}(\tau)] \cup [c_{RH}(\tau),\tau]$, where $c_{RH}(\tau)$ is chosen in a way that
\begin{equation}
L_{0,\tau} \subset RH_{a,b,\upsilon}, \ \ c_{RH}(\tau) \in R_{a,b,\upsilon}, \ \ |c_{RH}(\tau)| \leq |\tau| \label{defin_cRH}  
\end{equation}
for all $\tau \in RH_{a,b,\upsilon}$.
\begin{prop} Let $\gamma_{0},\gamma_{1} \geq 0$ and $\gamma_{2} \geq 1$ be integers. We take for granted that
\begin{equation}
\gamma_{2} \geq b(\gamma_{0} + \gamma_{1} + 2) \label{cond_gamma012} 
\end{equation}
holds. Then, for any $\epsilon$ given in $\dot{D}(0,\epsilon_{0})$, the map
$v(\tau,z) \mapsto \tau \int_{L_{0,\tau}} (\tau - s)^{\gamma_0} s^{\gamma_1} \partial_{z}^{-\gamma_{2}} v(s,z) ds$
is a bounded linear operator from $EG_{(\sigma_{1},RH_{a,b,\upsilon},\epsilon,\delta)}$ into itself. Furthermore, we get a constant
$C_{5}>0$ (depending on $\gamma_{0},\gamma_{1},\gamma_{2}$, $\sigma_{1}$ and $b$) independent of $\epsilon$, such that
\begin{equation}
|| \tau \int_{L_{0,\tau}} (\tau - s)^{\gamma_0} s^{\gamma_1} \partial_{z}^{-\gamma_{2}} v(s,z) ds ||_{(\sigma_{1},RH_{a,b,\upsilon},\epsilon,\delta)}
\leq C_{5}|\epsilon|^{\gamma_{0}+\gamma_{1}+2} \delta^{\gamma_2} ||v(\tau,z)||_{(\sigma_{1},RH_{a,b,\upsilon},\epsilon,\delta)}
\label{norm_conv_partialz_v_C5}
\end{equation}
for all $v(\tau,z) \in EG_{(\sigma_{1},RH_{a,b,\upsilon},\epsilon,\delta)}$, all $\epsilon \in \dot{D}(0,\epsilon_{0})$.
\end{prop}
\begin{proof}
Take $v(\tau,z) = \sum_{\beta \geq 0} v_{\beta}(\tau) z^{\beta}/\beta!$ in $EG_{(\sigma_{1},RH_{a,b,\upsilon},\epsilon,\delta)}$. In view of
Definition 5,
\begin{multline}
|| \tau \int_{L_{0,\tau}} (\tau - s)^{\gamma_0} s^{\gamma_1} \partial_{z}^{-\gamma_{2}}
v(s,z) ds ||_{(\sigma_{1},RH_{a,b,\upsilon},\epsilon,\delta)}\\
= \sum_{\beta \geq \gamma_{2}}
|| \tau \int_{L_{0,\tau}} (\tau - s)^{\gamma_0} s^{\gamma_1} v_{\beta - \gamma_{2}}(s) ds ||_{(\beta,\sigma_{1},RH_{a,b,\upsilon},\epsilon)}
\delta^{\beta}/\beta ! \label{defin_norm_convolution_partial_z_v}
\end{multline}
\begin{lemma}
One can choose a constant $C_{5.1}>0$ (depending on $\gamma_{0},\gamma_{1},\gamma_{2}$ and $\sigma_{1}$) such that
\begin{multline}
|| \tau \int_{L_{0,\tau}} (\tau - s)^{\gamma_0} s^{\gamma_1} v_{\beta - \gamma_{2}}(s) ds ||_{(\beta,\sigma_{1},RH_{a,b,\upsilon},\epsilon)}
\\
\leq C_{5.1} |\epsilon|^{\gamma_{0} + \gamma_{1} + 2}(\beta + 1)^{b(\gamma_{0}+\gamma_{1}+2)}
||v_{\beta - \gamma_{2}}(\tau) ||_{(\beta - \gamma_{2},\sigma_{1},RH_{a,b,\upsilon},\epsilon)} \label{norm_conv_v_beta_minus_gamma2}
\end{multline}
for all $\beta \geq \gamma_{2}$.
\end{lemma}
\begin{proof} By construction of $L_{0,\tau}$, we can split the integral in two parts
\begin{multline*}
\tau \int_{L_{0,\tau}} (\tau - s)^{\gamma_0} s^{\gamma_1} v_{\beta - \gamma_{2}}(s) ds =
\tau \int_{0}^{c_{RH}(\tau)} (\tau - s)^{\gamma_0} s^{\gamma_1} v_{\beta - \gamma_{2}}(s) ds \\
+
\tau \int_{c_{RH}(\tau)}^{\tau} (\tau - s)^{\gamma_0} s^{\gamma_1} v_{\beta - \gamma_{2}}(s) ds
\end{multline*}
We first provide estimates for
$$ L_{1} = || \tau \int_{0}^{c_{RH}(\tau)} (\tau - s)^{\gamma_0} s^{\gamma_1}
v_{\beta - \gamma_{2}}(s) ds ||_{(\beta,\sigma_{1},RH_{a,b,\upsilon},\epsilon)}. $$
We carry out the next factorization
\begin{multline*}
\frac{1}{|\tau|} \exp \left( -\frac{\sigma_{1}}{|\epsilon|}r_{b}(\beta) |\tau| \right) |\tau|
\left| \int_{0}^{c_{RH}(\tau)} (\tau - s)^{\gamma_0} s^{\gamma_1}
v_{\beta - \gamma_{2}}(s) ds \right|\\
= \exp \left( -\frac{\sigma_{1}}{|\epsilon|}r_{b}(\beta) |\tau| \right) \left| \int_{0}^{c_{RH}(\tau)} (\tau - s)^{\gamma_0} s^{\gamma_1}
\{ \frac{1}{|s|} \exp \left( - \frac{\sigma_{1}}{|\epsilon|}r_{b}(\beta - \gamma_{2})|s| \right) v_{\beta - \gamma_{2}}(s) \} \right. \\
\left. \times 
|s| \exp \left( \frac{\sigma_{1}}{|\epsilon|}r_{b}(\beta - \gamma_{2})|s| \right) ds \right|.
\end{multline*}
We deduce that
\begin{equation}
L_{1} \leq C_{5.1.1}(\beta,\epsilon) ||v_{\beta - \gamma_{2}}(\tau)||_{(\beta - \gamma_{2},\sigma_{1},RH_{a,b,\upsilon},\epsilon)} 
\end{equation}
where
\begin{multline*}
C_{5.1.1}(\beta,\epsilon) = \sup_{\tau \in RH_{a,b,\upsilon}}
\exp \left( -\frac{\sigma_{1}}{|\epsilon|}r_{b}(\beta) |\tau| \right) \int_{0}^{1} |\tau - c_{RH}(\tau)u|^{\gamma_0}
|c_{RH}(\tau)|^{\gamma_{1} + 2} u^{\gamma_{1}+1}\\
\times \exp \left( \frac{\sigma_{1}}{|\epsilon|}r_{b}(\beta - \gamma_{2}) |c_{RH}(\tau)u| \right) du.
\end{multline*}
As a consequence of the shape of $L_{0,\tau}$ through (\ref{defin_cRH}), according to the inequalities
(\ref{difference_s_b_r_b}), (\ref{A_1_bounds}) and taking account of the rough estimates
$|\tau - c_{RH}(\tau)u|^{\gamma_0} \leq 2^{\gamma_{0}}|\tau|^{\gamma_0}$
for $0 \leq u \leq 1$, we get
\begin{multline}
C_{5.1.1}(\beta,\epsilon) \leq 2^{\gamma_0} \sup_{\tau \in RH_{a,b,\upsilon}} |\tau|^{\gamma_{0}+\gamma_{1}+2}
\exp \left( -\frac{\sigma_{1}}{|\epsilon|}(r_{b}(\beta) - r_{b}(\beta - \gamma_{2})) |\tau| \right)\\
\leq 2^{\gamma_0} \sup_{x \geq 0} x^{\gamma_{0}+\gamma_{1}+2}
\exp \left( -\frac{\sigma_{1}}{|\epsilon|}\frac{\gamma_2}{(\beta + 1)^{b}} x \right)\\
\leq
2^{\gamma_0} |\epsilon|^{\gamma_{0}+\gamma_{1}+2}
\left( \frac{\gamma_{0}+\gamma_{1}+2}{\sigma_{1} \gamma_{2}} \right)^{\gamma_{0}+\gamma_{1}+2} \exp( -(\gamma_{0}+\gamma_{1}+2) )
(\beta + 1)^{b(\gamma_{0}+\gamma_{1}+2)} \label{C511_bds}
\end{multline}
for all $\beta \geq \gamma_{2}$, all $\epsilon \in \dot{D}(0,\epsilon_{0})$.

In a second part, we seek bounds for 
$$ L_{2} = ||\tau \int_{c_{RH}(\tau)}^{\tau} (\tau - s)^{\gamma_0}
s^{\gamma_1} v_{\beta - \gamma_{2}}(s) ds||_{(\beta,\sigma_{1},RH_{a,b,\upsilon},\epsilon)}. $$
As above, we achieve the factorization
\begin{multline*}
\frac{1}{|\tau|} \exp \left( -\frac{\sigma_{1}}{|\epsilon|}r_{b}(\beta) |\tau| \right) |\tau|
\left| \int_{c_{RH}(\tau)}^{\tau} (\tau - s)^{\gamma_0} s^{\gamma_1}
v_{\beta - \gamma_{2}}(s) ds \right|\\
= \exp \left( -\frac{\sigma_{1}}{|\epsilon|}r_{b}(\beta) |\tau| \right) \left| \int_{c_{RH}(\tau)}^{\tau} (\tau - s)^{\gamma_0} s^{\gamma_1}
\{ \frac{1}{|s|} \exp \left( - \frac{\sigma_{1}}{|\epsilon|}r_{b}(\beta - \gamma_{2})|s| \right) v_{\beta - \gamma_{2}}(s) \} \right. \\
\left. \times 
|s| \exp \left( \frac{\sigma_{1}}{|\epsilon|}r_{b}(\beta - \gamma_{2})|s| \right) ds \right|.
\end{multline*}
It follows that
\begin{equation}
L_{2} \leq C_{5.1.2}(\beta,\epsilon) ||v_{\beta - \gamma_{2}}(\tau)||_{(\beta - \gamma_{2},\sigma_{1},RH_{a,b,\upsilon},\epsilon)} 
\end{equation}
with
\begin{multline*}
C_{5.1.2}(\beta,\epsilon) = \sup_{\tau \in RH_{a,b,\upsilon}}
\exp \left( -\frac{\sigma_{1}}{|\epsilon|}r_{b}(\beta) |\tau| \right)
\int_{0}^{1} |\tau - c_{RH}(\tau)|^{\gamma_{0}+1}(1-u)^{\gamma_{0}}\\
\times |(1-u)c_{RH}(\tau) + u\tau|^{\gamma_{1}+1}
\exp \left( \frac{\sigma_{1}}{|\epsilon|}r_{b}(\beta - \gamma_{2})|(1-u)c_{RH}(\tau) + u\tau| \right) du.
\end{multline*}
By construction of the path $L_{0,\tau}$ by means of (\ref{defin_cRH}), bearing in mind
(\ref{difference_s_b_r_b}), (\ref{A_1_bounds}) and owing to the bounds
$|\tau - c_{RH}(\tau)|^{\gamma_{0}+1} \leq 2^{\gamma_{0}+1}|\tau|^{\gamma_{0}+1}$ with
$|(1-u)c_{RH}(\tau) + u\tau| \leq |\tau|$ for $0 \leq u \leq 1$, we obtain
\begin{multline}
C_{5.1.2}(\beta,\epsilon) \leq 2^{\gamma_{0}+1} \sup_{\tau \in RH_{a,b,\upsilon}} |\tau|^{\gamma_{0}+\gamma_{1}+2}
\exp \left( -\frac{\sigma_{1}}{|\epsilon|}(r_{b}(\beta) - r_{b}(\beta - \gamma_{2})) |\tau| \right)\\
\leq 2^{\gamma_{0}+1} |\epsilon|^{\gamma_{0}+\gamma_{1}+2}
\left( \frac{\gamma_{0}+\gamma_{1}+2}{\sigma_{1} \gamma_{2}} \right)^{\gamma_{0}+\gamma_{1}+2} \exp( -(\gamma_{0}+\gamma_{1}+2) )
(\beta + 1)^{b(\gamma_{0}+\gamma_{1}+2)}
\end{multline}
for all $\beta \geq \gamma_{2}$, all $\epsilon \in \dot{D}(0,\epsilon_{0})$. The lemma 10 follows.
\end{proof}
Gathering the expansion (\ref{defin_norm_convolution_partial_z_v}) and the upper bounds
(\ref{norm_conv_v_beta_minus_gamma2}), we get
\begin{multline}
|| \tau \int_{L_{0,\tau}} (\tau - s)^{\gamma_0} s^{\gamma_1} \partial_{z}^{-\gamma_{2}}
v(s,z) ds ||_{(\sigma_{1},RH_{a,b,\upsilon},\epsilon,\delta)}\\
\leq \sum_{\beta \geq \gamma_{2}} C_{5.1} |\epsilon|^{\gamma_{0}+\gamma_{1}+2}(\beta + 1)^{b(\gamma_{0}+\gamma_{1}+2)}
\frac{(\beta-\gamma_{2})!}{\beta!} ||v_{\beta - \gamma_{2}}(\tau)||_{(\beta - \gamma_{2},\sigma_{1},RH_{a,b,\upsilon},\epsilon)}
\delta^{\gamma_{2}} \frac{\delta^{\beta - \gamma_{2}}}{(\beta - \gamma_{2})!} \label{norm_conv_partialz_v_C51}
\end{multline}
Keeping in mind the guess (\ref{cond_gamma012}), we obtain a constant $C_{5.2}>0$ (depending on $\gamma_{0},\gamma_{1},\gamma_{2}$ and $b$)
for which
\begin{equation}
(\beta + 1)^{b(\gamma_{0}+\gamma_{1}+2)} \frac{(\beta - \gamma_{2})!}{\beta !} \leq C_{5.2} \label{bds_beta_gamma012}
\end{equation}
holds for all $\beta \geq \gamma_{2}$. Piling up (\ref{norm_conv_partialz_v_C51}) and (\ref{bds_beta_gamma012}) grants the result 
(\ref{norm_conv_partialz_v_C5}).
\end{proof}

\begin{prop}
Let $\gamma_{0},\gamma_{1} \geq 0$ be integers. Let $\sigma_{1},\sigma_{1}'>0$ be real numbers such that $\sigma_{1} > \sigma_{1}'$.
Then, for all $\epsilon \in \dot{D}(0,\epsilon_{0})$, the linear operator
$v(\tau,z) \mapsto \tau \int_{L_{0,\tau}} (\tau - s)^{\gamma_0}s^{\gamma_1}v(s,z) ds$ is bounded from
$(EG_{(\sigma_{1}',RH_{a,b,\upsilon},\epsilon,\delta)}, ||.||_{(\sigma_{1}',RH_{a,b,\upsilon},\epsilon,\delta)})$ into
$(EG_{(\sigma_{1},RH_{a,b,\upsilon},\epsilon,\delta)}, ||.||_{(\sigma_{1},RH_{a,b,\upsilon},\epsilon,\delta)})$. In addition, we can select
a constant $\check{C}_{5}>0$ (depending on $\gamma_{0},\gamma_{1},\sigma_{1}$ and $\sigma_{1}'$) with
\begin{equation}
|| \tau \int_{L_{0,\tau}} (\tau - s)^{\gamma_0}s^{\gamma_1}v(s,z) ds ||_{(\sigma_{1},RH_{a,b,\upsilon},\epsilon,\delta)} \leq
\check{C}_{5} |\epsilon|^{\gamma_{0}+\gamma_{1}+2} ||v(\tau,z)||_{(\sigma_{1}',RH_{a,b,\upsilon},\epsilon,\delta)}
\end{equation}
for all $v(\tau,z) \in EG_{(\sigma_{1}',RH_{a,b,\upsilon},\epsilon,\delta)}$, for all $\epsilon \in \dot{D}(0,\epsilon_{0})$.
\end{prop}
\begin{proof} Pick up some $v(\tau,z) = \sum_{\beta \geq 0} v_{\beta}(\tau) z^{\beta}/\beta!$ in
$EG_{(\sigma_{1}',RH_{a,b,\upsilon},\epsilon,\delta)}$. Owing to Definition 5,
\begin{multline}
|| \tau \int_{L_{0,\tau}} (\tau - s)^{\gamma_0} s^{\gamma_1}
v(s,z) ds ||_{(\sigma_{1},RH_{a,b,\upsilon},\epsilon,\delta)}\\
= \sum_{\beta \geq 0}
|| \tau \int_{L_{0,\tau}} (\tau - s)^{\gamma_0} s^{\gamma_1} v_{\beta}(s) ds ||_{(\beta,\sigma_{1},RH_{a,b,\upsilon},\epsilon)}
\delta^{\beta}/\beta ! \label{defin_norm_convolution_v_sigma1}
\end{multline}
\begin{lemma}
One can assign a constant $\check{C}_{5}>0$ (depending on $\gamma_{0},\gamma_{1},\sigma_{1}$ and $\sigma_{1}'$) such that
\begin{equation}
|| \tau \int_{L_{0,\tau}} (\tau - s)^{\gamma_0} s^{\gamma_1} v_{\beta}(s) ds ||_{(\beta,\sigma_{1},RH_{a,b,\upsilon},\epsilon)}
\leq \check{C}_{5} |\epsilon|^{\gamma_{0} + \gamma_{1} + 2}
||v_{\beta}(\tau) ||_{(\beta,\sigma_{1}',RH_{a,b,\upsilon},\epsilon)} \label{norm_conv_v_beta_sigma1_sigma1_prim}
\end{equation}
for all $\beta \geq 0$.
\end{lemma}
\begin{proof} As above, we first cut the integral into two pieces
$$
\tau \int_{L_{0,\tau}} (\tau - s)^{\gamma_0} s^{\gamma_1} v_{\beta}(s) ds =
\tau \int_{0}^{c_{RH}(\tau)} (\tau - s)^{\gamma_0} s^{\gamma_1} v_{\beta}(s) ds
+
\tau \int_{c_{RH}(\tau)}^{\tau} (\tau - s)^{\gamma_0} s^{\gamma_1} v_{\beta}(s) ds
$$
We first request estimates for
$$ \check{L}_{1} = || \tau \int_{0}^{c_{RH}(\tau)} (\tau - s)^{\gamma_0} s^{\gamma_1}
v_{\beta}(s) ds ||_{(\beta,\sigma_{1},RH_{a,b,\upsilon},\epsilon)}. $$
We do the next factorization
\begin{multline*}
\frac{1}{|\tau|} \exp \left( -\frac{\sigma_{1}}{|\epsilon|}r_{b}(\beta) |\tau| \right) |\tau|
\left| \int_{0}^{c_{RH}(\tau)} (\tau - s)^{\gamma_0} s^{\gamma_1}
v_{\beta}(s) ds \right|\\
= \exp \left( -\frac{\sigma_{1}}{|\epsilon|}r_{b}(\beta) |\tau| \right) \left| \int_{0}^{c_{RH}(\tau)} (\tau - s)^{\gamma_0} s^{\gamma_1}
\{ \frac{1}{|s|} \exp \left( - \frac{\sigma_{1}'}{|\epsilon|}r_{b}(\beta)|s| \right) v_{\beta}(s) \} \right. \\
\left. \times 
|s| \exp \left( \frac{\sigma_{1}'}{|\epsilon|}r_{b}(\beta)|s| \right) ds \right|.
\end{multline*}
which leads to
\begin{equation}
\check{L}_{1} \leq \check{C}_{5.1}(\beta,\epsilon) ||v_{\beta}(\tau)||_{(\beta,\sigma_{1}',RH_{a,b,\upsilon},\epsilon)} 
\end{equation}
where
\begin{multline*}
\check{C}_{5.1}(\beta,\epsilon) = \sup_{\tau \in RH_{a,b,\upsilon}}
\exp \left( -\frac{\sigma_{1}}{|\epsilon|}r_{b}(\beta) |\tau| \right) \int_{0}^{1} |\tau - c_{RH}(\tau)u|^{\gamma_0}
|c_{RH}(\tau)|^{\gamma_{1} + 2} u^{\gamma_{1}+1}\\
\times \exp \left( \frac{\sigma_{1}'}{|\epsilon|}r_{b}(\beta) |c_{RH}(\tau)u| \right) du.
\end{multline*}
Due to the constraints (\ref{defin_cRH}) and keeping in view the bounds (\ref{checkA1_bds}), we see that
\begin{multline}
\check{C}_{5.1}(\beta,\epsilon) \leq 2^{\gamma_0} \sup_{\tau \in RH_{a,b,\upsilon}} |\tau|^{\gamma_{0}+\gamma_{1}+2}
\exp \left( -\frac{\sigma_{1} - \sigma_{1}'}{|\epsilon|} r_{b}(\beta) |\tau| \right)\\
\leq
2^{\gamma_0} \sup_{x \geq 0} x^{\gamma_{0}+\gamma_{1}+2} \exp \left( -\frac{\sigma_{1} - \sigma_{1}'}{|\epsilon|} r_{b}(\beta) x \right)
\leq 2^{\gamma_0} |\epsilon|^{\gamma_{0}+\gamma_{1}+2}
\left( \frac{(\gamma_{0}+\gamma_{1}+2)e^{-1}}{\sigma_{1} - \sigma_{1}'} \right)^{\gamma_{0}+\gamma_{1}+2}
\end{multline}
for all $\beta \geq 0$, $\epsilon \in \dot{D}(0,\epsilon_{0})$.

Next in order, we point at 
$$ \check{L}_{2} = ||\tau \int_{c_{RH}(\tau)}^{\tau} (\tau - s)^{\gamma_0}
s^{\gamma_1} v_{\beta}(s) ds||_{(\beta,\sigma_{1},RH_{a,b,\upsilon},\epsilon)}. $$
As before, we accomplish a factorization
\begin{multline*}
\frac{1}{|\tau|} \exp \left( -\frac{\sigma_{1}}{|\epsilon|}r_{b}(\beta) |\tau| \right) |\tau|
\left| \int_{c_{RH}(\tau)}^{\tau} (\tau - s)^{\gamma_0} s^{\gamma_1}
v_{\beta}(s) ds \right|\\
= \exp \left( -\frac{\sigma_{1}}{|\epsilon|}r_{b}(\beta) |\tau| \right) \left| \int_{c_{RH}(\tau)}^{\tau} (\tau - s)^{\gamma_0} s^{\gamma_1}
\{ \frac{1}{|s|} \exp \left( - \frac{\sigma_{1}'}{|\epsilon|}r_{b}(\beta)|s| \right) v_{\beta}(s) \} \right. \\
\left. \times 
|s| \exp \left( \frac{\sigma_{1}'}{|\epsilon|}r_{b}(\beta)|s| \right) ds \right|
\end{multline*}
which entails
\begin{equation}
\check{L}_{2} \leq \check{C}_{5.2}(\beta,\epsilon) ||v_{\beta}(\tau)||_{(\beta,\sigma_{1}',RH_{a,b,\upsilon},\epsilon)} 
\end{equation}
with
\begin{multline*}
\check{C}_{5.2}(\beta,\epsilon) = \sup_{\tau \in RH_{a,b,\upsilon}}
\exp \left( -\frac{\sigma_{1}}{|\epsilon|}r_{b}(\beta) |\tau| \right)
\int_{0}^{1} |\tau - c_{RH}(\tau)|^{\gamma_{0}+1}(1-u)^{\gamma_{0}}\\
\times |(1-u)c_{RH}(\tau) + u\tau|^{\gamma_{1}+1}
\exp \left( \frac{\sigma_{1}'}{|\epsilon|}r_{b}(\beta)|(1-u)c_{RH}(\tau) + u\tau| \right) du.
\end{multline*}
By reason of the restriction (\ref{defin_cRH}) and by taking a glance at the bounds (\ref{checkA1_bds}), we deduce
\begin{multline}
\check{C}_{5.2}(\beta,\epsilon) \leq 2^{\gamma_{0}+1} \sup_{\tau \in RH_{a,b,\upsilon}} |\tau|^{\gamma_{0}+\gamma_{1}+2}
\exp \left( -\frac{\sigma_{1} - \sigma_{1}'}{|\epsilon|} r_{b}(\beta) |\tau| \right)\\
\leq 2^{\gamma_{0}+1} |\epsilon|^{\gamma_{0}+\gamma_{1}+2}
\left( \frac{(\gamma_{0}+\gamma_{1}+2)e^{-1}}{\sigma_{1} - \sigma_{1}'} \right)^{\gamma_{0}+\gamma_{1}+2}
\end{multline}
provided that $\beta \geq 0$, $\epsilon \in \dot{D}(0,\epsilon_{0})$. Hence, Lemma 11 is verified. 
\end{proof}
Finally, according to (\ref{defin_norm_convolution_v_sigma1}) we notice that Proposition 14 is just a byproduct of the lemma 11 above.
\end{proof}

The proof of the next proposition mirrors in a genuine way the one of Proposition 4.

\begin{prop} Let us consider some holomorphic function $c(\tau,z,\epsilon)$ on $\mathring{RH}_{a,b,\upsilon} \times D(0,\rho) \times 
D(0,\epsilon_{0})$, continuous on $RH_{a,b,\upsilon} \times D(0,\rho) \times D(0,\epsilon_{0})$, for a radius $\rho>0$, bounded therein
by a constant $M_{c}>0$. Fix some $0 < \delta < \rho$. Then, the linear operator
$v(\tau,z) \mapsto c(\tau,z,\epsilon)v(\tau,z)$ is bounded from
$(EG_{(\sigma_{1},RH_{a,b,\upsilon},\epsilon,\delta)},||.||_{(\sigma_{1},RH_{a,b,\upsilon},\epsilon,\delta)})$ into itself,
provided that $\epsilon \in \dot{D}(0,\epsilon_{0})$. Additionally, a constant $C_{6}>0$ (depending on
$M_{c},\delta,\rho$) independent of $\epsilon$ exists in a way that
\begin{equation}
||c(\tau,z,\epsilon)v(\tau,z)||_{(\sigma_{1},RH_{a,b,\upsilon},\epsilon,\delta)} \leq C_{6}
||v(\tau,z)||_{(\sigma_{1},RH_{a,b,\upsilon},\epsilon,\delta)}
\end{equation}
for all $v \in EG_{(\sigma_{1},RH_{a,b,\upsilon},\epsilon,\delta)}$.
\end{prop}

\subsection{Banach spaces of holomorphic functions with super exponential growth on $L-$shaped domains}

We will refer to the notations of Sections 3.1 and 4.1 within this subsection. Namely, we set a closed horizontal strip $J$ as defined
in (\ref{defin_strip_J}) where $c$ is chosen different from 0 among the family of sectors $\{ J_{k} \}_{k \in \llbracket -n,n \rrbracket}$
built up at the onset of the subsection 3.1 and a closed rectangle $R_{c,d,\upsilon}$ as displayed in 
(\ref{R_ab_u_a_positive}) and (\ref{R_ab_u_a_negative}) for some negative $\upsilon>0$. The set $RJ_{c,d,\upsilon}$ stands for the
$L-$shaped domain $J \cup R_{c,d,\upsilon}$. See Figure~\ref{fig52}.

\begin{figure}
	\centering
		\includegraphics[width=0.4\textwidth]{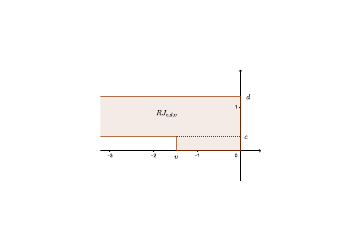}
		\includegraphics[width=0.4\textwidth]{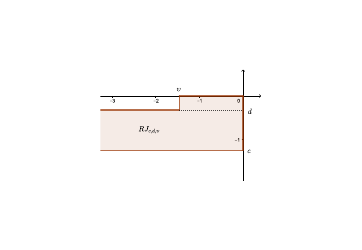}
		\caption{Examples of sets $RJ_{c,d,\upsilon}=J\cup R_{c,d,\upsilon}$}
		\label{fig52}
\end{figure}

\begin{defin}
Let $\underline{\varsigma} = (\sigma_{1},\varsigma_{2},\varsigma_{3})$ where $\sigma_{1},\varsigma_{2},\varsigma_{3}>0$ are assumed to be
positive real numbers and let $\beta \geq 0$ be an integer. For all $\epsilon \in \dot{D}(0,\epsilon_{0})$, we define
$SEG_{(\beta,\underline{\varsigma},RJ_{c,d,\upsilon},\epsilon)}$ as the vector space of holomorphic functions
$v(\tau)$ on $\mathring{RJ}_{c,d,\upsilon}$, continuous on $RJ_{c,d,\upsilon}$ for which
$$ ||v(\tau)||_{(\beta,\underline{\varsigma},RJ_{c,d,\upsilon},\epsilon)} =
\sup_{\tau \in RJ_{c,d,\upsilon}} \frac{|v(\tau)|}{|\tau|} \exp \left( -\frac{\sigma_1}{|\epsilon|} r_{b}(\beta) |\tau|
- \varsigma_{2}r_{b}(\beta)\exp( \varsigma_{3}|\tau| ) \right) 
$$
is finite. Let $\delta>0$ be some positive number. The set $SEG_{(\underline{\varsigma},RJ_{c,d,\upsilon},\epsilon,\delta)}$ stands
for the vector space of all formal series $v(\tau,z) = \sum_{\beta \geq 0} v_{\beta}(\tau) z^{\beta}/\beta!$ with coefficients
$v_{\beta}(\tau)$ belonging to $SEG_{(\beta,\underline{\varsigma},RJ_{c,d,\upsilon},\epsilon)}$ and whose norm
$$ ||v(\tau,z)||_{(\underline{\varsigma},RJ_{c,d,\upsilon},\epsilon,\delta)} = \sum_{\beta \geq 0}
||v_{\beta}(\tau)||_{(\beta,\underline{\varsigma},RJ_{c,d,\upsilon},\epsilon)} \frac{\delta^{\beta}}{\beta !}
$$
is finite. The space $SEG_{(\underline{\varsigma},RJ_{c,d,\upsilon},\epsilon,\delta)}$ equipped with this norm is a Banach space.
\end{defin}
The next statement can be checked exactly in the same manner as Proposition 5 1).

\begin{prop}
Let $v(\tau,z) \in SEG_{(\underline{\varsigma},RJ_{c,d,\upsilon},\epsilon,\delta)}$. Fix some $0 < \delta_{1} < 1$. Then, we get a constant
$C_{7}>0$ (depending on $||v||_{(\underline{\varsigma},RJ_{c,d,\upsilon},\epsilon,\delta)}$ and $\delta_{1}$) fulfilling
\begin{equation}
|v(\tau,z)| \leq C_{7}|\tau| \exp \left( \frac{\sigma_1}{|\epsilon|}\zeta(b)|\tau| + \varsigma_{2}\zeta(b) \exp(\varsigma_{3}|\tau|) \right) 
\end{equation}
for all $\tau \in RJ_{c,d,\upsilon}$, all $z \in D(0,\delta_{1}\delta)$.
\end{prop}

In the upcoming propositions, we plan to analyze the same convolution maps and multiplication by bounded holomorphic functions as worked out
in Propositions 13,14 and 15 but operating on
the Banach spaces disclosed in Definition 6. As in Section 4.1, $L_{0,\tau}$ stands for a path defined as a union
$[0,c_{RJ}(\tau)] \cup [c_{RJ}(\tau),\tau]$, where $c_{RJ}(\tau)$ is selected with the next properties:
\begin{equation}
L_{0,\tau} \subset RJ_{c,d,\upsilon}, \ \ c_{RJ}(\tau) \in R_{c,d,\upsilon}, \ \ |c_{RJ}(\tau)| \leq |\tau| \label{defin_cRJ}
\end{equation}
whenever $\tau \in RJ_{c,d,\upsilon}$.
\begin{prop}
Let $\gamma_{0},\gamma_{1} \geq 0$ and $\gamma_{2} \geq 1$ be integers. We assume that
\begin{equation}
\gamma_{2} \geq b(\gamma_{0} + \gamma_{1} + 2) 
\end{equation}
holds. Then, for all $\epsilon \in \dot{D}(0,\epsilon_{0})$, the linear operator
$v(\tau,z) \mapsto \tau \int_{L_{0,\tau}} (\tau - s)^{\gamma_{0}} s^{\gamma_1} \partial_{z}^{-\gamma_{2}}v(s,z) ds$ is bounded from
$SEG_{(\underline{\varsigma},RJ_{c,d,\upsilon},\epsilon,\delta)}$ into itself. In addition, one gets a constant $C_{8}>0$
(depending on $\gamma_{0},\gamma_{1},\gamma_{2},\sigma_{1}$ and $b$) independent of $\epsilon$, such that
\begin{equation}
|| \tau \int_{L_{0,\tau}} (\tau - s)^{\gamma_0} s^{\gamma_1} \partial_{z}^{-\gamma_{2}}
v(s,z) ds ||_{(\underline{\varsigma},RJ_{c,d,\upsilon},\epsilon,\delta)}
\leq C_{8}|\epsilon|^{\gamma_{0}+\gamma_{1}+2} \delta^{\gamma_2} ||v(\tau,z)||_{(\underline{\varsigma},RJ_{c,d,\upsilon},\epsilon,\delta)}
\label{norm_conv_partialz_v_C8}
\end{equation}
for all $v(\tau,z) \in SEG_{(\underline{\varsigma},RJ_{c,d,\upsilon},\epsilon,\delta)}$, all $\epsilon \in \dot{D}(0,\epsilon_{0})$.
\end{prop}
\begin{proof}
Only a brief outline of the proof will be presented hereafter since its content resembles the one displayed in Proposition 13. Namely, it boils
down to show the next lemma.
\begin{lemma}
Take $v_{\beta - \gamma_{2}}(\tau) \in SEG_{(\beta - \gamma_{2},\underline{\varsigma},RJ_{c,d,\upsilon},\epsilon)}$ for all $\beta \geq \gamma_{2}$.
One can sort a constant $C_{8.1}>0$ (depending on $\gamma_{0},\gamma_{1},\gamma_{2},\sigma_{1}$) for which
\begin{multline}
|| \tau \int_{L_{0,\tau}} (\tau - s)^{\gamma_0} s^{\gamma_1}
v_{\beta - \gamma_{2}}(s) ds ||_{(\beta,\underline{\varsigma},RJ_{c,d,\upsilon},\epsilon)} \\
\leq
C_{8.1} |\epsilon|^{\gamma_{0}+\gamma_{1}+2}(\beta + 1)^{b(\gamma_{0}+\gamma_{1}+2)}
||v_{\beta - \gamma_{2}}(\tau)||_{(\beta - \gamma_{2},\underline{\varsigma},RJ_{c,d,\upsilon},\epsilon)}
\end{multline}
\end{lemma}
\begin{proof} As before, we depart from the break up of the convolution product in two pieces
\begin{multline*}
\tau \int_{L_{0,\tau}} (\tau - s)^{\gamma_0} s^{\gamma_1} v_{\beta - \gamma_{2}}(s) ds =
\tau \int_{0}^{c_{RJ}(\tau)} (\tau - s)^{\gamma_0} s^{\gamma_1} v_{\beta - \gamma_{2}}(s) ds \\
+
\tau \int_{c_{RJ}(\tau)}^{\tau} (\tau - s)^{\gamma_0} s^{\gamma_1} v_{\beta - \gamma_{2}}(s) ds
\end{multline*}
We demand estimates for the first part
$$ LJ_{1} = || \tau \int_{0}^{c_{RJ}(\tau)} (\tau - s)^{\gamma_0} s^{\gamma_1}
v_{\beta - \gamma_{2}}(s) ds ||_{(\beta,\underline{\varsigma},RJ_{c,d,\upsilon},\epsilon)}. $$
We perform a factorization
\begin{multline*}
\frac{1}{|\tau|} \exp \left( -\frac{\sigma_{1}}{|\epsilon|}r_{b}(\beta) |\tau| -\varsigma_{2}r_{b}(\beta) \exp( \varsigma_{3}|\tau| ) \right)
|\tau|
\left| \int_{0}^{c_{RJ}(\tau)} (\tau - s)^{\gamma_0} s^{\gamma_1}
v_{\beta - \gamma_{2}}(s) ds \right|\\
= \exp \left( -\frac{\sigma_{1}}{|\epsilon|}r_{b}(\beta) |\tau| - \varsigma_{2}r_{b}(\beta) \exp( \varsigma_{3} |\tau| ) \right)
\left| \int_{0}^{c_{RJ}(\tau)} (\tau - s)^{\gamma_0} s^{\gamma_1} \right. \\
\times 
\{ \frac{1}{|s|} \exp \left( - \frac{\sigma_{1}}{|\epsilon|}r_{b}(\beta - \gamma_{2})|s| - \varsigma_{2}r_{b}(\beta - \gamma_{2})
\exp( \varsigma_{3}|s| ) \right) v_{\beta - \gamma_{2}}(s) \} \\
\left. \times 
|s| \exp \left( \frac{\sigma_{1}}{|\epsilon|}r_{b}(\beta - \gamma_{2})|s| + \varsigma_{2}r_{b}(\beta - \gamma_{2})
\exp( \varsigma_{3}|s|) \right) ds \right|.
\end{multline*}
which induces
\begin{equation}
LJ_{1} \leq C_{8.1.1}(\beta,\epsilon) ||v_{\beta - \gamma_{2}}(\tau)||_{(\beta - \gamma_{2},\underline{\varsigma},RJ_{c,d,\upsilon},\epsilon)} 
\end{equation}
with
\begin{multline*}
C_{8.1.1}(\beta,\epsilon) = \sup_{\tau \in RJ_{c,d,\upsilon}}
\exp \left( -\frac{\sigma_{1}}{|\epsilon|}r_{b}(\beta) |\tau| - \varsigma_{2}r_{b}(\beta)\exp( \varsigma_{3} |\tau| ) \right)
\int_{0}^{1} |\tau - c_{RJ}(\tau)u|^{\gamma_0} \\
\times |c_{RJ}(\tau)|^{\gamma_{1} + 2} u^{\gamma_{1}+1}
\exp \left( \frac{\sigma_{1}}{|\epsilon|}r_{b}(\beta - \gamma_{2}) |c_{RJ}(\tau)u| +
\varsigma_{2}r_{b}(\beta - \gamma_{2})\exp( \varsigma_{3} |c_{RJ}(\tau) u| ) \right) du.
\end{multline*}
According to the properties (\ref{defin_cRJ}), we observe in particular that
\begin{multline}
-\varsigma_{2}r_{b}(\beta)\exp( \varsigma_{3}|\tau| ) + \varsigma_{2}r_{b}(\beta - \gamma_{2})\exp( \varsigma_{3}|c_{RJ}(\tau)|u )\\
\leq
\varsigma_{2}(r_{b}(\beta - \gamma_{2}) - r_{b}(\beta))\exp( \varsigma_{3}|\tau| ) \leq 0
\end{multline}
for all $\tau \in RJ_{c,d,\upsilon}$, all $0 \leq u \leq 1$. In addition, taking into account the bounds
(\ref{difference_s_b_r_b}), (\ref{A_1_bounds}), we get in a similar way as in (\ref{C511_bds}) that
\begin{multline*}
C_{8.1.1}(\beta,\epsilon) \leq 2^{\gamma_0} \sup_{\tau \in RJ_{c,d,\upsilon}} |\tau|^{\gamma_{0}+\gamma_{1}+2}
\exp \left( -\frac{\sigma_{1}}{|\epsilon|}(r_{b}(\beta) - r_{b}(\beta - \gamma_{2})) |\tau| \right)\\
\leq 2^{\gamma_0} \sup_{x \geq 0} x^{\gamma_{0}+\gamma_{1}+2}
\exp \left( -\frac{\sigma_{1}}{|\epsilon|}\frac{\gamma_2}{(\beta + 1)^{b}} x \right)\\
\leq
2^{\gamma_0} |\epsilon|^{\gamma_{0}+\gamma_{1}+2}
\left( \frac{\gamma_{0}+\gamma_{1}+2}{\sigma_{1} \gamma_{2}} \right)^{\gamma_{0}+\gamma_{1}+2} \exp( -(\gamma_{0}+\gamma_{1}+2) )
(\beta + 1)^{b(\gamma_{0}+\gamma_{1}+2)}
\end{multline*}
for all $\beta \geq \gamma_{2}$, all $\epsilon \in \dot{D}(0,\epsilon_{0})$.

In the last part, we aim attention at
$$ LJ_{2} = ||\tau \int_{c_{RJ}(\tau)}^{\tau} (\tau - s)^{\gamma_0}
s^{\gamma_1} v_{\beta - \gamma_{2}}(s) ds||_{(\beta,\underline{\varsigma},RJ_{c,d,\upsilon},\epsilon)}. $$
As aforementioned, we achieve a factorization
\begin{multline*}
\frac{1}{|\tau|} \exp \left( -\frac{\sigma_{1}}{|\epsilon|}r_{b}(\beta) |\tau| -\varsigma_{2}r_{b}(\beta)\exp( \varsigma_{3}|\tau| )
\right) |\tau|
\left| \int_{c_{RJ}(\tau)}^{\tau} (\tau - s)^{\gamma_0} s^{\gamma_1}
v_{\beta - \gamma_{2}}(s) ds \right|\\
= \exp \left( -\frac{\sigma_{1}}{|\epsilon|}r_{b}(\beta) |\tau| -\varsigma_{2}r_{b}(\beta)\exp( \varsigma_{3}|\tau| )
\right) \left| \int_{c_{RJ}(\tau)}^{\tau} (\tau - s)^{\gamma_0} s^{\gamma_1} \right.\\
\times 
\{ \frac{1}{|s|} \exp \left( - \frac{\sigma_{1}}{|\epsilon|}r_{b}(\beta - \gamma_{2})|s| -\varsigma_{2}r_{b}(\beta-\gamma_{2})
\exp(\varsigma_{3}|s|) \right) v_{\beta - \gamma_{2}}(s) \} \\
\left. \times 
|s| \exp \left( \frac{\sigma_{1}}{|\epsilon|}r_{b}(\beta - \gamma_{2})|s| + \varsigma_{2}r_{b}(\beta - \gamma_{2})
\exp(\varsigma_{3}|s|) \right) ds \right|.
\end{multline*}
It follows that
\begin{equation}
LJ_{2} \leq C_{8.1.2}(\beta,\epsilon) ||v_{\beta - \gamma_{2}}(\tau)||_{(\beta - \gamma_{2},\underline{\varsigma},RJ_{c,d,\upsilon},\epsilon)} 
\end{equation}
with
\begin{multline*}
C_{8.1.2}(\beta,\epsilon) = \sup_{\tau \in RJ_{c,d,\upsilon}}
\exp \left( -\frac{\sigma_{1}}{|\epsilon|}r_{b}(\beta) |\tau| -\varsigma_{2}r_{b}(\beta)\exp(\varsigma_{3}|\tau|) \right)
\int_{0}^{1} |\tau - c_{RJ}(\tau)|^{\gamma_{0}+1}(1-u)^{\gamma_{0}}\\
\times |(1-u)c_{RJ}(\tau) + u\tau|^{\gamma_{1}+1}
\exp \left( \frac{\sigma_{1}}{|\epsilon|}r_{b}(\beta - \gamma_{2})|(1-u)c_{RJ}(\tau) + u\tau| \right. \\
\left. +\varsigma_{2}r_{b}(\beta - \gamma_{2}) \exp( \varsigma_{3}|(1-u)c_{RJ}(\tau) + u\tau| ) \right) du.
\end{multline*}
Taking a glance at the features (\ref{defin_cRJ}) of the path $L_{0,\tau}$, we notice that
\begin{multline}
-\varsigma_{2}r_{b}(\beta)\exp(\varsigma_{3}|\tau|)
+\varsigma_{2}r_{b}(\beta - \gamma_{2}) \exp( \varsigma_{3}|(1-u)c_{RJ}(\tau) + u\tau| )\\
\leq -\varsigma_{2}(r_{b}(\beta) - r_{b}(\beta - \gamma_{2})) \exp( \varsigma_{3}|\tau|) \leq 0
\end{multline}
for all $\tau \in RJ_{c,d,\upsilon}$, all $0 \leq u \leq 1$. Keeping in mind (\ref{difference_s_b_r_b}), (\ref{A_1_bounds}), we obtain as
above
\begin{multline*}
C_{8.1.2}(\beta,\epsilon) \leq 2^{\gamma_{0}+1} \sup_{\tau \in RJ_{c,d,\upsilon}} |\tau|^{\gamma_{0}+\gamma_{1}+2}
\exp \left( -\frac{\sigma_{1}}{|\epsilon|}(r_{b}(\beta) - r_{b}(\beta - \gamma_{2})) |\tau| \right)\\
\leq
2^{\gamma_{0}+1} |\epsilon|^{\gamma_{0}+\gamma_{1}+2}
\left( \frac{\gamma_{0}+\gamma_{1}+2}{\sigma_{1} \gamma_{2}} \right)^{\gamma_{0}+\gamma_{1}+2} \exp( -(\gamma_{0}+\gamma_{1}+2) )
(\beta + 1)^{b(\gamma_{0}+\gamma_{1}+2)} 
\end{multline*}
for all $\beta \geq \gamma_{2}$, all $\epsilon \in \dot{D}(0,\epsilon_{0})$. Lemma 12 follows.
\end{proof}
\end{proof}

\begin{prop} Take $\gamma_{0}$ and $\gamma_{1}$ as non negative integers. Let us select 
$\underline{\varsigma} = (\sigma_{1},\varsigma_{2},\varsigma_{3})$ and
$\underline{\varsigma}' = (\sigma_{1}',\varsigma_{2}',\varsigma_{3}')$ two tuples of positive real numbers in order that
\begin{equation}
\sigma_{1} > \sigma_{1}', \ \ \varsigma_{2} > \varsigma_{2}', \ \  \varsigma_{3} = \varsigma_{3}'.
\end{equation}
Then, for all $\epsilon \in \dot{D}(0,\epsilon_{0})$, the map
$v(\tau,z) \mapsto \tau \int_{L_{0,\tau}} (\tau - s)^{\gamma_{0}} s^{\gamma_1} v(s,z) ds$ is a linear bounded operator from
$SEG_{(\underline{\varsigma}',RJ_{c,d,\upsilon},\epsilon,\delta)}$ into $SEG_{(\underline{\varsigma},RJ_{c,d,\upsilon},\epsilon,\delta)}$.
Besides, one can choose a constant $\check{C}_{8}>0$
(depending on $\gamma_{0},\gamma_{1},\sigma_{1}$ and $\sigma_{1}'$) independent of $\epsilon$, such that
\begin{equation}
|| \tau \int_{L_{0,\tau}} (\tau - s)^{\gamma_0} s^{\gamma_1}
v(s,z) ds ||_{(\underline{\varsigma},RJ_{c,d,\upsilon},\epsilon,\delta)}
\leq \check{C}_{8}|\epsilon|^{\gamma_{0}+\gamma_{1}+2} ||v(\tau,z)||_{(\underline{\varsigma}',RJ_{c,d,\upsilon},\epsilon,\delta)}
\label{norm_conv_v_checkC8}
\end{equation}
for all $v(\tau,z) \in SEG_{(\underline{\varsigma},RJ_{c,d,\upsilon},\epsilon,\delta)}$, all $\epsilon \in \dot{D}(0,\epsilon_{0})$.
\end{prop}
\begin{proof} As above, we only concentrate on the main part of the proof since it is very close to the one of Proposition 14. More precisely,
we are scaled down to prove the next lemma.
\begin{lemma} Let $v_{\beta}(\tau)$ belonging to $SEG_{(\beta,\underline{\varsigma}',RJ_{c,d,\upsilon},\epsilon)}$.
One can sort a constant $\check{C}_{8}>0$ (depending on $\gamma_{0},\gamma_{1},\sigma_{1}$ and $\sigma_{1}'$) such that
\begin{equation}
|| \tau \int_{L_{0,\tau}} (\tau - s)^{\gamma_0} s^{\gamma_1}
v_{\beta}(s) ds ||_{(\beta,\underline{\varsigma},RJ_{c,d,\upsilon},\epsilon)}
\leq
\check{C}_{8} |\epsilon|^{\gamma_{0}+\gamma_{1}+2}
||v_{\beta}(\tau)||_{(\beta,\underline{\varsigma}',RJ_{c,d,\upsilon},\epsilon)}
\end{equation}
for all $\beta \geq 0$. 
\end{lemma}
\begin{proof} We first downsize the integral in two pieces
$$
\tau \int_{L_{0,\tau}} (\tau - s)^{\gamma_0} s^{\gamma_1} v_{\beta}(s) ds =
\tau \int_{0}^{c_{RJ}(\tau)} (\tau - s)^{\gamma_0} s^{\gamma_1} v_{\beta}(s) ds 
+
\tau \int_{c_{RJ}(\tau)}^{\tau} (\tau - s)^{\gamma_0} s^{\gamma_1} v_{\beta}(s) ds
$$
We ask for bounds regarding
$$ \check{LJ}_{1} = || \tau \int_{0}^{c_{RJ}(\tau)} (\tau - s)^{\gamma_0} s^{\gamma_1}
v_{\beta}(s) ds ||_{(\beta,\underline{\varsigma},RJ_{c,d,\upsilon},\epsilon)}. $$
The next factorization holds
\begin{multline*}
\frac{1}{|\tau|} \exp \left( -\frac{\sigma_{1}}{|\epsilon|}r_{b}(\beta) |\tau| -\varsigma_{2}r_{b}(\beta) \exp( \varsigma_{3}|\tau| ) \right)
|\tau|
\left| \int_{0}^{c_{RJ}(\tau)} (\tau - s)^{\gamma_0} s^{\gamma_1}
v_{\beta}(s) ds \right|\\
= \exp \left( -\frac{\sigma_{1}}{|\epsilon|}r_{b}(\beta) |\tau| - \varsigma_{2}r_{b}(\beta) \exp( \varsigma_{3} |\tau| ) \right)
\left| \int_{0}^{c_{RJ}(\tau)} (\tau - s)^{\gamma_0} s^{\gamma_1} \right. \\
\times 
\{ \frac{1}{|s|} \exp \left( - \frac{\sigma_{1}'}{|\epsilon|}r_{b}(\beta)|s| - \varsigma_{2}'r_{b}(\beta)
\exp( \varsigma_{3}|s| ) \right) v_{\beta}(s) \} \\
\left. \times 
|s| \exp \left( \frac{\sigma_{1}'}{|\epsilon|}r_{b}(\beta)|s| + \varsigma_{2}'r_{b}(\beta)
\exp( \varsigma_{3}|s|) \right) ds \right|.
\end{multline*}
which induces
\begin{equation}
\check{LJ}_{1} \leq \check{C}_{8.1}(\beta,\epsilon) ||v_{\beta}(\tau)||_{(\beta,\underline{\varsigma}',RJ_{c,d,\upsilon},\epsilon)} 
\end{equation}
where
\begin{multline*}
\check{C}_{8.1}(\beta,\epsilon) = \sup_{\tau \in RJ_{c,d,\upsilon}}
\exp \left( -\frac{\sigma_{1}}{|\epsilon|}r_{b}(\beta) |\tau| - \varsigma_{2}r_{b}(\beta)\exp( \varsigma_{3} |\tau| ) \right)
\int_{0}^{1} |\tau - c_{RJ}(\tau)u|^{\gamma_0} \\
\times |c_{RJ}(\tau)|^{\gamma_{1} + 2} u^{\gamma_{1}+1}
\exp \left( \frac{\sigma_{1}'}{|\epsilon|}r_{b}(\beta) |c_{RJ}(\tau)u| +
\varsigma_{2}'r_{b}(\beta)\exp( \varsigma_{3} |c_{RJ}(\tau) u| ) \right) du.
\end{multline*}
In accordance with the construction of the path $L_{0,\tau}$ described in (\ref{defin_cRJ}), we grant that
\begin{equation}
-\varsigma_{2}r_{b}(\beta)\exp( \varsigma_{3}|\tau| ) + \varsigma_{2}'r_{b}(\beta)\exp( \varsigma_{3}|c_{RJ}(\tau)|u )
\leq (\varsigma_{2}' - \varsigma_{2})r_{b}(\beta) \exp( \varsigma_{3}|\tau| ) \leq 0
\end{equation}
for all $\tau \in RJ_{c,d,\upsilon}$, all $0 \leq u \leq 1$.

Besides, taking into account the bounds (\ref{checkA1_bds}), we deduce
\begin{multline}
\check{C}_{8.1}(\beta,\epsilon) \leq 2^{\gamma_0} \sup_{\tau \in RJ_{c,d,\upsilon}} |\tau|^{\gamma_{0}+\gamma_{1}+2}
\exp \left( -\frac{\sigma_{1} - \sigma_{1}'}{|\epsilon|} r_{b}(\beta) |\tau| \right)\\
\leq
2^{\gamma_0} \sup_{x \geq 0} x^{\gamma_{0}+\gamma_{1}+2} \exp \left( -\frac{\sigma_{1} - \sigma_{1}'}{|\epsilon|} r_{b}(\beta) x \right)
\leq 2^{\gamma_0} |\epsilon|^{\gamma_{0}+\gamma_{1}+2}
\left( \frac{(\gamma_{0}+\gamma_{1}+2)e^{-1}}{\sigma_{1} - \sigma_{1}'} \right)^{\gamma_{0}+\gamma_{1}+2}
\end{multline}
for all $\beta \geq 0$, $\epsilon \in \dot{D}(0,\epsilon_{0})$.

In a second part, we concentrate on
$$ \check{LJ}_{2} = ||\tau \int_{c_{RJ}(\tau)}^{\tau} (\tau - s)^{\gamma_0}
s^{\gamma_1} v_{\beta}(s) ds||_{(\beta,\underline{\varsigma},RJ_{c,d,\upsilon},\epsilon)}. $$
Again we use a factorization
\begin{multline*}
\frac{1}{|\tau|} \exp \left( -\frac{\sigma_{1}}{|\epsilon|}r_{b}(\beta) |\tau| -\varsigma_{2}r_{b}(\beta)\exp( \varsigma_{3}|\tau| )
\right) |\tau|
\left| \int_{c_{RJ}(\tau)}^{\tau} (\tau - s)^{\gamma_0} s^{\gamma_1}
v_{\beta}(s) ds \right|\\
= \exp \left( -\frac{\sigma_{1}}{|\epsilon|}r_{b}(\beta) |\tau| -\varsigma_{2}r_{b}(\beta)\exp( \varsigma_{3}|\tau| )
\right) \left| \int_{c_{RJ}(\tau)}^{\tau} (\tau - s)^{\gamma_0} s^{\gamma_1} \right.\\
\times 
\{ \frac{1}{|s|} \exp \left( - \frac{\sigma_{1}'}{|\epsilon|}r_{b}(\beta)|s| -\varsigma_{2}'r_{b}(\beta)
\exp(\varsigma_{3}|s|) \right) v_{\beta}(s) \} \\
\left. \times 
|s| \exp \left( \frac{\sigma_{1}'}{|\epsilon|}r_{b}(\beta)|s| + \varsigma_{2}'r_{b}(\beta)
\exp(\varsigma_{3}|s|) \right) ds \right|.
\end{multline*}
that induces
\begin{equation}
\check{LJ}_{2} \leq \check{C}_{8.2}(\beta,\epsilon) ||v_{\beta}(\tau)||_{(\beta,\underline{\varsigma}',RJ_{c,d,\upsilon},\epsilon)} 
\end{equation}
with
\begin{multline*}
\check{C}_{8.2}(\beta,\epsilon) = \sup_{\tau \in RJ_{c,d,\upsilon}}
\exp \left( -\frac{\sigma_{1}}{|\epsilon|}r_{b}(\beta) |\tau| -\varsigma_{2}r_{b}(\beta)\exp(\varsigma_{3}|\tau|) \right)
\int_{0}^{1} |\tau - c_{RJ}(\tau)|^{\gamma_{0}+1}(1-u)^{\gamma_{0}}\\
\times |(1-u)c_{RJ}(\tau) + u\tau|^{\gamma_{1}+1}
\exp \left( \frac{\sigma_{1}'}{|\epsilon|}r_{b}(\beta)|(1-u)c_{RJ}(\tau) + u\tau| \right. \\
\left. +\varsigma_{2}'r_{b}(\beta) \exp( \varsigma_{3}|(1-u)c_{RJ}(\tau) + u\tau| ) \right) du.
\end{multline*}
The construction of $L_{0,\tau}$ through (\ref{defin_cRJ}) entails
\begin{multline}
-\varsigma_{2}r_{b}(\beta)\exp(\varsigma_{3}|\tau|)
+\varsigma_{2}'r_{b}(\beta) \exp( \varsigma_{3}|(1-u)c_{RJ}(\tau) + u\tau| )\\
\leq -(\varsigma_{2} - \varsigma_{2}')r_{b}(\beta) \exp( \varsigma_{3}|\tau|) \leq 0
\end{multline}
for all $\tau \in RJ_{c,d,\upsilon}$, all $0 \leq u \leq 1$. 

According to the bounds (\ref{checkA1_bds}), we get
\begin{multline}
\check{C}_{8.2}(\beta,\epsilon) \leq 2^{\gamma_{0}+1} \sup_{\tau \in RJ_{c,d,\upsilon}} |\tau|^{\gamma_{0}+\gamma_{1}+2}
\exp \left( -\frac{\sigma_{1} - \sigma_{1}'}{|\epsilon|} r_{b}(\beta) |\tau| \right)\\
\leq 2^{\gamma_{0}+1} |\epsilon|^{\gamma_{0}+\gamma_{1}+2}
\left( \frac{(\gamma_{0}+\gamma_{1}+2)e^{-1}}{\sigma_{1} - \sigma_{1}'} \right)^{\gamma_{0}+\gamma_{1}+2}
\end{multline}
for all $\beta \geq 0$, $\epsilon \in \dot{D}(0,\epsilon_{0})$. Finally, Lemma 13 is justified.
\end{proof}
\end{proof}

The proof of the next proposition is a straightforward adaptation of the one disclosed in Proposition 4 and will therefore be overlooked.

\begin{prop} Let us consider some holomorphic function $c(\tau,z,\epsilon)$ on $\mathring{RJ}_{c,d,\upsilon} \times D(0,\rho) \times 
D(0,\epsilon_{0})$, continuous on $RJ_{c,d,\upsilon} \times D(0,\rho) \times D(0,\epsilon_{0})$, for a radius $\rho>0$, bounded therein
by a constant $M_{c}>0$. Fix some $0 < \delta < \rho$. Then, the linear operator
$v(\tau,z) \mapsto c(\tau,z,\epsilon)v(\tau,z)$ is bounded from
$SEG_{(\underline{\varsigma},RJ_{c,d,\upsilon},\epsilon,\delta)}$ into itself,
provided that $\epsilon \in \dot{D}(0,\epsilon_{0})$. Additionally, a constant $C_{9}>0$ (depending on
$M_{c},\delta,\rho$) independent of $\epsilon$ exists in a way that
\begin{equation}
||c(\tau,z,\epsilon)v(\tau,z)||_{(\underline{\varsigma},RJ_{c,d,\upsilon},\epsilon,\delta)} \leq C_{9}
||v(\tau,z)||_{(\underline{\varsigma},RJ_{c,d,\upsilon},\epsilon,\delta)}
\end{equation}
for all $v \in SEG_{(\underline{\varsigma},RJ_{c,d,\upsilon},\epsilon,\delta)}$.
\end{prop}

\subsection{Continuity bounds for linear convolution operators acting on the Banach spaces
$EG_{(\sigma_{1},S_{d} \cup D(0,r),\epsilon,\delta)}$}

We keep the notations of Section 3.2. By means of the statement of the next two propositions, we inspect linear maps constructed as
convolution
products acting on the Banach spaces of functions with exponential growth on sectors mentioned in Definition 2. In the sequel,
a sector $S_{d}$ will denote one the sector $S_{d_p}$, $0 \leq p \leq \iota-1$ just introduced after Definition 4. For all
$\tau \in S_{d} \cup D(0,r)$, $L_{0,\tau}$ merely denotes the segment $[0,\tau]$ which obviously belong to
$S_{d} \cup D(0,r)$.

\begin{prop}
Take $\gamma_{0},\gamma_{1} \geq 0$ and $\gamma_{2} \geq 1$ among the set of integers. Assume that
\begin{equation}
\gamma_{2} \geq b(\gamma_{0} + \gamma_{1} + 2) 
\end{equation}
holds. Then, for any given $\epsilon$ in $\dot{D}(0,\epsilon_{0})$, the map
$v(\tau,z) \mapsto \tau \int_{L_{0,\tau}} (\tau - s)^{\gamma_0}s^{\gamma_1} \partial_{z}^{-\gamma_2}v(s,z) ds$ represents a bounded linear
operator from $EG_{(\sigma_{1},S_{d} \cup D(0,r),\epsilon,\delta)}$ into itself. Moreover, there exists a constant
$C_{10}>0$ (depending on $\gamma_{0},\gamma_{1},\gamma_{2}$, $\sigma_{1}$ and $b$) independent of $\epsilon$, for which
\begin{multline}
|| \tau \int_{L_{0,\tau}} (\tau - s)^{\gamma_0} s^{\gamma_1} \partial_{z}^{-\gamma_{2}}
v(s,z) ds ||_{(\sigma_{1},S_{d} \cup D(0,r),\epsilon,\delta)}\\
\leq C_{10}|\epsilon|^{\gamma_{0}+\gamma_{1}+2} \delta^{\gamma_2} ||v(\tau,z)||_{(\sigma_{1},S_{d} \cup D(0,r),\epsilon,\delta)}
\label{norm_conv_partialz_v_C10}
\end{multline}
provided that $v(\tau,z) \in EG_{(\sigma_{1},S_{d} \cup D(0,r),\epsilon,\delta)}$ and $\epsilon \in \dot{D}(0,\epsilon_{0})$.
\end{prop}
\begin{proof} Since the proof mirrors the one presented for Proposition 13, we only focus attention at the next lemma.
\begin{lemma} Let $v_{\beta - \gamma_{2}}(\tau)$ belonging to $EG_{(\beta - \gamma_{2},\sigma_{1},S_{d} \cup D(0,r),\epsilon)}$.
Then, one can select a constant $C_{10.1}>0$ (depending on $\gamma_{0},\gamma_{1},\gamma_{2}$ and $\sigma_{1}$) such that
\begin{multline}
|| \tau \int_{L_{0,\tau}} (\tau - s)^{\gamma_0} s^{\gamma_1} v_{\beta - \gamma_{2}}(s) ds ||_{(\beta,\sigma_{1},S_{d} \cup D(0,r),\epsilon)}
\\
\leq C_{10.1} |\epsilon|^{\gamma_{0} + \gamma_{1} + 2}(\beta + 1)^{b(\gamma_{0}+\gamma_{1}+2)}
||v_{\beta - \gamma_{2}}(\tau) ||_{(\beta - \gamma_{2},\sigma_{1},S_{d} \cup D(0,r),\epsilon)} \label{norm_conv_v_beta_minus_gamma2_C10.1}
\end{multline}
for all $\beta \geq \gamma_{2}$.
\end{lemma}
\begin{proof} We first perform a factorization 
\begin{multline*}
\frac{1}{|\tau|} \exp \left( -\frac{\sigma_{1}}{|\epsilon|}r_{b}(\beta) |\tau| \right) |\tau|
\left| \int_{0}^{\tau} (\tau - s)^{\gamma_0} s^{\gamma_1}
v_{\beta - \gamma_{2}}(s) ds \right|\\
= \exp \left( -\frac{\sigma_{1}}{|\epsilon|}r_{b}(\beta) |\tau| \right) \left| \int_{0}^{\tau} (\tau - s)^{\gamma_0} s^{\gamma_1}
\{ \frac{1}{|s|} \exp \left( - \frac{\sigma_{1}}{|\epsilon|}r_{b}(\beta - \gamma_{2})|s| \right) v_{\beta - \gamma_{2}}(s) \} \right. \\
\left. \times 
|s| \exp \left( \frac{\sigma_{1}}{|\epsilon|}r_{b}(\beta - \gamma_{2})|s| \right) ds \right|.
\end{multline*}
We deduce that
\begin{equation}
|| \tau \int_{0}^{\tau} (\tau - s)^{\gamma_0} s^{\gamma_1} v_{\beta - \gamma_{2}}(s) ds ||_{(\beta,\sigma_{1},S_{d} \cup D(0,r),\epsilon)}
\leq C_{10.1}(\beta,\epsilon) ||v_{\beta - \gamma_{2}}(\tau)||_{(\beta - \gamma_{2},\sigma_{1},S_{d} \cup D(0,r),\epsilon)} 
\end{equation}
where $C_{10.1}(\beta,\epsilon)$ fulfills the next bounds, with the help of (\ref{difference_s_b_r_b}), (\ref{A_1_bounds}),
\begin{multline}
C_{10.1}(\beta,\epsilon) = \sup_{\tau \in S_{d} \cup D(0,r)}
\exp \left( -\frac{\sigma_{1}}{|\epsilon|}r_{b}(\beta) |\tau| \right) \int_{0}^{1} |\tau|^{\gamma_{0}+\gamma_{1}+2} (1-u)^{\gamma_0}
u^{\gamma_{1}+1}\\
\times \exp \left( \frac{\sigma_{1}}{|\epsilon|}r_{b}(\beta - \gamma_{2}) |\tau| u \right) du\\
\leq \sup_{\tau \in S_{d} \cup D(0,r)} |\tau|^{\gamma_{0}+\gamma_{1}+2}
\exp \left( -\frac{\sigma_{1}}{|\epsilon|}(r_{b}(\beta) - r_{b}(\beta - \gamma_{2})) |\tau| \right)\\
\leq \sup_{x \geq 0} x^{\gamma_{0}+\gamma_{1}+2}
\exp \left( -\frac{\sigma_{1}}{|\epsilon|}\frac{\gamma_2}{(\beta + 1)^{b}} x \right)\\
\leq |\epsilon|^{\gamma_{0}+\gamma_{1}+2}
\left( \frac{\gamma_{0}+\gamma_{1}+2}{\sigma_{1} \gamma_{2}} \right)^{\gamma_{0}+\gamma_{1}+2} \exp( -(\gamma_{0}+\gamma_{1}+2) )
(\beta + 1)^{b(\gamma_{0}+\gamma_{1}+2)} \label{C10.1_bds}
\end{multline}
for all $\beta \geq \gamma_{2}$, all $\epsilon \in \dot{D}(0,\epsilon_{0})$. This yields the lemma 14.
\end{proof}
\end{proof}

\begin{prop} Let $\gamma_{0},\gamma_{1} \geq 0$ chosen among the set of integers. Let $\sigma_{1},\sigma_{1}'>0$ be real numbers
satisfying $\sigma_{1} > \sigma_{1}'$.
Then, for all $\epsilon \in \dot{D}(0,\epsilon_{0})$, the linear map
$v(\tau,z) \mapsto \tau \int_{L_{0,\tau}} (\tau - s)^{\gamma_0}s^{\gamma_1}v(s,z) ds$ is a bounded operator from
$EG_{(\sigma_{1}',S_{d} \cup D(0,r),\epsilon,\delta)}$ into
$EG_{(\sigma_{1},S_{d} \cup D(0,r),\epsilon,\delta)}$. Furthermore, we can get
a constant $\check{C}_{10}>0$ (depending on $\gamma_{0},\gamma_{1},\sigma_{1}$ and $\sigma_{1}'$) with
\begin{equation}
|| \tau \int_{L_{0,\tau}} (\tau - s)^{\gamma_0}s^{\gamma_1}v(s,z) ds ||_{(\sigma_{1},S_{d} \cup D(0,r),\epsilon,\delta)} \leq
\check{C}_{10} |\epsilon|^{\gamma_{0}+\gamma_{1}+2} ||v(\tau,z)||_{(\sigma_{1}',S_{d} \cup D(0,r),\epsilon,\delta)}
\end{equation}
for all $v(\tau,z) \in EG_{(\sigma_{1}',S_{d} \cup D(0,r),\epsilon,\delta)}$, for all $\epsilon \in \dot{D}(0,\epsilon_{0})$.
\end{prop}
\begin{proof}
The proof mimics the one of Proposition 14 and is based on the next lemma
\begin{lemma}
One can attach a constant $\check{C}_{10}>0$ (depending on $\gamma_{0},\gamma_{1},\sigma_{1}$ and $\sigma_{1}'$) such that
\begin{equation}
|| \tau \int_{L_{0,\tau}} (\tau - s)^{\gamma_0} s^{\gamma_1} v_{\beta}(s) ds ||_{(\beta,\sigma_{1},S_{d} \cup D(0,r),\epsilon)}
\leq \check{C}_{10} |\epsilon|^{\gamma_{0} + \gamma_{1} + 2}
||v_{\beta}(\tau) ||_{(\beta,\sigma_{1}',S_{d} \cup D(0,r),\epsilon)} \label{norm_conv_v_beta_sigma1_sigma1_prim_checkC10}
\end{equation}
for all $\beta \geq 0$.
\end{lemma}
\begin{proof} 
We apply the next factorization
\begin{multline*}
\frac{1}{|\tau|} \exp \left( -\frac{\sigma_{1}}{|\epsilon|}r_{b}(\beta) |\tau| \right) |\tau|
\left| \int_{0}^{\tau} (\tau - s)^{\gamma_0} s^{\gamma_1}
v_{\beta}(s) ds \right|\\
= \exp \left( -\frac{\sigma_{1}}{|\epsilon|}r_{b}(\beta) |\tau| \right) \left| \int_{0}^{\tau} (\tau - s)^{\gamma_0} s^{\gamma_1}
\{ \frac{1}{|s|} \exp \left( - \frac{\sigma_{1}'}{|\epsilon|}r_{b}(\beta)|s| \right) v_{\beta}(s) \} \right. \\
\left. \times 
|s| \exp \left( \frac{\sigma_{1}'}{|\epsilon|}r_{b}(\beta)|s| \right) ds \right|.
\end{multline*}
which entails
\begin{equation}
|| \tau \int_{0}^{\tau} (\tau - s)^{\gamma_0} s^{\gamma_1} v_{\beta}(s) ds ||_{(\beta,\sigma_{1},S_{d} \cup D(0,r),\epsilon)}
\leq \check{C}_{10}(\beta,\epsilon) ||v_{\beta}(\tau)||_{(\beta,\sigma_{1}',S_{d} \cup D(0,r),\epsilon)} 
\end{equation}
for $\check{C}_{10}(\beta,\epsilon)$ submitted to the next bounds, keeping in view (\ref{checkA1_bds}),
\begin{multline}
\check{C}_{10}(\beta,\epsilon) = \sup_{\tau \in S_{d} \cup D(0,r)}
\exp \left( -\frac{\sigma_{1}}{|\epsilon|}r_{b}(\beta) |\tau| \right) \int_{0}^{1} |\tau|^{\gamma_{0}+\gamma_{1}+2} (1 - u)^{\gamma_0}
u^{\gamma_{1}+1}\\
\times \exp \left( \frac{\sigma_{1}'}{|\epsilon|}r_{b}(\beta) |\tau| u \right) du\\
\leq \sup_{\tau \in S_{d} \cup D(0,r)} |\tau|^{\gamma_{0}+\gamma_{1}+2}
\exp \left( -\frac{\sigma_{1} - \sigma_{1}'}{|\epsilon|} r_{b}(\beta) |\tau| \right)\\
\leq \sup_{x \geq 0} x^{\gamma_{0}+\gamma_{1}+2} \exp \left( -\frac{\sigma_{1} - \sigma_{1}'}{|\epsilon|} r_{b}(\beta) x \right)
\leq |\epsilon|^{\gamma_{0}+\gamma_{1}+2}
\left( \frac{(\gamma_{0}+\gamma_{1}+2)e^{-1}}{\sigma_{1} - \sigma_{1}'} \right)^{\gamma_{0}+\gamma_{1}+2}
\end{multline}
for all $\beta \geq 0$, $\epsilon \in \dot{D}(0,\epsilon_{0})$. Lemma 15 follows.
\end{proof}
\end{proof}

\subsection{An accessory convolution problem with rational coefficients}

We set $\mathcal{B}$ as a finite subset of $\mathbb{N}^{3}$. For any $\underline{l}=(l_{0},l_{1},l_{2}) \in \mathcal{B}$, we consider
a bounded holomorphic function $d_{\underline{l}}(z,\epsilon)$ on a polydisc $D(0,\rho) \times D(0,\epsilon_{0})$ for some radii
$\rho,\epsilon_{0}>0$.
Let $S_{\mathcal{B}} \geq 1$ be an integer and $P_{\mathcal{B}}(\tau)$ be a polynomial (not identically equal to 0) with complex coefficients
which is either constant or whose roots that are located
in the open right halfplane $\mathbb{C}_{+} = \{ z \in \mathbb{C} / \mathrm{Re}(z) > 0 \}$. We introduce the following notations. When
$\underline{l} = (l_{0},l_{1},l_{2}) \in \mathcal{B}$, we put $d_{l_{0},l_{1}} = l_{0} - 2l_{1}$ and assume that $d_{l_{0},l_{1}} \geq 1$,
we also set $A_{l_{1},p}$ as real numbers for all $1 \leq p \leq l_{1}-1$ when $l_{1} \geq 2$. When $\tau \in \mathbb{C}$, the symbol
$L_{0,\tau}$ stands for a path in $\mathbb{C}$ joining $0$ and $\tau$ as constructed in the previous subsections.

We concentrate on the next convolution equation
\begin{multline}
\partial_{z}^{S_{\mathcal{B}}}v(\tau,z,\epsilon) = \sum_{\underline{l} = (l_{0},l_{1},l_{2}) \in \mathcal{B}}
\frac{d_{\underline{l}}(z,\epsilon)}{P_{\mathcal{B}}(\tau)} \left \{ \frac{ \epsilon^{l_{1} - l_{0}} \tau}{\Gamma( d_{l_{0},l_{1}} )}
\int_{L_{0,\tau}} (\tau - s)^{d_{l_{0},l_{1}} - 1} s^{l_1} \partial_{z}^{l_2}v(s,z,\epsilon) \frac{ds}{s} \right. \\
\left . +
\sum_{1 \leq p \leq l_{1}-1} A_{l_{1},p} \frac{ \epsilon^{l_{1} - l_{0}} \tau }{\Gamma( d_{l_{0},l_{1}} + (l_{1} -p) )}
\int_{L_{0,\tau}} (\tau - s)^{d_{l_{0},l_{1}} + (l_{1} - p) - 1} s^{p} \partial_{z}^{l_2}v(s,z,\epsilon) \frac{ds}{s} 
 \right \} + w(\tau,z,\epsilon) \label{ACP_forcterm_w}
\end{multline}
where $w(\tau,z,\epsilon)$ stands for solutions of the equation (\ref{1_aux_CP}) that are constructed in Propositions 10 and 11. We
use the convention that the
sum $\sum_{1 \leq p \leq l_{1}-1}$ is reduced to 0 when $l_{1}=1$.

In the next assertion, we build solutions to the convolution equation (\ref{ACP_forcterm_w}) within the three families of Banach spaces
described in Definitions 2, 5 and 6.

\begin{prop} 1) We ask for the next constraints\\
a) There exists a real number $b>1$ such that for all $\underline{l}=(l_{0},l_{1},l_{2}) \in \mathcal{B}$,
\begin{equation}
S_{\mathcal{B}} \geq b(l_{0} - l_{1}) + l_{2} \ \ , \ \ S_{\mathcal{B}} > l_{2} \ \ , \ \ l_{1} \geq 1 \label{cond_S_B_b_l}
\end{equation}
holds.\\
b) For all $0 \leq j \leq S_{\mathcal{B}} - 1$, we set $\tau \mapsto v_{j}(\tau,\epsilon)$ as a function that belongs to the Banach space
$EG_{(0,\sigma_{1}',RH_{a,b,\upsilon},\epsilon)}$, for all $\epsilon \in \dot{D}(0,\epsilon_{0})$, for a $L-$shaped domain
$RH_{a,b,\upsilon}$ displayed at the onset of Subsection 4.1 and some real number $\sigma_{1}'>0$. Furthermore, we assume the existence of
positive real numbers $J,\delta>0$ for which
\begin{equation}
\sum_{j=0}^{S_{\mathcal{B}} - 1 -h} ||v_{j+h}(\tau,\epsilon)||_{(0,\sigma_{1}',RH_{a,b,\upsilon},\epsilon)} \frac{\delta^{j}}{j!} \leq J
\label{i_d_v_j_less_J}
\end{equation}
for any $0 \leq h \leq S_{\mathcal{B}} - 1$, for $\epsilon \in \dot{D}(0,\epsilon_{0})$.

Then, for any given $\sigma_{1} > \sigma_{1}'$, for a suitable choice of constants $\Lambda>0$ and $0 < \delta < \rho$, the equation
(\ref{ACP_forcterm_w}) where the forcing term $w(\tau,z,\epsilon)$ needs to be supplanted by $w_{HJ_n}(\tau,z,\epsilon)$ along with the initial data
\begin{equation}
(\partial_{z}^{j}v)(\tau,0,\epsilon) = v_{j}(\tau,\epsilon) \ \ , \ \ 0 \leq j \leq S_{\mathcal{B}} - 1 \label{ACP_forcterm_w_i_d} 
\end{equation}
has a unique solution $v(\tau,z,\epsilon)$ in the space $EG_{(\sigma_{1},RH_{a,b,\upsilon},\epsilon,\delta)}$, for all
$\epsilon \in \dot{D}(0,\epsilon_{0})$ and is submitted to the bounds
\begin{equation}
||v(\tau,z,\epsilon)||_{(\sigma_{1},RH_{a,b,\upsilon},\epsilon,\delta)} \leq \delta^{S_{\mathcal{B}}}\Lambda + J
\label{norm_v_RHab_less_J}
\end{equation}
for all $\epsilon \in \dot{D}(0,\epsilon_{0})$.\\
2) We need the following restrictions to hold\\
a) There exists a real number $b>1$ for which (\ref{cond_S_B_b_l}) occurs.\\
b) For all $0 \leq j \leq S_{\mathcal{B}}-1$, we define $\tau \mapsto v_{j}(\tau,\epsilon)$ as a function that belongs to the Banach space
$SEG_{(0,\underline{\varsigma}',RJ_{c,d,\upsilon},\epsilon)}$, for any $\epsilon \in \dot{D}(0,\epsilon_{0})$, for some
$L-$shaped domain $RJ_{c,d,\upsilon}$ described at the beginning of Subsection 4.2 and for some tuple
$\underline{\varsigma}'=(\sigma_{1}',\varsigma_{2}',\varsigma_{3}')$ with $\sigma_{1}'>0$,$\varsigma_{2}'>0$ and $\varsigma_{3}'>0$. Moreover,
we can select real numbers $J,\delta>0$ with
$$
\sum_{j=0}^{S_{\mathcal{B}} - 1 -h} ||v_{j+h}(\tau,\epsilon)||_{(0,\underline{\varsigma}',RJ_{c,d,\upsilon},\epsilon)} \frac{\delta^{j}}{j!} \leq J
$$
for any $0 \leq h \leq S_{\mathcal{B}} - 1$, for $\epsilon \in \dot{D}(0,\epsilon_{0})$.

Then, for any given tuple $\underline{\varsigma} = (\sigma_{1},\varsigma_{2},\varsigma_{3})$ with $\sigma_{1} > \sigma_{1}'$,
$\varsigma_{2}>\varsigma_{2}'$ and $\varsigma_{3}=\varsigma_{3}'$, for an appropriate choice of constants $\Lambda>0$ and
$0 < \delta < \rho$, the equation (\ref{ACP_forcterm_w}) where the forcing term $w(\tau,z,\epsilon)$ must be interchanged with
$w_{HJ_n}(\tau,z,\epsilon)$ together with the initial data (\ref{ACP_forcterm_w_i_d}) possesses a unique
solution $v(\tau,z,\epsilon)$ in the space $SEG_{(\underline{\varsigma},RJ_{c,d,\upsilon},\epsilon,\delta)}$ which suffers the bounds 
\begin{equation}
||v(\tau,z,\epsilon)||_{(\underline{\varsigma},RJ_{c,d,\upsilon},\epsilon,\delta)} \leq \delta^{S_{\mathcal{B}}}\Lambda + J
\label{norm_v_RJcd_less_J}
\end{equation}
for all $\epsilon \in \dot{D}(0,\epsilon_{0})$.\\
3) We request the next assumptions\\
a) For a suitable real number $b>1$, the inequalities (\ref{cond_S_B_b_l}) hold.\\
b) For each $0 \leq j \leq S_{\mathcal{B}}-1$, we single out a function $\tau \mapsto v_{j}(\tau,\epsilon)$ belonging to the Banach space
$EG_{(0,\sigma_{1}',S_{d} \cup D(0,r),\epsilon)}$, for all $\epsilon \in \dot{D}(0,\epsilon_{0})$, where $S_{d}$ is one of sectors
$S_{d_p}$, $0 \leq p \leq \iota-1$ displayed after Definition 4, for some real number $\sigma_{1}'>0$. Furthermore, we assume that no root
of $P_{\mathcal{B}}(\tau)$ is located in $\bar{S}_{d} \cup \bar{D}(0,r)$. We impose the existence of two real numbers $J,\delta>0$ in a way that
$$
\sum_{j=0}^{S_{\mathcal{B}} - 1 -h} ||v_{j+h}(\tau,\epsilon)||_{(0,\sigma_{1}',S_{d} \cup D(0,r),\epsilon)} \frac{\delta^{j}}{j!} \leq J
$$
holds for any $0 \leq h \leq S_{\mathcal{B}} - 1$, for $\epsilon \in \dot{D}(0,\epsilon_{0})$.

Then, for any given $\sigma_{1}>\sigma_{1}'$, for an adequate guess of constants $\Lambda>0$ and $0 < \delta < \rho$, the equation
(\ref{ACP_forcterm_w}) where the forcing term $w(\tau,z,\epsilon)$ shall be replaced by $w_{S_d}(\tau,z,\epsilon)$ accompanied by the
initial data (\ref{ACP_forcterm_w_i_d}) has a unique solution $v(\tau,z,\epsilon)$ in the space
$EG_{(\sigma_{1},S_{d} \cup D(0,r),\epsilon,\delta)}$ withstanding the bounds
\begin{equation}
||v(\tau,z,\epsilon)||_{(\sigma_{1},S_{d} \cup D(0,r),\epsilon,\delta)} \leq \delta^{S_{\mathcal{B}}}\Lambda + J
\label{norm_v_Sd_less_J}
\end{equation}
for all $\epsilon \in \dot{D}(0,\epsilon_{0})$.
\end{prop}
\begin{proof} The proof will only be concerned with a thorough inspection of the first point 1) since a similar discourse holds
for the second (resp. third) point by merely replacing Propositions 13, 14 and 15 by Propositions 17, 18 and 19 (resp. 20, 21 and 8).

We keep the notations of the subsection 3.1 and we depart from a lemma dealing with the forcing term
$w(\tau,z,\epsilon)$ of the equation (\ref{ACP_forcterm_w}).
\begin{lemma}
1) The formal series $w_{HJ_n}(\tau,z,\epsilon)$ built in (\ref{formal_wHJn}) belongs both to the spaces\\
$EG_{(\sigma_{1},RH_{a,b,\upsilon},\epsilon,\delta)}$ and $SEG_{(\underline{\varsigma},RJ_{c,d,\upsilon},\epsilon,\delta)}$
for the tuples $\underline{\sigma},\underline{\varsigma}$ and $\delta$ considered in Proposition 10, for any choice of
$\upsilon<0$, provided that the sector
$H$ from $RH_{a,b,\upsilon}$ belongs to the set $\{ H_{k} \}_{k \in \llbracket -n,n \rrbracket}$ and
$J$ out of $RJ_{c,d,\upsilon}$ appertain to $\{ J_{k} \}_{k \in \llbracket -n,n \rrbracket}$. Moreover, there exist constants
$\tilde{C}_{RH_{a,b,\upsilon}}>0$ and $\tilde{C}_{RJ_{c,d,\upsilon}}>0$ for which
\begin{equation}
||w_{HJ_n}(\tau,z,\epsilon)||_{(\sigma_{1},RH_{a,b,\upsilon},\epsilon,\delta)} \leq \tilde{C}_{RH_{a,b,\upsilon}}, \ \
||w_{HJ_n}(\tau,z,\epsilon)||_{(\underline{\varsigma},RJ_{c,d,\upsilon},\epsilon,\delta)} \leq \tilde{C}_{RJ_{c,d,\upsilon}}
\label{bds_wHJn_on_Lshaped}
\end{equation}
for all $\epsilon \in \dot{D}(0,\epsilon_{0})$.\\
2) The formal series $w_{S_{d_p}}(\tau,z,\epsilon)$ defined in (\ref{defin_w_S_dp}) is located in the space
$EG_{(\sigma_{1},S_{d_p} \cup D(0,r),\epsilon,\delta)}$. Besides, there exists a constant $\tilde{C}_{S_{d_p}}>0$ with
\begin{equation}
||w_{S_{d_p}}(\tau,z,\epsilon)||_{(\sigma_{1},S_{d_p} \cup D(0,r),\epsilon,\delta)} \leq \tilde{C}_{S_{d_p}} 
\end{equation}
whenever $\epsilon \in \dot{D}(0,\epsilon_{0})$.
\end{lemma}
\begin{proof} We focus on the first point 1). According to (\ref{formal_wHJn}), the formal series
$w_{HJ_n}(\tau,z,\epsilon)$ has the following expansion
$w_{HJ_n}(\tau,z,\epsilon) = \sum_{\beta \geq 0} w_{\beta}(\tau,\epsilon) z^{\beta}/\beta!$ where
$w_{\beta}(\tau,\epsilon)$ stand for holomorphic functions on $\mathring{HJ}_{n} \times \dot{D}(0,\epsilon_{0})$, continuous
on $HJ_{n} \times \dot{D}(0,\epsilon_{0})$, for all $\beta \geq 0$. Besides, the estimates (\ref{norm_wHJn_Hk}) and
(\ref{norm_wHJn_Jk}) hold.

We first prove that $w_{HJ_n}(\tau,z,\epsilon)$ belongs to $EG_{(\sigma_{1},RH_{a,b,\upsilon},\epsilon,\delta)}$.
We need upper bounds for the quantity
$$ Rw_{a,b}(\beta,\epsilon) = \sup_{\tau \in R_{a,b,\upsilon}}
\frac{|w_{\beta}(\tau,\epsilon)|}{|\tau|} \exp \left( -\frac{\sigma_{1}}{|\epsilon|}r_{b}(\beta) |\tau| \right). $$
Since $R_{a,b,\upsilon} \subset HJ_{n} = \cup_{k \in \llbracket -n,n \rrbracket} H_{k} \cup J_{k}$, we get in particular the coarse bounds
\begin{multline}
Rw_{a,b}(\beta,\epsilon) \leq \sum_{k \in \llbracket -n,n \rrbracket} \sup_{\tau \in R_{a,b,\upsilon} \cap H_{k}}
\frac{|w_{\beta}(\tau,\epsilon)|}{|\tau|} \exp \left( -\frac{\sigma_{1}}{|\epsilon|}r_{b}(\beta)|\tau| \right)\\
+ \sum_{k \in \llbracket -n,n \rrbracket} \sup_{\tau \in R_{a,b,\upsilon} \cap J_{k}}
\frac{|w_{\beta}(\tau,\epsilon)|}{|\tau|} \exp \left( -\frac{\sigma_{1}}{|\epsilon|}r_{b}(\beta)|\tau| \right). \label{Rw_defin}
\end{multline}
The sums above are taken over the integers $k$ for which $R_{a,b,\upsilon} \cap H_{k}$ and
$R_{a,b,\upsilon} \cap J_{k}$ are not empty. But, we observe that
\begin{multline}
\sup_{\tau \in R_{a,b,\upsilon} \cap H_{k}}
\frac{|w_{\beta}(\tau,\epsilon)|}{|\tau|} \exp \left( -\frac{\sigma_{1}}{|\epsilon|}r_{b}(\beta)|\tau| \right)\\
\leq
\sup_{\tau \in H_{k}}
\frac{|w_{\beta}(\tau,\epsilon)|}{|\tau|} \exp \left( -\frac{\sigma_{1}}{|\epsilon|}r_{b}(\beta)|\tau| +
\sigma_{2}s_{b}(\beta)\exp( \sigma_{3} |\tau| ) \right) = ||w_{\beta}(\tau,\epsilon)||_{(\beta,\underline{\sigma},H_{k},\epsilon)} \label{Rw_first}
\end{multline}
and if one set
$$ \mathcal{C}_{a,b,\upsilon,k} = \sup_{\tau \in R_{a,b,\upsilon} \cap J_{k}}
\exp \left( \varsigma_{2} \zeta(b) \exp( \varsigma_{3}|\tau| ) \right) $$
we see that
\begin{multline}
\sup_{\tau \in R_{a,b,\upsilon} \cap J_{k}}
\frac{|w_{\beta}(\tau,\epsilon)|}{|\tau|} \exp \left( -\frac{\sigma_{1}}{|\epsilon|}r_{b}(\beta)|\tau| \right)
= \sup_{\tau \in R_{a,b,\upsilon} \cap J_{k}}
\frac{|w_{\beta}(\tau,\epsilon)|}{|\tau|} \exp \left( -\frac{\sigma_{1}}{|\epsilon|}r_{b}(\beta)|\tau| \right)\\
\times 
\exp( - \varsigma_{2}r_{b}(\beta) \exp( \varsigma_{3}|\tau| ) )
\times \exp( \varsigma_{2}r_{b}(\beta) \exp( \varsigma_{3}|\tau| ) ) \leq
\mathcal{C}_{a,b,\upsilon,k} \sup_{\tau \in J_{k}}
\frac{|w_{\beta}(\tau,\epsilon)|}{|\tau|} \exp \left( -\frac{\sigma_{1}}{|\epsilon|}r_{b}(\beta)|\tau| \right)\\
\times 
\exp( - \varsigma_{2}r_{b}(\beta) \exp( \varsigma_{3}|\tau| ) )
=\mathcal{C}_{a,b,\upsilon,k} ||w_{\beta}(\tau,\epsilon)||_{(\beta,\underline{\varsigma},J_{k},\epsilon)}. \label{Rw_second}
\end{multline}
Hence, gathering (\ref{Rw_defin}) and (\ref{Rw_first}), (\ref{Rw_second}) yields
\begin{equation}
Rw_{a,b}(\beta,\epsilon) \leq \sum_{k \in \llbracket -n,n \rrbracket} ||w_{\beta}(\tau,\epsilon)||_{(\beta,\underline{\sigma},H_{k},\epsilon)}
+ \sum_{k \in \llbracket -n,n \rrbracket} \mathcal{C}_{a,b,\upsilon,k}||w_{\beta}(\tau,\epsilon)||_{(\beta,\underline{\varsigma},J_{k},\epsilon)}
\label{Rw_third}
\end{equation}
Now, we notice that
\begin{multline}
||w_{\beta}(\tau,\epsilon)||_{(\beta,\sigma_{1},RH_{a,b,\upsilon},\epsilon)} \leq
\sup_{\tau \in R_{a,b,\upsilon}}
\frac{|w_{\beta}(\tau,\epsilon)|}{|\tau|} \exp \left( -\frac{\sigma_{1}}{|\epsilon|}r_{b}(\beta)|\tau| \right)\\
+ \sup_{\tau \in H} \frac{|w_{\beta}(\tau,\epsilon)|}{|\tau|} \exp \left( -\frac{\sigma_{1}}{|\epsilon|}r_{b}(\beta)|\tau|
+ \sigma_{2}s_{b}(\beta)\exp( \sigma_{3}|\tau| ) \right) = Rw_{a,b}(\beta,\epsilon) +
||w_{\beta}(\tau,\epsilon)||_{(\beta,\underline{\sigma},H,\epsilon)} \label{norm_w_beta_RHabu}
\end{multline}
Finally, clustering (\ref{Rw_third}) and (\ref{norm_w_beta_RHabu}) yields that
\begin{equation}
||w_{HJ}(\tau,z,\epsilon)||_{(\sigma_{1},RH_{a,b,\upsilon},\epsilon,\delta)} \leq
\sum_{k \in \llbracket -n,n \rrbracket} \tilde{C}_{H_{k}} + \sum_{k \in \llbracket -n,n \rrbracket} \mathcal{C}_{a,b,\upsilon,k}
\tilde{C}_{J_{k}} + \tilde{C}_{H}
\end{equation}
for all $\epsilon \in \dot{D}(0,\epsilon_{0})$, bearing in mind the notations within the bounds (\ref{norm_wHJn_Hk}) and (\ref{norm_wHJn_Jk}).

In a second step, we show that $w_{HJ_n}(\tau,z,\epsilon)$ belongs to
$SEG_{(\underline{\varsigma},RJ_{c,d,\upsilon},\epsilon,\delta)}$.
We search for upper bounds concerning
$$ RJw_{c,d}(\beta,\epsilon) = \sup_{\tau \in R_{c,d,\upsilon}}
\frac{|w_{\beta}(\tau,\epsilon)|}{|\tau|} \exp \left( -\frac{\sigma_{1}}{|\epsilon|}r_{b}(\beta)
|\tau| - \varsigma_{2}r_{b}(\beta)\exp( \varsigma_{3} |\tau| ) \right). $$
According to the inclusion $R_{c,d,\upsilon} \subset HJ_{n} = \cup_{k \in \llbracket -n,n \rrbracket} H_{k} \cup J_{k}$, we observe that
\begin{multline}
RJw_{c,d}(\beta,\epsilon) \leq \sum_{k \in \llbracket -n,n \rrbracket} \sup_{\tau \in R_{c,d,\upsilon} \cap H_{k}}
\frac{|w_{\beta}(\tau,\epsilon)|}{|\tau|} \exp \left( -\frac{\sigma_{1}}{|\epsilon|}r_{b}(\beta)|\tau| -
\varsigma_{2}r_{b}(\beta)\exp( \varsigma_{3}|\tau| ) \right)\\
+ \sum_{k \in \llbracket -n,n \rrbracket} \sup_{\tau \in R_{c,d,\upsilon} \cap J_{k}}
\frac{|w_{\beta}(\tau,\epsilon)|}{|\tau|} \exp \left( -\frac{\sigma_{1}}{|\epsilon|}r_{b}(\beta)|\tau| -
\varsigma_{2}r_{b}(\beta)\exp( \varsigma_{3}|\tau| ) \right). \label{RJw_defin}
\end{multline}
As above, the sums belonging to the latter inequalities are performed over the integers $k$ for which $R_{c,d,\upsilon} \cap H_{k}$ and
$R_{c,d,\upsilon} \cap J_{k}$ are not empty. Furthermore, we see that
\begin{multline}
\sup_{\tau \in R_{c,d,\upsilon} \cap H_{k}}
\frac{|w_{\beta}(\tau,\epsilon)|}{|\tau|} \exp \left( -\frac{\sigma_{1}}{|\epsilon|}r_{b}(\beta)|\tau| -
\varsigma_{2}r_{b}(\beta)\exp( \varsigma_{3} |\tau| ) \right)\\
\leq
\sup_{\tau \in H_{k}}
\frac{|w_{\beta}(\tau,\epsilon)|}{|\tau|} \exp \left( -\frac{\sigma_{1}}{|\epsilon|}r_{b}(\beta)|\tau| +
\sigma_{2}s_{b}(\beta)\exp( \sigma_{3} |\tau| ) \right) = ||w_{\beta}(\tau,\epsilon)||_{(\beta,\underline{\sigma},H_{k},\epsilon)} \label{RJw_first}
\end{multline}
and
\begin{multline}
\sup_{\tau \in R_{c,d,\upsilon} \cap J_{k}}
\frac{|w_{\beta}(\tau,\epsilon)|}{|\tau|} \exp \left( -\frac{\sigma_{1}}{|\epsilon|}r_{b}(\beta)|\tau| -
\varsigma_{2}r_{b}(\beta) \exp( \varsigma_{3} |\tau| ) \right)\\
\leq
\sup_{\tau \in J_{k}}
\frac{|w_{\beta}(\tau,\epsilon)|}{|\tau|} \exp \left( -\frac{\sigma_{1}}{|\epsilon|}r_{b}(\beta)|\tau|
- \varsigma_{2}r_{b}(\beta) \exp( \varsigma_{3} |\tau| ) \right)
= ||w_{\beta}(\tau,\epsilon)||_{(\beta,\underline{\varsigma},J_{k},\epsilon)}. \label{RJw_second}
\end{multline}
As a result, collecting (\ref{RJw_defin}) and (\ref{RJw_first}), (\ref{RJw_second}) leads to
\begin{equation}
RJw_{c,d}(\beta,\epsilon) \leq \sum_{k \in \llbracket -n,n \rrbracket} ||w_{\beta}(\tau,\epsilon)||_{(\beta,\underline{\sigma},H_{k},\epsilon)}
+ \sum_{k \in \llbracket -n,n \rrbracket} ||w_{\beta}(\tau,\epsilon)||_{(\beta,\underline{\varsigma},J_{k},\epsilon)}
\label{RJw_third}
\end{equation}
Besides, we remark that
\begin{multline}
||w_{\beta}(\tau,\epsilon)||_{(\beta,\underline{\varsigma},RJ_{c,d,\upsilon},\epsilon)} \leq
\sup_{\tau \in R_{c,d,\upsilon}}
\frac{|w_{\beta}(\tau,\epsilon)|}{|\tau|} \exp \left( -\frac{\sigma_{1}}{|\epsilon|}r_{b}(\beta)|\tau|
-\varsigma_{2}r_{b}(\beta) \exp( \varsigma_{3}|\tau| ) \right)\\
+ \sup_{\tau \in J} \frac{|w_{\beta}(\tau,\epsilon)|}{|\tau|} \exp \left( -\frac{\sigma_{1}}{|\epsilon|}r_{b}(\beta)|\tau|
- \varsigma_{2}r_{b}(\beta)\exp( \varsigma_{3}|\tau| ) \right) = RJw_{c,d}(\beta,\epsilon) +
||w_{\beta}(\tau,\epsilon)||_{(\beta,\underline{\varsigma},J,\epsilon)} \label{norm_w_beta_RJcdu}
\end{multline}
At last, storing up (\ref{RJw_third}) and (\ref{norm_w_beta_RJcdu}) returns the bounds
\begin{equation}
||w_{HJ}(\tau,z,\epsilon)||_{(\underline{\varsigma},RJ_{c,d,\upsilon},\epsilon,\delta)} \leq
\sum_{k \in \llbracket -n,n \rrbracket} \tilde{C}_{H_{k}} + \sum_{k \in \llbracket -n,n \rrbracket}
\tilde{C}_{J_{k}} + \tilde{C}_{J}
\end{equation}
for all $\epsilon \in \dot{D}(0,\epsilon_{0})$, in accordance with the bounds (\ref{norm_wHJn_Hk}) and (\ref{norm_wHJn_Jk}).

The second point 2) has already been checked in the proof of Proposition 11.
\end{proof}

Let us introduce the function
$$ V_{S_{\mathcal{B}}}(\tau,z,\epsilon) = \sum_{j=0}^{S_{\mathcal{B}}-1} v_{j}(\tau,\epsilon) \frac{z^j}{j!} $$
with $v_{j}(\tau,\epsilon)$ disclosed in 1)b) above. We set a map $B_{\epsilon}$ described as follows
\begin{multline*}
B_{\epsilon}(H(\tau,z)) := \sum_{\underline{l} = (l_{0},l_{1},l_{2}) \in \mathcal{B}}
\frac{d_{\underline{l}}(z,\epsilon)}{P_{\mathcal{B}}(\tau)} \left \{ \frac{ \epsilon^{l_{1} - l_{0}} \tau}{\Gamma( d_{l_{0},l_{1}} )}
\int_{L_{0,\tau}} (\tau - s)^{d_{l_{0},l_{1}} - 1} s^{l_1} \partial_{z}^{l_{2}-S_{\mathcal{B}}}H(s,z) \frac{ds}{s} \right. \\
\left . +
\sum_{1 \leq p \leq l_{1}-1} A_{l_{1},p} \frac{ \epsilon^{l_{1} - l_{0}} \tau }{\Gamma( d_{l_{0},l_{1}} + (l_{1} -p) )}
\int_{L_{0,\tau}} (\tau - s)^{d_{l_{0},l_{1}} + (l_{1} - p) - 1} s^{p} \partial_{z}^{l_{2}-S_{\mathcal{B}}}H(s,z) \frac{ds}{s} 
 \right \}\\
 + \sum_{\underline{l} = (l_{0},l_{1},l_{2}) \in \mathcal{B}}
\frac{d_{\underline{l}}(z,\epsilon)}{P_{\mathcal{B}}(\tau)} \left \{ \frac{ \epsilon^{l_{1} - l_{0}} \tau}{\Gamma( d_{l_{0},l_{1}} )}
\int_{L_{0,\tau}} (\tau - s)^{d_{l_{0},l_{1}} - 1} s^{l_1} \partial_{z}^{l_{2}}V_{S_{\mathcal{B}}}(s,z,\epsilon) \frac{ds}{s} \right. \\
\left . +
\sum_{1 \leq p \leq l_{1}-1} A_{l_{1},p} \frac{ \epsilon^{l_{1} - l_{0}} \tau }{\Gamma( d_{l_{0},l_{1}} + (l_{1} -p) )}
\int_{L_{0,\tau}} (\tau - s)^{d_{l_{0},l_{1}} + (l_{1} - p) - 1} s^{p} \partial_{z}^{l_{2}}V_{S_{\mathcal{B}}}(s,z,\epsilon) \frac{ds}{s} 
 \right \}\\
 + w_{HJ_n}(\tau,z,\epsilon)
\end{multline*}
In the next lemma, we explain why $B_{\epsilon}$ induces a Lipschitz shrinking map on the space\\
$EG_{(\sigma_{1},RH_{a,b,\upsilon},\epsilon,\delta)}$, for any given $\sigma_{1} > \sigma_{1}'$.
\begin{lemma}
We take for granted that the restriction (\ref{cond_S_B_b_l}) hold. Let us choose a positive real number $J$ and $\delta>0$ with
(\ref{i_d_v_j_less_J}). Then, if $\delta>0$ is close enough to 0,\\
a) We can select a constant $\Lambda>0$ for which
\begin{equation}
||B_{\epsilon}(H(\tau,z))||_{(\sigma_{1},RH_{a,b,\upsilon},\epsilon,\delta)} \leq \Lambda \label{B_epsilon_ball_in_ball} 
\end{equation}
for any $H(\tau,z) \in B(0,\Lambda)$, for all $\epsilon \in \dot{D}(0,\epsilon_{0})$, where $B(0,\Lambda)$ stands for the closed
ball centered at 0 with radius $\Lambda>0$ in $EG_{(\sigma_{1},RH_{a,b,\upsilon},\epsilon,\delta)}$.\\
b) The map $B_{\epsilon}$ is shrinking in the sense that
\begin{equation}
||B_{\epsilon}(H_{1}(\tau,z)) - B_{\epsilon}(H_{2}(\tau,z))||_{(\sigma_{1},RH_{a,b,\upsilon},\epsilon,\delta)} \leq
\frac{1}{2} ||H_{1}(\tau,z) - H_{2}(\tau,z)||_{(\sigma_{1},RH_{a,b,\upsilon},\epsilon,\delta)} \label{B_epsilon_shrink}
\end{equation}
occurs whenever $H_{1},H_{2}$ belong to $B(0,\Lambda)$, for all $\epsilon \in \dot{D}(0,\epsilon_{0})$.
\end{lemma}
\begin{proof}
According to the inequality $r_{b}(\beta) \geq r_{b}(0)$ for all $\beta \geq 0$, we observe that for all $0 \leq h \leq S_{\mathcal{B}}-1$
and $0 \leq j \leq S_{\mathcal{B}}-1-h$,
$$ ||v_{j+h}(\tau,\epsilon)||_{(j,\sigma_{1}',RH_{a,b,\upsilon},\epsilon)} \leq ||v_{j+h}(\tau,\epsilon)||_{(0,\sigma_{1}',RH_{a,b,\upsilon},\epsilon)} $$
holds. As a consequence, the function $\partial_{z}^{h}V_{S_{\mathcal{B}}}(\tau,z,\epsilon)$ is located in
$EG_{(\sigma_{1}',RH_{a,b,\upsilon},\epsilon,\delta)}$ with the upper estimates
\begin{equation}
||\partial_{z}^{h}V_{S_{\mathcal{B}}}(\tau,z,\epsilon)||_{(\sigma_{1}',RH_{a,b,\upsilon},\epsilon,\delta)}
\leq \sum_{j=0}^{S_{\mathcal{B}}-1-h} ||v_{j+h}(\tau,\epsilon)||_{(0,\sigma_{1}',RH_{a,b,\upsilon},\epsilon)} \frac{\delta^j}{j!} \leq J.
\label{norm_partial_z_h_V_SB_lessJ}
\end{equation}
We first concentrate our attention on the bounds (\ref{B_epsilon_ball_in_ball}). Let $H(\tau,z)$ in
$EG_{(\sigma_{1},RH_{a,b,\upsilon},\epsilon,\delta)}$ submitted to
$||H(\tau,z)||_{(\sigma_{1},RH_{a,b,\upsilon},\epsilon,\delta)} \leq \Lambda$. Assume that $0 < \delta < \rho$. We set
$$ M_{\mathcal{B},\underline{l}} = \sup_{\tau \in RH_{a,b,\upsilon}, \epsilon \in \dot{D}(0,\epsilon),z \in D(0,\rho)}
\left| \frac{d_{\underline{l}}(z,\epsilon)}{P_{\mathcal{B}}(\tau)} \right| $$
for all $\underline{l} \in \mathcal{B}$. Under the oversight of (\ref{cond_S_B_b_l}) and due to Propositions 13 and 15, we get constants
$C_{5}>0$ (depending on $\underline{l},S_{\mathcal{B}},\sigma_{1},b$) and $C_{6}>0$ (depending on
$M_{\mathcal{B},\underline{l}},\delta,\rho$) such that
\begin{multline}
|| \frac{d_{\underline{l}}(z,\epsilon)}{P_{\mathcal{B}}(\tau)} \epsilon^{l_{1} - l_{0}} \tau
\int_{L_{0,\tau}} (\tau - s)^{d_{l_{0},l_{1}} - 1} s^{l_1} \partial_{z}^{l_{2}-S_{\mathcal{B}}}H(s,z) \frac{ds}{s}
||_{(\sigma_{1},RH_{a,b,\upsilon},\epsilon,\delta)} \\
\leq C_{6}C_{5} \delta^{S_{\mathcal{B}}-l_{2}}
||H(\tau,z)||_{(\sigma_{1},RH_{a,b,\upsilon},\epsilon,\delta)} \label{norm_partial_z_H_1}
\end{multline}
and
\begin{multline}
|| \frac{d_{\underline{l}}(z,\epsilon)}{P_{\mathcal{B}}(\tau)} \epsilon^{l_{1} - l_{0}} \tau
\int_{L_{0,\tau}} (\tau - s)^{d_{l_{0},l_{1}} + (l_{1} - p) - 1} s^{p} \partial_{z}^{l_{2}-S_{\mathcal{B}}}H(s,z) \frac{ds}{s}
||_{(\sigma_{1},RH_{a,b,\upsilon},\epsilon,\delta)} \\
\leq C_{6}C_{5} \delta^{S_{\mathcal{B}}-l_{2}}
||H(\tau,z)||_{(\sigma_{1},RH_{a,b,\upsilon},\epsilon,\delta)} \label{norm_partial_z_H_2}
\end{multline}
for all $1 \leq p \leq l_{1}-1$. Besides, keeping in mind Propositions 14 and 15 with the help of
(\ref{norm_partial_z_h_V_SB_lessJ}), two constants $\check{C}_{5}>0$ (depending on $\underline{l},\sigma_{1},\sigma_{1}'$) and
$C_{6}>0$ (depending on $M_{\mathcal{B},\underline{l}},\delta,\rho$) are obtained for which
\begin{multline}
|| \frac{d_{\underline{l}}(z,\epsilon)}{P_{\mathcal{B}}(\tau)} \epsilon^{l_{1} - l_{0}} \tau
\int_{L_{0,\tau}} (\tau - s)^{d_{l_{0},l_{1}} - 1} s^{l_1} \partial_{z}^{l_{2}}V_{S_{\mathcal{B}}}(s,z,\epsilon) \frac{ds}{s}
||_{(\sigma_{1},RH_{a,b,\upsilon},\epsilon,\delta)} \\
\leq C_{6}\check{C}_{5}
||\partial_{z}^{l_2}V_{S_{\mathcal{B}}}(\tau,z,\epsilon)||_{(\sigma_{1}',RH_{a,b,\upsilon},\epsilon,\delta)} \leq
C_{6}\check{C}_{5}J \label{norm_V_SB_1}
\end{multline}
together with
\begin{multline}
|| \frac{d_{\underline{l}}(z,\epsilon)}{P_{\mathcal{B}}(\tau)} \epsilon^{l_{1} - l_{0}} \tau
\int_{L_{0,\tau}} (\tau - s)^{d_{l_{0},l_{1}} + (l_{1} - p) - 1} s^{p} \partial_{z}^{l_{2}}V_{S_{\mathcal{B}}}(s,z,\epsilon)
\frac{ds}{s} ||_{(\sigma_{1},RH_{a,b,\upsilon},\epsilon,\delta)} \\
\leq C_{6}\check{C}_{5}
||\partial_{z}^{l_2}V_{S_{\mathcal{B}}}(\tau,z,\epsilon)||_{(\sigma_{1}',RH_{a,b,\upsilon},\epsilon,\delta)} \leq
C_{6}\check{C}_{5}J \label{norm_V_SB_2}
\end{multline}
for all $1 \leq p \leq l_{1}-1$.

At last, from Lemma 16 1), one can select a constant $\tilde{C}_{RH_{a,b,\upsilon}}>0$ for which the first inequality of
(\ref{bds_wHJn_on_Lshaped}) holds. We choose $\delta>0$ small enough and $\Lambda>0$ sufficiently large such that
\begin{multline}
\sum_{\underline{l}=(l_{0},l_{1},l_{2}) \in \mathcal{B}}
\frac{C_{6}C_{5}\delta^{S_{\mathcal{B}}-l_{2}} }{\Gamma(d_{l_{0},l_{1}})} \Lambda + \sum_{1 \leq p \leq l_{1}-1}
|A_{l_{1},p}| \frac{C_{6}C_{5} \delta^{S_{\mathcal{B}}-l_{2}}}{\Gamma( d_{l_{0},l_{1}} + (l_{1}-p) )} \Lambda \\
+
\sum_{\underline{l}=(l_{0},l_{1},l_{2}) \in \mathcal{B}}
\frac{C_{6}\check{C}_{5}}{\Gamma(d_{l_{0},l_{1}})}J + \sum_{1 \leq p \leq l_{1}-1} |A_{l_{1},p}|
\frac{C_{6}\check{C}_{5}}{\Gamma( d_{l_{0},l_{1}} + (l_{1}-p) )} J + \tilde{C}_{RH_{a,b,\upsilon}} \leq \Lambda
\label{constraints_delta_Lambda_data_l_B}
\end{multline}
holds. Finally, gathering (\ref{norm_partial_z_H_1}), (\ref{norm_partial_z_H_2}), (\ref{norm_V_SB_1}), (\ref{norm_V_SB_2})
and (\ref{constraints_delta_Lambda_data_l_B}) implies (\ref{B_epsilon_ball_in_ball}).

In a second phase, we show that $B_{\epsilon}$ represents a shrinking map on the ball $B(0,\Lambda)$. Namely, let $H_{1},H_{2}$ be taken in the
ball $B(0,\Lambda)$. The bounds (\ref{norm_partial_z_H_1}) and (\ref{norm_partial_z_H_2}) just established above entail
\begin{multline}
|| \frac{d_{\underline{l}}(z,\epsilon)}{P_{\mathcal{B}}(\tau)} \epsilon^{l_{1} - l_{0}} \tau
\int_{L_{0,\tau}} (\tau - s)^{d_{l_{0},l_{1}} - 1} s^{l_1} \partial_{z}^{l_{2}-S_{\mathcal{B}}}(H_{1}(s,z) - H_{2}(s,z)) \frac{ds}{s}
||_{(\sigma_{1},RH_{a,b,\upsilon},\epsilon,\delta)} \\
\leq C_{6}C_{5} \delta^{S_{\mathcal{B}}-l_{2}}
||H_{1}(\tau,z) - H_{2}(\tau,z)||_{(\sigma_{1},RH_{a,b,\upsilon},\epsilon,\delta)} \label{norm_partial_z_H_1_shrink}
\end{multline}
in a row with
\begin{multline}
|| \frac{d_{\underline{l}}(z,\epsilon)}{P_{\mathcal{B}}(\tau)} \epsilon^{l_{1} - l_{0}} \tau
\int_{L_{0,\tau}} (\tau - s)^{d_{l_{0},l_{1}} + (l_{1} - p) - 1} s^{p}
\partial_{z}^{l_{2}-S_{\mathcal{B}}}(H_{1}(s,z) - H_{2}(s,z)) \frac{ds}{s}
||_{(\sigma_{1},RH_{a,b,\upsilon},\epsilon,\delta)} \\
\leq C_{6}C_{5} \delta^{S_{\mathcal{B}}-l_{2}}
||H_{1}(\tau,z) - H_{2}(\tau,z)||_{(\sigma_{1},RH_{a,b,\upsilon},\epsilon,\delta)} \label{norm_partial_z_H_2_shrink}
\end{multline}
for all $1 \leq p \leq l_{1}-1$. We take $\delta>0$ small scaled in order that
\begin{equation}
\sum_{\underline{l} = (l_{0},l_{1},l_{2}) \in \mathcal{B}} \frac{C_{6}C_{5}}{\Gamma(d_{l_{0},l_{1}})} \delta^{S_{\mathcal{B}}-l_{2}}
+ \sum_{1 \leq p \leq l_{1}-1} |A_{l_{1},p}| \frac{C_{6}C_{5}}{\Gamma(d_{l_{0},l_{1}} + (l_{1}-p))} \delta^{S_{\mathcal{B}}-l_{2}}
\leq \frac{1}{2} \label{constraints_delta_data_l_B_shrink}
\end{equation}
As a result, we obtain (\ref{B_epsilon_shrink}).

In conclusion, we set $\delta>0$ and $\Lambda>0$ in a way that
(\ref{constraints_delta_Lambda_data_l_B}) and (\ref{constraints_delta_data_l_B_shrink}) are concurrently fulfilled. Lemma 17 follows.
\end{proof}
Assume the restriction (\ref{cond_S_B_b_l}) holds. Take the constants $J,\Lambda$ and $\delta$ as in Lemma 17. The initial data
$v_{j}(\tau,\epsilon)$, $0 \leq j \leq S_{\mathcal{B}}-1$ and $\sigma_{1}'$ are chosen in a way that (\ref{i_d_v_j_less_J}) occurs.
In view of the points a) and b) of Lemma 17 and according to the classical contractive mapping theorem on complete metric spaces, we notice that
the map $B_{\epsilon}$ carries a unique fixed point named $H(\tau,z,\epsilon)$ (that relies analytically upon
$\epsilon \in \dot{D}(0,\epsilon_{0})$) inside the closed ball $B(0,\Lambda) \subset EG_{(\sigma_{1},RH_{a,b,\upsilon},\epsilon,\delta)}$ for all
$\epsilon \in \dot{D}(0,\epsilon_{0})$. In other words, $B_{\epsilon}(H(\tau,z,\epsilon))$ equates $H(\tau,z,\epsilon)$ with
$||H(\tau,z,\epsilon)||_{(\sigma_{1},RH_{a,b,\upsilon},\epsilon,\delta)} \leq \Lambda$. As a consequence, the expression
$$ v(\tau,z,\epsilon) = \partial_{z}^{-S_{\mathcal{B}}}H(\tau,z,\epsilon) + V_{S_{\mathcal{B}}}(\tau,z,\epsilon) $$
fufills the convolution equation (\ref{ACP_forcterm_w}) with initial data (\ref{ACP_forcterm_w_i_d}). In the last step, we explain
the reason why $v(\tau,z,\epsilon)$ shall belong to $EG_{(\sigma_{1},RH_{a,b,\upsilon},\epsilon,\delta)}$. Indeed, if one expands
$H(\tau,z,\epsilon)$ into a formal series in $z$, $H(\tau,z,\epsilon) = \sum_{\beta \geq 0} H_{\beta}(\tau,\epsilon) z^{\beta}/\beta!$, one checks
that
$$ ||\partial_{z}^{-S_{\mathcal{B}}}H(\tau,z,\epsilon)||_{(\sigma_{1},RH_{a,b,\upsilon},\epsilon,\delta)}
= \sum_{\beta \geq S_{\mathcal{B}}} ||H_{\beta - S_{\mathcal{B}}}(\tau,\epsilon)||_{(\beta,\sigma_{1},RH_{a,b,\upsilon},\epsilon)}
\delta^{\beta}/\beta! $$
From $r_{b}(\beta) \geq r_{b}(\beta - S_{\mathcal{B}})$, we notice that
$$ ||H_{\beta - S_{\mathcal{B}}}(\tau,\epsilon) ||_{(\beta,\sigma_{1},RH_{a,b,\upsilon},\epsilon)}
 \leq ||H_{\beta - S_{\mathcal{B}}}(\tau,\epsilon) ||_{(\beta - S_{\mathcal{B}},\sigma_{1},RH_{a,b,\upsilon},\epsilon)} $$
for all $\beta \geq S_{\mathcal{B}}$. Hence,
\begin{multline}
||\partial_{z}^{-S_{\mathcal{B}}}H(\tau,z,\epsilon)||_{(\sigma_{1},RH_{a,b,\upsilon},\epsilon,\delta)} \\
\leq \sum_{\beta \geq S_{\mathcal{B}}} \left( \frac{(\beta - S_{\mathcal{B}})!}{\beta!} \delta^{S_{\mathcal{B}}} \right)
||H_{\beta - S_{\mathcal{B}}}(\tau,\epsilon)||_{(\beta - S_{\mathcal{B}},\sigma_{1},RH_{a,b,\upsilon},\epsilon)}
\frac{\delta^{\beta - S_{\mathcal{B}}}}{(\beta - S_{\mathcal{B}})!} \\
\leq \delta^{S_{\mathcal{B}}}
||H(\tau,z,\epsilon)||_{(\sigma_{1},RH_{a,b,\upsilon},\epsilon,\delta)} \label{norm_partial_z_H_less_norm_H}
\end{multline}
Altogether, according to (\ref{norm_partial_z_h_V_SB_lessJ}) and (\ref{norm_partial_z_H_less_norm_H}), it follows that
$v(\tau,z,\epsilon)$ belongs to $EG_{(\sigma_{1},RH_{a,b,\upsilon},\epsilon,\delta)}$ with the upper bounds (\ref{norm_v_RHab_less_J}).
\end{proof}

\section{Sectorial analytic solutions in a complex parameter for a singularly perturbed differential Cauchy problem}

Let $\mathcal{B}$ be a finite set in $\mathbb{N}^{3}$. For all $\underline{l} = (l_{0},l_{1},l_{2}) \in \mathcal{B}$, we set
$d_{\underline{l}}(z,\epsilon)$ as a bounded holomorphic function on a polydisc $D(0,\rho) \times D(0,\epsilon_{0})$ for given radii
$\rho,\epsilon_{0}>0$. Let $S_{\mathcal{B}} \geq 1$ be an integer and let $P_{\mathcal{B}}(\tau)$ be a polynomial (not identically equal to 0)
with complex coefficients which is either constant or whose complex roots that are asked to lie in the open right halfplane $\mathbb{C}_{+}$ and
are imposed to avoid all the closed sets
$\bar{S}_{d_p} \cup \bar{D}(0,r)$, for $0 \leq p \leq \iota-1$, where the sectors $S_{d_p}$ and the disc $D(0,r)$ are introduced just after Definition 4.
We aim attention at the next partial differential Cauchy problem
\begin{equation}
P_{\mathcal{B}}(\epsilon t^{2}\partial_{t}) \partial_{z}^{S_{\mathcal{B}}} y(t,z,\epsilon) =
\sum_{\underline{l}=(l_{0},l_{1},l_{2}) \in \mathcal{B}} d_{\underline{l}}(z,\epsilon)
t^{l_0}\partial_{t}^{l_1}\partial_{z}^{l_{2}}y(t,z,\epsilon) + u(t,z,\epsilon) \label{SPCP_second}
\end{equation}
for given initial data
\begin{equation}
(\partial_{z}^{j}y)(t,0,\epsilon) = \psi_{j}(t,\epsilon) \label{SPCP_second_i_d} 
\end{equation}
for $0 \leq j \leq S_{\mathcal{B}}-1$, where $u(t,z,\epsilon)$ belongs to the sets of solutions to the Cauchy problem
(\ref{SPCP_first}), (\ref{SPCP_first_i_d}) constructed in Section 3.3 and displayed as
$\{ u_{\mathcal{E}_{HJ_n}^{k}} \}_{k \in \llbracket -n,n \rrbracket}$ or
$\{ u_{\mathcal{E}_{S_{d_p}}} \}_{0 \leq p \leq \iota-1}$.

We require the forthcoming constraints on the set $\mathcal{B}$ to hold. There exists a real number $b>1$ such that
\begin{equation}
S_{\mathcal{B}} \geq b(l_{0} - l_{1}) + l_{2} \ \ , \ \ S_{\mathcal{B}} > l_{2} \ \ , \ \ l_{1} \geq 1
\label{SB_underline_l_constraints}
\end{equation}
holds for all $\underline{l}=(l_{0},l_{1},l_{2}) \in \mathcal{B}$ and we assume the existence of an integer $d_{l_{0},l_{1}} \geq 1$ for which
\begin{equation}
l_{0} = 2l_{1} + d_{l_{0},l_{1}}, \label{d_l01_defin}
\end{equation}
for all $\underline{l}=(l_{0},l_{1},l_{2}) \in \mathcal{B}$. With the help of (\ref{d_l01_defin}), according to the formula
(8.7) p. 3630 from \cite{taya}, one can expand the differential operators
\begin{equation}
t^{l_0}\partial_{t}^{l_1} = t^{d_{l_{0},l_{1}}}(t^{2l_{1}}\partial_{t}^{l_1})
= t^{d_{l_{0},l_{1}}} \left( (t^{2}\partial_{t})^{l_1} + \sum_{1 \leq p \leq l_{1}-1} A_{l_{1},p}
t^{(l_{1}-p)} (t^{2}\partial_{t})^{p} \right) \label{Tahara_expansion_diff_op}
\end{equation}
for suitable real numbers $A_{l_{1},p}$, with $1 \leq p \leq l_{1}-1$ for $l_{1} \geq 1$ (with the convention that the
sum $\sum_{1 \leq p \leq l_{1}-1}$ is reduced to 0 when $l_{1}=1$).\medskip

In the sequel, we explain how we build up the initial data $\psi_{j}(t,\epsilon)$, $0 \leq j \leq S_{\mathcal{B}}-1$. We take for granted
that all the constraints disclosed at the beginning of Subsection 3.3 hold. 
We depart from a family of functions $\tau \mapsto v_{j}(\tau,\epsilon)$, $0 \leq j \leq S_{\mathcal{B}}-1$, which are holomorphic on
the disc $D(0,r)$, on each sector $S_{d_p}$, $0 \leq p \leq \iota-1$ and on the interior of the domain $HJ_{n}$ defined at the onset of the
Section 3.1 for some integer $n \geq 1$ and relies analytically on $\epsilon$ over $\dot{D}(0,\epsilon_{0})$. Furthermore, we require the next
additional properties.\medskip

\noindent a) For all $0 \leq j \leq S_{\mathcal{B}}-1$, all $k \in \llbracket -n,n \rrbracket$, the function $\tau \mapsto v_{j}(\tau,\epsilon)$
belongs to
the Banach spaces $EG_{(0,\sigma_{1}',RH_{a_{k},b_{k},\upsilon_{k}},\epsilon)}$ and
$SEG_{(0,\underline{\varsigma}',RJ_{c_{k},d_{k},\upsilon_{k}},\epsilon)}$ for all $\epsilon \in \dot{D}(0,\epsilon_{0})$,
where $\sigma_{1}'>0$ and the tuple $\underline{\varsigma}'=(\sigma_{1}',\varsigma_{2}',\varsigma_{3}')$ satisfies
$\varsigma_{2}'>0,\varsigma_{3}'>0$, the real numbers $a_{k},b_{k},c_{k},d_{k}$ are defined at the outstart of Subsection 3.1 and
$\upsilon_{k}>0$ is a real number suitably chosen in a way that $\upsilon_{k} < \mathrm{Re}(A_{k})$, where
$A_{k}$ is a point inside the strip $H_{k}$ defined through (\ref{SPCP_first_i_d_k}) and (\ref{choice_a_k}). Besides, for any
$0 \leq j \leq S_{\mathcal{B}}-1$, there exists a constant $J_{v_j}>0$ (independent of $\epsilon$) such that
\begin{equation}
||v_{j}(\tau,\epsilon)||_{(0,\sigma_{1}',RH_{a_{k},b_{k},\upsilon_{k}},\epsilon)} \leq J_{v_j} \ \ , \ \
||v_{j}(\tau,\epsilon)||_{(0,\underline{\varsigma}',RJ_{c_{k},d_{k},\upsilon_{k}},\epsilon)} \leq J_{v_j}
\label{normRHJ_vj_Ivj}
\end{equation}
for all $k \in \llbracket -n,n \rrbracket$, all $\epsilon \in \dot{D}(0,\epsilon_{0})$.\medskip

\noindent b) For all $0 \leq j \leq S_{\mathcal{B}}-1$, all $0 \leq p \leq \iota-1$, the map $\tau \mapsto v_{j}(\tau,\epsilon)$
appertains to the Banach space $EG_{(0,\sigma_{1}',S_{d_p} \cup D(0,r),\epsilon)}$ for all $\epsilon \in \dot{D}(0,\epsilon_{0})$, where
$\sigma_{1}'>0$. Furthermore, for each $0 \leq j \leq S_{\mathcal{B}}-1$, we have a constant $J_{v_j}>0$ (independent of $\epsilon$) for which
\begin{equation}
||v_{j}(\tau,\epsilon)||_{(0,\sigma_{1}',S_{d_p} \cup D(0,r),\epsilon)} \leq J_{v_j} \label{normSdp_vj_Ivj}
\end{equation}
for all $0 \leq p \leq \iota-1$, all $\epsilon \in \dot{D}(0,\epsilon_{0})$.\medskip

\noindent 1) We construct a first set of initial data
\begin{equation}
\psi_{j,\mathcal{E}_{HJ_n}^{k}}(t,\epsilon) = \int_{P_k} v_{j}(u,\epsilon) \exp(-\frac{u}{\epsilon t}) \frac{du}{u}
\label{defin_psi_j_HJ_i_d}
\end{equation}
for all $k \in \llbracket -n,n \rrbracket$, where the integration path is the same as the one involved in (\ref{SPCP_first_i_d_k}). The same
proof as the one presented in Lemma 8 justifies that
\begin{lemma}
The Laplace transform $\psi_{j,\mathcal{E}_{HJ_n}^{k}}(t,\epsilon)$ represents a bounded holomorphic function
on $(\mathcal{T} \cap D(0,r_{\mathcal{T}})) \times \mathcal{E}_{HJ_n}^{k}$ for a suitable radius $r_{\mathcal{T}}>0$, where
$\mathcal{T}$ and $\mathcal{E}_{HJ_n}^{k}$ are bounded open sectors described in Definition 3.
\end{lemma}
2) For any $0 \leq j \leq S_{\mathcal{B}}-1$, we set up a second family of initial data
\begin{equation}
\psi_{j,\mathcal{E}_{S_{d_p}}}(t,\epsilon) = \int_{L_{\gamma_{d_p}}} v_{j}(u,\epsilon) \exp( -\frac{u}{\epsilon t} ) \frac{du}{u}
\label{defin_psi_j_Sd_i_d}
\end{equation}
where the integration path is a halfline with direction $\gamma_{d_p}$ described in (\ref{relation_gamma_epsilon_t}) and 
(\ref{Laplace_varphi_j_along_halfline}). Following similar lines of arguments as in Lemma 9, we observe that
\begin{lemma}
The Laplace integral $\psi_{j,\mathcal{E}_{S_{d_p}}}(t,\epsilon)$ defines a bounded holomorphic function on
$(\mathcal{T} \cap D(0,r_{\mathcal{T}})) \times \mathcal{E}_{S_{d_p}}$ for a convenient radius $r_{\mathcal{T}}>0$, where
$\mathcal{T}$ and $\mathcal{E}_{S_{d_p}}$ are bounded open sectors displayed in Definition 4.
\end{lemma}
We are now in position to set forth the second main result of our work.
\begin{theo}
Under all the restrictions assumed above till the unfolding of Section 5, provided that the real number $\delta>0$ is chosen close enough to 0,
the following statements arise.\medskip

\noindent 1) 1.1) The Cauchy problem (\ref{SPCP_second}) where $u(t,z,\epsilon)$ stands for $u_{\mathcal{E}_{HJ_n}^{k}}(t,z,\epsilon)$ with initial
data (\ref{SPCP_second_i_d}) given by (\ref{defin_psi_j_HJ_i_d}) has a bounded holomorphic solution
$y_{\mathcal{E}_{HJ_n}^{k}}(t,z,\epsilon)$ on a domain
$(\mathcal{T} \cap D(0,r_{\mathcal{T}})) \times D(0,\delta \delta_{1}) \times \mathcal{E}_{HJ_n}^{k}$ for some radius $r_{\mathcal{T}}>0$ chosen
close to 0 and $0 < \delta_{1} < 1$. Besides, $y_{\mathcal{E}_{HJ_n}^{k}}$ can be expressed through a special Laplace transform
\begin{equation}
y_{\mathcal{E}_{HJ_n}^{k}}(t,z,\epsilon) = \int_{P_k} v_{HJ_n}(u,z,\epsilon) \exp(-\frac{u}{\epsilon t}) \frac{du}{u} \label{defin_Laplace_y_EHJnk}
\end{equation}
where $v_{HJ_n}(\tau,z,\epsilon)$ determines a holomorphic function on $\mathring{HJ}_{n} \times D(0,\delta \delta_{1}) \times
\dot{D}(0,\epsilon_{0})$, continuous on $HJ_{n} \times D(0,\delta \delta_{1}) \times
\dot{D}(0,\epsilon_{0})$, submitted to the next restrictions. For any choice of $\sigma_{1}>0$ and a tuple
$\underline{\varsigma} = (\sigma_{1},\varsigma_{2},\varsigma_{3})$ with
\begin{equation}
\sigma_{1} > \sigma_{1}' \ \ , \ \ \varsigma_{2} > \varsigma_{2}' \ \ , \ \ \varsigma_{3} = \varsigma_{3}' \label{cond_sigma_varsigma_theo2}
\end{equation}
one obtains constants $C_{H_k}^{v}>0$ and $C_{J_k}^{v}>0$ (independent of $\epsilon$) with
\begin{equation}
|v_{HJ_n}(\tau,z,\epsilon)| \leq C_{H_k}^{v}|\tau| \exp( \frac{\sigma_{1}}{|\epsilon|} \zeta(b) |\tau| ) \label{bds_vHJn_Hk} 
\end{equation}
for all $\tau \in H_{k}$, all $z \in D(0,\delta \delta_{1})$ and
\begin{equation}
|v_{HJ_n}(\tau,z,\epsilon)| \leq C_{J_k}^{v}|\tau|
\exp( \frac{\sigma_{1}}{|\epsilon|} \zeta(b) |\tau| + \varsigma_{2}\zeta(b) \exp( \varsigma_{3}|\tau| ) ) \label{bds_vHJn_Jk}
\end{equation}
for all $\tau \in J_{k}$, all $z \in D(0,\delta \delta_{1})$, whenever $\epsilon \in \dot{D}(0,\epsilon_{0})$, for
all $k \in \llbracket -n,n \rrbracket$.\\
1.2) Let $k \in \llbracket -n,n \rrbracket$ with $k \neq n$. Then, there exist constants $M_{k,1},M_{k,2}>0$ and $M_{k,3}>1$
independent of $\epsilon$, such that
\begin{equation}
| y_{\mathcal{E}_{HJ_{n}}^{k+1}}(t,z,\epsilon) - y_{\mathcal{E}_{HJ_{n}}^{k}}(t,z,\epsilon) |
\leq M_{k,1} \exp( -\frac{M_{k,2}}{|\epsilon|} \mathrm{Log} \frac{M_{k,3}}{|\epsilon|} ) \label{log_flat_difference_yk_plus_1_minus_yk_HJn}
\end{equation}
for all $t \in \mathcal{T} \cap D(0,r_{\mathcal{T}})$, all $\epsilon \in \mathcal{E}_{HJ_{n}}^{k} \cap
\mathcal{E}_{HJ_{n}}^{k+1} \neq \emptyset$ and
all $z \in D(0,\delta \delta_{1})$.\medskip

\noindent 2) 2.1) The Cauchy problem (\ref{SPCP_second}) where $u(t,z,\epsilon)$ must be replaced by $u_{\mathcal{E}_{S_{d_p}}}(t,z,\epsilon)$ along
with initial
data (\ref{SPCP_second_i_d}) given by (\ref{defin_psi_j_Sd_i_d}) possesses a bounded holomorphic solution
$y_{\mathcal{E}_{S_{d_p}}}(t,z,\epsilon)$ on a domain
$(\mathcal{T} \cap D(0,r_{\mathcal{T}})) \times D(0,\delta \delta_{1}) \times \mathcal{E}_{S_{d_p}}$ for some radius $r_{\mathcal{T}}>0$ chosen
small enough and $0 < \delta_{1} < 1$. Moreover, $y_{\mathcal{E}_{S_{d_p}}}$ appears to be a Laplace transform
\begin{equation}
y_{\mathcal{E}_{S_{d_p}}}(t,z,\epsilon) = \int_{L_{\gamma_{d_p}}} v_{S_{d_p}}(u,z,\epsilon) \exp(-\frac{u}{\epsilon t}) \frac{du}{u}
\label{defin_y_ESdp}
\end{equation}
where $v_{S_{d_p}}(\tau,z,\epsilon)$ represents a holomorphic function on $(S_{d_p} \cup D(0,r)) \times D(0,\delta \delta_{1}) \times
\dot{D}(0,\epsilon_{0})$, continuous on $(\bar{S}_{d_p} \cup \bar{D}(0,r)) \times D(0,\delta \delta_{1}) \times
\dot{D}(0,\epsilon_{0})$ that conforms the next demand: For any choice of $\sigma_{1}>\sigma_{1}'$, one can select
a constant $C_{S_{d_p}}^{v}>0$ (independent of $\epsilon$) with
\begin{equation}
|v_{S_{d_p}}(\tau,z,\epsilon)| \leq C_{S_{d_p}}^{v}|\tau| \exp( \frac{\sigma_{1}}{|\epsilon|} \zeta(b) |\tau| ) \label{bds_vSdp} 
\end{equation}
for all $\tau \in S_{d_p} \cup D(0,r)$, all $z \in D(0,\delta \delta_{1})$, all $\epsilon \in \dot{D}(0,\epsilon_{0})$.\\
2.2) Let $0 \leq p \leq \iota-2$. We can find two constants $M_{p,1},M_{p,2}>0$ independent of $\epsilon$, such that
\begin{equation}
| y_{\mathcal{E}_{S_{d_{p+1}}}}(t,z,\epsilon) - y_{\mathcal{E}_{S_{d_p}}}(t,z,\epsilon) |
\leq M_{p,1} \exp( -\frac{M_{p,2}}{|\epsilon|} ) \label{exp_flat_difference_yk_plus_1_minus_yk_Sdp}
\end{equation}
for all $t \in \mathcal{T} \cap D(0,r_{\mathcal{T}})$, all $\epsilon \in \mathcal{E}_{S_{d_{p+1}}} \cap
\mathcal{E}_{S_{d_p}} \neq \emptyset$ and
all $z \in D(0,\delta \delta_{1})$.\medskip

\noindent 3) The next additional bounds hold among the two families described above : There exist constants $M_{n,1},M_{n,2}>0$
(independent of $\epsilon$) with
\begin{equation}
| y_{\mathcal{E}_{HJ_n}^{-n}}(t,z,\epsilon) - y_{\mathcal{E}_{S_{d_0}}}(t,z,\epsilon) | \leq M_{n,1} \exp( -\frac{M_{n,2}}{|\epsilon|} )
\label{difference_y_HJn_Sd0}
\end{equation}
for all $\epsilon \in \mathcal{E}_{HJ_n}^{-n} \cap \mathcal{E}_{S_{d_0}}$ and
\begin{equation}
| y_{\mathcal{E}_{HJ_n}^{n}}(t,z,\epsilon) - y_{\mathcal{E}_{S_{d_{\iota-1}}}}(t,z,\epsilon) | \leq M_{n,1} \exp( -\frac{M_{n,2}}{|\epsilon|} )
\label{difference_y_HJn_Sdiota}
\end{equation}
for all $\epsilon \in \mathcal{E}_{HJ_n}^{n} \cap \mathcal{E}_{S_{d_{\iota-1}}}$ whenever $t \in \mathcal{T} \cap D(0, r_{\mathcal{T}})$ and
$z \in D(0,\delta \delta_{1})$.
\end{theo}
\begin{proof} We consider the convolution equation (\ref{ACP_forcterm_w}) with the forcing term
$w(\tau,z,\epsilon) = w_{HJ_n}(\tau,z,\epsilon)$ for given initial data
\begin{equation}
(\partial_{z}^{j}v)(\tau,0,\epsilon) = v_{j}(\tau,\epsilon) \ \ , \ \ 0 \leq j \leq S_{\mathcal{B}}-1. \label{ACP_forcterm_w_iv_d_j}
\end{equation}
We certify that the problem (\ref{ACP_forcterm_w}) along with (\ref{ACP_forcterm_w_iv_d_j}) carries a unique formal solution
\begin{equation}
v_{HJ_n}(\tau,z,\epsilon) = \sum_{\beta \geq 0} v_{\beta}(\tau,\epsilon) \frac{z^{\beta}}{\beta !} \label{defin_vHJn} 
\end{equation}
where $v_{\beta}(\tau,\epsilon)$ are holomorphic on $\mathring{HJ}_{n} \times \dot{D}(0,\epsilon_{0})$, continuous on
$HJ_{n} \times \dot{D}(0,\epsilon_{0})$. Indeed, if one develops
$d_{\underline{l}}(z,\epsilon) = \sum_{\beta \geq 0} d_{\underline{l},\beta}(\epsilon) z^{\beta}/\beta!$ as Taylor expansion at $z=0$, the formal
series (\ref{defin_vHJn}) solves (\ref{ACP_forcterm_w}), (\ref{ACP_forcterm_w_iv_d_j}) if and only if the next recursion formula holds true
\begin{multline}
v_{\beta + S_{\mathcal{B}}}(\tau,\epsilon) = \sum_{\underline{l} = (l_{0},l_{1},l_{2}) \in \mathcal{B}}
\frac{\epsilon^{l_{1}-l_{0}}\tau}{\Gamma(d_{l_{0},l_{1}}) P_{\mathcal{B}}(\tau)}
\sum_{\beta_{1} + \beta_{2} = \beta} \frac{d_{\underline{l},\beta_{1}}(\epsilon)}{\beta_{1}!}\\
\times 
\int_{L_{0,\tau}} (\tau-s)^{d_{l_{0},l_{1}}-1} s^{l_1} \frac{v_{\beta_{2}+l_{2}}(s,\epsilon)}{\beta_{2}!} \frac{ds}{s} \beta! +
\sum_{\underline{l}=(l_{0},l_{1},l_{2}) \in \mathcal{B}} \sum_{1 \leq p \leq l_{1}-1} A_{l_{1},p}\\
\times
\frac{\epsilon^{l_{1} - l_{0}}\tau}{\Gamma(d_{l_{0},l_{1}} + (l_{1}-p))P_{\mathcal{B}}(\tau)}
\sum_{\beta_{1} + \beta_{2} = \beta} \frac{d_{\underline{l},\beta_{1}}(\epsilon)}{\beta_{1}!}
\int_{L_{0,\tau}} (\tau-s)^{d_{l_{0},l_{1}} + (l_{1}-p)-1} s^{p} \\
\times \frac{v_{\beta_{2}+l_{2}}(s,\epsilon)}{\beta_{2}!} \frac{ds}{s}\beta! \label{recursion_v_beta}
+ w_{\beta}(\tau,\epsilon)
\end{multline}
for all $\beta \geq 0$, where $w_{\beta}(\tau,\epsilon)$ are the Taylor coefficients of the forcing term $w_{HJ_n}(\tau,z,\epsilon)$
in the variable $z$ which solve the recursion (\ref{recursion_w_beta}). Since the initial data $v_{j}(\tau,\epsilon)$,
$0 \leq j \leq S_{\mathcal{B}}-1$ and all the functions $w_{\beta}(\tau,\epsilon)$, $\beta \geq 0$, define holomorphic functions
on $\mathring{HJ}_{n} \times \dot{D}(0,\epsilon_{0})$, continuous on $HJ_{n} \times \dot{D}(0,\epsilon_{0})$, the recursion
(\ref{recursion_v_beta}) is well defined provided that $L_{0,\tau}$ stands for any path joining $0$ and $\tau$ that remains inside the
domain $HJ_{n}$. Furthermore, all $v_{n}(\tau,\epsilon)$ for $n \geq S_{\mathcal{B}}$ represent holomorphic functions on
$\mathring{HJ}_{n} \times \dot{D}(0,\epsilon_{0})$, continuous on $HJ_{n} \times \dot{D}(0,\epsilon_{0})$.

Bearing in mind all the assumptions set above since the beginning of Section 5, we observe in particular that the conditions
1)a)b) and 2)a)b) asked in Proposition 22 are satisfied. Therefore, the next features hold:\\
1) The formal series $v_{HJ_n}(\tau,z,\epsilon)$ belongs to the Banach spaces
$EG_{(\sigma_{1},RH_{a_{k},b_{k},\upsilon_{k}},\epsilon,\delta)}$, for all $\epsilon \in \dot{D}(0,\epsilon_{0})$, all
$k \in \llbracket -n,n \rrbracket$, for any $\sigma_{1} > \sigma_{1}'$ and one can sort a constant $C_{H_k}^{v}>0$ for which
\begin{equation}
||v_{HJ_{n}}(\tau,z,\epsilon)||_{(\sigma_{1},RH_{a_{k},b_{k},\upsilon_{k}},\epsilon,\delta)} \leq C_{H_k}^{v} \label{norm_vHJn_RHk}
\end{equation}
for all $\epsilon \in \dot{D}(0,\epsilon_{0})$.\\
2) The formal series $v_{HJ_n}(\tau,z,\epsilon)$ appertains to the Banach spaces
$SEG_{(\underline{\varsigma},RJ_{c_{k},d_{k},\upsilon_{k}},\epsilon,\delta)}$, whenever $\epsilon \in \dot{D}(0,\epsilon_{0})$
and $k \in \llbracket -n,n \rrbracket$, provided that the tuple $\underline{\varsigma}$ is chosen as in
(\ref{cond_sigma_varsigma_theo2}). Furthermore, one can get a constant $C_{J_k}^{v}>0$ with
\begin{equation}
||v_{HJ_{n}}(\tau,z,\epsilon)||_{(\underline{\varsigma},RJ_{c_{k},d_{k},\upsilon_{k}},\epsilon,\delta)} \leq C_{J_k}^{v} \label{norm_vHJn_RJk} 
\end{equation}
for all $\epsilon \in \dot{D}(0,\epsilon_{0})$. As a consequence of (\ref{norm_vHJn_RHk}), (\ref{norm_vHJn_RJk}), with the help of
Proposition 12 and 16, we deduce that $v_{HJ_n}(\tau,z,\epsilon)$ represents a holomorphic function on
$\mathring{HJ}_{n} \times D(0,\delta \delta_{1}) \times \dot{D}(0,\epsilon_{0})$, continuous on
$HJ_{n} \times D(0,\delta \delta_{1}) \times \dot{D}(0,\epsilon_{0})$ for some $0 < \delta_{1} < 1$, that withstands the bounds
(\ref{bds_vHJn_Hk}) and (\ref{bds_vHJn_Jk}). By application of a similar proof as in Lemma 8, one can show that
for each $k \in \llbracket -n,n \rrbracket$, the function $y_{\mathcal{E}_{HJ_n}^{k}}(t,z,\epsilon)$ defined as 
(\ref{defin_Laplace_y_EHJnk}) represents a bounded holomorphic function on
$(\mathcal{T} \cap D(0,r_{\mathcal{T}})) \times D(0,\delta_{1}\delta) \times \mathcal{E}_{HJ_n}^{k}$, for some fixed radius
$r_{\mathcal{T}}>0$ and $0 < \delta_{1} < 1$. In addition, following exactly the same reasoning as in Proposition 10 2), one can obtain
the estimates (\ref{log_flat_difference_yk_plus_1_minus_yk_HJn}).

It remains to show that $y_{\mathcal{E}_{HJ_n}^{k}}(t,z,\epsilon)$ actually solves the problem (\ref{SPCP_second}), (\ref{SPCP_second_i_d}).
In accordance with the expansion (\ref{Tahara_expansion_diff_op}), we are scaled down to prove that
\begin{lemma} The next identity
\begin{multline}
t^{d_{l_{0},l_{1}}} (t^{2}\partial_{t})^{l_1} y_{\mathcal{E}_{HJ_n}^{k}}(t,z,\epsilon) =
\frac{\epsilon^{-(d_{l_{0},l_{1}}+l_{1})}}{\Gamma(d_{l_{0},l_{1}})} \int_{P_{k}}
u \int_{L_{0,u}} (u-s)^{d_{l_{0},l_{1}}-1}s^{l_{1}}\\
\times v_{HJ_n}(s,z,\epsilon) \frac{ds}{s} \exp( -\frac{u}{\epsilon t} ) \frac{du}{u} \label{expansion_diff_op_Laplace_y_HJnk}
\end{multline}
holds 
for all $t \in \mathcal{T} \cap D(0,r_{\mathcal{T}})$, $\epsilon \in \mathcal{E}_{HJ_n}^{k}$, all given positive integers
$d_{l_{0},l_{1}},l_{1} \geq 1$. We recall that the path $P_{k}$ is the union of a segment $P_{k,1}$ joining 0 and a prescribed
point $A_{k} \in H_{k}$ and of a horizontal halfline $P_{k,2} = \{ A_{k} - s / s \geq 0 \}$ and here $L_{0,u}$ stands for the union
$[0,c_{RH}(u)] \cup [c_{RH}(u),u]$ where $c_{RH}(u)$ is chosen in a way that
$$ L_{0,u} \subset RH_{a_{k},b_{k},\upsilon_{k}} \ \ , \ \ c_{RH}(u) \in R_{a_{k},b_{k},\upsilon_{k}} \ \ , \ \ |c_{RH}(u)| \leq |u| $$
for all $u \in P_{k} \subset RH_{a_{k},b_{k},\upsilon_{k}}$ (Notice that this last inclusion stems from the assumption
$\upsilon_{k} < \mathrm{Re}(A_{k})$).
\end{lemma}
\begin{proof} We first specify an appropriate choice for the points $c_{RH}(u)$ that will simplify the computations, namely\\
1) When $u$ belongs to $P_{k,1} \subset R_{a_{k},b_{k},\upsilon_{k}}$, then we select $c_{RH}(u)$ somewhere inside the segment $[0,u]$, in
that case $L_{0,u}=[0,u]$.\\
2) For $u \in P_{k,2}$, we choose $c_{RH}(u)=A_{k}$. Hence $L_{0,u}$ becomes the union of the segments $[0,A_{k}]$ and $[A_{k},u]$.

As a result, the right handside of the equality (\ref{expansion_diff_op_Laplace_y_HJnk}) can be written
\begin{multline*}
R = \frac{\epsilon^{-(d_{l_{0},l_{1}} + l_{1})}}{\Gamma(d_{l_{0},l_{1}})}
\left\{ \int_{P_{k,1}} ( \int_{[0,u]} (u-s)^{d_{l_{0},l_{1}}-1}s^{l_1}v_{HJ_n}(s,z,\epsilon) \frac{ds}{s} ) \exp( -\frac{u}{\epsilon t} )
\right. du\\
 + \int_{P_{k,2}} \left( \int_{[0,A_{k}]} (u-s)^{d_{l_{0},l_{1}}-1}s^{l_1}v_{HJ_n}(s,z,\epsilon) \frac{ds}{s} \right.  \\
+ \left. \left. \int_{[A_{k},u]} (u-s)^{d_{l_{0},l_{1}}-1}s^{l_1}v_{HJ_n}(s,z,\epsilon) \frac{ds}{s} \right) \exp( -\frac{u}{\epsilon t} ) du
\right\}
\end{multline*}
for all $t \in \mathcal{T} \cap D(0,r_{\mathcal{T}})$, $\epsilon \in \mathcal{E}_{HJ_n}^{k}$. Now, with the help of the Fubini theorem and a
path deformation argument, we can express each piece of $R$ as some truncated Laplace transforms
of $v_{HJ_n}(\tau,z,\epsilon)$. Namely,
\begin{multline*}
\int_{P_{k,1}} ( \int_{[0,u]} (u-s)^{d_{l_{0},l_{1}}-1}s^{l_1}v_{HJ_n}(s,z,\epsilon) \frac{ds}{s} ) \exp( -\frac{u}{\epsilon t} ) du \\
= \int_{[0,A_{k}]} \left( \int_{[s,A_{k}]} (u-s)^{d_{l_{0},l_{1}}-1} \exp( - \frac{u}{\epsilon t} ) du \right) s^{l_1}v_{HJ_n}(s,z,\epsilon)
\frac{ds}{s} \\
= \int_{[0,A_{k}]} \left( \int_{[0,A_{k}-s]} (u')^{d_{l_{0},l_{1}}-1} \exp( -\frac{u'}{\epsilon t} ) du' \right)
s^{l_{1}} v_{HJ_n}(s,z,\epsilon) \exp( -\frac{s}{\epsilon t} ) \frac{ds}{s}
\end{multline*}
and
\begin{multline*}
\int_{P_{k,2}} \left( \int_{[0,A_{k}]} (u-s)^{d_{l_{0},l_{1}}-1}s^{l_1}v_{HJ_n}(s,z,\epsilon) \frac{ds}{s} \right)
\exp( -\frac{u}{\epsilon t} ) du \\
= \int_{[0,A_{k}]} \left( \int_{P_{k,2}} (u-s)^{d_{l_{0},l_{1}}-1} \exp( -\frac{u}{\epsilon t} ) du \right)
s^{l_1} v_{HJ_n}(s,z,\epsilon) \frac{ds}{s}\\
= \int_{[0,A_{k}]} \left( \int_{P_{k,2}-s} (u')^{d_{l_{0},l_{1}}-1} \exp( -\frac{u'}{\epsilon t} ) du' \right)
s^{l_1} v_{HJ_n}(s,z,\epsilon) \exp( -\frac{s}{\epsilon t} ) \frac{ds}{s}
\end{multline*}
where $P_{k,2}-s$ denotes the path $\{ A_{k}-h-s / h \geq 0 \}$, together with
\begin{multline*}
\int_{P_{k,2}} \left( \int_{[A_{k},u]} (u-s)^{d_{l_{0},l_{1}}-1}s^{l_1}v_{HJ_n}(s,z,\epsilon) \frac{ds}{s} \right) \exp( -\frac{u}{\epsilon t} )
du \\
= \int_{P_{k,2}} \left( \int_{P_{s;2}} (u-s)^{d_{l_{0},l_{1}}-1} \exp( -\frac{u}{\epsilon t} ) du \right) s^{l_1}v_{HJ_n}(s,z,\epsilon)
\frac{ds}{s} \\
= \int_{P_{k,2}} \left( \int_{\mathbb{R}_{-}} (u')^{d_{l_{0},l_{1}}-1} \exp( -\frac{u'}{\epsilon t} ) du' \right)
s^{l_1}v_{HJ_n}(s,z,\epsilon) \exp( -\frac{s}{\epsilon t} ) \frac{ds}{s}
\end{multline*}
where $P_{s;2} = \{ s - h / h \geq 0 \}$ and $\mathbb{R}_{-}$ stands for the path $\{-h / h \geq 0 \}$,
for all $t \in \mathcal{T} \cap D(0,r_{\mathcal{T}})$, $\epsilon \in \mathcal{E}_{HJ_n}^{k}$. On the other hand, a path deformation argument
and the very definition of the Gamma function yields
\begin{multline*}
\int_{[0,A_{k}-s]} (u')^{d_{l_{0},l_{1}}-1} \exp( -\frac{u'}{\epsilon t} ) du' +
\int_{P_{k,2}-s} (u')^{d_{l_{0},l_{1}}-1} \exp( -\frac{u'}{\epsilon t} ) du' \\
=
\int_{\mathbb{R}_{-}} (u')^{d_{l_{0},l_{1}}-1} \exp( -\frac{u'}{\epsilon t} ) du' = \Gamma(d_{l_{0},l_{1}}) (\epsilon t)^{d_{l_{0},l_{1}}}
\end{multline*}
for all $s \in [0,A_{k}]$, all $t \in \mathcal{T} \cap D(0,r_{\mathcal{T}})$, $\epsilon \in \mathcal{E}_{HJ_n}^{k}$. By clustering
the above estimates, we can rewrite the quantity $R$ as
\begin{equation}
R = t^{d_{l_{0},l_{1}}} \epsilon^{-l_{1}}\int_{P_k}s^{l_1}v_{HJ_n}(s,z,\epsilon) \exp( -\frac{s}{\epsilon t} ) \frac{ds}{s}
= t^{d_{l_{0},l_{1}}} (t^{2}\partial_{t})^{l_1}y_{\mathcal{E}_{HJ_n}^{k}}(t,z,\epsilon) 
\end{equation}
for all $t \in \mathcal{T} \cap D(0,r_{\mathcal{T}})$, $\epsilon \in \mathcal{E}_{HJ_n}^{k}$. Lemma 20 follows.
\end{proof}
In order to discuss the second point 2) of the statement, let us concentrate on the equation
(\ref{ACP_forcterm_w}) equipped with the forcing term $w(\tau,z,\epsilon) = w_{S_{d_p}}(\tau,z,\epsilon)$ for given initial
data (\ref{ACP_forcterm_w_iv_d_j}). We must check that the problem (\ref{ACP_forcterm_w}), (\ref{ACP_forcterm_w_iv_d_j}) has a unique
formal series solution
\begin{equation}
v_{S_{d_p}}(\tau,z,\epsilon) = \sum_{\beta \geq 0} v_{\beta}(\tau,\epsilon) \frac{z^{\beta}}{\beta!} \label{defin_vSdp}
\end{equation}
where $v_{\beta}(\tau,\epsilon)$ are holomorphic on $(S_{d_p} \cup D(0,r)) \times \dot{D}(0,\epsilon_{0})$, continuous
on $(\bar{S}_{d_p} \cup \bar{D}(0,r)) \times \dot{D}(0,\epsilon_{0})$. Indeed, the formal expansion
(\ref{defin_vSdp}) solves (\ref{ACP_forcterm_w}), (\ref{ACP_forcterm_w_iv_d_j}) if and only if $v_{\beta}(\tau,\epsilon)$ fulfills the
recursion (\ref{recursion_v_beta}) for all $\beta \geq 0$, where $w_{\beta}(\tau,\epsilon)$ represent the Taylor coefficients of the
forcing term $w_{S_{d_{p}}}(\tau,\epsilon)$ which are implemented by the recursion (\ref{recursion_w_beta}). As a consequence, all the
coefficients $v_{n}(\tau,\epsilon)$ for $n \geq S_{\mathcal{B}}$ define holomorphic functions on
$(S_{d_p} \cup D(0,r)) \times \dot{D}(0,\epsilon_{0})$, continuous on
$(\bar{S}_{d_p} \cup \bar{D}(0,r)) \times \dot{D}(0,\epsilon_{0})$ in view of the fact that it is already the case for
$w_{\beta}(\tau,\epsilon)$, $\beta \geq 0$ and the initial conditions (\ref{ACP_forcterm_w_iv_d_j}).

In accordance with the whole set of requirements made since the onset of Section 5, we can see that the constraints 3)a)b) imposed
in Proposition 22 are obeyed. Hence, the formal series $v_{S_{d_p}}(\tau,z,\epsilon)$ belongs to the Banach spaces
$EG_{(\sigma_{1},S_{d_p} \cup D(0,r),\epsilon,\delta)}$ for all $\epsilon \in \dot{D}(0,\epsilon_{0})$, for any
$\sigma_{1} > \sigma_{1}'$ and a constant $C_{S_{d_p}}^{v}>0$ is given for which
$$ ||v_{S_{d_p}}(\tau,z,\epsilon)||_{(\sigma_{1},S_{d_p} \cup D(0,r),\epsilon,\delta)} \leq C_{S_{d_p}}^{v} $$
for all $\epsilon \in \dot{D}(0,\epsilon_{0})$. As a byproduct, bearing in mind Proposition 5 2), $v_{S_{d_p}}(\tau,z,\epsilon)$ defines
a holomorphic function on $(S_{d_p} \cup D(0,r)) \times D(0,\delta \delta_{1}) \times \dot{D}(0,\epsilon_{0})$, continuous
on $(\bar{S}_{d_p} \cup \bar{D}(0,r)) \times D(0,\delta \delta_{1}) \times \dot{D}(0,\epsilon_{0})$, for some $0 < \delta_{1} < 1$ that
suffers the bounds (\ref{bds_vSdp}). By application of the same arguments as in Lemma 9, one can prove that the function
$y_{\mathcal{E}_{S_{d_p}}}(t,z,\epsilon)$ defined as (\ref{defin_y_ESdp}) induces a bounded holomorphic function on
$(\mathcal{T} \cap D(0,r_{\mathcal{T}})) \times D(0,\delta \delta_{1}) \times \mathcal{E}_{S_{d_p}}$. Moreover, an analogous reasoning
as the one in Proposition 11 2) leads to the bounds (\ref{exp_flat_difference_yk_plus_1_minus_yk_Sdp}).

Lastly, we notice that $y_{\mathcal{E}_{S_{d_p}}}(t,z,\epsilon)$ shall solve the problem (\ref{SPCP_second}), (\ref{SPCP_second_i_d}).
Bearing in mind the operators unfoldings (\ref{Tahara_expansion_diff_op}), this follows from the observation that the next identity holds
 \begin{multline}
t^{d_{l_{0},l_{1}}} (t^{2}\partial_{t})^{l_1} y_{\mathcal{E}_{S_{d_p}}}(t,z,\epsilon) =
\frac{\epsilon^{-(d_{l_{0},l_{1}}+l_{1})}}{\Gamma(d_{l_{0},l_{1}})} \int_{L_{\gamma_{d_p}}}
u \int_{0}^{u} (u-s)^{d_{l_{0},l_{1}}-1}s^{l_{1}}\\
\times v_{S_{d_{p}}}(s,z,\epsilon) \frac{ds}{s} \exp( -\frac{u}{\epsilon t} ) \frac{du}{u} \label{expansion_diff_op_Laplace_y_Sdp}
\end{multline}
for all $t \in \mathcal{T} \cap D(0,r_{\mathcal{T}})$, $\epsilon \in \mathcal{E}_{S_{d_p}}$, all given positive integers
$d_{l_{0},l_{1}},l_{1} \geq 1$. Its proof remains a straightforward adaptation of the one of Lemma 20 and is therefore omitted.

Ultimately, we are left to testify the estimates (\ref{difference_y_HJn_Sd0}) and (\ref{difference_y_HJn_Sdiota}). Again, this follows
from paths deformations methods which mirrors the lines of arguments detailed in the proof of Theorem 1 3).
\end{proof}

Since the forcing term $u(t,z,\epsilon)$ in the equation (\ref{SPCP_second}) in particular solves the Cauchy problem
(\ref{SPCP_first}), (\ref{SPCP_first_i_d}), we deduce that the functions $y_{\mathcal{E}_{HJ_n}^{k}}(t,z,\epsilon)$ and
$y_{\mathcal{E}_{S_{d_p}}}(t,z,\epsilon)$ themselves solve a Cauchy problem with holomorphic coefficients in the vicinity of the
origin in $\mathbb{C}^{3}$. Namely,
\begin{corol} Let us introduce the next differential and linear fractional operators
\begin{multline*}
\mathcal{P}_{1}(t,z,\epsilon,\{ m_{k,t,\epsilon} \}_{k \in I_{\mathcal{A}}},\partial_{t},\partial_{z}) = 
P(\epsilon t^{2}\partial_{t})\partial_{z}^{S}
- \sum_{\underline{k} = (k_{0},k_{1},k_{2}) \in \mathcal{A}}
c_{\underline{k}}(z,\epsilon) m_{k_{2},t,\epsilon}(t^{2}\partial_{t})^{k_0} \partial_{z}^{k_1},\\
\mathcal{P}_{2}(t,z,\epsilon,\partial_{t},\partial_{z}) =
P_{\mathcal{B}}(\epsilon t^{2}\partial_{t}) \partial_{z}^{S_{\mathcal{B}}} -
\sum_{\underline{l}=(l_{0},l_{1},l_{2}) \in \mathcal{B}} d_{\underline{l}}(z,\epsilon)
t^{l_0}\partial_{t}^{l_1}\partial_{z}^{l_{2}}
\end{multline*}
where $m_{k_{2},t,\epsilon}$ stands for the Moebius operator
$m_{k_{2},t,\epsilon}(u(t,z,\epsilon)) = u(\frac{t}{1 + k_{2} \epsilon t},z,\epsilon)$.

Then, the functions $y_{\mathcal{E}_{HJ_n}^{k}}(t,z,\epsilon)$, for $k \in \llbracket -n,n \rrbracket$ and
$y_{\mathcal{E}_{S_{d_p}}}(t,z,\epsilon)$ for $0 \leq p \leq \iota-1$ are actual holomorphic solutions to the next
Cauchy problem
$$ \mathcal{P}_{1}(t,z,\epsilon,\{ m_{k,t,\epsilon} \}_{k \in I_{\mathcal{A}}},\partial_{t},\partial_{z})
\mathcal{P}_{2}(t,z,\epsilon,\partial_{t},\partial_{z})y(t,z,\epsilon) = 0 $$
whose coefficients are holomorphic w.r.t $z$ and $\epsilon$ near on a neighborhood of the origin and polynomial in $t$,
under the constraints
$$
\left\{ \begin{aligned}
         (\partial_{z}^{j}y)(t,0,\epsilon) = \psi_{j}(t,\epsilon) \ \ , \ \ 0 \leq j \leq S_{\mathcal{B}}-1 \\
         (\partial_{z}^{j}\mathcal{P}_{2}(t,z,\epsilon,\partial_{t},\partial_{z})y)(t,0,\epsilon) = \varphi_{j}(t,\epsilon) \ \ , \ \
         0 \leq j \leq S-1.
        \end{aligned} \right.
$$
\end{corol}

\section{Parametric Gevrey asymptotic expansions with two levels 1 and $1^{+}$ for the analytic solutions to the Cauchy problems displayed
in Sections 3 and 5}

\subsection{A version of the Ramis-Sibuya Theorem involving two levels}

Within this section we state a version of a variant of a classical cohomological criterion in the framework of Gevrey asymptotics known as the
Ramis-Sibuya Theorem (see \cite{hssi}, Theorem XI-2-3) obtained by the first author in the work \cite{ma}. Besides, in view of the recent results
on so-called $\mathbb{M}-$summability for strongly regular sequences $\mathbb{M} = (M_{n})_{n \geq 0}$ obtained by the authors and
J. Sanz, we can provide sufficient conditions which gives rise to the special situation that involves both 1 and $1^{+}$ summability.\medskip

\noindent We depart from the definitions of Gevrey 1 and $1^{+}$ asymptotics.\medskip

\noindent Let $(\mathbb{F},||.||_{\mathbb{F}})$ be a Banach space over $\mathbb{C}$. The set $\mathbb{F}[[\epsilon]]$ stands for the
space of all formal series $\sum_{k \geq 0} a_{k} \epsilon^{k}$ with coefficients $a_{k}$ belonging to $\mathbb{F}$ for all integers $k \geq 0$.
We consider $f : \mathcal{F} \rightarrow \mathbb{F}$ be a holomorphic function on a bounded open sector $\mathcal{F}$
centered at 0 and $\hat{f}(\epsilon) = \sum_{k \geq 0} a_{k} \epsilon^{k} \in \mathbb{F}[[\epsilon]]$ be a formal series.

\begin{defin} The function $f$ is said to possess the formal series $\hat{f}$
as $1-$Gevrey asymptotic expansion if, for any closed proper subsector $\mathcal{W} \subset \mathcal{F}$ centered at 0, there exist
$C,M>0$ such that
\begin{equation}
 ||f(\epsilon) - \sum_{k=0}^{N-1} a_{k} \epsilon^{k}||_{\mathbb{F}} \leq CM^{N}(N/e)^{N}|\epsilon|^{N} \label{f_expansion_Gevrey_1}
\end{equation}
for all $N \geq 1$, all $\epsilon \in \mathcal{W}$. When the aperture of $\mathcal{F}$ is slightly larger than $\pi$, then according to
the Watson's lemma (see \cite{ba2}, Proposition 11), $f$ is the unique holomorphic function on $\mathcal{F}$ satisfying
(\ref{f_expansion_Gevrey_1}). The function $f$ is then called the $1-$sum of $\hat{f}$ on $\mathcal{F}$ and can be reconstructed from $\hat{f}$ using
Borel/Laplace transforms as detailed in Chapter 3 of \cite{ba1}.
\end{defin}

\begin{defin} We say that $f$ has the formal series $\hat{f}$
as $1^{+}-$Gevrey asymptotic expansion if, for any closed proper subsector $\mathcal{W} \subset \mathcal{F}$ centered at 0,
there exist $C,M>0$ such that
\begin{equation}
||f(\epsilon) - \sum_{k=0}^{N-1} a_{k} \epsilon^{k}||_{\mathbb{F}} \leq CM^{N}(N/\mathrm{Log} N)^{N}|\epsilon|^{N} \label{f_expansion_Gevrey_1_plus}
\end{equation}
for all $N \geq 2$, all $\epsilon \in \mathcal{W}$. In particular, the formal series $\hat{f}$ is itself of $1^{+}-$Gevrey type, meaning that
there exist two constants $C',M'>0$ such that $||a_{k}||_{\mathbb{F}} \leq C'M'^{k}(k/\mathrm{Log} k)^{k}$ for all $k \geq 2$.
Provided that the aperture of $\mathcal{F}$ is slightly larger than $\pi$, Theorem 3.1 in \cite{lamasa} ensures the unicity of the analytic
function $f$ fulfilling the estimates (\ref{f_expansion_Gevrey_1_plus}) on $\mathcal{F}$
(see the next remark underneath). In that case, $\hat{f}$ is named
$\mathbb{M}-$summable on $\mathcal{F}$ for the strongly regular sequence $\mathbb{M} = (M_{n})_{n \geq 0}$ where
$M_{n} = (\frac{n}{\mathrm{Log}(n+2)})^{n}$ and $f$ denotes the $\mathbb{M}-$sum of $\hat{f}$ on $\mathcal{F}$.
For brevity of notation, we will call it also $1^{+}-$sum. As explained in \cite{lamasa}, the $1^{+}-$sum $f$ can be recovered from
the formal expansions $\hat{f}$ with the help of an analog of a Borel/Laplace procedure. It is worthwhile noting that this notion
of $1^{+}-$summability has to be distinguished from the notion of $1^{+}-$summability introduced
in the papers of G. Immink whose sums are defined on domains which are not sectors, see \cite{im1},\cite{im2},\cite{im3}.
\end{defin}

{\bf Remark :} The strongly regular sequence $\mathbb{M}$ stated above is equivalent, in the sense that the functional spaces associated to them coincide, to $\mathbb{M}_{\alpha,\beta}=(n!^{\alpha}\prod_{m=0}^{n}\log^{\beta}(e+m))_{n\ge 0}$, for $\alpha=1,\beta=-1$. In this case, one has $\omega(\mathbb{M})=1$, meaning that unicity of the sum $f$ in (\ref{f_expansion_Gevrey_1_plus}) is guaranteed, for a prescribed asymptotic expansion, when departing from a sector of opening larger than $\pi$. The criteria leans on the divergence of a series of positive real numbers, see~\cite{korem}.

\medskip

We consider the set of sectors $ \underline{\mathcal{E}} = \{ \mathcal{E}_{HJ_n}^{k} \}_{k \in \llbracket -n,n \rrbracket}
\cup \{ \mathcal{E}_{S_{d_p}} \}_{0 \leq p \leq \iota - 1}$ constructed in Section 3.3 that fufills the constraints 3),4) and 5).
The set $\underline{\mathcal{E}}$ forms a so-called good covering in $\mathbb{C}^{\ast}$ as given in Definition 3 of \cite{ma}.

We rephrase the version of the Ramis-Sibuya which entails both $1-$Gevrey and $1^{+}-$Gevrey asymptotics displayed in \cite{ma}
for the specific covering $\underline{\mathcal{E}}$ with additional informations concerning $1$ and $1^{+}$ summability.\medskip
 
\begin{prop} Let $(\mathbb{F},||.||_{\mathbb{F}})$ be a Banach space over $\mathbb{C}$. For all $k \in \llbracket -n,n \rrbracket$
and $0 \leq p \leq \iota - 1$, let $G_{k}$ be a holomorphic function from $\mathcal{E}_{HJ_n}^{k}$ into $(\mathbb{F},||.||_{\mathbb{F}})$
and $\breve{G}_{p}$ be a holomorphic function from $\mathcal{E}_{S_{d_p}}$ into $(\mathbb{F},||.||_{\mathbb{F}})$.

\noindent We consider a cocycle $\underline{\Delta}(\epsilon)$ defined as the set of functions
$\breve{\Delta}_{p} = \breve{G}_{p+1}(\epsilon) - \breve{G}_{p}(\epsilon)$ for $0 \leq p \leq \iota-2$ when 
$\epsilon \in \mathcal{E}_{S_{d_{p+1}}} \cap \mathcal{E}_{S_{d_p}}$, $\Delta_{k}(\epsilon) = G_{k}(\epsilon) - G_{k+1}(\epsilon)$
for $-n \leq k \leq n-1$ and $\epsilon \in \mathcal{E}_{HJ_n}^{k} \cap \mathcal{E}_{HJ_n}^{k+1}$ together with
$\Delta_{-n,0}(\epsilon) = \breve{G}_{0}(\epsilon) - G_{-n}(\epsilon)$ on $\mathcal{E}_{S_{d_0}} \cap \mathcal{E}_{HJ_n}^{-n}$
and $\Delta_{\iota-1,n}(\epsilon) = G_{n}(\epsilon) - \breve{G}_{\iota-1}(\epsilon)$ on
$\mathcal{E}_{HJ_n}^{n} \cap \mathcal{E}_{S_{d_{\iota-1}}}$.

\noindent We make the next assumptions:\\
1) The functions $G_{k}$ and $\breve{G}_{p}$ are bounded as $\epsilon$ tends to 0 on their domains of definition.\\
2) For all $0 \leq p \leq \iota-2$, $\breve{\Delta}_{p}(\epsilon)$ and both $\Delta_{-n,0}(\epsilon)$, $\Delta_{\iota-1,n}(\epsilon)$ are
exponentially flat. This means that one can sort constants $\breve{K}_{p},\breve{M}_{p}>0$ and $K_{-n,0},M_{-n,0}>0$ with
$K_{\iota-1,n},M_{\iota-1,n}>0$ such that
\begin{multline}
|| \breve{\Delta}_{p}(\epsilon) ||_{\mathbb{F}} \leq \breve{K}_{p}\exp( - \frac{\breve{M}_{p}}{|\epsilon|} ) \ \
\mbox{for $\epsilon \in \mathcal{E}_{S_{d_{p+1}}} \cap \mathcal{E}_{S_{d_{p}}}$ },\\
||\Delta_{-n,0}(\epsilon)||_{\mathbb{F}} \leq K_{-n,0}\exp( - \frac{M_{n,0}}{|\epsilon|} ) \ \
\mbox{for $\epsilon \in \mathcal{E}_{HJ_n}^{-n} \cap \mathcal{E}_{S_{d_0}}$ },\\
||\Delta_{\iota-1,n}(\epsilon)||_{\mathbb{F}} \leq K_{\iota-1,n} \exp( -\frac{M_{\iota-1,n}}{|\epsilon|} ) \ \
\mbox{for $\epsilon \in \mathcal{E}_{HJ_n}^{n} \cap \mathcal{E}_{S_{d_{\iota-1}}}$.} \label{cond_Delta_cocycle_exp_flat}
\end{multline}
3) For $-n \leq k \leq n-1$, $\Delta_{k}(\epsilon)$ are super-exponentially flat on
$\mathcal{E}_{HJ_n}^{k+1} \cap \mathcal{E}_{HJ_n}^{k}$. This signifies that one can pick up constants $K_{k},M_{k}>0$ and $L_{k}>1$ such that
\begin{equation}
|| \Delta_{k}(\epsilon) ||_{\mathbb{F}} \leq K_{k} \exp( -\frac{M_k}{|\epsilon|} \mathrm{Log} \frac{L_k}{|\epsilon|} )
\label{cond_Delta_cocycle_log_exp_flat}
\end{equation}
for all $\epsilon \in \mathcal{E}_{HJ_n}^{k+1} \cap \mathcal{E}_{HJ_n}^{k}$.

Then, there exist a convergent power series $a(\epsilon) \in \mathbb{F}\{ \epsilon \}$ near $\epsilon=0$ and two formal series
$\hat{G}^{1}(\epsilon),\hat{G}^{2}(\epsilon) \in \mathbb{F} [[ \epsilon ]]$ with the property that $G_{k}(\epsilon)$ and $\breve{G}_{p}(\epsilon)$
admit the next decomposition
\begin{equation}
G_{k}(\epsilon) = a(\epsilon) + G_{k}^{1}(\epsilon) + G_{k}^{2}(\epsilon) \ \ , \ \
\breve{G}_{p}(\epsilon) = a(\epsilon) + \breve{G}_{p}^{1}(\epsilon) + \breve{G}_{p}^{2}(\epsilon)
\end{equation}
for $k \in \llbracket -n,n \rrbracket$, $0 \leq p \leq \iota-1$, where $G_{k}^{1}(\epsilon)$ (resp. $G_{k}^{2}(\epsilon)$) are
holomorphic on $\mathcal{E}_{HJ_n}^{k}$ and have $\hat{G}^{1}(\epsilon)$ (resp. $\hat{G}^{2}(\epsilon)$) as $1-$Gevrey (resp.
$1^{+}-$Gevrey) asymptotic expansion on $\mathcal{E}_{HJ_n}^{k}$ and where
$\breve{G}_{p}^{1}$ (resp. $\breve{G}_{p}^{2}(\epsilon)$) are holomorphic on $\mathcal{E}_{S_{d_p}}$ and possesses
$\hat{G}^{1}(\epsilon)$ (resp. $\hat{G}^{2}(\epsilon)$) as $1-$Gevrey (resp. $1^{+}-$Gevrey) asymptotic expansion
on $\mathcal{E}_{S_{d_p}}$. Besides, the functions $G_{-n}^{2}(\epsilon)$,$G_{n}^{2}(\epsilon)$ and $\breve{G}_{h}^{2}(\epsilon)$
for $0 \leq h \leq \iota-1$ turn out to be the restriction of a common holomorphic function denoted $G^{2}(\epsilon)$
on the large sector $\mathcal{E}_{HS} = \mathcal{E}_{HJ_n}^{-n} \cup \bigcup_{h=0}^{\iota-1} \mathcal{E}_{S_{d_h}} \cup \mathcal{E}_{HJ_n}^{n}$
which determines the $1^{+}-$sum of $\hat{G}^{2}(\epsilon)$ on $\mathcal{E}_{HS}$. Moreover, $\check{G}_{p}^{1}(\epsilon)$ represents the $1-$sum of $\hat{G}^{1}(\epsilon)$ on $\mathcal{E}_{S_{d_p}}$ whenever the aperture of
$\mathcal{E}_{S_{d_p}}$ is strictly larger than $\pi$.
\end{prop}
\begin{proof} Since the notations used here are rather different from the ones within the result enounced in \cite{ma} and in order
to explain the
part of the proposition concerning $1$ and $1^{+}$ summability which is not mentioned in our previous work \cite{ma}, we have decided
to present a sketch of proof of the statement.

We consider a first cocycle $\underline{\Delta}^{1}(\epsilon)$ defined by the next family of functions
\begin{multline}
\breve{\Delta}_{p}^{1}(\epsilon) = \breve{\Delta}_{p}(\epsilon) \ \ \mbox{for $0 \leq p \leq \iota-2$ on
$\mathcal{E}_{S_{d_{p+1}}} \cap \mathcal{E}_{S_{d_p}}$},\\
\Delta_{-n,0}^{1}(\epsilon) = \Delta_{-n,0}(\epsilon) \ \ \mbox{on $\mathcal{E}_{S_{d_0}} \cap \mathcal{E}_{HJ_n}^{-n}$},
\ \ \Delta_{\iota-1,n}^{1}(\epsilon) = \Delta_{\iota-1,n}(\epsilon) \ \ \mbox{on $\mathcal{E}_{HJ_n}^{n} \cap \mathcal{E}_{S_{d_{\iota-1}}}$},\\
\Delta_{k}^{1}(\epsilon) = 0 \ \ \mbox{for $-n \leq k \leq n-1$ on $\mathcal{E}_{HJ_n}^{k+1} \cap \mathcal{E}_{HJ_n}^{k}$},
\label{cocycle_1_delta}
\end{multline}
and a second cocycle $\underline{\Delta}^{2}(\epsilon)$ described by the forthcoming set of functions
\begin{multline}
\breve{\Delta}_{p}^{2}(\epsilon) = 0 \ \ \mbox{for $0 \leq p \leq \iota-2$ on
$\mathcal{E}_{S_{d_{p+1}}} \cap \mathcal{E}_{S_{d_p}}$}, \\
\Delta_{-n,0}^{2}(\epsilon) = 0 \ \ \mbox{on $\mathcal{E}_{S_{d_0}} \cap \mathcal{E}_{HJ_n}^{-n}$}, \ \ \Delta_{\iota-1,n}^{2} = 0,
\ \ \mbox{on $\mathcal{E}_{HJ_n}^{n} \cap \mathcal{E}_{S_{d_{\iota-1}}}$},\\
\Delta_{k}^{2}(\epsilon) = \Delta_{k}(\epsilon) \ \
\mbox{for $-n \leq k \leq n-1$ on $\mathcal{E}_{HJ_n}^{k+1} \cap \mathcal{E}_{HJ_n}^{k}$}. \label{cocycle_2_delta}
\end{multline}
The next lemma restate Lemma 14 from \cite{ma}.
\begin{lemma} For all $k \in \llbracket -n,n \rrbracket$, all $0 \leq p \leq \iota-1$, there exist bounded holomorphic functions
$G_{k}^{1} : \mathcal{E}_{HJ_n}^{k} \rightarrow \mathbb{C}$ and $\breve{G}_{p}^{1}:\mathcal{E}_{S_{d_p}} \rightarrow \mathbb{C}$
that satisfy the property
\begin{multline}
\breve{\Delta}_{p}^{1}(\epsilon) =  \breve{G}_{p+1}^{1}(\epsilon) - \breve{G}_{p}^{1}(\epsilon) \ \ \mbox{for $0 \leq p \leq \iota-2$ on
$\mathcal{E}_{S_{d_{p+1}}} \cap \mathcal{E}_{S_{d_p}}$},\\
\Delta_{-n,0}^{1}(\epsilon) = \breve{G}_{0}^{1}(\epsilon) - G_{-n}^{1}(\epsilon)
\ \ \mbox{on $\mathcal{E}_{S_{d_0}} \cap \mathcal{E}_{HJ_n}^{-n}$}, \ \ \Delta_{\iota-1,n}^{1}(\epsilon) =
G_{n}^{1}(\epsilon) - \breve{G}_{\iota-1}^{1}(\epsilon) \ \ \mbox{on $\mathcal{E}_{HJ_n}^{n} \cap \mathcal{E}_{S_{d_{\iota-1}}}$},\\
\Delta_{k}^{1}(\epsilon) = G_{k}^{1}(\epsilon) - G_{k+1}^{1}(\epsilon) \ \ \mbox{for $-n \leq k \leq n-1$ on
$\mathcal{E}_{HJ_n}^{k+1} \cap \mathcal{E}_{HJ_n}^{k}$}. \label{cocycle_1_delta_split}
\end{multline}
Furthermore, one can get coefficients $\varphi_{m}^{1} \in \mathbb{F}$, for $m \geq 0$ such that\\
1) For all $k \in \llbracket -n,n \rrbracket$,
any closed proper subsector $\mathcal{W} \subset \mathcal{E}_{HJ_n}^{k}$, centered at 0, there exist constants $K_{k},M_{k}>0$ with
\begin{equation}
||G_{k}^{1}(\epsilon) - \sum_{m=0}^{N-1} \varphi_{m}^{1} \epsilon^{m} ||_{\mathbb{F}} \leq
K_{k}(M_{k})^{N}(\frac{N}{e})^{N} |\epsilon|^{N}
\end{equation}
for all $\epsilon \in \mathcal{W}$, all $N \geq 1$.\\
2) For $0 \leq p \leq \iota-1$, any closed proper subsector $\mathcal{W} \subset \mathcal{E}_{S_{d_p}}$, centered at 0,
one can grab constants $K_{p},M_{p}>0$ with
\begin{equation}
||\breve{G}_{p}^{1}(\epsilon) - \sum_{m=0}^{N-1} \varphi_{m}^{1} \epsilon^{m} ||_{\mathbb{F}} \leq
K_{p}(M_{p})^{N}(\frac{N}{e})^{N} |\epsilon|^{N} \label{expansion_breveG_p1}
\end{equation}
for all $\epsilon \in \mathcal{W}$, all $N \geq 1$.
\end{lemma}
Likewise, the next lemma recapitulates Lemma 15 from \cite{ma}.
\begin{lemma} For all $k \in \llbracket -n,n \rrbracket$, all $0 \leq p \leq \iota-1$, one can find bounded holomorphic functions
$G_{k}^{2} : \mathcal{E}_{HJ_n}^{k} \rightarrow \mathbb{C}$ and $\breve{G}_{p}^{2}:\mathcal{E}_{S_{d_p}} \rightarrow \mathbb{C}$
that obey to the next demand
 \begin{multline}
\breve{\Delta}_{p}^{2}(\epsilon) =  \breve{G}_{p+1}^{2}(\epsilon) - \breve{G}_{p}^{2}(\epsilon) \ \ \mbox{for $0 \leq p \leq \iota-2$ on
$\mathcal{E}_{S_{d_{p+1}}} \cap \mathcal{E}_{S_{d_p}}$},\\
\Delta_{-n,0}^{2}(\epsilon) = \breve{G}_{0}^{2}(\epsilon) - G_{-n}^{2}(\epsilon)
\ \ \mbox{on $\mathcal{E}_{S_{d_0}} \cap \mathcal{E}_{HJ_n}^{-n}$}, \ \ \Delta_{\iota-1,n}^{2}(\epsilon) =
G_{n}^{2}(\epsilon) - \breve{G}_{\iota-1}^{2}(\epsilon) \ \ \mbox{on $\mathcal{E}_{HJ_n}^{n} \cap \mathcal{E}_{S_{d_{\iota-1}}}$},\\
\Delta_{k}^{2}(\epsilon) = G_{k}^{2}(\epsilon) - G_{k+1}^{2}(\epsilon) \ \ \mbox{for $-n \leq k \leq n-1$ on
$\mathcal{E}_{HJ_n}^{k+1} \cap \mathcal{E}_{HJ_n}^{k}$}. \label{cocycle_2_delta_split}
\end{multline}
Moreover, one can obtain coefficients $\varphi_{m}^{2} \in \mathbb{F}$, for $m \geq 0$ such that\\
1) For all $k \in \llbracket -n,n \rrbracket$,
any closed proper subsector $\mathcal{W} \subset \mathcal{E}_{HJ_n}^{k}$, centered at 0, one can find constants $K_{k},M_{k}>0$ with
\begin{equation}
||G_{k}^{2}(\epsilon) - \sum_{m=0}^{N-1} \varphi_{m}^{2} \epsilon^{m} ||_{\mathbb{F}} \leq
K_{k}(M_{k})^{N}(\frac{N}{\mathrm{Log} N})^{N} |\epsilon|^{N} \label{expansion_G_k2}
\end{equation}
for all $\epsilon \in \mathcal{W}$, all $N \geq 2$.\\
2) For $0 \leq p \leq \iota-1$, any closed proper subsector $\mathcal{W} \subset \mathcal{E}_{S_{d_p}}$, centered at 0,
one can grasp constants $K_{p},M_{p}>0$ with
\begin{equation}
||\breve{G}_{p}^{2}(\epsilon) - \sum_{m=0}^{N-1} \varphi_{m}^{2} \epsilon^{m} ||_{\mathbb{F}} \leq
K_{p}(M_{p})^{N}(\frac{N}{\mathrm{Log} N})^{N} |\epsilon|^{N} \label{expansion_breveG_p2}
\end{equation}
for all $\epsilon \in \mathcal{W}$, all $N \geq 2$.
\end{lemma}
We introduce the bounded holomorphic functions
$$
a_{k}(\epsilon) = G_{k}(\epsilon) - G_{k}^{1}(\epsilon) - G_{k}^{2}(\epsilon) \ \ \mbox{for $\epsilon \in \mathcal{E}_{HJ_n}^{k}$}, \ \
\breve{a}_{p}(\epsilon) = \breve{G}_{p}(\epsilon) - \breve{G}_{p}^{1}(\epsilon) - \breve{G}_{p}^{2}(\epsilon) \ \
\mbox{for $\epsilon \in \mathcal{E}_{S_{d_p}}$}.
$$
for $k \in \llbracket -n,n \rrbracket$ and $0 \leq p \leq \iota-1$. By construction, we notice that
\begin{multline*}
a_{k}(\epsilon) - a_{k+1}(\epsilon) = G_{k}(\epsilon) - G_{k}^{1}(\epsilon) - G_{k}^{2}(\epsilon) - G_{k+1}(\epsilon) +
G_{k+1}^{1}(\epsilon) + G_{k+1}^{2}(\epsilon)\\
= G_{k}(\epsilon) - G_{k+1}(\epsilon) - \Delta_{k}^{1}(\epsilon) - \Delta_{k}^{2}(\epsilon)
= G_{k}(\epsilon) - G_{k+1}(\epsilon) - \Delta_{k}(\epsilon) = 0
\end{multline*}
for $-n \leq k \leq n-1$ on $\mathcal{E}_{HJ_n}^{k+1} \cap \mathcal{E}_{HJ_n}^{k}$ together with
\begin{multline*}
\breve{a}_{p+1}(\epsilon) - \breve{a}_{p}(\epsilon) = \breve{G}_{p+1}(\epsilon) - \breve{G}_{p+1}(\epsilon)
- \breve{\Delta}_{p}^{1}(\epsilon) - \breve{\Delta}_{p}^{2}(\epsilon) = \breve{G}_{p+1}(\epsilon) - \breve{G}_{p+1}(\epsilon)
- \breve{\Delta}_{p}(\epsilon) = 0
\end{multline*}
for $0 \leq p \leq \iota-2$ on $\mathcal{E}_{S_{d_{p+1}}} \cap \mathcal{E}_{S_{d_{p}}}$. Furthermore,
\begin{multline*}
\breve{a}_{0}(\epsilon) - a_{-n}(\epsilon) = \breve{G}_{0}(\epsilon) - \breve{G}_{0}^{1}(\epsilon) -
\breve{G}_{0}^{2}(\epsilon) - G_{-n}(\epsilon) + G_{-n}^{1}(\epsilon) + G_{-n}^{2}(\epsilon)\\
= \breve{G}_{0}(\epsilon) - G_{-n}(\epsilon) - \Delta_{-n,0}^{1}(\epsilon) - \Delta_{-n,0}^{2}(\epsilon) =
\breve{G}_{0}(\epsilon) - G_{-n}(\epsilon) - \Delta_{-n,0}(\epsilon) = 0
\end{multline*}
for $\epsilon \in \mathcal{E}_{HJ_n}^{-n} \cap \mathcal{E}_{S_{d_0}}$ and
\begin{multline*}
a_{n}(\epsilon) - \breve{a}_{\iota-1}(\epsilon) = G_{n}(\epsilon) - G_{n}^{1}(\epsilon) - G_{n}^{2}(\epsilon)
- \breve{G}_{\iota-1}(\epsilon) + \breve{G}_{\iota-1}^{1}(\epsilon) + \breve{G}_{\iota-1}^{2}(\epsilon)\\
= G_{n}(\epsilon) - \breve{G}_{\iota-1}(\epsilon) - \Delta_{\iota-1,n}^{1}(\epsilon) - \Delta_{\iota-1,n}^{2}(\epsilon)
= G_{n}(\epsilon) - \breve{G}_{\iota-1}(\epsilon) - \Delta_{\iota-1,n}(\epsilon) = 0
\end{multline*}
whenever $\epsilon \in \mathcal{E}_{HJ_n}^{n} \cap \mathcal{E}_{S_{d_{\iota-1}}}$.

As a result, the functions $a_{k}(\epsilon)$ on $\mathcal{E}_{HJ_n}^{k}$ and $\breve{a}_{p}(\epsilon)$ on
$\mathcal{E}_{S_{d_p}}$ are the restriction of a common holomorphic bounded function $a(\epsilon)$ on $D(0,\epsilon_{0}) \setminus \{ 0 \}$.
The origin is therefore a removable singularity and $a(\epsilon)$ defines a convergent power series on $D(0,\epsilon_{0})$.

As a consequence, one can write
$$
G_{k}(\epsilon) = a(\epsilon) + G_{k}^{1}(\epsilon) + G_{k}^{2}(\epsilon) \ \ \mbox{on $\mathcal{E}_{HJ_n}^{k}$}, \ \
\breve{G}_{p}(\epsilon) = a(\epsilon) + \breve{G}_{p}^{1}(\epsilon) + \breve{G}_{p}^{2}(\epsilon) \ \ \mbox{on
$\mathcal{E}_{S_{d_p}}$}
$$
for all $k \in \llbracket -n,n \rrbracket$, $0 \leq p \leq \iota-1$. Moreover, $G_{k}^{1}(\epsilon)$ (resp. $G_{k}^{2}(\epsilon)$)
have $\hat{G}^{1}(\epsilon) = \sum_{m \geq 0} \varphi_{m}^{1} \epsilon^{m}$ (resp.
$\hat{G}^{2}(\epsilon) = \sum_{m \geq 0} \varphi_{m}^{2} \epsilon^{m}$) as $1-$Gevrey (resp.
$1^{+}-$Gevrey) asymptotic expansion on $\mathcal{E}_{HJ_n}^{k}$ and
$\breve{G}_{p}^{1}$ (resp. $\breve{G}_{p}^{2}(\epsilon)$) possesses
$\hat{G}^{1}(\epsilon)$ (resp. $\hat{G}^{2}(\epsilon)$) as $1-$Gevrey (resp. $1^{+}-$Gevrey) asymptotic expansion
on $\mathcal{E}_{S_{d_p}}$.

By the very definition of the cocycles $\underline{\Delta}^{1}(\epsilon)$ and $\underline{\Delta}^{2}(\epsilon)$ given by
(\ref{cocycle_1_delta}) and (\ref{cocycle_2_delta}), in accordance with the constraints
(\ref{cocycle_1_delta_split}) and (\ref{cocycle_2_delta_split}), we get in particular that
\begin{multline*}
G_{n}^{2}(\epsilon) = \breve{G}_{\iota-1}^{2}(\epsilon) \ \ \mbox{on $\mathcal{E}_{S_{d_{\iota-1}}} \cap \mathcal{E}_{HJ_n}^{n}$}, \ \
G_{-n}^{2}(\epsilon) = \breve{G}_{0}^{2}(\epsilon) \ \ \mbox{on $\mathcal{E}_{S_{d_0}} \cap \mathcal{E}_{HJ_n}^{-n}$},\\
\breve{G}_{p+1}^{2}(\epsilon) = \breve{G}_{p}^{2}(\epsilon) \ \ \mbox{on $\mathcal{E}_{S_{d_{p+1}}} \cap \mathcal{E}_{S_{d_p}}$}
\end{multline*}
for all $0 \leq p \leq \iota-2$. For that reason, we see that $G_{-n}^{2}(\epsilon)$,$G_{n}^{2}(\epsilon)$ and
$\breve{G}_{p}^{2}(\epsilon)$ are the restrictions of a common holomorphic function denoted $G^{2}(\epsilon)$
on the large sector
$\mathcal{E}_{HS} = \mathcal{E}_{HJ_n}^{-n} \cup \bigcup_{h=0}^{\iota-1} \mathcal{E}_{S_{d_h}} \cup \mathcal{E}_{HJ_n}^{n}$ with aperture larger
than $\pi$. In addition, from the expansions (\ref{expansion_G_k2}) and (\ref{expansion_breveG_p2}) we deduce that $G^{2}(\epsilon)$
defines the $1^{+}-$sum of $\hat{G}^{2}(\epsilon)$ on $\mathcal{E}_{HS}$. Finally, when the aperture of $\mathcal{E}_{S_{d_p}}$ is strictly larger than $\pi$, in view of the expansion (247) it turns out that $\check{G}_{p}^{1! }$ defines the $1-$sum of $\hat{G}^{1}(\epsilon)$ on $\mathcal{E}_{S_{d_p}}$.
\end{proof}

\subsection{Existence of multiscale parametric Gevrey asymptotic expansions for the analytic solutions to the problems
(\ref{SPCP_first}), (\ref{SPCP_first_i_d}) and (\ref{SPCP_second}), (\ref{SPCP_second_i_d})} 

We are now ready to enounce the third main result of this work, which reveals a fine structure of two Gevrey orders 1 and $1^{+}$ for the
solutions $u_{\mathcal{E}_{HJ_n}^{k}}$ and $u_{\mathcal{E}_{S_{d_p}}}$ (resp. $y_{\mathcal{E}_{HJ_n}^{k}}$ and $y_{\mathcal{E}_{S_{d_p}}}$)
regarding the parameter $\epsilon$.

\begin{theo} Let us assume that all the requirements asked in Theorem 1 (resp. Theorem 2) are fulfilled. Then, there exist\\
- An holomorphic function $a(t,z,\epsilon)$ (resp. $b(t,z,\epsilon)$) on the domain
$(\mathcal{T} \cap D(0,r_{\mathcal{T}})) \times D(0,\delta \delta_{1}) \times D(0,\hat{\epsilon}_{0})$ for some
$0 < \hat{\epsilon}_{0} < \epsilon_{0}$,\\
- Two formal series
$$ \hat{u}^{j}(t,z,\epsilon) = \sum_{k \geq 0} u_{k}^{j}(t,z) \epsilon^{k} \in \mathbb{F}[[ \epsilon ]] \ \ , \ \ j=1,2 $$
(resp.
$$ \hat{y}^{j}(t,z,\epsilon) = \sum_{k \geq 0} y_{k}^{j}(t,z) \epsilon^{k} \in \mathbb{F}[[ \epsilon ]] \ \ , \ \ j=1,2) $$
whose coefficients $u_{k}^{j}(t,z)$ (resp. $y_{k}^{j}(t,z)$) belong to the Banach space
$\mathbb{F} = \mathcal{O}( (\mathcal{T} \cap D(0,r_{\mathcal{T}})) \times D(0,\delta \delta_{1}) )$ of bounded holomorphic functions
on the set $(\mathcal{T} \cap D(0,r_{\mathcal{T}})) \times D(0,\delta \delta_{1})$ endowed with the supremum norm,\\
which are submitted to the next features:\\
A) For each $k \in \llbracket -n,n \rrbracket$, the function $u_{\mathcal{E}_{HJ_n}^{k}}(t,z,\epsilon)$ (resp.
$y_{\mathcal{E}_{HJ_n}^{k}}(t,z,\epsilon)$) admits a decomposition
$$ u_{\mathcal{E}_{HJ_n}^{k}}(t,z,\epsilon) = a(t,z,\epsilon) + u_{\mathcal{E}_{HJ_n}^{k}}^{1}(t,z,\epsilon) +
u_{\mathcal{E}_{HJ_n}^{k}}^{2}(t,z,\epsilon) $$
(resp.
$$ y_{\mathcal{E}_{HJ_n}^{k}}(t,z,\epsilon) = b(t,z,\epsilon) + y_{\mathcal{E}_{HJ_n}^{k}}^{1}(t,z,\epsilon) +
y_{\mathcal{E}_{HJ_n}^{k}}^{2}(t,z,\epsilon) ) $$
where $u_{\mathcal{E}_{HJ_n}^{k}}^{1}(t,z,\epsilon)$ (resp. $y_{\mathcal{E}_{HJ_n}^{k}}^{1}(t,z,\epsilon)$) is bounded holomorphic
on $(\mathcal{T} \cap D(0,r_{\mathcal{T}})) \times D(0,\delta \delta_{1}) \times \mathcal{E}_{HJ_n}^{k}$ and possesses
$\hat{u}^{1}(t,z,\epsilon)$ (resp. $\hat{y}^{1}(t,z,\epsilon)$) as $1-$Gevrey asymptotic expansion on $\mathcal{E}_{HJ_n}^{k}$, meaning
that for any closed subsector $\mathcal{W} \subset \mathcal{E}_{HJ_n}^{k}$, there exist two constants $C,M>0$ with
$$ \sup_{t \in \mathcal{T} \cap D(0,r_{\mathcal{T}}), z \in D(0,\delta \delta_{1})}
|u_{\mathcal{E}_{HJ_n}^{k}}^{1}(t,z,\epsilon) - \sum_{k=0}^{N-1} u_{k}^{1}(t,z) \epsilon^{k}| \leq CM^{N}(\frac{N}{e})^{N} |\epsilon|^{N} $$
(resp. 
$$ \sup_{t \in \mathcal{T} \cap D(0,r_{\mathcal{T}}), z \in D(0,\delta \delta_{1})}
|y_{\mathcal{E}_{HJ_n}^{k}}^{1}(t,z,\epsilon) - \sum_{k=0}^{N-1} y_{k}^{1}(t,z) \epsilon^{k}| \leq CM^{N}(\frac{N}{e})^{N} |\epsilon|^{N}) $$
for all $N \geq 1$, all $\epsilon \in \mathcal{W}$ and $u_{\mathcal{E}_{HJ_n}^{k}}^{2}(t,z,\epsilon)$ (resp.
$y_{\mathcal{E}_{HJ_n}^{k}}^{2}(t,z,\epsilon)$) is bounded holomorphic
on $(\mathcal{T} \cap D(0,r_{\mathcal{T}})) \times D(0,\delta \delta_{1}) \times \mathcal{E}_{HJ_n}^{k}$ and carries
$\hat{u}^{2}(t,z,\epsilon)$ (resp. $\hat{y}^{2}(t,z,\epsilon)$) as $1^{+}-$Gevrey asymptotic expansion on $\mathcal{E}_{HJ_n}^{k}$, in other words,
for any closed subsector $\mathcal{W} \subset \mathcal{E}_{HJ_n}^{k}$, one can get two constants $C,M>0$ with
$$ \sup_{t \in \mathcal{T} \cap D(0,r_{\mathcal{T}}), z \in D(0,\delta \delta_{1})}
|u_{\mathcal{E}_{HJ_n}^{k}}^{2}(t,z,\epsilon) - \sum_{k=0}^{N-1} u_{k}^{2}(t,z) \epsilon^{k}| \leq CM^{N}(\frac{N}{\mathrm{Log} N})^{N}
|\epsilon|^{N} $$
(resp. 
$$ \sup_{t \in \mathcal{T} \cap D(0,r_{\mathcal{T}}), z \in D(0,\delta \delta_{1})}
|y_{\mathcal{E}_{HJ_n}^{k}}^{2}(t,z,\epsilon) - \sum_{k=0}^{N-1} y_{k}^{2}(t,z) \epsilon^{k}| \leq CM^{N}(\frac{N}{\mathrm{Log} N})^{N}
|\epsilon|^{N}) $$
for all $N \geq 2$, all $\epsilon \in \mathcal{W}$.\medskip

B) For each $0 \leq p \leq \iota - 1$, the function $u_{\mathcal{E}_{S_{d_p}}}(t,z,\epsilon)$ (resp.
$y_{\mathcal{E}_{S_{d_p}}}(t,z,\epsilon)$) can be split into three pieces
$$ u_{\mathcal{E}_{S_{d_p}}}(t,z,\epsilon) = a(t,z,\epsilon) + u_{\mathcal{E}_{S_{d_p}}}^{1}(t,z,\epsilon) +
u_{\mathcal{E}_{S_{d_p}}}^{2}(t,z,\epsilon) $$
(resp.
$$ y_{\mathcal{E}_{S_{d_p}}}(t,z,\epsilon) = b(t,z,\epsilon) + y_{\mathcal{E}_{S_{d_p}}}^{1}(t,z,\epsilon) +
y_{\mathcal{E}_{S_{d_p}}}^{2}(t,z,\epsilon) ) $$
where $u_{\mathcal{E}_{S_{d_p}}}^{1}(t,z,\epsilon)$ (resp. $y_{\mathcal{E}_{S_{d_p}}}^{1}(t,z,\epsilon)$) is bounded holomorphic
on $(\mathcal{T} \cap D(0,r_{\mathcal{T}})) \times D(0,\delta \delta_{1}) \times \mathcal{E}_{S_{d_p}}$ and has
$\hat{u}^{1}(t,z,\epsilon)$ (resp. $\hat{y}^{1}(t,z,\epsilon)$) as $1-$Gevrey asymptotic expansion on $\mathcal{E}_{S_{d_p}}$
and $u_{\mathcal{E}_{S_{d_p}}}^{2}(t,z,\epsilon)$ (resp.
$y_{\mathcal{E}_{S_{d_p}}}^{2}(t,z,\epsilon)$) is bounded holomorphic
on $(\mathcal{T} \cap D(0,r_{\mathcal{T}})) \times D(0,\delta \delta_{1}) \times \mathcal{E}_{S_{d_p}}$ and possesses
$\hat{u}^{2}(t,z,\epsilon)$ (resp. $\hat{y}^{2}(t,z,\epsilon)$) as $1^{+}-$Gevrey asymptotic expansion on $\mathcal{E}_{S_{d_p}}$.\medskip

Furthermore, the functions $u_{\mathcal{E}_{HJ_n}^{-n}}^{2}(t,z,\epsilon)$ (resp. $y_{\mathcal{E}_{HJ_n}^{-n}}^{2}(t,z,\epsilon)$),
$u_{\mathcal{E}_{HJ_n}^{n}}^{2}(t,z,\epsilon)$ (resp. $y_{\mathcal{E}_{HJ_n}^{n}}^{2}(t,z,\epsilon)$)
and all $u_{\mathcal{E}_{S_{d_h}}}^{2}(t,z,\epsilon)$ (resp. $y_{\mathcal{E}_{S_{d_h}}}^{2}(t,z,\epsilon)$) for $0 \leq h \leq \iota-1$, are
the restrictions of a common holomorphic function $u^{2}(t,z,\epsilon)$ (resp. $y^{2}(t,z,\epsilon)$)
defined on the large domain $(\mathcal{T} \cap D(0,r_{\mathcal{T}})) \times D(0,\delta \delta_{1}) \times \mathcal{E}_{HS}$,
where $\mathcal{E}_{HS} = \mathcal{E}_{HJ_n}^{-n} \cup_{h=0}^{\iota - 1} \mathcal{E}_{S_{d_h}} \cup \mathcal{E}_{HJ_n}^{n}$
which represents the $1^{+}-$sum of $\hat{u}^{2}(t,z,\epsilon)$ (resp. $\hat{y}^{2}(t,z,\epsilon)$) on $\mathcal{E}_{HS}$ w.r.t $\epsilon$.
Beside, $u_{\mathcal{E}_{S_{d_p}}}^{1}(t,z,\epsilon)$ (resp. $y_{\mathcal{E}_{S_{d_p}}}^{1}(t,z,\epsilon)$) is the
$1-$sum of $\hat{u}^{1}(t,z,\epsilon)$ (resp. $\hat{y}^{1}(t,z,\epsilon)$) on each $\mathcal{E}_{S_{d_p}}$ w.r.t $\epsilon$ whenever its aperture is strictly larger than $\pi$.
\end{theo}
\begin{proof}

For all $k \in \llbracket -n,n \rrbracket$, we set forth a holomorphic function $G_{k}$ described as
$G_{k}(\epsilon) := (t,z) \mapsto u_{\mathcal{E}_{HJ_n}^{k}}(t,z,\epsilon)$
(resp. $G_{k}(\epsilon) := (t,z) \mapsto y_{\mathcal{E}_{HJ_n}^{k}}(t,z,\epsilon)$) which defines, by construction, a bounded and
holomorphic function from $\mathcal{E}_{HJ_n}^{k}$ into the Banach space $\mathbb{F} =
\mathcal{O}( (\mathcal{T} \cap D(0,r_{\mathcal{T}})) \times D(0,\delta \delta_{1})$ equipped with the supremum norm. For all
$0 \leq p \leq \iota-1$, we set up a holomorphic function $\breve{G}_{p}$ given by
$\breve{G}_{p}(\epsilon) := (t,z) \mapsto u_{\mathcal{E}_{S_{d_p}}}(t,z,\epsilon)$ (resp.
$\breve{G}_{p}(\epsilon) := (t,z) \mapsto y_{\mathcal{E}_{S_{d_p}}}(t,z,\epsilon)$) which yields a bounded holomorphic function
from $\mathcal{E}_{S_{d_p}}$ into $\mathbb{F}$. We deduce that the assumption 1) of Proposition 23 is satisfied.

Furthermore, according to the bounds
(\ref{difference_u_Sdp_exp_small}) together with (\ref{difference_u_HJn_Sd0}) and (\ref{difference_u_HJn_Sdiota}) concerning the
functions $u_{\mathcal{E}_{S_{d_p}}}$, $0 \leq p \leq \iota - 2$ and
$u_{\mathcal{E}_{HJ_{n}}^{-n}}$, $u_{\mathcal{E}_{HJ_{n}}^{n}}$, $u_{\mathcal{E}_{S_{d_{\iota-1}}}}$ (resp. to the bounds
(\ref{exp_flat_difference_yk_plus_1_minus_yk_Sdp}) in a row with (\ref{difference_y_HJn_Sd0}) and
(\ref{difference_y_HJn_Sdiota}) dealing with the functions $y_{\mathcal{E}_{S_{d_p}}}$, $0 \leq p \leq \iota-2$ and
$y_{\mathcal{E}_{HJ_{n}}^{-n}}$, $y_{\mathcal{E}_{HJ_{n}}^{n}}$, $y_{\mathcal{E}_{S_{d_{\iota-1}}}}$), we observe that the
bounds (\ref{cond_Delta_cocycle_exp_flat}) are fulfilled for the functions
$\breve{\Delta}_{p}(\epsilon) = \breve{G}_{p+1}(\epsilon) - \breve{G}_{p}(\epsilon)$, $0 \leq p \leq \iota-2$ and
$\Delta_{-n,0}(\epsilon) = \breve{G}_{0}(\epsilon) - G_{-n}(\epsilon)$,
$\Delta_{\iota-1,n}(\epsilon) = G_{n}(\epsilon) - \breve{G}_{\iota-1}(\epsilon)$. As a result, Assumption 2) of Proposition 23 holds.

At last, keeping in mind the estimates (\ref{log_flat_difference_uk_plus_1_minus_uk_HJn}) for the maps
$u_{\mathcal{E}_{HJ_n}^{k}}$, $k \in \llbracket -n,n \rrbracket$, $k \neq n$ (resp. the estimates
(\ref{log_flat_difference_yk_plus_1_minus_yk_HJn}) for the maps $y_{\mathcal{E}_{HJ_n}^{k}}$, $k \in \llbracket -n,n \rrbracket$, $k \neq n$),
we conclude that the upper bounds (\ref{cond_Delta_cocycle_log_exp_flat}) are justified for the functions
$\Delta_{k}(\epsilon) = G_{k}(\epsilon) - G_{k+1}(\epsilon)$, $-n \leq k \leq n-1$. Hence, Assumption 3) of Proposition 23 holds true.

Accordingly, the proposition 23 gives rise to the existence of\\
- A convergent series $(t,z) \mapsto a(t,z,\epsilon) := a(\epsilon)$ (resp. $(t,z) \mapsto b(t,z,\epsilon) := a(\epsilon)$) belonging
to $\mathbb{F}\{ \epsilon \}$,\\
- Two formal series $(t,z) \mapsto \hat{u}^{j}(t,z,\epsilon) := \hat{G}^{j}(\epsilon)$
(resp. $(t,z) \mapsto \hat{y}^{j}(t,z,\epsilon) := \hat{G}^{j}(\epsilon)$) in $\mathbb{F}[[\epsilon]]$, $j=1,2$,\\
- $\mathbb{F}-$valued holomorphic functions $(t,z) \mapsto u_{\mathcal{E}_{HJ_n}^{k}}^{j}(t,z,\epsilon) := G_{k}^{j}(\epsilon)$
(resp. $(t,z) \mapsto y_{\mathcal{E}_{HJ_n}^{k}}^{j}(t,z,\epsilon) := G_{k}^{j}(\epsilon)$) on $\mathcal{E}_{HJ_n}^{k}$, for all
$k \in \llbracket -n,n \rrbracket$, $j=1,2$,\\
- $\mathbb{F}-$valued holomorphic functions $(t,z) \mapsto u_{\mathcal{E}_{S_{d_p}}}^{j}(t,z,\epsilon) := \breve{G}_{p}^{j}(\epsilon)$
(resp. $(t,z) \mapsto y_{\mathcal{E}_{S_{d_p}}}^{j}(t,z,\epsilon) := \breve{G}_{p}^{j}(\epsilon)$) on $\mathcal{E}_{S_{d_p}}$, for all
$0 \leq p \leq \iota-1$, $j=1,2$,\\
that accomplish the statement of Theorem 3.
\end{proof}

\end{document}